\documentclass[oneside, abstract=true, DIV=12]{scrartcl}

\usepackage{preamble}

\title{Equivariant algebraic models for relative~self-equivalences and block~diffeomorphisms}
\author{Alexander Berglund\thanks{Stockholm University. E-mail: \href{mailto:alexb@math.su.se}{\texttt{alexb@math.su.se}}.} \and Robin Stoll\thanks{University of Cambridge. E-mail: \href{mailto:rs2348@cam.ac.uk}{\texttt{rs2348@cam.ac.uk}}.}}
\date{\today}

\begin{document}

\maketitle

\begin{abstract}
  We construct rational models for classifying spaces of self-equivalences of bundles over simply connected finite CW-complexes relative to a given simply connected subcomplex.
  Via work of Berglund--Madsen and Krannich this specializes to rational models for classifying spaces of block diffeomorphism groups of simply connected smooth manifolds of dimension at least $6$ with simply connected boundary.
  The main application is a formula for the rational cohomology of these classifying spaces in terms of the cohomology of arithmetic groups and dg Lie algebras.
  We furthermore prove that our models are compatible with gluing constructions, and deduce that the model for block diffeomorphisms is compatible with boundary connected sums of manifolds whose boundary is a sphere.
  As in preceding work of Berglund--Zeman on spaces of self--homotopy equivalences, a key idea is to study equivariant algebraic models for nilpotent coverings of the classifying spaces.
\end{abstract}

\tableofcontents

{
\renewcommand*{\thetheorem}{\Alph{theorem}}
\renewcommand*{\thecorollary}{\Alph{theorem}}

\section{Introduction}

The block diffeomorphism group $\BlDiff[\bdry](M)$ of a smooth manifold $M$ with boundary is a central object in the pseudo-isotopy approach to the homotopy theory of diffeomorphism groups.
Specifically, there is a map
\[ \B \Diff[\bdry](M)  \longto  \B \BlDiff[\bdry](M) \]
whose homotopy fiber can, in a range, be described in terms of algebraic $K$-theory by work of Weiss--Williams \cite{WW}, Igusa \cite{Igu}, and Waldhausen--Jahren--Rognes \cite{WBR} (see also the survey \cite{Rog}).
Algebraic $K$-theory is often computable rationally, so the main obstacle to computing the rational cohomology of $\B \Diff[\bdry](M)$ in the range where this comparison applies is knowledge of the rational cohomology of $\B \BlDiff[\bdry](M)$.
In this paper we provide a model for the rational homotopy type of $\B \BlDiff[\bdry](M)$ in terms of a dg Lie algebra with an action of an arithmetic group; in particular this yields a \emph{formula} for the rational cohomology of $\B \BlDiff[\bdry](M)$.

\begin{theorem}[see \cref{thm:block_diff}] \label{thm:intro_block}
  Let $M$ be a simply connected compact smooth manifold of dimension $d \geq 6$ with simply connected (in particular non-empty) boundary.
  Assume that the rational Pontryagin classes of $\bdry M \setminus *$ are trivial.
  Then there is a rational equivalence
  \[ \B \BlDiff[\bdry](M)  \req  \hcoinv {\nerve[\big] {\lieBlTor{\bdry}(M)}} {\BlGamma{\bdry}(M)} \]
  where $\nerve {\lieBlTor{\bdry}(M)}$ denotes the geometric realization of a nilpotent dg Lie algebra $\lieBlTor{\bdry}(M)$ equipped with an algebraic action of an arithmetic subgroup $\BlGamma{\bdry}(M)$ of a reductive algebraic group $\BlRed{\bdry}(M)$.
  In fact $\nerve {\lieBlTor{\bdry}(M)}$ is equivariantly rationally equivalent to a normal covering space $\B \unip{\BlDiff[\bdry](M)}$ of $\B \BlDiff[\bdry](M)$ with deck transformation group $\BlGamma{\bdry}(M)$.
\end{theorem}

\Cref{thm:intro_block} has the following consequence for the rational cohomology of $\B \BlDiff[\bdry](M)$.
It implies, but is stronger than, collapse of the rational Serre spectral sequence of the homotopy fiber sequence
\begin{equation} \label{eq:torelli sequence}
  \B \unip{\BlDiff[\bdry](M)}  \longto  \B \BlDiff[\bdry](M)  \longto  \B \BlGamma{\bdry}(M)
\end{equation}
associated to the normal covering $\B \unip{\BlDiff[\bdry](M)}$.

\begin{corollary}[see \cref{cor:coho_block_diff}] \label{cor:intro_block_coho}
  In the situation of \cref{thm:intro_block}, there is an isomorphism of graded algebras
  \[ \Coho * \bigl( \B \BlDiff[\bdry](M); \QQ \bigr)  \iso  \Coho * \bigl( \BlGamma{\bdry}(M); \Coho * \bigl( \B \unip{\BlDiff[\bdry](M)}; \QQ \bigr) \bigr) \]
  and a $\BlGamma{\bdry}(M)$-equivariant isomorphism of graded algebras
  \[ \Coho * \bigl( \B \unip{\BlDiff[\bdry](M)}; \QQ \bigr)  \iso  \CEcoho * \bigl( \lieBlTor{\bdry}(M) \bigr) \]
  where $\CEcoho *$ denotes Chevalley--Eilenberg cohomology.
\end{corollary}

Let us now explain the terms in the statement of \cref{thm:intro_block} in more detail.
Recall that the group $\Eaut(\rat M)$ of homotopy classes of self-equivalences of the rationalization $\rat M$ can be equipped with the structure of a linear algebraic group over $\QQ$ (cf.\ \cite{Sul,Wil}); the analogous statement for the group $\Eaut[\bdry](\rat M)$ of homotopy classes of self-equivalences of $\rat M$ that restrict to the identity on $\rat{(\bdry M)}$ was proven by Espic--Saleh \cite{ES}.
The algebraic group $\BlRed{\bdry}(M)$ is defined to be the maximal reductive quotient of the algebraic subgroup $\Eaut[\bdry](\rat M)_{\pont} \subseteq \Eaut[\bdry](\rat M)$ of those self-equivalences that fix the total rational Pontryagin class $\pont(M) \in \Coho * (M; \QQ)$.
The arithmetic group $\BlGamma{\bdry}(M)$ is the image of the (block) mapping class group $\hg 0 (\BlDiff[\bdry](M))$ in $\BlRed{\bdry}(M)$.
The dg Lie algebra $\lieBlTor{\bdry}(M)$ is defined explicitly in terms of a rational model for the triple $(M, \bdry M, \bdry M \setminus *)$ and the rational Pontryagin classes of $M$; more details are given in \cref{thm:intro_main} below.
The space $\B \unip{\BlDiff[\bdry](M)}$ is the covering of $\B \BlDiff[\bdry](M)$ associated to the kernel of the surjection $\hg 0 (\BlDiff[\bdry](M)) \to \BlGamma{\bdry}(M)$; by definition this kernel is the preimage of the unipotent radical of $\Eaut[\bdry](\rat M)_{\pont}$, which is what the subscript $\mathrm{u}$ indicates.
The covering $\B \unip{\BlDiff[\bdry](M)}$ can be thought of as a variant of the block Torelli space: it is the classifying space for the group of block diffeomorphisms that act trivially on the semi-simplification (i.e.\ the direct sum of the composition factors) of the $\Eaut[\bdry](\rat M)_{\pont}$-representation $\Ho * (M, \bdry M; \QQ)$.

The first step towards proving \cref{thm:intro_block} is work of Berglund--Madsen \cite[Theorem~1.1]{BM} and Krannich \cite[Theorem~2.2]{Kra22} that provides, for a simply connected compact smooth manifold $M$ of dimension $d \ge 6$ with non-empty boundary, a rational equivalence
\begin{equation} \label{eq:intro_block_diff}
  \B \BlDiff[\bdry](M)  \req  \widetilde{\B} \bdlaut{\bdry}{\bdry_0}{\B\SO}(\sttang{M})
\end{equation}
to a finite covering of the classifying space $\B \bdlaut{\bdry}{\bdry_0}{\B\SO}(\sttang{M})$ of the topological monoid of self-equivalences of the (oriented) stable tangent bundle $\sttang{M}$ of $M$ that restrict to the identity of $\sttang{M}$ over $\bdry_0 M \defeq \bdry M \setminus *$ and whose underlying self-equivalence $M \to M$ restricts to the identity on $\bdry M$ (see \cref{prop:block_diff}).
Taking this as a starting point, we set out to construct rational models for more general classifying spaces of the form $\B \bdlaut{A}{B}{\B G}(\xi)$ for bundles $\xi$ over simply connected finite CW-complexes $X$. 
Specializing this to the case of a bundle with trivial structure group, we also obtain a model for the classifying space $\B \aut[A](X)$ of self-equivalences of $X$ that restrict to the identity on $A \subseteq X$.
Rational models for $\B \bdlaut{A}{B}{\B G}(\xi)$ and $\B \aut[A](X)$ were previously only known for certain nilpotent coverings, see \cite{Ber,BS,FFM23}.

\begin{theorem}[see \cref{cor:Baut_eq}] \label{thm:intro_main}
  Let $B \subseteq A \subseteq X$ be cofibrations of simply connected (and in particular non-empty) finite CW-complexes and let $\xi$ be a bundle over $X$ with connected structure group $G$.
  Assume that $\Coho * (\B G; \QQ)$ is of finite type and free as a graded commutative algebra and that $\xi$ is rationally trivial over $B$.
  Then there is a rational equivalence
  \[ \B \bdlaut{A}{B}{\B G}(\xi)  \req  \hcoinv {\nerve[\big] { \lie g_A^B(\xi) } } {\arithEbdlaut{A}{B}(\xi)} \]
  where $\nerve {\lie g_A^B(\xi)}$ denotes the geometric realization of a nilpotent dg Lie algebra $\lie g_A^B(\xi)$ equipped with an algebraic action of an arithmetic subgroup $\arithEbdlaut{A}{B}(\xi)$ of a reductive algebraic group $\redEbdlaut{\rat A}{\rat B}(\rat \xi)$.
  In fact $\nerve {\lie g_A^B(\xi)}$ is equivariantly rationally equivalent to a normal covering space $\B \unip{\bdlaut{A}{B}{\B G}(\xi)}$ of $\B \bdlaut{A}{B}{\B G}(\xi)$ with deck transformation group $\arithEbdlaut{A}{B}(\xi)$.
  
  The dg Lie algebra is explicitly given as
  \[ \lie g_A^B(\xi)  \defeq  \trunc 0 { \Hom \bigl( \indec[L_B](L_X), \hg {* + 1}(G) \tensor \QQ \bigr) }  \rsemidir[\tilde \rho_*] \Deru[\rho](L_X \rel L_A),\]
  where $L_B \to L_A \to L_X$ are relative dg Lie models for the inclusions $B \to A \to X$ such that $L_A \to L_X$ is minimal, and $\rho \colon L_X \to \hg {* + 1}(\B G) \tensor \QQ$ is a dg Lie model for the classifying map of $\xi$ such that $\rho(L_B) = 0$.
\end{theorem}

In fact, we prove a slightly more general statement that applies to arbitrary finite covers of $\B \bdlaut{A}{B}{\B G}(\xi)$.
Via \eqref{eq:intro_block_diff} this immediately specializes to \cref{thm:intro_block}; in particular the dg Lie algebra $\lieBlTor{\bdry}(M)$ is simply $\lie g_{\bdry}^{\bdry_0}(\sttang M)$.
Similarly the following consequence of \cref{thm:intro_main} specializes to \cref{cor:intro_block_coho}.

\begin{corollary}[see \cref{cor:cohomology}] \label{cor:intro_coho}
  In the situation of \cref{thm:intro_main}, there is an isomorphism of graded algebras
  \[ \Coho * \bigl( \B \bdlaut{A}{B}{\B G}(\xi); \QQ \bigr)  \iso  \Coho * \bigl( \arithEbdlaut{A}{B}(\xi); \Coho * \bigl( \B \unip{\bdlaut{A}{B}{\B G}(\xi)}; \QQ \bigr) \bigr) \]
  and a $\arithEbdlaut{A}{B}(\xi)$-equivariant isomorphism
  \[ \Coho * \bigl( \B \unip{\bdlaut{A}{B}{\B G}(\xi)}; \QQ \bigr)  \iso  \CEcoho * \bigl( \lie g_A^B(\xi) \bigr) \]
   of graded algebras.
\end{corollary}

In fact, this result (and hence also \cref{cor:intro_block_coho}) applies more generally to local systems of rational vector spaces on $\B \bdlaut{A}{B}{\B G}(\xi)$ that factor through $\arithEbdlaut{A}{B}(\xi)$; we obtain an isomorphism of algebras whenever this local system takes values in algebras.
\Cref{cor:intro_coho} is deduced from \cref{thm:intro_main} via a result providing a cdga model for $\B \bdlaut{A}{B}{\B G}(\xi)$, see \cref{cor:forms_eq}.

Our strategy for proving \cref{thm:intro_main} is heavily inspired by work of Berglund--Zeman \cite{BZ}, who constructed an equivariant rational model for a normal covering of the classifying space $\B \aut(X)$ of self-equivalences of $X$, together with its action by deck transformations.
Note, however, that our results do not imply theirs since we require $A$ to be simply connected and hence non-empty.
Nonetheless, we expect that combining our approaches would lead to a model for the classifying space $\B \aut(\xi)$ of unrestricted self-equivalences of a bundle $\xi$.
(We chose to not pursue this in the present paper because it does not immediately specialize to results for block diffeomorphisms.)

Let us now explain the statement of \cref{thm:intro_main} in more detail.
The notion of a minimal relative dg Lie model was introduced by Espic--Saleh \cite{ES} and specializes to the classical notion of a minimal dg Lie algebra when $L_A$ is trivial.
The term $\Deru[\rho](L_X \rel L_A)$ denotes an explicit dg Lie subalgebra of the dg Lie algebra of derivations of $L_X$ that vanish on $L_A$.
It admits an outer (i.e.\ twisted) action, which depends on $\rho$, on the non-negative truncation
\[ \trunc 0 { \Hom \bigl( \indec[L_B](L_X), \hg {* + 2}(\B G) \tensor \QQ \bigr) } \]
of the abelian dg Lie algebra of graded maps from the chain complex of indecomposables of $L_X$ relative to $L_B$ (whose homology is isomorphic to $\Ho * (X, B; \QQ)$) to the graded vector space $\hg {* + 2}(\B G) \tensor \QQ$.
The symbol $\rsemidir[\tilde \rho_*]$ denotes the twisted semi-direct product with respect to this outer action.
The dg Lie algebra $\lie g_A^B(\xi)$ is a generalization of a model of Berglund \cite{Ber} for a certain nilpotent covering of $\B \bdlaut{A}{*}{\B G}(\xi)$; in fact we rely on this work as an input.

As explained above, the group $\Eaut[\rat A](\rat X) \defeq \hg 0 ( \aut[\rat A](\rat X) )$ is an algebraic group by work of Espic--Saleh \cite{ES}; we write $\Eaut[\rat A](\rat X)_{\eqcl{\rat \xi}}$ for the algebraic subgroup of those self-equivalences $f$ such that $\rat \xi \after f \eq \rat \xi$, and denote its maximal reductive quotient by $\redEbdlaut{\rat A}{\rat B}(\rat \xi)$.
The group $\arithEbdlaut{A}{B}(\xi)$ is defined to be the image of $\hg 0 (\bdlaut{A}{B}{\B G}(\xi))$ in this quotient.
A result of Kupers \cite[Proposition~3.4]{Kup} (building on the work of Espic--Saleh) implies that $\arithEbdlaut{A}{B}(\xi)$ is an arithmetic subgroup of $\redEbdlaut{\rat A}{\rat B}(\rat \xi)$.

The group $\arithEbdlaut{A}{B}(\xi)$ acts on $\lie g_A^B(\xi)$ by conjugation via any choice of section of the composite map of algebraic groups
\begin{equation} \label{eq:intro_proj}
  \Aut[L_A](L_X)_\rho  \xlongto{\langle \blank \rangle}  \Eaut[\rat A](\rat X)_{\eqcl{\rat \xi}}  \xlongto{\pr}  \redEbdlaut{\rat A}{\rat B}(\rat \xi)
\end{equation}
where $\Aut[L_A](L_X)_\rho$ is the algebraic group of those automorphisms $f$ of $L_X$ such that $\restrict f {L_A} = \id$ and $\rho \after f = \rho$.
The key point is that the group $\redEbdlaut{\rat A}{\rat B}(\rat \xi)$ is also the maximal reductive quotient of $\Aut[L_A](L_X)_\rho$; therefore, the desired section is provided by the existence of Levi decompositions for algebraic groups in characteristic zero (a theorem due to Mostow \cite{Mos}).
The section is not unique, but any two are conjugate to each other.
Apart from the choice of this section the rational equivalence of \cref{thm:intro_main} and the isomorphism of \cref{cor:intro_coho} are canonical.

\begin{remark*}
  The condition of \cref{thm:intro_main} that $\Coho * (\B G; \QQ)$ is free as a graded commutative algebra is equivalent to $\B G$ being rationally equivalent to a product of Eilenberg--Mac~Lane spaces, which in turn is equivalent to the abelian graded Lie algebra $\hg {* + 1}(\B G) \tensor \QQ$ being a dg Lie model for $\B G$.
  This could likely be weakened to only requiring $\B G$ to be coformal, i.e.\ that the graded Lie algebra $\hg {* + 1}(\B G) \tensor \QQ$ is a dg Lie model for $\B G$, in exchange for replacing $\Hom ( \indec[L_B](L_X), \hg {* + 2}(\B G) \tensor \QQ )$ with a more complicated object, see \cref{rem:more_general_with_CE}.
  This would also allow to remove the condition of \cref{thm:intro_main} that $\restrict \xi B$ is rationally trivial and the condition of \cref{thm:intro_block} that the rational Pontryagin classes of $\bdry M \setminus *$ vanish.
\end{remark*}

\subsubsection*{Manifolds with boundary a sphere and applications}

Some of the main applications of the present work are to manifolds $M$ whose boundary is homeomorphic to a sphere.
For such $M$, the algebraic model for $\B \BlDiff[\bdry](M)$ of \cref{thm:intro_block} admits a simplification that we now will discuss.
The inclusions $\bdry_0 M \to \bdry M \to M$ are in this case modeled by the inclusions of dg Lie algebras
\[ 0 \longto  \freelie \omega_M  \longto  \freelie V_M \]
where $\freelie V_M$ denotes the free graded Lie algebra, equipped with a certain differential, on the desuspension $V_M$ of $\rHo * (M; \QQ)$, and $\omega_M \in \freelie V_M$ is a distinguished cycle dual to the intersection pairing on $V_M$.
\Cref{thm:intro_main} does not apply directly, because the inclusion $\freelie \omega_M \to \freelie V_M$ is not quasi-free (which is part of being a relative dg Lie model), but the following result proves that the conclusion remains valid.
This yields more explicit versions of \cref{thm:intro_block} and \cref{cor:intro_block_coho}.

\begin{theorem}[see \cref{thm:manifolds_block}] \label{thm:intro_boundary_sphere}
  Let $M$ be an oriented simply connected compact smooth manifold of dimension $d\geq 6$ such that $\bdry M$ is homeomorphic to $\Sphere {d-1}$.
  Then the dg Lie algebra $\lieBlTor{\bdry}(M)$ of \cref{thm:intro_block} is $\BlGamma{\bdry}(M)$-equivariantly quasi-isomorphic to the nilpotent dg Lie algebra
  \[ \lieBlTor{0}(M)  \defeq  \trunc 0 { \Hom \bigl( V_M, \hg {* + 1} (\SO) \tensor \QQ \bigr) } \rsemidir[{\pont}_*] \Deru[\pont](\freelie V_M \rel \omega_M) \]
  where $\pont \colon V_M \to \hg {* + 1} (\SO) \tensor \QQ$ denotes the map of degree $-1$ given by the Pontryagin classes $\pont[i] \colon \Ho {4i} (M; \QQ) \to \QQ \iso \hg {4i - 1}(\SO) \tensor \QQ$.
  In particular there is a rational equivalence
  \begin{equation} \label{eq:intro_rat_eq_boundary_sphere}
    \B \BlDiff[\bdry](M)  \req  \hcoinv {\nerve[\big] {\lieBlTor{0}(M)}} {\BlGamma{\bdry}(M)}
  \end{equation}
  and \cref{cor:intro_block_coho} also holds when replacing $\lieBlTor{\bdry}(M)$ by $\lieBlTor{0}(M)$.
\end{theorem}

\Cref{thm:intro_boundary_sphere} is a significant improvement of the model for block diffeomorphisms of manifolds with boundary a sphere of Berglund--Madsen \cite[Theorem 1.2]{BM}, in that the action of the deck transformation group is incorporated.
This simplifies and significantly generalizes several key steps in Berglund--Madsen's computation of the stable rational cohomology of the block diffeomorphism group of the manifold
\[ W_{g,1} = \#^g (S^n \times S^n) \setminus \openDisk{2n} \]
for $n\geq 3$. Specifically, we obtain the collapse of the spectral sequence of \eqref{eq:torelli sequence}, and in particular the injectivity of the map $\Coho * (\BlGamma{\bdry}(M);\QQ) \to \Coho * (\B \BlDiff[\bdry](M);\QQ)$, for arbitrary $M$, whereas Berglund--Madsen obtained these results only for the specific manifold $W_{g,1}$ and only stably as $g$ goes to infinity.
See Berglund--Zeman \cite[§5.2]{BZ} for a parallel discussion about self-equivalences.
There, the applications to \emph{relative} self-equivalences of $W_{g,1}$ are only sketched; in this paper we provide full details for arbitrary manifolds.
A version of \cref{thm:intro_boundary_sphere} for relative self-equivalences is given in \cref{thm:manifolds}.

Using the results of the present paper, the second author \cite{Sto} computed the stable rational cohomology of $\B \aut[\bdry](W^{k,l}_{g,1})$ for certain $k < l$, where $W^{k,l}_{g,1}$ denotes the $g$-fold connected sum of $S^k \times S^l$ with a disk removed.
In forthcoming work he will use \cref{thm:intro_block,thm:intro_boundary_sphere} to give a complete computation of the stable rational cohomology of $\B \BlDiff[\bdry](W^{k,l}_{g,1})$.
This will also lead to a computation of $\Coho * (\B \Diff[\bdry](W^{k,l}_{g,1}); \QQ)$ in degrees up to about $k + l$.
This is a significant extension of work of Ebert--Reinhold \cite{ER}, who computed the stable rational cohomology of $\B \BlDiff[\bdry](W^{n,n+1}_{g,1})$ and $\B \Diff[\bdry](W^{n,n+1}_{g,1})$ in degrees up to $n - 3$.

\subsubsection*{Compatibility with gluing constructions and forgetful maps}

The work of Stoll \cite{Sto} and Berglund--Madsen \cite{BM} (as well as the forthcoming follow-up of the former for block diffeomorphisms) mentioned in the previous subsection is mainly concerned with computing the \emph{stable} cohomology of $\B \aut[\bdry](W^{k,l}_{g,1})$ and $\B \BlDiff[\bdry](W^{k,l}_{g,1})$, where the stabilization maps are given by taking the boundary connected sum with a copy of $\Sphere k \times \Sphere l$ with an open disk removed\footnote{This is often equivalently phrased as attaching a copy of $\Sphere k \times \Sphere l$ with two open disks removed along one of its boundary components.}.
In particular it is desirable for the isomorphism obtained by combining \cref{cor:intro_block_coho,thm:intro_boundary_sphere} to be compatible with boundary connected sums.

Motivated by this, we study more generally the compatibility of the model of \cref{thm:intro_main} with gluing constructions (i.e.\ pushouts); this is also of independent interest.
More concretely, given cofibrations of simply connected finite CW-complexes
\[ C  \longto  A  \longto  B  \longto  X   \qquad \text{and} \qquad  C  \longto  A'  \longto  B'  \longto  X' \]
and bundles $\xi$ over $X$ and $\xi'$ over $X'$ that agree on $C$, there is a canonical map
\begin{equation} \label{eq:intro_gluing}
  \bdlaut{A}{B}{\B G}(\xi) \times \bdlaut{A'}{B'}{\B G}(\xi')  \longto  \bdlaut{A \cop_C A'}{B \cop_C B'}{\B G}(\xi \cop_C \xi')
\end{equation}
of topological monoids.
We prove the following.

\begin{theorem}[see \cref{cor:Baut_eq_natural}] \label{thm:intro_gluing}
  Assume that the algebraic representations of $\Eaut[\rat A](\rat X)_{\eqcl{\rat \xi}}$ on $\Ho * (X, A; \QQ)$ and of $\Eaut[\rat A'](\rat X')_{\eqcl{\rat \xi'}}$ on $\Ho * (X', A'; \QQ)$ are semi-simple.
  Then the map \eqref{eq:intro_gluing} corresponds to an explicit map of dg Lie algebras under the rational equivalence of \cref{thm:intro_main}.
\end{theorem}

The assumption that the representations on the relative homologies are semi-simple guarantees that there is an induced map $\arithEbdlaut{A}{B}(\xi) \times \arithEbdlaut{A'}{B'}(\xi') \to \arithEbdlaut{A \cop_C A'}{B \cop_C B'}(\xi \cop_C \xi')$, which is generally not true (see \cref{rem:gluing_condition}).

Still motivated by the application to boundary connected sums, we furthermore study the compatibility of the model of \cref{thm:intro_main} with restriction of the fixed subspace, i.e.\ maps of the form
\begin{equation} \label{eq:intro_forget}
  \bdlaut{A}{B}{\B G}(\xi)  \longto  \bdlaut{A'}{B'}{\B G}(\xi)
\end{equation}
where $A' \subseteq A$ and $B' \subseteq B$.
Here there is a complication: given a minimal dg Lie model $L_{A'} \to L_A \to L_X$ of the sequence $A' \to A \to X$, the composite $L_{A'} \to L_X$ is not necessarily minimal.
To apply our methods, we thus need to consider the minimal model $L_{A'} \to L_X'$ of this composite.
Due to this, we obtain neither a direct map of dg Lie algebras, nor a map in their homotopy category, corresponding to \eqref{eq:intro_forget} (see \cref{rem:forget_no_map} for further comments).\footnote{For $A$ a point it is claimed in \cite[Theorem~0.2]{FFM23} that the forgetful map $\aut[\rat B](\rat X) \to \aut[\rat A](\rat X)$ corresponds to a direct map of dg Lie algebras; there is, however, a gap in the argument.}
However, we do prove a certain compatibility of the model of \cref{thm:intro_main} with forgetful maps as in \eqref{eq:intro_forget}, see \cref{cor:Baut_eq_forget}.
Combined with \cref{thm:intro_gluing} this allows to deduce the following result (as well as a version for self-equivalences, see \cref{thm:manifolds_gluing}).

\begin{corollary}[see \cref{thm:manifolds_block_gluing}]
  Assume that the algebraic representations of $\Eaut[\bdry](\rat M)_{\pont}$ on $\Ho * (M, \bdry M; \QQ)$ and of $\Eaut[\bdry](\rat N)_{\pont}$ on $\Ho * (N, \bdry N; \QQ)$ are semi-simple.
  Then the rational equivalence \eqref{eq:intro_rat_eq_boundary_sphere} of \cref{thm:intro_boundary_sphere} is compatible with taking the boundary connected sum $M \bdryconnsum N$.
\end{corollary}

\subsubsection*{Related work}
The first models for classifying spaces of self-equivalences were obtained for the universal covering of $\B \aut(X)$ by Sullivan \cite{Sul}, in terms of derivations of a Sullivan model of $X$, and by Tanré \cite{Tan}, in terms of derivations of a dg Lie model (as in this article).
This was only recently generalized to a larger class of nilpotent coverings in work of Félix--Fuentes--Murillo \cite{FFM22}.
For the universal covering of $\B \aut(\rat X)$, Lazarev \cite{Laz} described the action of the deck transformation group on the homotopy groups.
Before that, it had already been suggested by Sullivan \cite[p.~314]{Sul} that the specific covering $\B \unip {\aut(\rat X)}$ could allow to incorporate the deck transformation action on the space level, granting access to the whole classifying space $\B \aut(\rat X)$.
This idea seems to have been largely overlooked until it was rediscovered and made precise by Berglund--Zeman \cite{BZ}, who furthermore used it to describe the space $\B \aut(X)$ and deduce the striking consequences for its rational cohomology.
Our \cref{thm:intro_main,cor:intro_coho} are counterparts of their results for $\B \aut(X)$, and we follow their strategy in that we incorporate the action of the deck transformation group into known models for covers of $\B \bdlaut{A}{B}{\B G}(\xi)$.

In the relative situation, a model for the universal covering of $\B \aut[A](X)$ in terms of derivations of a dg Lie model was obtained by Berglund--Saleh \cite{BS}.
This was recently generalized to a larger class of nilpotent coverings, again by Félix--Fuentes--Murillo \cite{FFM23} (which appeared while work on the present article was ongoing).
Their results differ from ours in that they restrict their attention to the case that $A$ and $X$ are rational spaces, and in that they do not incorporate the deck transformation action and thus do not obtain a model for the classifying space itself.
In exchange, they allow $A$ and $X$ to be merely nilpotent, instead of requiring them to be simply connected.

In the more general case of bundles, a model for the covering of $\B \bdlaut{A}{*}{\B G}(\xi)$ associated to the kernel of the map $\hg 0 (\bdlaut{A}{*}{\B G}(\xi)) \to \Eaut[A](X)$ was constructed by Berglund \cite{Ber} using the work of Berglund--Saleh.
This is the model we build upon in the present article.
It had also been used by Berglund--Madsen \cite{BM} to construct a model for a certain nilpotent covering of $\B \BlDiff[\bdry](M)$ in the case that the boundary of $M$ is a sphere.

Let us also highlight again the work of Espic--Saleh \cite{ES}, who proved the algebraicity of $\Eaut[\rat A](\rat X)$ as well as related results we fundamentally rely upon.
Lastly, work of Lindell--Saleh \cite{LS} contains an explicit formula for the comparison of the model of Berglund--Saleh with the universal covering of $\B \aut[A](X)$, which we also use as an input.

\subsubsection*{Acknowledgments}

The authors would like to thank Thomas Blom, Ronno Das, Mario Fuentes Rumí, João Lobo Fernandes, Jan McGarry Furriol, Manuel Krannich, Nils Prigge, Oscar Randal-Williams, Bashar Saleh, and Jan Steinebrunner for various useful discussions, and Sander Kupers and Samuel Muñoz-Echániz for helpful comments on an earlier version of this article.

Alexander Berglund was supported by the Swedish 
Research Council through grant no.~2021-03946, and Robin Stoll was partially supported by a postdoctoral scholarship of the Knut and Alice Wallenberg foundation.

}

\section{Preliminaries}

In this section, we fix our notation and conventions, recall classical constructions and results, and prove various basic lemmas that we will need throughout the rest of the paper.

\subsection{Basic conventions and notation}

\begin{convention}
  The base field is $\QQ$.
  All gradings are by $\ZZ$.
  Unless otherwise stated, group actions are from the left.
\end{convention}

\begin{notation}
  We denote by $\Vect$ the category of vector spaces and linear maps, and by $\GrVect$ the category of graded vector spaces and grading preserving linear maps.
  Sometimes we will implicitly consider an ungraded vector space (such as $\QQ$) as a graded vector space concentrated in degree $0$.
  Given a graded vector space $V$, we will denote by $\shift[n] V$ its shift by $n$, i.e.\ we set $(\shift[n] V)_k \defeq V_{k-n}$ (if $n = 1$, we omit it from the notation).
\end{notation}

\begin{notation}
  Let $R$ be a commutative ring.
  A \emph{chain complex} over $R$ is a graded $R$-module equipped with a differential of degree $-1$, and a \emph{cochain complex} over $R$ is a graded $R$-module equipped with a differential of degree $1$.
  We denote by $\Ch[R]$ and $\coCh[R]$ the categories of (co)chain complexes over $R$, equipped with their usual symmetric monoidal structures.
  When $R = \QQ$, we omit it from the notation.
  A \emph{(commutative) chain algebra} is a (commutative) monoid object in $\Ch$, and similarly a \emph{(commutative) cochain algebra} is a (commutative) monoid object in $\coCh$.
\end{notation}

\begin{notation}
  Given a chain complex $C$ and $n \in \ZZ$, we denote by $\trunc n C$ the chain complex given by
  \[ \trunc {n} {C}_k  \defeq  \begin{cases*} C_k, & if $k > n$ \\ \Cycles n (C), & if $k = n$ \\ 0, & if $k < n$ \end{cases*} \]
  and the restriction of the differential of $C$.
  Here $\Cycles n(C) \subseteq C_n$ denotes the $n$-cycles.
  Note that the inclusion $\trunc n C \to C$ induces an isomorphism on homology in degrees $\ge n$.
\end{notation}

\begin{notation}
  Given $n \in \NN$, we denote by $\lincat n$ the linearly ordered set $\set{0 \le \dots \le n}$ considered as a category.
  We denote by $\Simplices$ the full subcategory of the category of categories spanned by the objects $\lincat n$ with $n \in \NN$, and write $\sSet \defeq \Fun(\opcat \Simplices, \Set)$ for the category of simplicial sets.
  We also denote by $\lincat n \in \sSet$ the object represented by $\lincat n$.
  As usual, we will denote the face and degeneracy maps of a simplicial set by $d_i$ and $s_i$, respectively.
\end{notation}

\begin{notation} \label{not:components}
  Given a simplicial set $X$ and a subset $C \subseteq \hg 0 (X)$, we denote by $X_C \subseteq X$ the simplicial subset of $X$ consisting of the components in $C$.
  When $C = \set{\eqcl x}$ for some $x \in X_0$, we simply write $X_x \defeq X_C$.
\end{notation}

\begin{notation}
  We denote by $\Top$ the category of topological spaces, and by $\CGWH$ its full subcategory of compactly generated weak Hausdorff spaces.
  Furthermore we denote by $\gr \blank$ the geometric realization as a functor $\Fun(\opcat \Simplices, \CGWH) \to \CGWH$ and, by restriction, as a functor $\sSet \to \CGWH$ (see e.g.\ \cite[§11]{May72}).
  We also write $\Sing \colon \Top \to \sSet$ for the singular simplicial complex functor.
\end{notation}

Recall that the geometric realization of a simplicial set, when formed in the category of all topological spaces, is a CW-complex and hence already an object of $\CGWH$; thus it does not matter in which of these two categories we perform the construction.
Furthermore recall that the functor $\gr \blank \colon \sSet \to \CGWH$ preserves finite limits (see e.g.\ \cite[Ch.~III, §3.1]{GZ}\footnote{Gabriel--Zisman prove that the functor preserves finite limits when considered as a functor to the category of compactly generated Hausdorff spaces (called ``Kelley spaces'' there); however the inclusion of that category into $\CGWH$ preserves limits since the inclusion of either category into all compactly generated spaces does so.}).
Lastly, recall that geometric realization sends Kan fibrations to Serre fibrations (see e.g.\ \cite[Theorem~10.10]{GJ}).

\begin{notation}
  Given a category $\cat C$ equipped with a class $\cat W$ of weak equivalences, i.e.\ some class of morphisms that includes all identities and fulfills the 2-out-of-3 property, we call the localization of $\cat C$ at $\cat W$ its \emph{homotopy category}.
  Its morphisms are represented by (finite) zig-zags of morphisms of $\cat C$ where all maps that ``point backwards'' are weak equivalences.
  Given two objects $X, Y \in \cat C$, we will reserve the notation $X \to Y$ for maps of $\cat C$ and the maps they represent in the homotopy category, and use squiggly arrows $X \hto Y$ for maps in the homotopy category that are merely represented by a zig-zag.
  Given a larger class of morphisms $\cat W' \supseteq \cat W$, we will say that a map $X \hto Y$ of the homotopy category is in $\cat W'$ if there exists a zig-zag that represents it where each map is in $\cat W'$.
\end{notation}

Note that the homotopy category of a category is not necessarily locally small anymore.
However, this will not matter for our use cases in this paper.

\subsection{Simplicial model categories and relative mapping complexes}

We will assume basic familiarity with the theory of (simplicial) model categories, as presented, for example, in the textbook of Hirschhorn \cite[Chapters~7--9]{Hir}.
In this subsection, we mainly fix our notation for simplicial mapping spaces and related objects, which play a central role in this article.

\begin{notation}
  Let $X$ and $Y$ be objects of a simplicial category $\cat M$.
  We denote by $\map(X, Y)$ the simplicial mapping complex of $\cat M$.
  We will say that two maps $f, g \colon X \to Y$ are \emph{simplicially homotopic}, denoted $f \eq g$, if they lie in the same connected component of $\map(X, Y)$.
  We write $\hmap X Y \defeq \hg 0 (\map(X, Y))$ for the set of maps $X \to Y$ up to simplicial homotopy.
\end{notation}

\begin{lemma} \label{rem:simplicial_homotopy}
  Let $\cat M$ be a simplicial model category.
  If $f, g \colon X \to Y$ are simplicially homotopic maps in $\cat M$, then they are (left and right) homotopic.
  If moreover $X$ is cofibrant and $Y$ is fibrant, then the reverse implication also holds.
  
  Furthermore, if $f \colon X \to Y$ is a weak equivalence between objects that are both fibrant and cofibrant, then it has a simplicial homotopy inverse $g \colon Y \to X$, and $g$ is again a weak equivalence.
\end{lemma}

\begin{proof}
  The first two statements are \cite[Propositions~9.5.23 and 9.5.24]{Hir}.
  The last claim follows from \cite[Theorem~7.5.10 and Proposition~9.5.16]{Hir}.
\end{proof}

\begin{lemma} \label{lemma:sh_we}
  In a simplicial model category, let $f, g \colon X \to Y$ be two simplicially homotopic maps.
  Then $f$ is a weak equivalence if and only if $g$ is.
\end{lemma}

\begin{proof}
  By \cref{rem:simplicial_homotopy}, two simplicially homotopic maps are left homotopic.
  Moreover, it is easy to see from the definition that, if a map is left homotopic to a weak equivalence, then it is a weak equivalence itself.
\end{proof}

Of central importance in this paper are \emph{relative} mapping spaces, i.e.\ the mapping spaces of an undercategory, as in the following definition.

\begin{definition} \label{def:undercategory_simplicial}
  In a simplicial category $\cat M$, let $i \colon A \to X$ and $j \colon A \to Y$ be two maps.
  Then we denote by $\map[A](X, Y) \subseteq \map(X, Y)$ the simplicial subset of maps under $A$, i.e.\ the pullback
  \[
  \begin{tikzcd}
    \map[A](X, Y) \rar[hook] \dar & \map(X, Y) \dar{i^*} \\
    \set j \rar[hook] & \map(A, Y)
  \end{tikzcd}
  \]
  of simplicial sets.
  This equips the undercategory $A \comma \cat M$ with the structure of a simplicial category.
  We write $\hmap[A] X Y \defeq \hg 0 (\map[A](X, Y))$ for the set of maps $X \to Y$ under $A$ up to simplicial homotopy relative to $A$.
\end{definition}

\begin{lemma} \label{lemma:slice_sm}
  Let $\cat M$ be a simplicial model category and $A \in \cat M$ an object.
  Then the undercategory $A \comma \cat M$ with the simplicial structure of \cref{def:undercategory_simplicial} is a simplicial model category such that a map is a weak equivalence, fibration, or cofibration if and only if its underlying map is one in $\cat M$.
  In particular an object $(f \colon A \to X) \in A \comma \cat M$ is fibrant if and only if $X$ is fibrant, and cofibrant if and only if $f$ is a cofibration.
\end{lemma}

\begin{proof}
  This is \cite[Ch.~II, §2, Proposition~6]{Qui}.
  Note, however, that what is called a ``closed model category'' there is a slightly weaker notion than what we call a ``model category''; using \cite[Theorem~7.6.5]{Hir} one sees that the statement is also true for the stronger version.
\end{proof}

\begin{definition}
  Let $A$ be an object, and $i \colon A \to X$ and $A \to Y$ be two morphisms of a simplicial model category $\cat M$.
  We denote by $\mapeq[A](X, Y) \subseteq \map[A](X, Y)$ the simplicial subset of those connected components consisting of weak equivalences (this makes sense by \Cref{lemma:sh_we}), and set $\hmapeq[A] X Y \defeq \hg 0 (\mapeq[A](X, Y))$.
  We furthermore write $\aut[A](X)$ for the simplicial monoid $\mapeq[A](X, X)$ and $\Eaut[A](X)$ for the monoid $\hg 0 (\aut[A](X))$.
  When $i$ is a cofibration and $X$ is fibrant, then $\aut[A](X)$ is group-like (by \cref{rem:simplicial_homotopy}), i.e.\ $\Eaut[A](X)$ is a group.
  When $A$ is the initial object of $\cat M$, we omit it from the notation.
\end{definition}

We note, for later use, the following lemma.
It applies in particular to pullback squares of the form
\[
\begin{tikzcd}
  \aut(f) \rar \dar & \aut(Y) \dar{f^*} \\
  \aut(X) \rar{f_*} & \map(X, Y)
\end{tikzcd}
\]
where $f \colon X \to Y$ is a morphism in a simplicial model category and $\aut(f)$ is the simplicial monoid of endomorphisms of $f$ in the arrow category that are pointwise weak equivalences.

\begin{lemma} \label{lemma:monoid_pullback}
  Let $M$ and $N$ be simplicial monoids, and let $X$ be an $M$-$N$-bimodule in simplicial sets.
  Furthermore, let $x_0 \in X$ be a point, and consider the maps $f \colon M \to X$ and $g \colon N \to X$ given by acting on $x_0$.
  Then the pullback $M \times_{X} N$ has a unique simplicial monoid structure such that the projections to $M$ and $N$ are maps of simplicial monoids.
  Moreover, given a commutative diagram of simplicial sets
  \[
  \begin{tikzcd}
  T \rar{b} \dar[swap]{a} & N \dar{g} \\
  M \rar{f} & X
  \end{tikzcd}
  \]
  such that $a$ and $b$ are maps of simplicial monoids, the induced map $T \to M \times_{X} N$ is a map of simplicial monoids.
\end{lemma}

\begin{proof}
  We define the product of $(m, n), (m', n') \in (M \times_{X} N)_k$ to be $(mm', nn')$.
  That the actions of $M$ and $N$ on $X$ commute guarantees that this is again an element of $(M \times_{X} N)_k$.
  The neutral element is given by the pair consisting of the neutral elements of $M$ and $N$.
  Using this construction, all of the claims are easy to verify.
\end{proof}

\subsection{Nilpotent and virtually nilpotent spaces}

In this subsection, we recall the classical notion of a nilpotent space, the less well-known notion of a virtually nilpotent space due to Dror--Dwyer--Kan \cite{DDK}, as well as their relationships to rational homotopy theory.
We begin by recalling the notion of a nilpotent action on a group; the definition we use is taken from Bousfield--Kan \cite[Ch.~II, 4.1]{BK}.\footnote{There is (at least) one other non-equivalent definition of a nilpotent action, which is for example used by Hilton \cite{Hil75}; however one can proof that they agree when the group being acted on is nilpotent.}

\begin{definition}
  Let $K$ be a group that acts on a group $G$, i.e.\ it comes equipped with a group homomorphism $K \to \Aut[\Grp](G)$.
  We say that this action is \emph{nilpotent} if $G$ has a finite normal series of subgroups (i.e.\ each $A_i$ is normal in $A_{i+1}$)
  \[ 1 = A_0 \subseteq A_1 \subseteq \dots \subseteq A_n = G \]
  such that each $A_i$ is fixed setwise by $K$ and each quotient $\quot {A_{i+1}} {A_i}$ is abelian with trivial $K$-action.
  We say that the action is \emph{virtually nilpotent} if there exists a finite index subgroup $K' \subseteq K$ such that the action of $K'$ on $G$ is nilpotent.
\end{definition}

Recall that the class of nilpotent $G$-actions is closed under taking subgroups, quotients, and extensions, see \cite[Ch.~II, Lemma~4.2]{BK}.
Note that a group $G$ is nilpotent (i.e.\ it admits a central series of finite length) if and only if the conjugation action of $G$ on itself is nilpotent.

\begin{definition}
  Let $X$ be a path-connected topological space or a connected simplicial set, and let $x_0 \in X$.
  We say that $X$ is \emph{nilpotent} if the action of $\hg 1 (X, x_0)$ on $\hg k (X, x_0)$ is nilpotent for every $k \ge 1$.
  Similarly, we say that $X$ is \emph{virtually nilpotent} if the action is virtually nilpotent for every $k \ge 1$.
  Both of these notions are independent of the choice of basepoint $x_0 \in X$.
\end{definition}

\begin{remark}
  Note that the subgroup of $\hg 1 (X, x_0)$ that acts nilpotently on $\hg k (X, x_0)$ can vary with $k$.
  In particular being virtually nilpotent is implied by, but weaker than, admitting a finite covering that is nilpotent.
\end{remark}

We now recall that rational homology equivalences between nilpotent spaces can be characterized by a condition on the induced map on homotopy groups.
The following definition stems from work of Hilton--Mislin--Roitberg \cite[Ch.~I, §1]{HMR}, though our terminology is the one of Berglund--Zeman \cite[Definition 2.1]{BZ}.

\begin{definition} \label{def:Q-iso}
  A group homomorphism $f \colon G \to H$ is \emph{$\QQ$-injective} when every element of $\ker(f)$ has finite order.
  It is $\QQ$-surjective when, for any $x \in H$, there exists $n \ge 1$ such that $x^n \in \im(f)$.
  It is a $\QQ$-isomorphism when it is both $\QQ$-injective and $\QQ$-surjective.
\end{definition}

Note that a map $f$ of abelian groups is a $\QQ$-isomorphism if and only if $f \tensor \QQ$ is an isomorphism.

\begin{definition}
  We say that a map $f \colon X \to Y$ of topological spaces or simplicial sets is a \emph{rational homotopy equivalence} if it induces a bijection on $\hg 0$ and the induced map $\hg n (X, x) \to \hg n (Y, f(x))$ is a $\QQ$-isomorphism for all $n \ge 1$ and all basepoints $x \in X$.
  We call $f$ a \emph{rational homology equivalence} if the induced map $\Ho * (X; \QQ) \to \Ho * (Y; \QQ)$ is an isomorphism.
  We call $f$ a \emph{rational equivalence} if it is both a rational homotopy equivalence and a rational homology equivalence.
\end{definition}

Recall that, when both $X$ and $Y$ are nilpotent, then $f \colon X \to Y$ being a rational homotopy equivalence is equivalent to it being a rational homology equivalence.
This follows, for example, from \cite[Ch.~V, Proposition~3.2]{BK} combined with \cite[Ch.~V, 2.5 and 2.7]{BK} and \cite[Theorem on p.~7 and Ch.~I, Lemma~1.5]{HMR}.
The following lemma tells us that, when the domain is only virtually nilpotent, one of the implications remains valid.

\begin{lemma}[Berglund--Zeman] \label{lemma:req}
  Let $f \colon X \to Y$ be a rational homotopy equivalence of simplicial sets and assume that $X$ is virtually nilpotent and that $Y$ is nilpotent.
  Then $f$ is a rational homology equivalence.
\end{lemma}

\begin{proof}
  This is \cite[Lemma 2.3]{BZ}.
\end{proof}

\subsection{Model structures on categories of spaces}

In this subsection, we recall various classical model structures on the categories $\sSet$ and $\CGWH$ and fix related conventions and notation.

\begin{convention}
  Unless stated otherwise, we consider the category $\sSet$ of simplicial sets to be equipped with the Kan--Quillen model structure (see e.g.\ \cite[Theorem~7.10.12]{Hir}).
  Its weak equivalences are those maps that induce weak homotopy equivalences on geometric realizations, its cofibrations are the pointwise injections, and its fibrations are the Kan fibrations.
  It is a simplicial model category when equipped with its usual simplicial structure (see e.g.\ \cite[Example~9.1.13]{Hir}).
\end{convention}

\begin{notation} \label{not:localization_functor}
  Given a set $S$ of morphisms of $\sSet$, we denote by $\FR_S \colon \sSet \to \sSet$ a simplicial $S$-localization functor as constructed in \cite[Theorem~4.3.8]{Hir}.\footnote{Note that $\sSet$ fulfills the necessary conditions by \cite[Proposition~4.1.4]{Hir}.}
  It is a simplicial functor whose image is contained in the $S$-local objects, and it comes equipped with a natural transformation $\id \to \FR_S$ that is pointwise a cofibration and $S$-local equivalence.
  If $S = \emptyset$, we omit it from the notation; in that case we call $\FR$ the \emph{fibrant replacement}.
  It takes values in the fibrant objects, and the natural transformation $\id \to \FR$ is pointwise a weak equivalence.
\end{notation}

As this paper is concerned with rational homotopy theory, the rational model structure on the category of simplicial sets will play an integral role.

\begin{definition}[Bousfield]
  The \emph{rational model structure} on the category of simplicial sets is the left Bousfield localization of the Kan--Quillen model structure at the class of rational homology equivalences; it equips $\sSet$ with the structure of a simplicial model category (see e.g.\ \cite[Ch.\ X, Theorem~3.2]{GJ}).
  Its weak equivalences are the rational homology equivalences and its cofibrations are the pointwise injections.
  We will call its fibrations \emph{rational fibrations} and its fibrant objects \emph{rational Kan complexes}.
  We furthermore call the corresponding simplicial localization functor of \cref{not:localization_functor} the \emph{rationalization} and denote it by $\rat {(\blank)}$.\footnote{Note that the rational model structure on $\sSet$ is cofibrantly generated by Smith's Recognition Theorem, see \cite[Theorem~1.7 and Example~3.2]{Bek}, and hence combinatorial.
    Thus, by \cite[Proposition~A.3.7.4]{LurHTT}, it is equal to the left Bousfield localization at some (small) set $S$ of morphisms of $\sSet$.
    In particular the $S$-local equivalences are exactly the rational homology equivalences, and the $S$-local objects are exactly the rational Kan complexes.}
  It takes values in rational Kan complexes and comes equipped with a natural transformation $\id \to \rat{(\blank)}$ that is pointwise a cofibration and rational homology equivalence.
\end{definition}

Recall that a rational homology equivalence between rational Kan complexes is a weak equivalence (see e.g.\ \cite[Theorem~3.2.13]{Hir}).

Since we will simultaneously use two different model structures on $\sSet$, to avoid confusion we use different notation for spaces of rational homology equivalences.

\begin{definition}
  Let $A$ be a simplicial set, and $A \to X$ and $A \to Y$ be two maps of simplicial sets.
  We denote by $\mapreq[A](X, Y) \subseteq \map[A](X, Y)$ the simplicial subset of those connected components consisting of rational homology equivalences.
  We write $\raut[A](X)$ for the simplicial monoid $\mapreq[A](X, X)$.
  When $A = \emptyset$, we omit it from the notation.
\end{definition}

Note that $\mapreq[A](X, Y) = \mapeq[A](X, Y)$ and $\raut[A](X) = \aut[A](X)$ when $X$ and $Y$ are rational Kan complexes.

The following lemma will turn out to be convenient.
It applies more generally to all functors $\FR_S$ as in \cref{not:localization_functor}, but we only state it for rationalization for simplicity.

\begin{lemma} \label{lemma:rat}
  The functor $\rat {(\blank)}$ preserves cofibrations.
  Furthermore, given a span $X \from A \to Y$ of cofibrations of simplicial sets, the induced map
  \[ \rat X \cop_{\rat A} \rat Y  \longto  \rat {(X \cop_A Y)} \]
  is a cofibration and a rational homology equivalence.
\end{lemma}

\begin{proof}
  We begin by proving the two claims about cofibrations at the same time; let $P \to Q$ be a cofibration of simplicial sets.
  Recall that we defined the functor $\rat {(\blank)}$ as in \cite[Proof of Theorem~4.3.8]{Hir}; it is the colimit of a sequence of functors
  \[ \id = E^0  \longto  E^1  \longto  \dots  \longto  E^\beta  \longto  \dots \]
  indexed over a certain regular cardinal.
  Note that, if a natural transformation between two directed diagrams of simplicial sets is pointwise a cofibration, then the induced map on colimits is a cofibration as well; this follows from the corresponding statement for sets and injections.
  Hence, by induction, it is enough to prove that if the statements hold for the functor $E^\beta$, then they also hold for the functor $E^{\beta + 1}$.
  At a simplicial set $K$, the latter functor is defined to be the pushout of the following span
  \[ E^\beta(K)  \longfrom   \coprod_{C \to D}  C \times \map \bigl( C, E^\beta(K) \bigr)  \longto  \coprod_{C \to D}  D \times \map \bigl( C, E^\beta(K) \bigr) \]
  where $C \to D$ runs over a certain set of cofibrations of $\sSet$.
  Fix one such cofibration and set  $M(K) \defeq \map(C, E^\beta(K))$.
  Then it is, by \cref{lemma:cofibration_pushouts} and the fact that coproducts preserve cofibrations (see e.g.\ \cite[Proposition~7.2.5]{Hir}), enough to prove that the induced maps from the pushouts of the spans
  \begin{gather*}
    C \times M(Q)  \longfrom  C \times M(P)  \longto  D \times M(P) \\
    C \times M(X \cop_A Y)  \longfrom  C \times \bigl( M(X) \cop_{M(A)} M(Y) \bigr)  \longto  D \times \bigl( M(X) \cop_{M(A)} M(Y) \bigr)
  \end{gather*}
  to $D(Q)$ and $D \times M(X \cop_A Y)$, respectively, are cofibrations.
  By \cite[Proposition~9.3.8]{Hir}, it now suffices to prove that the maps
  \begin{gather*}
    M(P)  \longto  M(Q) \\
    M(X) \cop_{M(A)} M(Y)   \longto  \map \bigl( C, E^\beta(X) \cop_{E^\beta(A)} E^\beta(Y) \bigr)  \longto  M(X \cop_A Y)
  \end{gather*}
  are cofibrations.
  For the upper and lower right-hand maps, this follows from the induction hypothesis, finishing the proof of the first statement.
  For the lower left-hand map it follows from the fact that, given a set $T$ and a span $V \from U \to W$ of injections of sets, the induced map
  \[ \Hom(T, V) \cop_{\Hom(T, U)} \Hom(T, W)  \longto  \Hom(T, V \cop_U W) \]
  is injective.
  
  Now consider the following two maps of simplicial sets
  \[ X \cop_A Y  \longto  \rat X \cop_{\rat A} \rat Y  \longto  \rat {(X \cop_A Y)} \]
  and note that the composite is a rational homology equivalence.
  The first map is also a rational homology equivalence as both of the pushouts are homotopy pushouts (since, as shown above, rationalization preserves cofibrations).
  Hence the second map is a rational homology equivalence, too.
\end{proof}

\begin{lemma} \label{lemma:compact-open}
  Let $i \colon A \to X$ be a cofibration and $j \colon A \to Y$ a map of simplicial sets such that $Y$ is a Kan complex.
  Then there is a natural weak homotopy equivalence
  \[ \gr {\map[A] (X, Y)}  \xlongto{\eq}  \Map[\gr A](\gr X, \gr Y) \]
  where $\Map[\gr A]$ denotes the space of continuous maps $f$ such that $f \after \gr i = \gr j$, equipped with the compact-open topology.
\end{lemma}

\begin{proof}
  By \cite[Theorem~2.1]{BS97} there is a commutative diagram of simplicial sets
  \[
  \begin{tikzcd}
    \map(X, Y) \rar{\eq} \dar[swap]{i^*} & \map(\gr X, \gr Y) \dar{\gr i^*} \\
    \map(A, Y) \rar{\eq} & \map(\gr A, \gr Y)
  \end{tikzcd}
  \]
  such that the horizontal maps are weak equivalences.\footnote{Note that they are implicitly using that $Y$ is a Kan complex to deduce that the canonical morphism $\map(X, Y) \to \map(X, \Sing(\gr Y))$ is a weak equivalence, even though this is not explicitly stated as an assumption.}
  Since $i$ and $\gr i$ are cofibrations and $Y$ is fibrant, the induced map on fibers is a weak equivalence as well.
  Taking geometric realizations and using the weak homotopy equivalence
  \[ \gr {\map[\gr A](\gr X, \gr Y)}  =  \gr[\big] { \Sing \bigl( \Map[\gr A](\gr X, \gr Y) \bigr){} }  \xlongto{\eq}  \Map[\gr A](\gr X, \gr Y) \]
  completes the proof.
\end{proof}

We will also have use, particularly in \cref{sec:manifolds}, for a simplicial model category of topological spaces, as well as a comparison of its mapping spaces with those of $\sSet$.

\begin{convention}
  We consider the category $\CGWH$ to be equipped with the Quillen model structure, i.e.\ the weak equivalences are the weak homotopy equivalences and the fibrations are the Serre fibrations.
  When $\CGWH$ is equipped with its usual structure of a simplicial category, this makes it a simplicial model category (see e.g.\ \cite[Example~9.1.15]{Hir}).
\end{convention}

\begin{lemma} \label{lemma:Sing_map}
  Let $i \colon A \to X$ and $j \colon A \to Y$ be two maps of $\CGWH$ such that $i$ is a cofibration and $A$ is cofibrant.
  Then the canonical map
  \[ \sigma \colon \map[A](X, Y)  \xlongto{\eq}  \map[\Sing(A)] \bigl( \Sing(X), \Sing(Y) \bigr) \]
  is a weak equivalence of simplicial sets.
  When $i = j$, it restricts to a weak equivalence
  \[ \aut[A](X)  \xlongto{\eq}  \aut[\Sing(A)] \bigl( \Sing(X) \bigr) \]
  of simplicial monoids.
\end{lemma}

\begin{proof}
  The map $\sigma$ is given, at a simplex $f \colon X \times \Simplex n \to Y$ of $\map[A](X, Y)$ by the composite map
  \[ \Sing(X) \times \lincat n  \longto  \Sing(X) \times \Sing(\Simplex n)  \iso  \Sing(X \times \Simplex n)  \xlongto{f_*}  \Sing (Y) \]
  of simplicial sets.
  We first treat the case that $A = \emptyset$ and consider the commutative diagram
  \[
  \begin{tikzcd}
    \map(X, Y) \rar{\sigma} \dar[swap]{\epsilon^*} & \map \bigl( \Sing(X), \Sing(Y) \bigr) \dar{\alpha} \\
    \map \bigl( \gr {\Sing(X)}, Y \bigr) & \lar[swap]{\epsilon_*} \map \bigl( \gr {\Sing(X)}, \gr {\Sing(Y)} \bigr)
  \end{tikzcd}
  \]
  where $\epsilon$ denotes the counit of the adjunction $\gr \blank \dashv \Sing$ and $\alpha$ is given by geometric realization and the canonical homeomorphism $\gr {\lincat n} \iso \Simplex n$.
  Since $\epsilon$ is a pointwise weak equivalence and both $X$ and $\gr{\Sing (X)}$ are cofibrant, both $\epsilon_*$ and $\epsilon^*$ are weak equivalences.
  The map $\alpha$ is a weak equivalence by \cite[Theorem~2.1]{BS97}, and hence so is $\sigma$.
  
  The general case, with $A$ not necessarily empty, follows by taking horizontal fibers in the commutative diagram
  \[
  \begin{tikzcd}
    \map(X, Y) \rar \dar[swap]{\sigma} & \map(A, Y) \dar{\sigma} \\
    \map \bigl( \Sing(X), \Sing(Y) \bigr) \rar & \map \bigl( \Sing(A), \Sing(Y) \bigr)
  \end{tikzcd}
  \]
  and using that $\Sing(\blank)$ preserves cofibrations and takes values in fibrant objects.
  That $\sigma$ restricts to a weak equivalence of simplicial monoids of self-equivalences follows from the fact that a map $f \colon X \to X$ of simplicial sets is a weak equivalence if and only if $\Sing(f)$ is.
\end{proof}

\subsection{Classifying spaces of simplicial monoids}

In this subsection, we recall the two-sided bar construction.
Unfortunately, the authors are not aware of a good reference in the simplicial setting.
Thus, we also state some of its basic properties and prove them by comparison to the topological version, which has been studied in detail by May \cite[§§7--8]{May}.

\begin{notation}
  We write $\sr \blank$ for the \emph{simplicial realization} functor from bisimplicial sets to simplicial sets, i.e.\ the functor given by precomposing with the diagonal of $\Simplices$.
\end{notation}

\begin{definition} \label{def:bar}
  Let $\cat C$ be a monoidal category and $(M, G, N)$ a triple where $G$ is a monoid object in $\cat C$, $M$ a right $G$-module, and $N$ a left $G$-module.
  We define a simplicial object $\B(M, G, N)_\bullet$ in $\cat C$ by $\B(M, G, N)_n \defeq M \tensor G^{n} \tensor N$ with face and degeneracy maps given by the formulas (abusively written with elements)
  \begin{align*}
    d_i(x_0, x_1, \dots, x_n, x_{n+1}) &\defeq (x_0, \dots, x_{i-1}, x_i x_{i+1}, x_{i+2}, \dots, x_{n+1}) \\
    s_i(x_0, x_1, \dots, x_n, x_{n+1}) &\defeq (x_0, \dots, x_i, 1, x_{i+1}, \dots, x_{n+1})
  \end{align*}
  where $x_0 \in M$, $x_{n+1} \in N$, and $x_j \in G$ for $1 \le j \le n$.
  The construction is functorial in maps of triples $(a, f, b)$ where $f$ is a map of monoid objects and $a$ and $b$ are module homomorphisms with respect to $f$.
  When $\cat C = \sSet$ or $\cat C = \CGWH$, we respectively write
  \[ \B(M, G, N)  \defeq  \sr { \B(M, G, N)_\bullet }  \qquad \text{and} \qquad  \B(M, G, N)  \defeq  \gr { \B(M, G, N)_\bullet } \]
  for the \emph{two-sided bar construction}.
\end{definition}

\begin{lemma} \label{lemma:bar_products}
  For $\cat C$ the category of simplicial sets, the functor $\B(\blank, \blank, \blank)$ preserves products.
\end{lemma}

\begin{proof}
  This follows from the fact that $\sr \blank$ preserves products.
\end{proof}

\begin{lemma} \label{lemma:bar_top}
  Let $(M, G, N)$ be a triple as in \cref{def:bar} in the category of simplicial sets.
  Then there is a natural isomorphism
  \[ \gr {\B(M, G, N)}  \xlongto{\iso}  \B \bigl( \gr M, \gr G, \gr N \bigr) \]
  of functors to $\CGWH$.
\end{lemma}

\begin{proof}
  First note that $\gr M$ and $\gr N$ are modules over $\gr G$ in $\CGWH$ since geometric realization preserves finite products.
  Then the claim follows from, again, the fact that $\gr \blank$ preserves finite products and the fact that the diagram
  \[
  \begin{tikzcd}
  \Fun(\Simplices, \sSet) \dar[swap]{\sr \blank} \rar{{\gr \blank}_*} & \Fun(\Simplices, \CGWH) \dar{\gr \blank} \\
  \sSet \rar{\gr \blank} & \CGWH
  \end{tikzcd}
  \]
  commutes up to natural isomorphism by \cite[Theorem~18.9.11]{Hir}.
\end{proof}

\begin{lemma} \label{lemma:bar_eq}
  Let $(a, f, b) \colon (M, G, N) \to (M', G', N')$ be a map of triples as in \cref{def:bar} in the category of simplicial sets.
  \begin{enumerate}
    \item
    If each of $a$, $f$, and $b$ is a weak equivalence, then so is $\B(a, f, b)$.
    
    \item
    Let $A$ be an abelian group.
    If each of $a$, $f$, and $b$ induces an isomorphism on homology with $A$-coefficients, then so does $\B(a, f, b)$.
  \end{enumerate}
\end{lemma}

\begin{proof}
  The first statement follows from \cref{lemma:bar_top} and \cite[Proposition~7.3, (ii)]{May}, since the geometric realization of a simplicial set is a CW-complex and thus $\B(\gr M, \gr G, \gr N)$ and $\B(\gr {M'}, \gr {G'}, \gr {N'})$ have the homotopy type of a CW-complex by \cite[Proposition~7.2]{May}.
  The second statement follows from \cite[Theorem~11.14]{May72} in the same way as the case $A = \ZZ$ in \cite[Proposition~7.3, (i)]{May}.
\end{proof}

\begin{definition}
  Let $G$ be a simplicial monoid.
  Then we define the \emph{classifying space} of $G$ to be the simplicial set $\B G \defeq \B(*, G, *)$.
\end{definition}

\begin{definition}
  Let $G$ be a simplicial monoid and $N$ a left $G$-module in simplicial sets.
  We define the simplicial set
  \[ \hcoinv N G  \defeq  \B(*, G, N) \]
  and call it the \emph{homotopy orbits} of $N$ (and analogously when $N$ is a right $G$-module).
  We write $\pr \colon \hcoinv N G \to \B G$ for the map induced by the constant map $N \to *$.
\end{definition}

\begin{lemma} \label{lemma:bar_fiber_sequence}
  Let $G$ be a simplicial monoid that is group-like (i.e.\ the monoid $\hg 0 (G)$ is a group), and let $M$ and $N$ be a right and a left $G$-module in simplicial sets.
  Then the following map, induced by the constant map $N \to *$, is a quasi-fibration
  \[ \pr  \colon  \B(M, G, N)  \longto  \B (M, G, *) \]
  i.e.\ the following is a fiber sequence
  \[ N  \xlongto{i}  \B(M, G, N)  \xlongto{\pr}  \B (M, G, *) \]
  where $i$ is induced by the inclusion $1 \to G$.
\end{lemma}

\begin{proof}
  This follows from \cite[Theorem~7.6]{May} and \cref{lemma:bar_top}.
\end{proof}

When $H$ and $G$ are group-like, the homotopy orbits $\hcoinv M G$ can be recovered from the right $H$-module $\B(M, G, H)$ by taking homotopy orbits, as in the following lemma.

\begin{lemma} \label{lemma:hquot_hcoinv}
  Let $f \colon G \to H$ be a map of group-like simplicial monoids and $M$ a right $G$-module.
  Then the following map of simplicial sets, induced by the constant map $H \to 1$,
  \[ \hcoinv {\B(M, G, H)} {H}  =  \B \bigl( \B(M, G, H), H, * \bigr)  \xlongto{\eq}  \B \bigl( \B(M, G, 1), 1, * \bigr)  \iso  \hcoinv M G \]
  is a weak equivalence.
  It is a monoidal natural transformation of strong symmetric monoidal functors from the category of pairs $(f, M)$, equipped with the cartesian monoidal structure, to simplicial sets.
\end{lemma}

\begin{proof}
  For topological monoids, it is shown in \cite[Remarks~8.9]{May} that the map is a weak equivalence when $M = *$; the same argument applies for general $M$.
  The simplicial version then follows from \cref{lemma:bar_top}.
  The monoidality is an elementary verification.
\end{proof}

The following lemma identifies a part of the homotopy action of the fundamental group of $\hcoinv M G$ on $M$ associated to the fiber sequence $M \to \hcoinv M G \to \B G$.

\begin{lemma} \label{lemma:bar_fiber_sequence_action}
  Let $G$ be a group-like simplicial monoid and $M$ a $G$-module in simplicial sets.
  Then there is a commutative diagram of topological spaces, functorial in the pair $(G, M)$,
  \[
  \begin{tikzcd}
  \gr M  \rar{\gr i} \dar{\eq}[swap]{\phi} & \gr {\hcoinv M G} \rar{\gr \pr} \dar{\eq}[swap]{\psi} & \gr {\B G} \dar[equal] \\
  F \rar & E \rar[two heads]{q} & \gr {\B G}
  \end{tikzcd}
  \]
  such that $q$ is a Serre fibration, $F$ is the fiber of $q$ and a CW-complex, and the vertical maps are homotopy equivalences.
  Furthermore it fulfills, for any basepoint $m_0 \in M_0$, that $\phi \after \gr{\rho_g}$ and $\lambda_g \after \phi$ are homotopic relative to $m_0$ for all $g \in G_0$ such that $g \act m_0 = m_0$.
  Here $\rho_g$ denotes the action of $g$ on $M$ and $\lambda_g$ denotes the pointed homotopy action on $F$ of the image of $\eqcl {(s_0(g), s_0(m_0))} \in \hg 1 (\hcoinv M G, m_0)$ under $\psi$ (as defined in e.g.\ \cite[p.~20]{MP}\footnote{\label{foot:holonomy}They only define this action for a Hurewicz fibration; however, the same argument works for a Serre fibration whose fiber is a CW-complex.}).
\end{lemma}

\begin{proof}
  Let $g \in G_0$, set $a \defeq (s_0(g), s_0(m_0)) \in (\hcoinv M G)_1$, and consider the outer commutative square
  \[
  \begin{tikzcd}
  M \cop_{m_0} \lincat 1 \dar[swap]{{\inc[0]} \cop \inc[m_0]} \ar{rr}{i \cop a} & & \hcoinv M G \dar{\pr} \\
  M \times \lincat 1 \rar{\pr} \ar[dashed]{urr}{h} & \lincat 1 \rar{g} & \B G
  \end{tikzcd}
  \]
  of simplicial sets.
  Note that there is a unique dashed lift $h$ such that both triangles commute; in degree $1$ it fulfills $h(m, \id[\lincat 1]) = (s_0(g), m)$, and the composite $h \after \inc[1] \colon M \to \hcoinv M G$ is equal to $i \after \rho_g$.
  
  Now we choose a factorization of the map $\pr \colon \hcoinv M G \to \B G$ of simplicial sets
  \[
  \begin{tikzcd}[column sep = 1.5em]
  \hcoinv M G \rar[tail]{\psi'} & E' \rar[two heads]{q'} & \B G
  \end{tikzcd}
  \]
  as a trivial cofibration $\psi'$ followed by a fibration $q'$.
  Let $j' \colon F' \to E'$ denote the fiber of $q'$ and note that the induced map $\phi' \colon M \to F'$ is a cofibration.
  Hence the left-hand vertical map in the following commutative square is a trivial cofibration
  \[
  \begin{tikzcd}
  F' \times \set 0  \cop_{M \times \set 0} M \times \lincat 1 \ar{rr}{j' \cop (\psi' \after h)} \dar[tail]{\eq}[swap]{\inc} & & E' \dar[two heads]{q'} \\
  F' \times \lincat 1 \rar{\pr} \ar[dashed]{urr}{t} & \lincat 1 \rar{g} & \B G
  \end{tikzcd}
  \]
  and thus we can find a dashed lift $t$ as indicated such that both triangles commute.
  Setting $\lambda_g = \restrict t {F' \times \set 1}$, we obtain that $\lambda_g \after \phi' = \phi' \after \rho_g$.
  
  We set $q \defeq \gr {q'}$, $\psi \defeq \gr {\psi'}$, and $\phi \defeq \gr {\phi'}$.
  Note that $q$ is a Serre fibration, and that $\gr {\lambda_g}$ is in the pointed homotopy class of the action of $\psi_* \eqcl {a}$ on the fiber $F \defeq \gr {F'}$.
  That $\phi$ is a weak equivalence follows from \cref{lemma:bar_fiber_sequence}.
  This finishes the proof.
\end{proof}

The following lemma provides, under certain conditions on a $G$-module $M$, a description of $\hcoinv M G$ as the classifying space of a simplicial submonoid of $M$.

\begin{lemma} \label{lemma:bar_stabilizer}
  Let $G$ be a group-like simplicial monoid, $N$ and $M$ a right and a left $G$-module in simplicial sets, and $m \in M_0$ a $0$-simplex.
  Assume that the map $\rho_m \colon G \to M$, given by acting on $m$, is a fibration of simplicial sets.
  Furthermore denote by $\Stab G (m) \subseteq G$ the stabilizer submonoid, i.e.\ the pullback of $\rho_m$ along the inclusion $\set m \to M$.
  Then the following map, induced by the inclusions, is a weak equivalence
  \begin{equation} \label{eq:bar_stabilizer}
    \B \bigl( N, \Stab G (m), \set m \bigr)  \longto  \B ( N, G, M )_{\eqcl m}
  \end{equation}
  where the subscript $\eqcl m$ denotes the connected components containing an element of the form $(n, m)$ for any $n \in N_0$.
\end{lemma}

\begin{proof}
  By definition, the square
  \[
  \begin{tikzcd}
  	\Stab G (m) \rar \dar & \set{m} \dar[hook] \\
  	G \rar{\rho_m} & M
  \end{tikzcd}
  \]
  is a pullback square; it remains a pullback square after applying $\gr \blank$.
  The latter is also a homotopy pullback square since $\rho_m$ being a Kan fibration implies that $\gr {\rho_m}$ is a Serre fibration.
  By \cite[Lemma~4.10]{BM} and \cref{lemma:bar_top}, this implies that, in the case $N = *$, the map \eqref{eq:bar_stabilizer} is a weak equivalence to the connected component containing $m$.
  
  When $N$ is non-trivial, we consider the commutative diagram
  \[
  \begin{tikzcd}
  	N \rar \dar[swap]{\id} & \B \bigl( N, \Stab G (m), \set m \bigr) \dar \rar & \B \bigl( *, \Stab G (m), \set m \bigr) \dar \\
  	N \rar & \B ( N, G, M )_{\eqcl m} \rar & \B(*, G, M)_{\eqcl m}
  \end{tikzcd}
  \]
  whose rows are fiber sequences by \cref{lemma:bar_fiber_sequence}.
  Since the two base spaces are connected and the map between them is a weak equivalence by the first paragraph, this finishes the proof.
\end{proof}

The following lemma shows that a homology equivalence of $G$-modules induces, on homotopy orbits, an isomorphism on homology with respect to all local coefficient systems pulled back from $\B G$.

\begin{lemma} \label{lemma:hcoinv_coho}
  Let $G$ be a simplicial monoid, $f \colon M \to N$ a map of $G$-modules in simplicial sets, $R$ a commutative ring, and $A$ a local coefficient system of $R$-modules on $\B G$.
  If $f$ induces an isomorphism on cohomology with coefficients in $R$, then the induced maps
  \[ \Ho * (\hcoinv M G; {\pr}^* A)  \longto  \Ho * (\hcoinv N G; {\pr}^* A)  \qquad \text{and} \qquad  \Coho * (\hcoinv N G; {\pr}^* A)  \longto  \Coho * (\hcoinv M G; {\pr}^* A) \]
  are isomorphisms of $R$-modules.
\end{lemma}

\begin{proof}
  There is a commutative diagram
  \[
  \begin{tikzcd}
  \gr M  \rar{\gr {i_M}} \dar[swap]{\gr f} & \gr {\hcoinv M G} \rar{\gr \pr} \dar & \gr {\B G} \dar[equal] \\
  \gr N  \rar{\gr {i_N}} & \gr {\hcoinv N G} \rar{\gr \pr} & \gr {\B G}
  \end{tikzcd}
  \]
  whose rows are fiber sequences by \cref{lemma:bar_fiber_sequence}.
  Applying the Serre spectral sequence for computing the cohomology of the total space with local coefficients (see e.g.\ \cite[Theorem~2.6]{Bro94}\footnote{The same comment as in \cref{foot:holonomy} applies.}) we obtain a map
  \[
  \begin{tikzcd}
  E_2^{p, q} = \Coho * \bigl( \B G; \Coho * ( N; i_N^* {\pr}^* A ) \bigr) \dar \rar[Rightarrow] & \Coho * ( \hcoinv N G; {\pr}^* A ) \dar \\
  E_2^{p, q} = \Coho * \bigl( \B G; \Coho * ( M; i_M^* {\pr}^* A ) \bigr) \rar[Rightarrow] & \Coho * ( \hcoinv M G; {\pr}^* A )
  \end{tikzcd}
  \]
  of converging first-quadrant spectral sequences.
  Since ${\pr} \after i_N$ is constant, the local coefficient system on $N$ is trivial, and hence the map is an isomorphism on the $E_2$-page.
  Thus it is also an isomorphism on the abutment, which proves the statement for cohomology.
  The proof for homology is analogous.
\end{proof}

\subsection{Block diffeomorphisms and self-equivalences of bundles} \label{sec:block_prelim}

In this subsection, we deduce from work of Krannich \cite[§2]{Kra22} that in dimension at least $6$ the classifying space of block diffeomorphisms of a manifold (relative to the non-empty boundary) is rationally equivalent to a classifying space of self-equivalences of its stable tangent bundle.\footnote{The authors would like to thank Manuel Krannich for pointing out to them how to deduce this statement from his results.}
This generalizes a result of Berglund--Madsen \cite[Theorem~1.1]{BM} for the case that the boundary is a sphere.
We begin by introducing the notation we will use for self-equivalences of bundles.

\begin{definition}
	Let $B \to A \to X$ and $\xi \colon X \to K$ be maps in a simplicial model category.
	We write $\Eaut[A](X)_{\eqcl{\xi}_B} \subseteq \Eaut[A](X)$ for the stabilizer of the element $\xi \in \hmap[B]{X}{K}$.
	When $B = *$, we omit it from the notation.
\end{definition}

\begin{definition} \label{def:bdlaut}
	Let $B \to A \to X$ and $\xi \colon X \to K$ be maps in a simplicial model category, and $G \subseteq \Eaut[A](X)_{\eqcl{\xi}_B}$ a submonoid.
	We write
	\[ \B \bdlaut{A}{B}{K}(\xi)_G  \defeq  \B \bigl( \map[B](X, K)_\xi, \aut[A](X)_G, * \bigr) \]
  where we recall that $\aut[A](X)_G$ denotes the components of $\aut[A](X)$ belonging to $G$ (cf.\ \cref{not:components}).
	We omit $G$ from the notation if it is all of $\Eaut[A](X)_{\eqcl{\xi}_B}$.
\end{definition}

When $K$ is the classifying space for some kind of bundles, the space $\B \bdlaut{A}{B}{K}(\xi)$ should be thought of as the classifying space of self-equivalences of the bundle classified by $\xi$ that are the identity over $B$ and whose underlying map $X \to X$ is the identity on $A$.
This is made precise in \cite[§2]{Ber}.

\begin{notation}
	Given an embedding of compact smooth manifolds $N \subseteq M$, we denote by $\B \BlDiff[N](M)$ the classifying space of the semi-simplicial group of block diffeomorphisms of $M$ that restrict to the identity on $N$, as in \cite[§1.4]{Kra22}.
\end{notation}

\begin{proposition}[Berglund--Madsen, Krannich] \label{prop:block_diff}
	Let $M$ be a simply connected compact smooth manifold of dimension $d \ge 6$ and $\Disk {d - 1} \subseteq \bdry M$ an embedded disk; furthermore write $\bdry_0 M \defeq \bdry M \setminus \openDisk {d - 1}$.
	Then, for any choice of classifying map $\sttang{M} \colon M \to \B \SO$ of the oriented stable tangent bundle of $M$, there is a rational equivalence in the homotopy category of topological spaces
	\[ \Phi \colon \B \BlDiff[\bdry](M)  \xlonghto{\req}  \gr[\big] { \B \bdlaut{\bdry}{\bdry_0}{\B \O}(\sttang{M})_G } \]
	where $G$ is the image of the map $\hg 0 (\BlDiff[\bdry](M)) \to \Eaut[\bdry](M)_{\eqcl{\sttang{M}}_{\bdry_0}}$.
  Moreover $G$ is a finite-index subgroup of $\Eaut[\bdry](M)_{\eqcl{\sttang{M}}_{\bdry_0}}$, and $\Phi$ induces an isomorphism on homology with respect to any local coefficient system in rational vector spaces pulled back from $\B G$.
\end{proposition}

\newcommand{\hAut}{\mathrm{hAut}}
\newcommand{\BlhAut}{\widetilde{\mathrm{hAut}}}
\newcommand{\Bun}{\mathrm{Bun}}

\begin{proof}
  The analogue for block \emph{homeomorphisms} of the statement that $G$ is a finite-index subgroup of $\Eaut[\bdry](M)_{\eqcl{\sttang{M}}_{\bdry_0}}$ was proven by Kupers \cite[Theorem~4.2~(2)]{Kup}; the proof of the (oriented) smooth case is analogous, replacing $\mathrm{Top}$ by $\SO$ everywhere (for closed manifolds, this is originally due to Sullivan \cite[Theorem~13.3]{Sul}).
  
	Throughout the rest of this proof, we will use the notation of \cite[§1]{Kra22}.
	In the pullback square
	\[
	\begin{tikzcd}
		\BlDiff[\bdry](M)_\bullet \rar \dar & \BlDiff[\bdry](\Disk{d - 1})_\bullet \dar \\
		\BlDiff(M)_\bullet \rar & \BlDiff(\bdry M)_\bullet
	\end{tikzcd}
	\]
	the lower horizontal map a Kan fibration (this can be seen by generalizing the argument of \cite[§2.2]{HLLR}), and hence the upper horizontal map is one as well.
	Thus there is a fiber sequence
	\[ \BlDiff[\bdry](\Disk{d - 1})_{\id}  \longto  \B \BlDiff[\bdry](M)  \xlongto{\beta}  \B \BlDiff[\bdry](M)_C \]
	where $C$ is the collection of connected components in the image of $\hg 0 (\BlDiff[\bdry](M))$.
	Note that $\hg k (\BlDiff[\bdry](\Disk{d - 1}))$ is, for each $k \ge 0$, isomorphic to the group of homotopy $(d + k)$-spheres (see e.g.\ \cite[§25.2]{KupBook}) and hence finite by work of Kervaire--Milnor \cite[Theorem~1.2]{KM}.
	Thus the subgroup $C \subseteq \hg 0 (\BlDiff[\bdry](M))$ has finite index and $\beta$ is a rational homotopy equivalence.
	Moreover $\BlDiff[\bdry](\Disk{d - 1})_{\id}$ is virtually nilpotent and thus has the rational homology of a point by \cref{lemma:req}; using the Serre spectral sequence for computing the homology of the total space with local coefficients (see e.g.\ \cite[Theorem~2.6]{Bro94}), this implies that $\beta$ induces an isomorphism on homology with coefficients in any local coefficient system with values in rational vector spaces.

	By \cite[Theorem~2.2]{Kra22} (and the fact that $C$ has finite index), the homotopy fiber of the block derivative map of \cite[§1.9]{Kra22}
	\[ \delta  \colon  \B \BlDiff[\bdry](M)_C  \longto  \B \BlhAut_{\bdry_0} (\sttang{M}, \Disk{d - 1})^\tau_D \]
	has finite homotopy groups (Krannich proves this for the non-oriented stable tangent bundle; however, since $\bdry_0 M$ is non-empty, every self-equivalence of $\sttang{M}$ relative to $\bdry_0 M$ is orientation preserving).
	Here $D$ denotes the collection of connected components in the image of $\hg 0 (\BlDiff[\bdry](M))$; in particular the homotopy fiber of $\delta$ is connected.
	Hence, as above, the map $\delta$ is a rational homotopy equivalence and it induces an isomorphism on homology with coefficients in any local coefficient system with values in rational vector spaces.
	
	Now consider the following commutative diagram of horizontal fiber sequences
	\[
	\begin{tikzcd}
		\widetilde F \rar \dar & \B \BlDiff[\bdry] (M)_C \rar \dar{\req} & \B \BlhAut_{\bdry_0} (M, \Disk{d - 1})_G \dar[equal] \\
		F \rar \dar[swap]{f} & \B \BlhAut_{\bdry_0} (\sttang{M}, \Disk{d - 1})^\tau_D \rar \dar{\iota} & \B \BlhAut_{\bdry_0} (M, \Disk{d - 1})_G \dar[equal] \\
		F' \rar & \B \BlhAut_{\bdry_0} (\sttang{M}, \Disk{d - 1})^\tau_H \rar & \B \BlhAut_{\bdry_0} (M, \Disk{d - 1})_G
	\end{tikzcd}
	\]
	where $G$ is the collection of connected components in the image of $\hg 0 (\BlDiff[\bdry](M))$, and $H$ denotes the collection of those components that map to $G$.
	By \cite[Proof of Theorem~2.2]{Kra22}, there is a weak equivalence $F' \eq \Map_{\bdry_0}(M, \B \O)_{\sttang{M}}$ and the fiber of $\widetilde F \to F'$ has finite $\hg 0$ (here we again use that $C$ has finite index).
	Hence $\hg 1 (\iota)$ is the inclusion of a finite-index subgroup, and so $\iota$ is a rational homotopy equivalence.
	Moreover, the space $F'$ can be seen to be nilpotent by combining \cite[Ch.~II, Theorems~2.2 and 2.5~(ii)]{HMR} (and \cref{lemma:compact-open}).
  Since $f$ is a finite covering, this implies that $F$ is nilpotent and hence $f$ is a rational homology equivalence.
  By the same argument as in the proof of \cref{lemma:hcoinv_coho}, this implies that $\iota$ induces an isomorphism on homology with respect to any local coefficient system in rational vector spaces pulled back from $\B G$.
	
	By \cite[Lemmas~A.3 and A.4]{Kra22} the canonical map
	\[ \hAut_{\bdry_0} (\sttang{M}, \Disk{d - 1})  \xlongto{\eq}  \BlhAut_{\bdry_0} (\sttang{M}, \Disk{d - 1})^\tau \]
	is a weak equivalence.
	Since the following two maps are Serre fibrations
	\[ \hAut_{\bdry_0} (\sttang{M}, \Disk{d - 1})  \longto  \hAut_{\bdry_0} (M, \Disk{d - 1})  \longto  \hAut_{\bdry} (\Disk{d - 1}) \]
	and the target of the composite map is contractible, the inclusion of its fiber
	\[ \hAut_{\bdry_0} (\sttang{M}; \bdry M)  \xlongto{\eq}  \hAut_{\bdry_0} (\sttang{M}, \Disk{d - 1}) \]
	is also a weak equivalence.
	Now consider the commutative diagram
	\[
	\begin{tikzcd}
		\hAut_{\bdry_0} (\sttang{M}; \bdry M)_H \rar \dar & \Bun_{\bdry_0}(\sttang{M}, \gamma_\infty) \dar \\
		\hAut_{\bdry} (M)_G \rar{(\sttang{M})_*} & \Map_{\bdry_0}(M, \B \O)_{\sttang{M}}
	\end{tikzcd}
	\]
	where $H$ and $G$ are as above and $\gamma_\infty$ denotes the universal stable vector bundle over $\B \O$.
	The diagram is a pullback square and the right-hand vertical map is a Serre fibration; hence the diagram is a homotopy pullback square as well.
	Moreover the top right-hand corner is weakly contractible by the universality of $\gamma_\infty$.
	Then, by the same argument as in \cite[Proof of Theorem~2.3]{Ber}, the two maps
	\[
	\begin{tikzcd}
		\B \bigl( \Bun_{\bdry_0}(\sttang{M}, \gamma_\infty), \hAut_{\bdry_0} (\sttang{M}; \bdry M)_H, * \bigr) \rar{\eq} \dar[swap]{\eq} & \B \hAut_{\bdry_0} (\sttang{M}; \bdry M)_H \\
		\B \bigl( \Map_{\bdry_0}(M, \B \O)_{\sttang{M}}, \hAut_{\bdry} (M)_G, * \bigr)
	\end{tikzcd}
	\]
	are weak equivalences.
	The bottom space is weakly equivalent to $\gr { \B \bdlaut{\bdry}{\bdry_0}{\B \O}(\sttang{M})_G }$ by \cref{lemma:compact-open,lemma:bar_top}, which completes the proof.
\end{proof}

\begin{remark} \label{rem:BlDiff_part_boundary}
  Krannich's result \cite[Theorem~2.2]{Kra22} in fact yields a rational description of the classifying space $\B \BlDiff[\bdry_1](M)$ for certain compact submanifolds $\bdry_1 M \subseteq \bdry M$ of codimension $0$.
  We chose to restrict ourselves to the case $\bdry_1 M = \bdry M$ for simplicity, but it seems plausible that our results (such as \cref{thm:block_diff} below) could be extended to the general case.
  A first step in this direction would be to construct rational models for spaces of self-equivalences of simply connected finite CW-complexes $X$ that restrict to a self-equivalence on a given simply connected subcomplex $C \subseteq X$.
  As far as we know, such models have not yet appeared in the literature.
\end{remark}

We furthermore note that the rational equivalence of the preceding proposition is compatible with boundary connected sums, in the following sense.

\begin{proposition} \label{prop:block_diff_gluing}
	Let $M$ and $N$ be simply connected compact smooth manifolds of dimension $d \ge 6$ and $D_M \subset \bdry M$ and $D_N \subset \bdry N$ embedded $(d-1)$-disks; write $\bdry_0 M \defeq \bdry M \setminus \interior{D}_M$ and $\bdry_0 N \defeq \bdry N \setminus \interior{D}_N$.
	Furthermore let $M \bdryconnsum N$ be the boundary connected sum of $M$ and $N$, formed along disks disjoint from $D_M$ and $D_N$, let $D_{M \bdryconnsum N} \subset \bdry(M \bdryconnsum N)$ be an embedded $(d-1)$-disk containing both $D_M$ and $D_N$, and write $\bdry_0 (M \bdryconnsum N) \defeq \bdry (M \bdryconnsum N) \setminus \interior{D}_{M \bdryconnsum N}$.
	Moreover choose a classifying map $\sttang{M \bdryconnsum N} \colon M \bdryconnsum N \to \B \SO$ of the oriented stable tangent bundle of $M \bdryconnsum N$, and let $\sttang{M}$ and $\sttang{N}$ be its restrictions to $M$ and $N$, respectively.
	Then the following diagram, with horizontal maps (products of) those of \cref{prop:block_diff}, commutes in the homotopy category of topological spaces
	\[
	\begin{tikzcd}
		\B \BlDiff[\bdry](M) \times \B \BlDiff[\bdry](N) \rar[squiggly]{\req} \dar & \gr[\big] { \B \bdlaut{\bdry}{\bdry_0}{\B \O}(\sttang{M})_{G_M} } \times \gr[\big] { \B \bdlaut{\bdry}{\bdry_0}{\B \O}(\sttang{N})_{G_N} } \dar \\
		\B \BlDiff[\bdry](M \bdryconnsum N) \rar[squiggly]{\req} & \gr[\big] { \B \bdlaut{\bdry}{\bdry_0}{\B \O}(\sttang{M \bdryconnsum N})_{G_{M \bdryconnsum N}} }
	\end{tikzcd}
	\]
	where $G_M$ is the image of the map $\hg 0 (\BlDiff[\bdry](M)) \to \Eaut[\bdry](M)$ and analogously for $G_N$ and $G_{M \bdryconnsum N}$.
\end{proposition}

\begin{proof}
	This is evident from the description of the zig-zag in the proof of \cref{prop:block_diff}.
\end{proof}

\subsection{Local systems}

In this subsection, we recall the notion of a local system on a simplicial set, following Halperin \cite[§12]{Hal} (see also \cite[§3.8.1]{BZ}).
It is more general than the usual notion of a local coefficient system defined as a functor out of the fundamental groupoid (see \cref{rem:local_system_groupoid} below).
On the way, we reformulate everything using more categorical language, and keep careful track of monoidal structures.
Our main application is to recall the definition, and prove properties, of polynomial differential forms with local coefficients; this is contained in the next subsection.

\begin{definition}
  Let $X$ be a simplicial set, and $\cat C$ a category.
  A \emph{local system on $X$ with values in $\cat C$} is a functor
  \[ F \colon \opcat{(\Simplices \comma X)}  \longto  \cat C \]
  where $\comma$ denotes the comma category and we consider $\Simplices$ as a subcategory of $\sSet$ via the Yoneda embedding.
  We denote by $\Loc[\cat C](X) \defeq \Fun(\opcat{(\Simplices \comma X)}, \cat C)$ the category of local systems on $X$ with values in $\cat C$.
  Given a map $g \colon Y \to X$ of simplicial sets, we obtain a local system $g^* F \defeq F \after \opcat{(g_*)}$ on $Y$; this yields a canonical functor $g^* \colon \Loc[\cat C](X) \to \Loc[\cat C](Y)$.
  For an object $C \in \cat C$, we denote by $\ul C$ the constant local system with value $C$.
  
  When $\cat C$ is complete, we furthermore define the \emph{global sections} of $F$ to be the object
  \[ F(X)  \defeq  \lim{} F  \in  \cat C \]
  which yields a functor $\Loc[\cat C](X) \to \cat C$.
  More generally, when $g \colon Y \to X$ is a simplicial set over $X$, we write $F(g)$ and sometimes $F(Y)$ for $(g^* F)(Y)$; this yields a canonical functor $F(\blank) \colon \opcat{(\sSet \comma X)} \to \cat C$.
  (Note that, when restricted to $\opcat{(\Simplices \comma X)}$, this functor agrees with $F$, so that the notation does not conflict.)
  
  When $\cat C$ is a monoidal category and $F$ and $G$ are local systems with values in $\cat C$ on $X$ and $Y$, respectively, then we define the \emph{external tensor product} $F \extensor G$ to be the functor
  \[ \opcat{\bigl( \Simplices \comma (X \times Y) \bigr)}  \iso  \opcat{(\Simplices \comma X)} \times \opcat{(\Simplices \comma Y)}  \xlongto{F \times G}  \cat C \times \cat C  \xlongto{\tensor}  \cat C \]
  which is a local system on $X \times Y$.
  When $X = Y$, we write $F \tensor G \defeq \Delta^*(F \extensor G)$ for the \emph{tensor product} of $F$ and $G$ (which is exactly the pointwise one); this canonically extends to a symmetric monoidal structure on $\Loc[\cat C](X)$.
\end{definition}

\begin{definition}
  Let $X$ be a simplicial set and $F$ a local system on $X$ with values in a category $\cat C$.
  \begin{itemize}
    \item
    When $\cat C$ is complete, we say that $F$ is \emph{extendable} when, for all $n \in \NN$ and all maps $\sigma \colon \lincat n \to X$, the map $(\sigma^* F)(\lincat n) \to (\sigma^* F)(\bdry \lincat n)$ is an epimorphism.
    
    \item
    We say that $F$ is \emph{groupoidal} if, for each map $\alpha$ of $\opcat{(\Simplices \comma X)}$, the map $F(\alpha)$ is an isomorphism.\footnote{This is called a ``local system of coefficients'' by Halperin.}
    
    \item
    When $\cat C$ is the category of cochain complexes, we say that $F$ is \emph{quasi-groupoidal} if $\Coho * \after F$ is groupoidal.\footnote{This is called a ``local system of differential coefficients'' by Halperin.}
  \end{itemize}
\end{definition}

\begin{remark} \label{rem:local_system_groupoid}
  By \cite[§III, Theorem~1.1]{GJ}, there is an equivalence of categories from the localization of $\Simplices \comma X$ at all maps to the fundamental groupoid of $X$.
  In particular, a groupoidal local system $F$ on $X$ is equivalently a functor $\widetilde F$ out of the fundamental groupoid of $X$ (this is what we called a ``local coefficient system'' in the preceding subsections).
\end{remark}

\begin{definition}
  Let $\cat C$ be a category.
  We denote by $\Loc[\cat C]$ the Grothendieck construction of the functor $\Loc[\cat C](\blank) \colon \opcat{\sSet} \to \Cat$.
  More explicitly, it is the category such that
  \begin{itemize}
    \item
    objects are pairs $(X, F)$ of a simplicial set $X$ and a local system $F$ on $X$ with values in $\cat C$,
    \item
    morphisms $(X, F) \to (Y, G)$ are given by pairs $(f, \phi)$ of a map of simplicial sets $f \colon Y \to X$ and a natural transformation $\phi \colon f^* F \to G$,
    \item
    and whose composition is defined in the obvious way.
  \end{itemize}
  When $\cat C$ is complete, the assignment $(X, F) \mapsto F(X)$ can be canonically extended to a functor $\Gamma \colon \Loc[\cat C] \to \cat C$.
  Moreover we denote by $\Locgrpd[\cat C] \subseteq \Loc[\cat C]$ the full subcategory spanned by those objects $(X, F)$ such that $F$ is groupoidal.
  
  When $\cat C$ is a (symmetric) monoidal category, then the assignment
  \[ \bigl( (X,F), (Y,G) \bigr)  \longmapsto  (X \times Y, F \extensor G) \]
  can be canonically extended to a (symmetric) monoidal structure on $\Loc[\cat C]$ with unit object $(\pt, \ul \unit)$, and the functor $\Gamma$ can be equipped with a lax (symmetric) monoidal structure.
  
  When $\cat C$ is the category of cochain complexes, we omit it from the notation.
\end{definition}

\begin{remark} \label{rem:Loc_monoid}
  Let $\cat C$ be a (symmetric) monoidal category, $X$ a simplicial set, and $F, G \in \Loc[\cat C](X)$.
  Then the maps
  \begin{gather*}
    (\pt, \ul \unit)  \xlongto{\const}  (X, \ul \unit) \\
    (X, F) \tensor (X, G)  =  (X \times X, F \extensor G)  \xlongto{\Diag[X]}  (X, F \tensor G)
  \end{gather*}
  equip the inclusion $\Loc[\cat C](X) \to \Loc[\cat C]$ with the structure of a lax (symmetric) monoidal functor.
  In particular, when $A$ is a (commutative) monoid in $\Loc[\cat C](X)$ (or, equivalently, when $A$ is a local system on $X$ with values in the category of (commutative) monoids in $\cat C$), then $(X, A)$ is a (commutative) monoid in $\Loc[\cat C]$.
\end{remark}

We will need the following notion of pushforward (or direct image) of a local system along a map of simplicial sets.
A special case of this is contained in \cite[19.16]{Hal}.

\begin{definition} \label{def:pushforward}
  Let $f \colon X \to Y$ be a map of simplicial sets, and $F$ a local system on $X$ with values in a complete category $\cat C$.
  Then we denote by $f_* F$ the local system on $Y$ given by the right Kan extension of $F$ along $\opcat {(f_*)} \colon \opcat{(\Simplices \comma X)} \to \opcat{(\Simplices \comma Y)}$.
  More explicitly we have, for $\sigma \colon \lincat n \to Y$, that $(f_* F)(\sigma) = F(\sigma^* X)$ where $\sigma^* X$ denotes the pullback of $f$ along $\sigma$.
  This construction canonically assembles into a functor from the pullback $\cat P_{\cat C}$ of the cospan of categories
  \[ \Fun(\lincat 1, \opcat \sSet)  \xlongto{\ev[1]}  \opcat \sSet  \longfrom  \Loc[\cat C] \]
  to $\Loc[\cat C]$ over the functor $\ev[0]$.
  
  When $\cat C$ is (symmetric) monoidal, then we equip $\cat P_{\cat C}$ with the pointwise (symmetric) monoidal structure and the functor above with the canonical structure of a lax (symmetric) monoidal functor: for maps $f \colon X \to Y$ and $f' \colon X' \to Y'$ and local systems $F$ and $F'$ over $X$ and $X'$, respectively, the structure map
  \[ \bigl( Y, f_*(F) \bigr) \tensor \bigl( Y', f'_*(F') \bigr)  =  \bigl( Y \times Y', f_*(F) \extensor f'_*(F') \bigr)  \longto \bigl( Y \times Y', (f \times f')_*(F \extensor F') \bigr) \]
  is given, at $\sigma = (\sigma_1, \sigma_2) \colon \lincat n \to Y \times Y'$, by the composite
  \[ F(\sigma_1^* X) \tensor F'(\sigma_2^* X')  \longto  (F \extensor F') \bigl( \sigma_1^* (X) \times \sigma_2^* (X') \bigr)  \longto  (F \extensor F') \bigl( \sigma^* (X \times X') \bigr) \]
  where we use that $\sigma = (\sigma_1 \times \sigma_2) \after \Diag[\lincat n]$.
\end{definition}

The following lemma is a generalization of \cite[Lemma~19.21]{Hal} with a more conceptual proof.

\begin{lemma} \label{lemma:pushforward_global_sections}
  Let $f \colon X \to Y$ be a map of simplicial sets, and $F$ a local system on $X$ with values in a complete category $\cat C$.
  Then there is a canonical natural isomorphism $\psi \colon F(X) \iso (f_* F)(Y)$ of functors $\cat P_{\cat C} \to \cat C$, where $\cat P_{\cat C}$ is as in \cref{def:pushforward}.
  When $\cat C$ is (symmetric) monoidal, then $\psi$ is a monoidal natural isomorphism between lax (symmetric) monoidal functors.
\end{lemma}

\begin{proof}
  The existence of $\psi$ follows from the fact that a limit is a right Kan extension to the terminal category, and the fact that right Kan extensions compose.
  Explicitly, the map $\psi$ is assembled from the maps $\tilde \sigma^* \colon F(X) \to F(\sigma^* X) = (f_* F)(\sigma)$, where $\sigma \colon \lincat n \to Y$ is any map and $\tilde \sigma \colon \sigma^* X \to X$ is the projection.
  Using this description one sees that $\psi$ is monoidal by unwinding the definitions.
\end{proof}

The following lemma is a generalization of part of \cite[Lemma~19.17]{Hal}.

\begin{lemma} \label{lemma:pushforward_extendable}
  Let $\cat C$ be a complete category that admits a faithful, limit-preserving functor to the category of sets.
  Furthermore, let $f \colon X \to Y$ be a map of simplicial sets, and $F$ an extendable local system on $X$ with values in $\cat C$.
  Then $f_* F$ is extendable.
\end{lemma}

\begin{proof}
  This follows from \cite[(Proof of) Proposition~12.21]{Hal}.
\end{proof}

\subsection{Polynomial forms with local coefficients} \label{sec:forms}

In this subsection, we recall the definition of the cochain complex of polynomial differential forms with coefficients in a local system, following Halperin \cite[§§13--15]{Hal} (with trivial coefficients this goes back to Sullivan \cite{Sul}).
We will also prove various related lemmas that we will need throughout this paper (generalizing results of \cite[§3.8.1]{BZ}).
For a later comparison, we begin by recalling the definition of the simplicial cochain complex from this perspective, see \cite[Definition~14.1]{Hal}.

\begin{definition}
  Let $X$ be a simplicial set and $F$ a groupoidal local system on $X$ with values in graded vector spaces.
  We write $\Cochains * (X; F)$ for the cochain complex in graded vector spaces where $\Cochains n (X; F)$ is the graded vector space with degree-$k$ part the vector space of all families $a = (a_\sigma)_{\sigma \in X_n}$ with $a_\sigma \in F(\sigma)_k$ and addition and scalar multiplication defined pointwise.
  Its differential $d$ is given by
  \[ d(a)_\sigma  \defeq  \sum_{i = 0}^{n + 1} (-1)^i \inv{d_i} (a_{\sigma \after \delta_i}) \]
  for $a \in \Cochains n (X; F)$ and $\sigma \in X_{n+1}$ (here $\delta_i \colon \lincat n \to \lincat {n+1}$ denotes the $i$-th face inclusion and $d_i$ its image under $F$).
  This assembles into a functor $\Cochains * \colon \Locgrpd[\GrVect] \to \coCh[\GrVect]$.
  
  For $p \le n$, let $\delta^{\mathrm F}_p \colon \lincat p \to \lincat n$ denote the inclusion onto the first $p + 1$ vertices, and $\delta^{\mathrm B}_p \colon \lincat p \to \lincat n$ the inclusion onto the last $p + 1$ vertices.
  Then the maps
  \begin{align*}
    \Cochains p (X; F) \tensor \Cochains q (Y; G) & \longto \Cochains {p+q} (X \times Y; F \extensor G) \\
    (a_\sigma)_{\sigma \in X_p} \tensor (b_{\sigma'})_{\sigma' \in Y_q} & \longmapsto \bigl( a_{\pr[X](\tau \after \delta^{\mathrm F}_p)} \tensor b_{\pr[Y](\tau \after \delta^{\mathrm B}_q)} \bigr)_{\tau \in (X \times Y)_{p+q}}
  \end{align*}
  and the canonical map $\QQ \to \Cochains * (\pt; \ul \QQ)$ equip $\Cochains *$ with the structure of a lax monoidal functor.
\end{definition}

\begin{remark}
  Let $F$ be a groupoidal local system on $X$ with values in graded vector spaces, and $\widetilde F$ the corresponding functor out of the fundamental groupoid of $X$ under the equivalence of \cref{rem:local_system_groupoid}.
  Then one can construct a monoidal natural isomorphism between the cochain complex $\Cochains * (X; F)$ and the usual simplicial cochain complex of $X$ with coefficients in $\widetilde F$ (as defined in e.g.\ \cite[p.~270]{Whi}).
\end{remark}

\begin{definition}
  We denote by $\sForms$ the simplicial commutative cochain algebra of polynomial differential forms on the simplices.
  More explicitly, we have
  \[ \sForms[n]  \defeq  \quot {\freegca(t_0, \dots, t_n, dt_0, \dots, dt_n)} {\textstyle (\sum_{i=0}^n t_i = 1, \sum_{i=0}^n dt_i = 0)} \]
  with $t_i$ in degree $0$ and $dt_i$ in degree $1$ and the unique differential such that $d(t_i) = dt_i$.
  A map $\sigma \colon \lincat n \to \lincat m$ induces the map $\sigma^* \colon \sForms[m] \to \sForms[n]$ uniquely determined by $\sigma^*(t_i) = \sum_{\sigma(j) = i} t_j$.
\end{definition}

\begin{definition}
  Let $X$ be a simplicial set.
  We write $\Omega^*_X$ for the composite functor
  \[ \opcat{(\Simplices \comma X)}  \longto  \opcat \Simplices  \xlongto{\sForms}  \coCh \]
  which is a local system on $X$ with values in cochain complexes.
  By \cref{rem:Loc_monoid}, the object $(X, \Omega^*_X) \in \Loc$ comes equipped with a canonical commutative monoid structure.
\end{definition}

\begin{definition}
  We denote the composite functor
  \[ \Loc  \xlongto{(\pt, \Omega^*_\pt) \tensor \blank}  \Loc  \xlongto{\Gamma}  \coCh \]
  by $\Forms *$.
  It is given by $\Forms *(X; F) = (\Omega^*_X \tensor F)(X)$.
  Since $(\pt, \Omega^*_\pt)$ is a commutative monoid object in $\Loc$, the functor $\Forms *$ can be canonically equipped with a lax symmetric monoidal structure.
  We write $\Forms * (X) \defeq \Forms * (X; \ul \QQ)$.
\end{definition}

It is a classical fact that polynomial differential forms provide a commutative cochain algebra model for the simplicial cochain complex.
This is the content of the following lemma.

\begin{lemma} \label{lemma:forms_cohomology}
  Let $X$ be a simplicial set, and let $M$ be a groupoidal local system on $X$ with values in graded vector spaces.
  Then there is a zig-zag of monoidal natural transformations that are pointwise quasi-isomorphisms
  \[ \Forms * (X; M)  \xlonghto{\eq}  \Cochains * (X; M) \]
  of lax monoidal functors $\Locgrpd[\GrVect] \to \coCh$ (here the grading on the right-hand side is by total degree).
  Thus, there is also a monoidal natural isomorphism
  \[ \Coho * \bigl( \Forms * (X; M) \bigr)  \iso  \Coho * (X; M) \]
  of lax symmetric monoidal functors $\Locgrpd[\GrVect] \to \GrVect$.
  
  In particular, a rational homology equivalence $f \colon Y \to X$ of simplicial sets induces a quasi-isomorphism $\Forms * (X) \to \Forms * (Y)$.
  If $f$ is even a weak equivalence, then it induces a quasi-isomorphism $\Forms * (X; M) \to \Forms * (Y; f^* M)$ for all $M$ as above.
\end{lemma}

\begin{proof}
  This is implicit in the proof of \cite[Theorem~14.18]{Hal}.
\end{proof}

The following lemma tells us that, under certain conditions, the functor $\Forms * (X; \blank)$ preserves quasi-isomorphisms of local systems.

\begin{lemma} \label{lemma:forms_coeff_quasi-iso}
  Let $X$ be a simplicial set, and let $f \colon F \to G$ be a pointwise quasi-isomorphism of local systems on $X$ with values in cochain complexes.
  Assume that each of $F$ and $G$ is either groupoidal or quasi-groupoidal and extendable.
  Then the induced map $\Forms * (X; F) \to \Forms *(X; G)$ is a quasi-isomorphism.
\end{lemma}

\begin{proof}
  This follows from \cite[Theorem~12.27]{Hal} since $\Forms * \tensor F$ and $\Forms * \tensor G$ are extendable by \cite[Remarks~13.11]{Hal}.
\end{proof}

The following definition and lemma provide a monoidal version of an argument contained in the proof of \cite[Theorem~3.41]{BZ}, with a more elementary proof.
This is an integral part of the argument for obtaining the cohomological results of this paper.

\begin{definition} \label{def:split_local_system}
  Let $X$ be a simplicial set and $F \in \Loc(X)$.
  We say that $F$ is \emph{split} if there exist quasi-isomorphisms $p_F \colon F \to \Coho * (F)$ and $\nabla_F \colon \Coho * (F) \to F$ in $\Loc(X)$ such that $\Coho * (p_F)$ is equal to the identity after composing with the canonical isomorphism $\Coho * (\Coho * (F)) \iso \Coho * (F)$.
  (Here we consider $\Coho * (F)$ as a local system on $X$ with trivial differential.)
\end{definition}

\begin{remark}
  Note that $\nabla_F$ can always be chosen to be a quasi-inverse to $p_F$.
  Moreover note that our definition of split is, at least a priori, weaker than the definition of Weibel \cite[Definition~1.4.1]{Wei} (see also \cite[Appendix~B]{BM}).
\end{remark}

\begin{lemma} \label{lemma:forms_split_coeff}
  Let $X$ be a simplicial set and $F$ a split, quasi-groupoidal, and extendable local system on $X$ with values in cochain complexes.
  Then there is a monoidal natural isomorphism
  \[ \Coho * \bigl( \Forms * (X; F) \bigr)  \iso  \Coho * \bigl( X; \Coho * (F) \bigr) \]
  of lax symmetric monoidal functors to $\GrVect$ from the full subcategory of $\Loc$ spanned by those objects $(X, F)$ such that $F$ is split, quasi-groupoidal, and extendable.
  (The grading on the right-hand side is by total degree.)
\end{lemma}

\begin{proof}
  Choose maps $p_F$ and $\nabla_F$ as in \cref{def:split_local_system}.
  Then $p_F$ induces an isomorphism of graded vector spaces
  \[ (p_F)_*  \colon  \Coho * \bigl( \Forms * (X; F) \bigr)  \xlongto{\iso}   \Coho * \bigl( \Forms * \bigl( X; \Coho * (F) \bigr) \bigr) \]
  by \cref{lemma:forms_coeff_quasi-iso} since $\Coho * (F)$ is groupoidal, and by the same argument $\nabla_F$ induces an isomorphism $(\nabla_F)_*$ in the other direction.
  Combining this with \cref{lemma:forms_cohomology}, we see that it is enough to prove that $(p_F)_*$ is a monoidal natural transformation.
  
  We begin by showing that $(p_F)_*$ is a natural transformation.
  This is clear for maps $f \colon Y \to X$ of simplicial sets (together with the identity $f^*F \to f^*F$), so we will restrict to the case of a map $\alpha \colon F \to G$ of local systems over a fixed $X$.
  Consider the (not necessarily commutative) diagram
  \[
  \begin{tikzcd}
  \Cocycles * (F) \rar{\inc} \dar[swap]{\Cocycles * (\alpha)} & F \rar{p_F} \dar[swap]{\alpha} & \Coho * (F) \dar{\Coho * (\alpha)} \\
  \Cocycles * (G) \rar{\inc} & G \rar{p_G} & \Coho * (G)
  \end{tikzcd}
  \]
  where $\Cocycles *$ denotes the subcomplex of cocycles.
  Note that the two horizontal composites are equal to the respective quotient map by the assumption that $p_F$ and $p_G$ induce the identity on cohomology; this implies that the outer square commutes.
  We want to prove that the right-hand square commutes after applying $\Coho *( \Forms * (X; \blank) )$; since $(\nabla_F)_*$ is an isomorphism that factors through $\Coho * ( \Forms * (X; \Cocycles * (F)))$ (because the differential of $\Coho * (F)$ is trivial), this is implied by commutativity of the outer square.
  By applying this naturality statement to the identity $F \to F$, it also follows that the isomorphism $(p_F)_*$ is independent of the choice of $p_F$.
  
  That $(p_F)_*$ is monoidal is now clear: given two pairs $(X, F)$ and $(Y, G)$ as in the statement, the two maps
  \begin{gather*}
  p_{F \extensor G}  \colon  F \extensor G  \xlongto{p_F \extensor p_G}  \Coho * (F) \extensor \Coho * (G)  \iso  \Coho * (F \extensor G) \\
  \nabla_{F \extensor G}  \colon  \Coho * (F \extensor G)  \iso  \Coho * (F) \extensor \Coho * (G)  \xlongto{\nabla_F \extensor \nabla_G}  F \extensor G
  \end{gather*}
  exhibit $F \extensor G$ as split in such a way that the diagram
  \[
  \begin{tikzcd}
    \Forms * (X; F) \tensor \Forms * (Y; G) \dar \rar{(p_F)_* \tensor (p_G)_*} &[40] \Forms * \bigl( X; \Coho * (F) \bigr) \tensor \Forms * \bigl( Y; \Coho * (G) \bigr) \dar \\
    \Forms * (X \times Y; F \extensor G) \rar{(p_F \extensor p_G)_*} \drar[bend right = 10][swap]{(p_{F \extensor G})_*} & \Forms * \bigl( X \times Y; \Coho * (F) \extensor \Coho * (G) \bigr) \dar{\iso} \\
     & \Forms * \bigl( X \times Y; \Coho * (F \extensor G) \bigr)
  \end{tikzcd}
  \]
  commutes already at the level of $\Forms *$.
  Furthermore note that $F \extensor G$ is indeed quasi-groupoidal and extendable; the latter by \cite[Theorem~12.37]{Hal} since $F \extensor G \iso \pr[X]^*(F) \tensor \pr[Y]^*(G)$ and both $\pr[X]^*(F)$ and $\pr[Y]^*(G)$ are extendable as $F$ and $G$ are.
\end{proof}

We now want to relate $\Forms * (\hcoinv X G)$ and $\Forms * (\B G; \Forms * (X))$ to each other.
When the preceding lemma applies, this also relates $\Coho * (\hcoinv X G)$ and $\Coho * (\B G; \Coho * (X))$.
In fact, we will do so for local systems coming from $G$.
We begin with an auxiliary definition.

\begin{definition}
  Let $\cat C$ be a category.
  We denote by $\LocGrp[\cat C]$ the category such that
  \begin{itemize}
    \item
    objects are triples $(G, X, M)$ of a group $G$, a simplicial set $X$ with a $G$-action, and an object $M$ of $\cat C$ with a $G$-action,
    \item
    morphisms $(G, X, M) \to (H, Y, N)$ are given by triples $(\alpha, f, \phi)$ of a group homomorphism $\alpha \colon H \to G$, an $H$-equivariant map of simplicial sets $f \colon Y \to \alpha^* X$, and an $H$-equivariant map $\phi \colon \alpha^* M \to N$,
    \item
    composition is defined in the obvious way.
  \end{itemize}
  When $\cat C$ is a monoidal category, then the assignment
  \[ \bigl( (G, X, M), (H, Y, N) \bigr) \longmapsto (G \times H, X \times Y, M \tensor N) \]
  can be canonically extended to a monoidal structure on $\LocGrp[\cat C]$.
  When $\cat C$ is symmetric monoidal, then $\LocGrp[\cat C]$ is also canonically equipped with a symmetric monoidal structure.
  When $\cat C$ is the category of cochain complexes, we omit it from the notation.
\end{definition}

\begin{definition} \label{def:BG_local_system}
  Let $\cat C$ be a category and $(G, X, M) \in \LocGrp[\cat C]$.
  Then we consider $M$ as a local system on $\hcoinv X G$ with values in $\cat C$ by setting $M(\sigma) = M$ for all $\sigma = (g_1, g_2, \dots, g_n, x) \in (\hcoinv X G)_n$ and letting all face and degeneracy maps act by the identity, except for $d_n$ which acts by multiplication with $g_n$.
  This assembles into a canonical functor $\LocGrp[\cat C] \to \Loc[\cat C]$ given on objects by $(G, X, M) \mapsto (\hcoinv X G, M)$.
  When $\cat C$ is (symmetric) monoidal, we equip this functor with its canonical strong (symmetric) monoidal structure (using \cref{lemma:bar_products}).
\end{definition}

There are two different maps from $\hcoinv {\B(X, G, H)} H$ to $\B H$.
The following lemma shows that, given an $H$-representation, pulling it back to $\hcoinv {\B(X, G, H)} H$ along these two maps yields isomorphic local systems.

\begin{lemma} \label{lemma:hquot_hcoinv_local_system}
  Let $f \colon G \to H$ be a map of simplicial monoids such that $H$ is a discrete group, and let $X$ be a right $G$-module in simplicial sets.
  Furthermore let $\cat C$ be a category and $M$ an object of $\cat C$ with an $H$-action.
  Consider the composite map of simplicial sets
  \[ \psi  \colon  \hcoinv {\B(X, G, H)} H  \longto  \hcoinv X G  \xlongto{\pr}  \B G  \xlongto{f}  \B H\]
  where the first morphism is the one of \cref{lemma:hquot_hcoinv}.
  Then there is a natural isomorphism, covering the identity of $\hcoinv {\B(X, G, H)} H$,
  \[ \Phi  \colon  \bigl( \hcoinv {\B(X, G, H)} H, M \bigr)  \xlongto{\iso}  \bigl( \hcoinv {\B(X, G, H)} H, \psi^*(M) \bigr) \]
  of functors from the appropriate category of pairs $(f, M)$ to $\Loc[\cat C]$; here $M$ is considered as a local system via \cref{def:BG_local_system} on both sides.
  When $\cat C$ is (symmetric) monoidal, then $\Phi$ is a monoidal natural isomorphism between strong (symmetric) monoidal functors.
\end{lemma}

\begin{proof}
  We define $\Phi$, at $\sigma = ((x, g_1, \dots, g_n, h), h_1, \dots, h_n) \in \B(\B(X, G, H), H, *)_n$, to be the map $\psi^*(M)(\sigma) = M \to M = M(\sigma)$ given by acting with $h \mult h_1 \mult \dots \mult h_n$.
  That this is indeed an isomorphism of local systems and a (monoidal) natural transformation in the claimed sense follows by chasing through the definitions.
\end{proof}

The following upgrades \cite[Proposition 3.38]{BZ} by incorporating non-trivial coefficients and monoidal properties.

\begin{lemma} \label{lemma:forms_orbits}
  Let $G$ be a group, $X$ a Kan complex with a $G$-action, and $M$ a $G$-representation in vector spaces.
  Then there is a zig-zag of monoidal natural transformations that are pointwise quasi-isomorphisms
  \[ \Forms * (\hcoinv X G; M)  \xlonghto{\eq}  \Forms * \bigl( \B G; \Forms * (X; \ul M) \bigr) \]
  between lax symmetric monoidal functors to $\coCh$ from the full subcategory of $\LocGrp[\Vect]$ spanned by those objects $(G, X, M)$ such that $X$ is a Kan complex.
  Here we consider $\Forms * (X; \ul M)$ and $M$ as local systems on $\B G$ and $\hcoinv X G$ via \cref{def:BG_local_system}; on $\Forms * (X; \ul M)$ an element $g \in G$ acts by the map induced by $(\inv g, g)$.
  On the right-hand side, we use the canonical lax symmetric monoidal functor given on objects by $(G, X, M) \mapsto (G, *, \Forms * (X; \ul M))$.
\end{lemma}

\begin{proof}
  We write $\pr \colon \hcoinv X G \to \B G$ for the map induced by the constant map $X \to *$; it is a fibration of simplicial sets by \cite[Theorem~3.1, (ii)]{vdBM} since $X$ is a Kan complex.
  We first define a quasi-isomorphism
  \[ \Psi  \colon  \Forms * (X; \ul M)  \xlongto{\eq}  {\pr}_* \bigl( \Omega^*_{\hcoinv X G} \tensor M \bigr) \]
  of local systems on $\B G$.
  To do so, we have to provide, for all $\sigma = (g_1, \dots, g_n) \in (\B G)_n$, a quasi-isomorphism
  \[ \Psi_\sigma  \colon  \Forms * (X; \ul M)  \xlongto{\eq}  {\pr}_* \bigl( \Omega^*_{\hcoinv X G} \tensor M \bigr) (\sigma)  =  \Forms * \bigl( \sigma^*(\hcoinv X G); \tilde \sigma^* (M) \bigr) \]
  where $\tilde \sigma \colon \sigma^*(\hcoinv X G) \to \hcoinv X G$ is the projection; we also write $\pr[\sigma] \colon \sigma^*(\hcoinv X G) \to \lincat n$ for the other projection.
  We define $\Psi_\sigma$ to be the map induced by $(\psi_\sigma, \phi_\sigma)$ where $\psi_\sigma \colon \sigma^*(\hcoinv X G) \to X$ is the map of simplicial sets given, at $\tau \colon \lincat k \to \sigma^*(\hcoinv X G)$, by
  \[ \psi_\sigma(\tau)  =  \Biggl( \prod_{i = 0}^{n - 1 - a_\tau} \inv {g_{n-i}} \Biggr) \act x_\tau \]
  where $x_\tau \in X_k$ is the image of $\tilde \sigma \after \tau$ under the projection $(\hcoinv X G)_k \to X_k$ and we set $a_\tau = \max ({\pr[\sigma]} \after \tau)$ (here we consider ${\pr[\sigma]} \after \tau$ as a map of linearly ordered sets $\set {0, \dots, k} \to \set {0, \dots, n}$).
  The map $\phi_\sigma \colon \psi_\sigma^* (\ul M) = \ul M \to \tilde \sigma^* (M)$ of local systems over $\sigma^*(\hcoinv X G)$ is given, at $\tau$ as above, by the map of chain complexes $\ul M(\tau) = M \to M = \tilde \sigma^* (M)(\tau)$ given by
  \[ m  \longmapsto  \Biggl( \prod_{i = a_\tau + 1}^{n} g_i \Biggr) \act m \]
  where $a_\tau$ is as above.
  Chasing through the definitions, one sees that indeed $\psi_\sigma$ is a map of simplicial sets, that $\phi_\sigma$ is a map of local systems, and that the maps $\Psi_\sigma$ assemble into a map $\Psi$ of local systems as claimed.
  The pair $(\pr[\sigma], \psi_\sigma)$ induces an isomorphism $\sigma^*(\hcoinv X G) \iso \lincat n \times X$, and hence $\psi_\sigma$ is a weak equivalence; since $\phi_\sigma$ is an isomorphism, \cref{lemma:forms_cohomology} then implies that $\Psi_\sigma$ is a quasi-isomorphism.
  Unwinding the definitions, one furthermore sees that $\Psi$ is a monoidal natural transformation between lax symmetric monoidal functors $\LocGrp[\Vect] \to \Loc$; here the lax symmetric monoidal structure on the functor
  \[ (G, X, M)  \longmapsto  \bigl( \B G, p_* ( \Omega^*_{\hcoinv X G} \tensor M ) \bigr) \]
  comes from \cref{def:pushforward}, \cref{lemma:bar_products}, and the multiplication of $\sForms$.
  
  The map $\Psi$ induces a monoidal natural transformation
  \[ \Forms * \bigl( \B G; \Forms * (X; \ul M) \bigr)  \xlongto{\eq}  \Forms * \bigl( \B G; p_* ( \Omega^*_{\hcoinv X G} \tensor M ) \bigr) \]
  which is a pointwise quasi-isomorphism by \cref{lemma:forms_coeff_quasi-iso} since $\Forms * (X; \ul M)$ is groupoidal and $p_* ( \Omega^*_{\hcoinv X G} \tensor M )$ is extendable (by \cref{lemma:pushforward_extendable} and the fact that $\Omega^*_E \tensor M$ is extendable by \cite[Remarks~13.11, 4.]{Hal}) and quasi-groupoidal (using that $\pr$ is a fibration and \cref{lemma:forms_cohomology}).
  Next we note that the unit map $\ul \QQ \to \Omega^*_{\B G}$ induces a monoidal natural transformation
  \[ p_* \bigl( \Omega^*_{\hcoinv X G} \tensor M \bigr) (\B G)  \xlongto{\eq}  \Forms * \bigl( \B G; p_* \bigl( \Omega^*_{\hcoinv X G} \tensor M \bigr) \bigr) \]
  which is a pointwise quasi-isomorphism by \cite[Theorem~13.12]{Hal}.
  Lastly, there is a monoidal natural isomorphism
  \[ \Forms * (\hcoinv X G; M)  \iso  p_* \bigl( \Omega^*_{\hcoinv X G} \tensor M \bigr) (\B G) \]
  by \cref{lemma:pushforward_global_sections}.
\end{proof}

\subsection{The Federer spectral sequence}

We will need the following spectral sequence for computing the homotopy groups of a mapping space.
It is originally due to Federer \cite{Fed}.

\begin{proposition}[Federer] \label{prop:Federer}
  Let $i \colon A \to X$ be a cofibration and $f \colon X \to K$ a map of simplicial sets such that $K$ is a simply connected Kan complex.
  Then there is an extended spectral sequence (in the sense of Bousfield--Kan \cite[Ch.~IX, §4.2]{BK})
  \[ E_2^{p,q}  =  \begin{lrcases*} \Coho p \bigl( X, A; \hg q(K) \bigr), & if $p \ge 0$ and $q \ge \max(2, p)$ \\ 0, & otherwise \end{lrcases*}  \longabuts  \hg {q - p} \bigl( \map[A](X, K), f \bigr) \]
  with the differential of the $r$-th page having bidegree $(r, r - 1)$.
  When $(X, A)$ has the homotopy type of a finite-dimensional relative CW-pair, the spectral sequence converges completely for $q - p \ge 1$ (in the sense of \cite[Ch.~IX, §5.3]{BK}).
  It is functorial in the pair $(i, f)$ in the appropriate sense: a morphism from $(i, f)$ to $(i' \colon A' \to X', f' \colon X' \to K')$ is given by a tuple $(\alpha \colon A' \to A, \chi \colon X' \to X, \kappa \colon K \to K')$ of maps such that $i \alpha = \chi i'$ and $\kappa f \chi = f'$.
\end{proposition}

\begin{proof}
  A treatment of this spectral sequence is contained in work of Kupers--Randal-Williams \cite[§5.2]{KR}; the statement above arises by applying their more general version to the fibration given by the projection $\gr X \times \gr K  \to  \gr X$ together with the section $(\id[\gr X], \gr f)$.
  The spectral sequence then converges to $\hg * ( \Map[\gr A](\gr X, \gr K), \gr f )$, where $\Map[\gr A](\gr X, \gr K)$ denotes the space of continuous maps from $\gr X$ to $\gr K$ that restrict to $\gr {f \after i}$ on $\gr A$, equipped with the compact-open topology.
  Since $i$ is a cofibration and $K$ is a Kan complex, there is a natural weak homotopy equivalence
  \[ \gr {\map[A] (X, K)}  \xlongto{\eq}  \Map[\gr A](\gr X, \gr K) \]
  by \cref{lemma:compact-open}.
  Lastly, note that a priori the cohomology appearing on the $E_2$-page has coefficients in the local system $\ul \pi_q(\gr K, \gr f)$ on $\gr X$ given on points by $x \mapsto \hg q (\gr K, \gr f(x))$ and on a path $\gamma$ by the isomorphism induced by $\gr f \after \gamma$.
  However, since $K$ is simply connected, any basepoint $k_0 \in K$ induces an isomorphism between $\ul \pi_q(\gr K, \gr f)$ and the constant local system with value $\hg q(K, k_0)$.
  Moreover, two different choices of basepoint $k_0$ yield canonically isomorphic groups.
\end{proof}

\subsection{Differential graded Lie algebras}

In this subsection, we recall basic properties of dg Lie algebras over a commutative $\QQ$-algebra, including the model structure on their category.

\begin{definition}
  Let $R$ be a commutative $\QQ$-algebra.
  A \emph{dg Lie algebra over $R$} is a chain complex $L$ of $R$-modules equipped with a dg Lie algebra structure such that the Lie bracket $\liebr \blank \blank \colon L \times L \to L$ is $R$-bilinear.
  A \emph{graded Lie algebra over $R$} is defined analogously.
  
  We denote by $\dgLie[R]$ the category of dg Lie algebras over $R$ and $R$-linear dg Lie algebra morphisms.
  Furthermore, we write $\freelie_R$ for the free dg Lie algebra functor, i.e.\ the left adjoint of the forgetful functor $\dgLie[R] \to \Ch[R]$, and we denote pushouts (and coproducts) in the category $\dgLie[R]$ by $L' \cop^R_L L''$.
  When $R = \QQ$, we omit it from the notation.
\end{definition}

\begin{definition}
  Let $R$ be a commutative $\QQ$-algebra, $L$ a dg Lie algebra over $R$, and $A$ a commutative chain algebra over $R$.
  Then we equip $L \tensor[R] A$ with the dg Lie algebra structure determined by
  \[ \liebr {l \tensor a} {l' \tensor a'}  =  (-1)^{\deg a \deg{l'}} \liebr {l} {l'} \tensor a a' \]
  for all homogeneous $l, l' \in L$ and $a, a' \in A$.
  Often we will use the same notation when $A$ is a commutative cochain algebra, in which case we consider it as a chain algebra by inverting the grading.
\end{definition}

\begin{lemma} \label{lemma:Lie_algebra_extension}
  Let $R \to S$ be a map of commutative $\QQ$-algebras.
  Then the extension-of-scalars functor $\blank \tensor[R] S \colon \dgLie[R] \to \dgLie[S]$ is left adjoint to the forgetful functor.
  In particular it preserves colimits.
  Moreover, there is a natural isomorphism $\freelie_S(\blank \tensor[R] S) \iso \freelie_R(\blank) \tensor S$ of functors $\Ch[R] \to \dgLie[S]$.
\end{lemma}

\begin{proof}
  The first statement follows from an elementary verification.
  The last statement follows from the fact that both functors are left adjoint to the forgetful functor.
\end{proof}

\begin{definition} \label{def:quasi-free}
  Let $R$ be a commutative $\QQ$-algebra.
  A map $L' \to L$ of dg Lie algebras over $R$ is \emph{quasi-free} if there exists a graded $R$-module $V$ that is free in each degree and an isomorphism $L \iso L' \cop^R \freelie_R(V)$ under $L'$ of graded Lie algebras over $R$.
  
  A dg Lie algebra $L$ over $R$ is \emph{quasi-free} if the unique map $0 \to L$ of dg Lie algebras over $R$ is quasi-free.
\end{definition}

The following is a generalization of the indecomposables of a dg Lie algebra to the relative setting.
It turns out to also have analogous properties.
In particular, when $L' \to L$ is quasi-free, the homology of the relative indecomposables is isomorphic to the Chevalley--Eilenberg homology of $L$ relative to $L'$; see (the proof of) \cref{lemma:indec_is_ho}.

\begin{definition} \label{def:indec}
  Let $R$ be a commutative $\QQ$-algebra and $L' \to L$ a map of dg Lie algebras over $R$.
  We define the \emph{indecomposables} of $L$ relative to $L'$ to be the dg Lie algebra over $R$
  \[ \indec[L'](L)  \defeq  \quot {\ol L} { \liebr {\ol L} {\ol L} } \]
  where $\ol L$ is the quotient of $L$ by the Lie ideal generated by the image of $L'$.
  Note that the Lie algebra structure of $\indec[L'](L)$ is trivial, so that we could also consider it to be simply a chain complex of $R$-modules.
\end{definition}

\begin{lemma} \label{lemma:quasi-free_indec}
  Let $R$ be a commutative $\QQ$-algebra, $L'$ a graded Lie algebra over $R$, and $V$ a graded $R$-module.
  Set $L \defeq L' \cop^R \freelie_R(V)$.
  Then the inclusion induces an isomorphism $V \to \indec[L'](L)$ of graded $R$-modules.
\end{lemma}

\begin{proof}
  Consider the surjective morphism of graded Lie algebras $L \to \freelie_R(V)$ given by the identity on $\freelie_R(V)$ and the trivial map on $L'$.
  Its kernel is precisely the Lie ideal generated by $L'$, and hence we have, in the notation of \cref{def:indec}, that $\ol L = \freelie_R(V)$.
  The claim then follows from the fact that the inclusion induces an isomorphism $V \to \quot {\freelie_R(V)} {\liebr {\freelie_R(V)} {\freelie_R(V)}}$.
\end{proof}

We will now recall the standard simplicial enrichment and model structure on the category of dg Lie algebras over $R$.
They do not quite combine to a simplicial model category structure in the usual sense (as defined in e.g.~\cite[§9.1.5]{Hir}), because the category of dg Lie algebras is not (co)tensored over simplicial sets.
However the most important axiom, SM7 (see \cref{lemma:SM7}), is satisfied.

\begin{definition} \label{def:dgLie_simplicial}
  Let $R$ be a commutative $\QQ$-algebra and $L$, $L'$, and $L''$ three dg Lie algebras over $R$.
  We write $\map(L, L')$ for the simplicial set $\Hom[{\dgLie[R]}](L, L \tensor[\QQ] \sForms)$.
  We furthermore define a composition $\map(L', L'') \times \map(L, L') \to \map(L, L'')$ by sending an $n$-simplex $(g, f)$ to the composite
  \[ L  \xlongto{f}  L' \tensor[\QQ] \sForms[n]  \xlongto{g \tensor {\id}}  L'' \tensor[\QQ] \sForms[n] \tensor[\QQ] \sForms[n]  \xlongto{{\id} \tensor \mu}  L'' \tensor[\QQ] \sForms[n] \]
  where $\mu$ is the multiplication of $\sForms[n]$.
  Together with the identity maps $L \to L \iso L  \tensor[\QQ] \sForms[0]$, this equips $\dgLie[R]$ with the structure of a simplicial category.\footnote{This is stated in \cite[§4.8.10]{Hin}, but also easy to check directly.}
  
  We write $\eq_{L}^R$ for the simplicial homotopy equivalence relation between morphisms of the simplicial category $L \comma \dgLie[R]$ (see \cref{def:undercategory_simplicial}).
\end{definition}

\begin{lemma}[Quillen, Hinich] \label{lemma:dgLie_model_structure}
  Let $R$ be a commutative $\QQ$-algebra.
  Then there exists a unique model structure on the category $\dgLie[R]$ such that:
  \begin{itemize}
    \item
    The weak equivalences are the underlying quasi-isomorphisms.
    
    \item
    The fibrations are the degreewise surjective maps.
  \end{itemize}
  Furthermore, this model structure fulfills:
  \begin{enumerate}
    \item
    If $i \colon L' \to L$ is a quasi-free map of dg Lie algebras over $R$ such that $L$ is concentrated in non-negative degrees, then $i$ is a cofibration.
    
    \item
    If $i \colon L' \to L$ is a cofibration and $\widetilde L$ an object of $\dgLie[R]$, then the precomposition map $i^* \colon \map(L, \widetilde L) \to \map(L', \widetilde L)$ is a fibration of simplicial sets.
    If $i$ is furthermore a quasi-isomorphism, then $i^*$ is a weak equivalence.
  \end{enumerate}
\end{lemma}

\begin{proof}
  The existence of the model structure is \cite[Theorem~4.1.1]{Hin}.
  When $L$ is concentrated in non-negative degrees, one sees, by choosing an isomorphism $L \iso L' \cop^R \freelie_R(V)$ under $L'$ and decomposing $V$ by homological degree, that $i$ is a ``standard cofibration'' in the terminology of \cite[§2.2.3]{Hin}; this proves property~1.
  Property~2 is \cite[Corollary~4.8.5]{Hin}.
\end{proof}

\begin{remark} \label{rem:dgLie_map_Kan}
  Note that property~2 of \cref{lemma:dgLie_model_structure} implies that $\map(L, \widetilde L)$ is a Kan complex when $L$ is cofibrant.
  More generally, when $L' \to L$ is a cofibration and $L' \to \widetilde L$ any map, then $\map[L'](L, \widetilde L)$ is a Kan complex.
\end{remark}

We now provide proofs for basic properties of the simplicial homotopy relation of maps of dg Lie algebras relative to another dg Lie algebra.
See also Espic--Saleh \cite[Definition~4.6]{ES} and, in the absolute case, Tanré \cite[Ch.~II, §5]{Tan}.

\begin{definition}
  For $q \in \QQ$, we denote by $\ev[q] \colon \sForms[1] \to \QQ$ the unique morphism of cochain algebras such that $\ev[q](t_0) = q$.
\end{definition}

\begin{lemma} \label{lemma:dgLie_homotopy}
  Let $R$ be a commutative $\QQ$-algebra and $i \colon L' \to L$ and $j \colon L' \to \widetilde L$ two maps of dg Lie algebras over $R$.
  Assume that $i$ is a cofibration.
  Then the following conditions on two maps $f, g \colon L \to \widetilde L$ under $L'$ are equivalent:
  \begin{enumerate}
    \item
    The maps $f$ and $g$ are simplicially homotopic in the simplicial category $L' \comma \dgLie[R]$.
    
    \item
    There exists a map $h \colon L \to \widetilde L \tensor[\QQ] \sForms[1]$ of dg Lie algebras over $R$ such that we have $({\id} \tensor \ev[0]) \after h = f$, $({\id} \tensor \ev[1]) \after h = g$, and $h(i(a)) = j(a) \tensor 1$ for all $a \in L'$.
    
    \item
    The maps $f$ and $g$ are right homotopic in the model category $L' \comma \dgLie[R]$.
  \end{enumerate}
\end{lemma}

\begin{proof}
  That 1 is equivalent to 2 follows from \cref{rem:dgLie_map_Kan} and the definition of the simplicial structure of $\sForms$.
  To see that 2 and 3 are equivalent, note that the unit map $\eta \colon \QQ \to \sForms[1]$ is a quasi-isomorphism, and that $(\ev[0], \ev[1]) \colon \sForms[1] \to \QQ \dirsum \QQ$ is surjective.
  This remains true upon tensoring with $\widetilde L$, which implies that
  \[ \widetilde L  \xlongto{{\id} \tensor \eta}  \widetilde L \tensor[\QQ] \sForms[1]  \xlongto{({\id} \tensor \ev[0], {\id} \tensor \ev[1])}  \widetilde L \times \widetilde L \]
  is a path object for $\widetilde L$ in $L' \comma \dgLie[R]$.
  Hence 2 implies 3, and the other implication follows from \cite[Proposition~7.4.10]{Hir}.
\end{proof}

\begin{remark} \label{rem:dgLie_homotopy_inverse}
  Let $R$ be a commutative $\QQ$-algebra and $L' \to L$ and $L' \to \widetilde L$ two cofibrations of dg Lie algebras over $R$.
  \Cref{lemma:dgLie_homotopy} implies, as in (the proof of) \cref{rem:simplicial_homotopy}, that if $f \colon L \to \widetilde L$ is a quasi-isomorphism under $L'$ of dg Lie algebras over $R$, then it has a simplicial homotopy inverse $g$ relative to $L'$, and $g$ is again a quasi-isomorphism.
\end{remark}

\begin{lemma} \label{lemma:dgLie_homotopic}
  Let $R$ be a commutative $\QQ$-algebra and $i \colon L' \to L$ and $L' \to \widetilde L$ two maps of dg Lie algebras over $R$ such that $i$ is a cofibration.
  Furthermore, let $f, g \colon L \to \widetilde L$ be two maps under $L'$ of dg Lie algebras over $R$.
  If $f \eq_{L'}^R g$, then their underlying maps of chain complexes are chain homotopic relative to $L'$.
\end{lemma}

\begin{proof}
  This follows as in \cite[§II.5, Remarque (27)]{Tan} (up to a difference in signs, stemming from the fact that Tanré uses the tensor product $\sForms[1] \tensor[\QQ] \widetilde L$ to define homotopy).
\end{proof}

\subsection{Lie models}

In this subsection, we recall how to model spaces by (complete) dg Lie algebras as well as various properties of the constructions involved.
This goes back to the beginning of rational homotopy theory due to Quillen \cite{Qui69}, though we will not use his original constructions.
We begin with some definitions.

\begin{definition}
  A dg Lie algebra $\lie g$ is of \emph{finite type} if $\lie g_n$ is finite dimensional for each $n \in \ZZ$.
  We denote by $\dgLieft$ the full subcategory of $\dgLie$ spanned by the dg Lie algebras of finite type.
\end{definition}

\begin{definition}
  A \emph{representation} of a Lie algebra $\lie g$ on a vector space $V$ is a Lie algebra homomorphism $\rho \colon \lie g \to \gl(V)$.
  We set $\Gamma_\rho^1 V \defeq V$ and $\Gamma_\rho^{k+1} V \defeq \rho(\lie g)(\Gamma_\rho^k V)$ for $k \ge 1$.
  The representation $\rho$ is \emph{nilpotent} when $\Gamma_\rho^k V = 0$ for some $k$.
  The Lie algebra $\lie g$ is \emph{nilpotent} when the adjoint representation $\ad \colon \lie g \to \gl(\lie g)$, given by $a \mapsto \liebr a \blank$, is nilpotent.
\end{definition}

\begin{definition}
  Let $f \colon L \to \widetilde L$ be a map of dg Lie algebras.
  An \emph{$f$-derivation} of degree $n \in \ZZ$ is a map of graded vector spaces $\theta \colon L \to \widetilde L$ of degree $n$ such that
  \[ \theta(\liebr x y)  =  \liebr {\theta(x)} {f(y)} + (-1)^{n \deg x} \liebr {f(x)} {\theta(y)} \]
  for all homogeneous $x, y \in L$.
  Given a map of dg Lie algebras $L' \to L$, we define $\Derrel{f}(L, \widetilde L \rel L')$ to be the chain complex consisting, in degree $n$, of the vector space of $f$-derivation of degree $n$ that vanish on the image of $L'$, with differential given by $\theta \mapsto d_{L'} \after \theta - (-1)^{\deg \theta} \theta \after d_L$.
  When $f = \id[L]$, then we simply write $\Der(L \rel L')$, and this comes equipped with a dg Lie algebra structure with Lie bracket given by $\liebr \theta \psi \defeq \theta \after \psi - (-1)^{\deg \theta \deg \psi} \psi \after \theta$.
\end{definition}

We will also require the notion of an action of a dg Lie algebra, as well as a twisted version that goes back to Tanré \cite[§VII.2]{Tan}; also see Berglund \cite[§3.5]{Ber}.

\begin{definition}
	An \emph{outer action} of a dg Lie algebra $\lie g$ on a dg Lie algebra $L$ consists of a map of graded Lie algebras $\alpha \colon \lie g \to \Der(L)$ and a map of chain complexes $\chi \colon \lie g \to L$ of degree~$-1$ such that the following two equations hold
	\begin{align*}
		\chi ( \liebr \theta \psi ) &= \chi(\theta) \act \psi + (-1)^{\deg \chi \deg \theta} \theta \act \chi(\psi) \\
		d(\theta \act a) &= d(\theta) \act a + (-1)^{\deg \theta} \theta \act d(a) + \liebr {\chi(\theta)} a
	\end{align*}
	where we write $\theta \act a \defeq \alpha(\theta)(a)$ and $a \act \theta \defeq - (-1)^{\deg a \deg \theta} \theta \act a$.
	We will often refer to an outer action simply by $\chi$ and omit mentioning $\alpha$.
	An \emph{action} is an outer action such that $\chi = 0$, i.e.\ a map of dg Lie algebras $\lie g \to \Der(L)$.
\end{definition}

\begin{definition}
	Let $L$ be a dg Lie algebra equipped with an outer action $\chi$ of a dg Lie algebra $\lie g$, and let $A$ be a commutative chain algebra.
	Then we equip the dg Lie algebra $L \tensor A$ with the outer action $\chi_A$ of the dg Lie algebra $\lie g \tensor A$ defined as
	\begin{align*}
		(\theta \tensor a) \act (x \tensor b)  &\defeq  (-1)^{\deg a \deg x} (\theta \act x) \tensor ab \\
		\chi_A(\theta \tensor a)  &\defeq  \chi(\theta) \tensor a
	\end{align*}
	for homogeneous elements $\theta \in \lie g$, $x \in L$, and $a, b \in A$.
\end{definition}

\begin{definition}
	An outer action of a non-negatively graded dg Lie algebra $\lie g$ on a dg Lie algebra $L$ is \emph{nilpotent} if the induced representation of $\lie g_0$ on $L_n$ is nilpotent for all $n \in \ZZ$.
	A non-negatively graded dg Lie algebra $\lie g$ is \emph{nilpotent} when the adjoint action $\ad \colon \lie g \to \Der(\lie g)$ is nilpotent.
	We denote by $\dgLienil$ the full subcategory of $\dgLie$ spanned by the nilpotent dg Lie algebras, and we set $\dgLienilft \defeq \dgLienil \intersect \dgLieft$.
\end{definition}

Note that a positively graded dg Lie algebra is automatically nilpotent.
We will also require the notion of complete dg Lie algebras, which are more general than nilpotent dg Lie algebras but still enjoy many of their properties.
See Buijs--Félix--Murillo--Tanré \cite[Ch.~3]{BFMT} for more background on complete dg Lie algebras.

\begin{definition}
	A \emph{filtered} dg Lie algebra is a dg Lie algebra $L$ equipped with an $\NNpos$-indexed descending filtration of chain complexes
	\[ L = \cfilt L 1 \supseteq \cfilt L 2 \supseteq \cfilt L 3 \supseteq \dots \]
	such that $\liebr {\cfilt L p} {\cfilt L q} \subseteq \cfilt L {p+q}$.
	A \emph{map} of filtered dg Lie algebras is a filtration-preserving map of dg Lie algebras.
	A \emph{complete} dg Lie algebra is a filtered dg Lie algebra $L$ such that the canonical map $L \to \lim{n} \quot L {\cfilt L n}$ is an isomorphism.
\end{definition}

Note that a nilpotent dg Lie algebra can be considered as a complete dg Lie algebra by equipping it with its lower central series filtration $\cfilt L n \defeq \Gamma_{\ad}^n L$ for $n \ge 1$.

We now introduce the model we will use for the geometric realization of a dg Lie algebra.
It is the functor $\MCb$ from dg Lie algebras to simplicial sets, originally introduced by Hinich \cite[§2]{Hin97}.

\begin{notation}
  Let $L$ be a dg Lie algebra.
  We write
  \[ \MC(L)  \defeq  \set { \tau \in L_{-1}  \mid  d(\tau) + \frac 1 2 \liebr \tau \tau = 0 } \]
  for the set of Maurer--Cartan elements of $L$.
\end{notation}

\begin{definition}[Hinich]
  For a dg Lie algebra $L$, we define the simplicial set
  \[ \MCb(L)  \defeq  \MC(L \tensor \sForms) \]
  and call it the \emph{realization} of $L$.
  Sometimes we will also use the shorter notation $\nerve L \defeq \MCb(L)$.
  For a filtered dg Lie algebra $L$, we furthermore define the simplicial set
  \[ \cMCb(L)  \defeq  \lim{n} \MCb \bigl( \quot L {\cfilt L n} \bigr) \iso \MC(L \ctensor \sForms) \]
  where $L \ctensor A \defeq \lim{n} (\quot L {\cfilt L n} \tensor A)$ denotes the completed tensor product.
\end{definition}

Note that $\cMCb(L) \iso \MCb(L)$ when $L$ is nilpotent and equipped with its lower central series filtration.
Also recall that the functor $\cMCb$ sends surjective maps of complete dg Lie algebras to Kan fibrations (see e.g.\ \cite[Theorem~5.4]{BerMC}), and hence complete dg Lie algebras to Kan complexes.
Moreover, it sends quasi-isomorphisms between non-negatively graded complete dg Lie algebras to weak equivalences (see e.g.\ \cite[Corollary~6.3]{BerMC}).
Lastly note that $\MCb$ sends injective maps of dg Lie algebras (in particular quasi-free maps) to cofibrations of simplicial sets.

\begin{definition}
  Let $\cat I$ be a small category.
  We say that a functor $L \colon \cat I \to \dgLienil$ \emph{models} a functor $X \colon \cat I \to \sSet$ if it comes equipped with the extra structure of a natural pointwise weak equivalence $\rat{(\blank)} \after X \to \MCb \after L$.
\end{definition}

The following well-known lemma implies that, in the preceding definition, we could also have asked for a pointwise rational homology equivalence $X \to \MCb \after L$ (at least in the positively graded case).

\begin{lemma}
  Let $L$ be a positively graded dg Lie algebra.
  Then $\MCb(L)$ is a rational Kan complex.
\end{lemma}

\begin{proof}
  Note that $\MCb(L)$ is simply connected by definition.
  Moreover, for $k \ge 2$, the abelian groups $\hg k (\MCb(L))$ are all uniquely divisible (i.e.\ $\QQ$-vector spaces), for example by \cite[Corollary~6.3]{BerMC}.
  This implies the claim by \cite[Ch.~V, Proposition~3.3]{BK} combined with \cite[Ch.~X, Corollary~3.3]{GJ}.
\end{proof}

The following lemma relates mapping complexes of dg Lie algebras and their realizations.

\begin{lemma} \label{lemma:map_MCb}
	Let $i \colon L_A \to L_X$ and $L_A \to L_Y$ be maps of nilpotent dg Lie algebras such that $L_A$ and $i$ are quasi-free and $L_Y$ is of finite type and positively graded.
	Then there is a natural weak equivalence of simplicial sets
	\[ \map[L_A](L_X, L_Y)  \xlongto{\eq}  \map[\MCb(L_A)] \bigl( \MCb(L_X), \MCb(L_Y) \bigr) \]
	which is given by $\MCb$ on $0$-simplices.
\end{lemma}

\begin{proof}
  \newcommand{\QL}{\mathcal{L}}
  \newcommand{\mc}{\mathfrak{mc}}
  Note that $\MCb(L) \iso \map(\mc, L)$, where $\mc$ is the free graded Lie algebra on an element $\tau$ of degree $-1$ with differential $d(\tau) = -\frac{1}{2} \liebr{\tau}{\tau}$.
  When $L_A = 0$, the map
  \begin{equation} \label{eq:pointed mapping space}
    \map(L_X, L_Y)  \longto  \map[*] \bigl( \MCb(L_X), \MCb(L_Y) \bigr)
  \end{equation}
  is defined to be the adjoint of the map
  \[ \map(L_X, L_Y) \times \map(\mc, L_X) \longto \map(\mc, L_Y) \]
  given by composition.
  We observe that \eqref{eq:pointed mapping space} is the composite
  \begin{align*}
    \map(L_X, L_Y) & \xlongto{\epsilon^*} \map \bigl( \QL \CEchains * (L_X), L_Y \bigr) \\
    & \xlongto{\cong} \MC \bigl( \Hom \bigl( \rCEchains * (L_X), L_Y \tensor \sForms \bigr) \bigr) \\
    & \xlongto{\alpha} \map[*] \bigl( \MCb(L_X), \MCb(L_Y) \bigr)
  \end{align*}
  where $\QL$ is Quillen's functor from counital dg coalgebras to dg Lie algebras (see e.g.\ \cite[§2.2]{Hin01}), $\epsilon \colon \QL \CEchains * (L_X) \to L_X$ is the canonical quasi-isomorphism (see e.g.\ \cite[Proposition~3.3.2~(2)]{Hin01}), $\rCEchains *$ denotes reduced Chevalley--Eilenberg chains, the middle isomorphism is the one of \cite[Theorem~2.2.5]{Hin01}, and $\alpha$ is the restriction of the weak equivalence
  \[ \alpha'  \colon  \MC \bigl( \Hom \bigl( \CEchains * (L_X), L_Y \tensor \sForms \bigr) \bigr)  \xlongto{\eq}  \map \bigl( \MCb(L_X), \MCb(L_Y) \bigr) \]
  of \cite[Theorem~3.16]{Ber}.
  That $\epsilon^*$ is a weak equivalence follows from \cref{lemma:dgLie_model_structure} and Kenny Brown’s lemma (see e.g.\ \cite[Corollary~7.7.2]{Hir}), since $\QL \CEchains * (L_X) = \freelie (\shift[-1] \rCEchains * (L_X))$ is quasi-free and non-negatively graded and hence cofibrant.
  By \cite[Remark~3.17]{Ber} there is a natural isomorphism of simplicial sets
  \[ \MC \bigl( \Hom \bigl( \CEchains * (L_X), L_Y \tensor \sForms \bigr) \bigr)  \iso  \cMCb \bigl( \Hom \bigl( \CEchains * (L_X), L_Y \bigr) \bigr) \]
  where on the right-hand side the complete filtration is induced by the degree filtration $L_Y = (L_Y)_{\ge 1} \supseteq (L_Y)_{\ge 2} \supseteq \dots$.
  Furthermore note that the map $\Hom ( \CEchains * (L_X), L_Y ) \to L_Y$ that evaluates at the unit is a surjective map of complete dg Lie algebras and hence induces a fibration upon applying $\cMCb$.
  By taking fibers, this implies that $\alpha$ is a weak equivalence.
  
  For the general version, we consider the commutative diagram
  \[
  \begin{tikzcd}
  	\map(L_X, L_Y) \rar{\eq} \dar[swap]{i^*} & \map[*] \bigl( \MCb(L_X), \MCb(L_Y) \bigr) \dar{\MCb(i)^*} \\
  	\map(L_A, L_Y) \rar{\eq} & \map[*] \bigl( \MCb(L_A), \MCb(L_Y) \bigr)
  \end{tikzcd}
  \]
  and note that the vertical maps are fibrations.
  Hence the induced map on fibers is a weak equivalence, which is what we wanted to show.
  That this map is given by $\MCb$ on $0$-simplices is seen by chasing through its construction.
\end{proof}

The following technical lemma will be useful later.

\begin{lemma} \label{lemma:MC_pushouts}
  Let $L_X \from L_A \to L_Y$ be quasi-free maps of positively graded dg Lie algebras.
  Then the induced map of simplicial sets
  \[ \MCb(L_X) \cop_{\MCb(L_A)} \MCb(L_Y)  \longto  \MCb \bigl( L_X \cop_{L_A} L_Y \bigr) \]
  is a cofibration and rational homology equivalence.
\end{lemma}

\begin{proof}
  We first prove that the map is a cofibration.
  By \cref{lemma:Lie_algebra_extension}, it is enough to prove that the map of sets
  \[ \MC(L_X) \cop_{\MC(L_A)} \MC(L_Y)  \longto  \MC \bigl( L_X \cop^R_{L_A} L_Y \bigr) \]
  is injective for $L_A$, $L_X$, and $L_Y$ dg Lie algebras over a commutative $\QQ$-algebra $R$.
  For ease of notation, we restrict to the case $R = \QQ$; the proof of the general case is analogous.
  Set $L_P \defeq L_X \cop_{L_A} L_Y$ and note that, since $\MC(L_X) \to \MC(L_P)$ and $\MC(L_Y) \to \MC(L_P)$ are injective, it is enough to prove that $\MC(L_X) \intersect \MC(L_Y) = \MC(L_A)$ in $\MC(L_P)$.
  Now, pick isomorphisms $L_X \iso L_A \cop \freelie (V)$ and $L_Y \iso L_A \cop \freelie (W)$ of graded Lie algebras under $L_A$.
  Then there is an induced isomorphism $L_X \cop_{L_A} L_Y \iso L_A \cop \freelie (V) \cop \freelie (W)$ of graded Lie algebras under $L_A$.
  This implies the claim.
  
  By a theorem of Félix--Fuentes--Murillo \cite[Theorem~0.1]{FFM} combined with work of Buijs--Félix--Murillo--Tanré \cite[Theorem~0.1]{BFMT17} or Robert-Nicoud \cite[Corollary~5.3]{Rob} (see also Berglund \cite[Theorem~8.1]{BerMC}), there is a zig-zag of pointwise weak equivalences between (the restriction of) $\MCb$ and the right Quillen equivalence constructed by Quillen \cite{Qui69} from the model category of positively graded dg Lie algebras to the category of $2$-reduced simplicial sets equipped with the rational model structure.
  This implies that the total right derived functor of $\MCb$ preserves homotopy colimits of positively graded dg Lie algebras.
  Since $\MCb$ preserves weak equivalences, it thus preserves homotopy colimits itself.
  As $L_A \to L_X$ and $\MCb(L_A) \to \MCb(L_X)$ are cofibrations, and hence both pushouts appearing above are homotopy pushouts, this finishes the proof.
\end{proof}

We will also need the construction $\expb(\lie g)$, introduced by Berglund \cite[§3.2]{Ber}, which models the loop space of the realization $\MCb(\lie g)$.
A concise treatment of its properties is provided in \cite[§§5--6]{BerMC}.

\begin{notation}
  Let $\lie g$ be a nilpotent Lie algebra.
  We denote by $\exp(\lie g)$ the group with underlying set $\lie g$ and multiplication given by the Baker--Campbell--Hausdorff formula; see e.g.\ \cite[§4.2]{BFMT} for its definition.
\end{notation}

\begin{definition}[Berglund]
  Let $\lie g$ be a nilpotent dg Lie algebra.
  We define the \emph{exponential group} of $\lie g$ to be the simplicial group
  \[ \expb(\lie g)  \defeq  \exp \bigl( \Cycles 0 ( \lie g \tensor \sForms) \bigr) \]
  where $\Cycles 0$ denotes the $0$-cycles.
\end{definition}

\begin{lemma} \label{lemma:exp_MC_products}
  The functors $\expb$, $\MCb$, and $\cMCb$ preserve products.
\end{lemma}

\begin{proof}
  For $\MCb$ and $\cMCb$ this follows directly from the definitions.
  For $\expb$ it follows from the fact that $\exp$ preserves products since the induced map $\exp(\lie g \times \lie h) \to \exp(\lie g) \times \exp(\lie h)$ is a bijection by definition.
\end{proof}

The following lemma generalizes \cite[Theorem~3.15]{Ber}; the proof we give here is inspired by an argument of Berglund--Saleh \cite[Proof of Proposition~3.11]{BS}.
See \cite[§9]{BerMC} for a comparison of the two approaches.

\begin{definition}
	Let $\chi$ be an outer action of a dg Lie algebra $\lie g$ on a dg Lie algebra $L$.
	We write $\lie g \lsemidir[\chi] L$ for the dg Lie algebra with underlying graded vector space $\lie g \dirsum L$ and
	\begin{align*}
		\liebr {(\theta, x)} {(\psi, y)}  &\defeq  \bigl( \liebr \theta \psi, \liebr x y + \theta \act y + x \act \psi \bigr) \\
		d(\theta, x)  &\defeq  \bigl( d(\theta), d(x) + \chi(\theta) \bigr)
	\end{align*}
	as Lie bracket and differential.
\end{definition}

\begin{lemma}[Berglund] \label{lemma:MC_outer}
  Let $L$ be a dg Lie algebra equipped with a nilpotent outer action $\chi$ of a nilpotent dg Lie algebra $\lie g$.
  Then there is a weak equivalence in the homotopy category of simplicial sets
  \begin{equation*}
    \MCb( \lie g \lsemidir[\chi] L )  \xlonghto{\eq}  \B \bigl( *, \expb(\lie g), \MCb(L) \bigr)
  \end{equation*}
  where the action $\Xi_\chi$ of $\expb(\lie g)$ on $\MCb(L)$ is given by
  \[ \Xi_\chi(\theta)(x)  \defeq  x + \sum_{n \ge 0} \frac{1}{(n + 1)!} (\theta \act \blank)^{\after n} \bigl( \theta \act x - \chi(\theta) \bigl) \]
  in each simplicial degree.
  Moreover, this weak equivalence can be lifted to a zig-zag of monoidal natural weak equivalences of strong symmetric monoidal functors from the category of triples $(L, \lie g, \chi)$, equipped with the cartesian monoidal structure, to $\sSet$.
\end{lemma}

\begin{proof}
  For any dg Lie algebra $\lie h$ and graded Lie subalgebra $\lie h' \subseteq \lie h$ such that the adjoint action of $\lie h'$ on $\lie h$ is nilpotent, there is a gauge action $\mathcal G$ of $\Gb(\lie h') \defeq \exp((\lie h' \tensor \sForms)_0)$ on $\MCb(\lie h)$, in each simplicial degree given by
  \[ \mathcal G(\theta)(\psi) = x + \sum_{n \ge 0} \frac{1}{(n + 1)!} \ad[\theta]^{\after n} \bigl( [\theta, \psi] - d(\theta) \bigr) \]
  (this can be shown as in \cite[§4.3]{BFMT}; see also \cite[Proof of Proposition~3.11]{BS} and \cite[§5]{BerMC}).
  Our assumptions imply that there is a gauge action $\mathcal G_\chi$ associated to $\lie g \subseteq \lie g \lsemidir[\chi] L$, which restricts to a map
  \[ \mathcal G'_\chi \colon \Gb(\lie g) \times \MCb( L )  \longto  \MCb( \lie g \lsemidir[\chi] L ) \]
  given by
  \[ \mathcal G'_\chi(\theta, x)  =  \Bigl( - \textstyle\sum_{n \ge 0} \tfrac 1 {(n + 1)!} \ad[\theta]^{\after n} \bigl( d(\theta) \bigr), x + \textstyle\sum_{n \ge 0} \frac{1}{(n + 1)!} (\theta \act \blank)^{\after n} \bigl( \theta \act x - \chi(\theta) \bigl) \Bigr) \]
  in each simplicial degree.
  Note that $\mathcal G'_\chi$ restricts to a map $\expb(\lie g) \times \MCb( L )  \to  \MCb( L )$, and that this is precisely the adjoint of the map $\Xi_\chi$.
  This also implies that there is an induced map
  \[ \Phi \colon \Gb(\lie g) \times_{\expb(\lie g)} \MCb( L )  \longto  \MCb( \lie g \lsemidir[\chi] L ) \]
  and we will now prove that it is an isomorphism.

  To this end note that the map $\pi \colon \MCb(\lie g \lsemidir[\chi] L) \to \MCb(\lie g)$ induced by the projection is $\Gb(\lie g)$-equivariant when both sides are equipped with the respective gauge action.
  Moreover, by \cite[Theorem~5.2~(2)]{BerMC}, the map $\gamma \colon \Gb(\lie g) \to \MCb(\lie g)$ given by acting on $0$ induces an isomorphism $\quot {\Gb(\lie g)} {\expb(\lie g)} \iso \MCb(\lie g)$.
  We now define an inverse of $\Phi$ by
  \[ \Psi(x) \defeq \bigl( \theta, \mathcal G_\chi(\inv \theta)(x) \bigr) \]
  where $\theta \in \inv \gamma(\pi(x))$ is arbitrary.
  First note that $\mathcal G_\chi(\inv \theta)(x) \in L$ since $\pi(\mathcal G_\chi(\inv \theta)(x)) = \mathcal G(\inv \theta)(\pi(x)) = 0$ by assumption.
  The value $\Psi(x)$ does not depend on the choice of $\theta$ since two such choices differ exactly by multiplication with an element of $\expb(\lie g)$.
  It is clear that $\Psi$ is a map of simplicial sets and inverse to $\Phi$.

  Lastly note that we obtain a zig-zag
  \begin{equation*} 
  	\B \bigl( *, \expb(\lie g), \MCb( L ) \bigr)  \xlongfrom{\eq}  \B \bigl( \Gb(\lie g), \expb(\lie g), \MCb( L ) \bigr)  \xlongto{\eq}  \Gb(\lie g) \times_{\expb(\lie g)} \MCb( L )
  \end{equation*}
  where the left-hand map is a weak equivalence since the simplicial group $\Gb(\lie g)$ is contractible by \cite[Lemma~5.1~(1)]{BerMC}, and the right-hand map is a weak equivalence by an argument as in \cite[(Proof of) Proposition~8.5]{May}.
  It is clear that all maps involved in the zig-zag, as well as $\Phi$, are natural and monoidal in the claimed sense.
  This finishes the proof.
\end{proof}

%

\begin{lemma} \label{lemma:htpy_exp}
  Let $\lie g$ be a nilpotent dg Lie algebra.
  Then there is a natural isomorphism
  \[ \hg 0 \bigl( \expb (\lie g) \bigr)  \iso  \exp \bigl( \Ho 0 (\lie g) \bigr) \]
  of groups, and there is, for any $k \in \NNpos$, a natural isomorphism
  \[ \hg k \bigl( \expb (\lie g) \bigr)  \iso  \Ho k (\lie g) \]
  of abelian groups.
  
  Moreover, the first isomorphism fits into a commutative diagram
  \[
  \begin{tikzcd}
  \exp[0] (\lie g) \rar[equal]  \dar[swap, two heads] & \exp (\lie g_0) \dar[two heads] \\
  \hg 0 \bigl( \expb (\lie g) \bigr) \rar{\iso} & \exp \bigl( \Ho 0 (\lie g) \bigr)
  \end{tikzcd}
  \]
  of groups.
\end{lemma}

\begin{proof}
  The first paragraph is \cite[Theorem~6.2]{BerMC}; the second follows from the explicit description of the isomorphism provided there.
\end{proof}

We conclude this subsection by recalling the Chevalley--Eilenberg (co)homology of a dg Lie algebra, and its relation to the realization functor $\MCb$ above.

\begin{notation}
  Let $\lie g$ be a dg Lie algebra and $M$ a vector space.
  We denote by $\CEchains * (\lie g)$ the \emph{Chevalley--Eilenberg chain complex} of $\lie g$ (see e.g.\ \cite[§2.1]{BFMT}) and by $\CEcochains * (\lie g; M)  \defeq  \Hom (\CEchains * (\lie g), M)$ the \emph{Chevalley--Eilenberg cochain complex} with coefficients in $M$.
  We furthermore write $\CEho * (\lie g)$ and $\CEcoho * (\lie g; M)$ for the homology of the respective (co)chain complex.
\end{notation}

\begin{lemma} \label{lemma:CE_products}
  Let $\lie g$ and $\lie h$ be dg Lie algebras, and let $M$ and $N$ be two vector spaces.
  Then the canonical map
  \[ \CEcochains * (\lie g; M) \tensor \CEcochains * (\lie h; N)  \longto  \CEcochains * (\lie g \times \lie h; M \tensor N) \]
  together with the canonical isomorphism $\QQ \to \CEcochains * (0; \QQ)$ equips $\CEcochains *$ with the structure of a lax symmetric monoidal functor $\opcat{\dgLie} \times \Vect \to \coCh$.
  Furthermore, if $\lie g$ and $\lie h$ are non-negatively graded and of finite type, then this map is an isomorphism.
\end{lemma}

\begin{proof}
  The first part follows from an elementary verification.
  For the second part, we note that there is a canonical isomorphism $\CEchains * (\lie g \times \lie h) \iso \CEchains * (\lie g) \tensor \CEchains * (\lie h)$ and that our conditions guarantee that both factors are finite dimensional in each degree.
\end{proof}

The following is a slight generalization of \cite[Proposition~2.21]{BZ}, incorporating coefficients and monoidality.

\begin{lemma} \label{lemma:MC_forms}
  Let $\lie g$ be a nilpotent dg Lie algebra of finite type and $M$ a vector space.
  Then there is a monoidal natural quasi-isomorphism
  \[ \Phi \colon \CEcochains * (\lie g; M)  \xlongto{\eq}  \Forms * \bigl( \MCb(\lie g); M \bigr) \]
  of lax symmetric monoidal functors $\opcat{(\dgLienilft)} \times \Vect \to \coCh$.
\end{lemma}

\begin{proof}
  First note that there is a canonical monoidal natural isomorphism of lax symmetric monoidal functors $\opcat \sSet \times \Vect \to \coCh$
  \[ \Forms * (X; M)  \xlongto{\iso}  \Hom[\sSet]( X, M \tensor \sForms ) \]
  where the monoidal structure on the right-hand side comes from the algebra structure of $\sForms$.
  Using this identification, we define the image $\Phi(\alpha)$ of a linear map $\alpha \colon \CEchains p (\lie g) \to M$ to be the composite map of simplicial sets
  \[ \MCb(\lie g)  \xlongto{\Psi}  \bigl( \CEchains * (\lie g) \tensor \sForms \bigr)_0  \xlongto{\alpha \tensor \id}  M \tensor \Omega_\bullet^p \]
  where $\Psi$ is the map constructed by Berglund in \cite[Corollary~3.6]{Ber}.
  It is an elementary verification that the map $\Psi$ is a monoidal natural transformation of lax symmetric monoidal functors in $\lie g$, and hence $\Phi$ is a monoidal natural transformation as well.
  
  We will now prove that $\Phi$ is a quasi-isomorphism when $M = \QQ$.
  By \cite[Corollary~3.6 and Proposition~3.7]{Ber}, the composite map of simplicial sets
  \[ \MCb(\lie g)  \xlongto{\Psi}  \bigl( \CEchains * (\lie g) \tensor \sForms \bigr)_0  \longto  \Hom \bigl( \CEcochains * (\lie g), \sForms \bigr) \]
  is an isomorphism $\Psi'$ onto the simplicial subset of the pointwise cochain algebra homomorphisms.
  Furthermore, note that $\CEcochains * (\lie g)$ is a Sullivan algebra (see e.g.\ \cite[Lemma~23.1]{FHT}) and hence a cofibrant cochain algebra.
  This implies that the canonical map
  \[ \CEcochains * (\lie g)  \xlongto{\eq}  \Forms * \bigl( \Hom[\Mon(\coCh)] \bigl( \CEcochains * (\lie g), \sForms \bigr) \bigr) \]
  is a quasi-isomorphism (see e.g.\ \cite[Theorem~9.4]{BG}).
  Since this map is identified with $\Phi$ under the isomorphism $\Psi'$, this implies the claim.
  
  Lastly, note that the functor $\CEcochains * (\lie g; M)$ preserves products in $M$, and hence also does so after taking cohomology.
  After taking cohomology, the codomain of $\Phi$ is isomorphic to $\Coho * ( \MCb(\lie g); M )$ by \cref{lemma:forms_cohomology}, and this functor also preserves products in $M$.
  This implies the statement for arbitrary $M$ since every vector space is a retract of a product of copies of $\QQ$, and isomorphisms are closed under retracts.
\end{proof}

\subsection{Algebraic groups}

\begin{notation}
  Let $R$ be a commutative ring.
  We denote by $\Algfg{R}$ the category of finitely generated commutative $R$-algebras.
\end{notation}

We will assume basic familiarity with the theory of algebraic groups, for example as presented in the textbook of Milne \cite[§§1--5]{Mil}.
We use the functor-of-points perspective, i.e.\ we consider the category of algebraic groups over a field $k$ to be a full subcategory of the category of functors from $\Algfg{k}$ to the category $\Grp$ of groups.

\begin{convention}
  By an \emph{algebraic group} we will mean a linear algebraic group over $\QQ$.
\end{convention}

Recall that an algebraic group is linear if and only if it is affine (see e.g.\ \cite[Remark~4.11]{Mil}).
Furthermore note that subgroups and quotients of affine algebraic groups are again affine (see e.g.\ \cite[Proposition~5.18]{Mil}), so that our notion of algebraic group is closed under these constructions.

\begin{definition}
	For $R$ a commutative ring and $M$ an $R$-module, we denote by $\GL^R(M)$ the group of $R$-module automorphisms of $M$.
	For $V$ a finite-dimensional vector space, we write $\aGL(V)$ for the algebraic group given on objects by $R \mapsto \GL^R(V \tensor R)$ and on morphisms by extension of scalars.
  We set $\aGL[n] \defeq \aGL(\QQ^n)$.
\end{definition}

We will also need the notion of an arithmetic subgroup of an algebraic group, see e.g.\ \cite{Ser} for more background.

\begin{definition}
	Let $G$ be an algebraic group.
	A subgroup $H \subseteq G(\QQ)$ is an \emph{arithmetic subgroup} of $G$ if there exists an $n \in \NN$ and an embedding $\iota \colon G \to \aGL[n]$ such that $\iota_\QQ(H) \intersect \GL[n](\ZZ)$ is a finite-index subgroup of both $\iota_\QQ(H)$ and $\GL[n](\ZZ)$.
\end{definition}

Recall that an algebraic group is called \emph{unipotent} when all of its simple algebraic representations are trivial.
Furthermore recall that the class of unipotent algebraic groups is closed under taking subgroups, quotients, and extensions (see e.g.\ \cite[Corollary~14.7]{Mil}), and that unipotent algebraic groups are connected (see e.g.\ \cite[Corollary~14.15]{Mil}).

\begin{definition}
  Let $G$ be an algebraic group.
  We call its maximal normal unipotent subgroup the \emph{unipotent radical} and denote it by $\unip G \subseteq G$ (for its existence see e.g.\ \cite[Corollary~14.8]{Mil}.\footnote{Note that affine algebraic groups are smooth (see e.g.\ \cite[Theorem~3.23]{Mil}).}).
\end{definition}

\begin{definition}
  We say that an algebraic group $G$ is \emph{reductive} if $\unip G$ is trivial.\footnote{Note that our definition differs from that of Milne, in that he additionally requires the algebraic group to be connected.}
\end{definition}

Note that an algebraic group $G$ is reductive if and only if its identity component is reductive, and recall that this is the case if and only if $G$ is \emph{linearly reductive}, i.e.\ if every finite-dimensional algebraic representation of $G$ is semi-simple (see e.g.\ \cite[Corollary~22.43]{Mil}).

\begin{lemma} \label{lemma:unipotent_radical}
  Let $G$ be an algebraic group and $U \subseteq G$ a normal unipotent algebraic subgroup.
  Then $U = \unip G$ if and only if the quotient $\quot G U$ is reductive.
\end{lemma}

\begin{proof}
  This follows easily from the fact that the class of unipotent algebraic groups is closed under extensions and quotients.
\end{proof}

We now recall the Levi decomposition of an algebraic group.
It plays an important role in this article.

\begin{definition}
  Let $G$ be an algebraic group.
  We call $\quot G {\unip G}$ the \emph{maximal reductive quotient} of $G$.
  A \emph{Levi subgroup} of $G$ is an algebraic subgroup $L \subseteq G$ such that the composite $L  \to  G  \to  \quot G {\unip G}$ is an isomorphism.
\end{definition}

Note that a Levi subgroup $L \subseteq G$ yields a section $\quot G {\unip G} \iso L \to G$ of the projection.
Moreover, a Levi subgroup always exists; that is the content of the following classical result of Mostow \cite{Mos}.

\begin{proposition}[Mostow] \label{prop:Mostow}
  Let $G$ be an algebraic group and $R \subseteq G$ a reductive algebraic subgroup.
  Then there exists a Levi subgroup $L \subseteq G$ such that $R \subseteq L$.
  In particular, any algebraic group has a Levi subgroup.
\end{proposition}

\begin{proof}
  By \cite[Ch.~VIII, Theorem~4.3]{Hoc}, there exists a Levi subgroup $L' \subseteq G$ and an element $t \in \unip G(\QQ)$ such that $t R \inv t \subseteq L'$.
  Hence $R$ is contained in $L \defeq \inv t L' t$, which is again a Levi subgroup.
\end{proof}

Note that forming the unipotent radical is not functorial, and hence neither is the maximal reductive quotient.
However, the following lemma shows that, when a map $G \to H$ of algebraic groups does induce a map $\quot G {\unip G} \to \quot H {\unip H}$, then we can choose compatible sections of the projection to the maximal reductive quotient.

\begin{lemma} \label{lemma:compatible_Levi}
  Let the following be a commutative diagram of algebraic groups
  \[
  \begin{tikzcd}
    & H \dar{p} \\
    R \rar{g} \urar[bend left]{f} & \quot H {\unip H}
  \end{tikzcd}
  \]
  such that $p$ is the quotient map and $R$ is reductive.
  Then there exists a section $s$ of $p$ such that $s \after g = f$.
\end{lemma}

\begin{proof}
  Since $R$ is linearly reductive, the image of $f$ is again a linearly reductive subgroup of $H$.
  By \cref{prop:Mostow}, there exists a Levi subgroup $L \subseteq H$ containing this image.
  Hence we have $\restrict p L \after f = g$ and thus, setting $s \defeq \inv {( \restrict p L )} \colon \quot H {\unip H} \to L \subseteq H$, that $f = s \after g$.
\end{proof}

We now recall the Lie algebra of an algebraic group and, further below, the exponential algebraic group of a nilpotent Lie algebra.
They form an important part of the structure theory of algebraic group in characteristic $0$.

\begin{notation}
  We write $\Qeps$ for the commutative $\QQ$-algebra $\quot {\QQ[\epsilon]} {\epsilon^2}$.
  Given an algebraic group $G$, we denote by
  \[ \Lie(G)  \defeq  \ker \bigl( G(\Qeps) \to G(\QQ) \bigr) \]
  the kernel of the map induced by the unique map of algebras sending $\epsilon$ to $0$, equipped with its usual Lie algebra structure (see e.g.\ \cite[§10b]{Mil}).
  It assembles into a functor $\Lie(\blank)$ from algebraic groups to Lie algebras.
\end{notation}

The following lemma describes the Lie algebra of the general linear group.
In fact, this uniquely determines the Lie algebra structure on $\Lie(G)$ for all $G$.

\begin{lemma} \label{lemma:Lie_GL}
  Let $V$ be a finite-dimensional vector space.
  Then there is a natural isomorphism of Lie algebras
  \[ \Lie \bigl( \aGL(V) \bigr)  \xlongto{\iso}  \gl(V) \]
  given by sending $f \in \Lie \bigl( \aGL(V) \bigr) \subseteq \GL^\Qeps \bigl( V \tensor \Qeps \bigr)$ to the unique endomorphism $g \in \End(V) = \gl(V)$ such that $f(v \tensor 1) = v \tensor 1 + g(v) \tensor \epsilon$ for all $v \in V$.
\end{lemma}

\begin{proof}
  This is contained in \cite[10.7]{Mil}.
\end{proof}

\begin{definition}
  Given a finite-dimensional nilpotent Lie algebra $\lie g$, we denote by $\aexp(\lie g)$ the algebraic group given, at an object $R \in \Algfg{\QQ}$, by $\aexp(\lie g)(R) \defeq \exp(\lie g \tensor R)$ and by scalar extension on morphisms (see e.g.\ \cite[Theorem~14.37, (a)]{Mil}).
\end{definition}

Recall that $\aexp(\blank)$ is an equivalence of categories from finite-dimensional nilpotent Lie algebras to unipotent algebraic groups, and that its inverse (up to natural isomorphisms) is given by the appropriate restriction of $\Lie(\blank)$ (see e.g.\ \cite[Theorem~14.37, (b)]{Mil}\footnote{Note that a unipotent algebraic group is automatically affine (see e.g.\ \cite[Theorem~14.5]{Mil}).}).
In particular $\Lie(\blank)$ takes unipotent algebraic groups to finite-dimensional nilpotent Lie algebras.
The following lemma provides an explicit description of the natural isomorphism relating one of the composites of the equivalence to the identity functor.

\begin{lemma} \label{lemma:Lie_exp}
  Let $\lie g$ be a finite-dimensional nilpotent Lie algebra.
  Then there is a natural isomorphism of Lie algebras
  \[ \lie g  \xlongto{\iso}  \Lie \bigl( \aexp(\lie g) \bigr) \]
  given by $x \mapsto x \tensor \epsilon \in \lie g \tensor \Qeps = \aexp(\lie g) \bigl( \Qeps \bigr)$.
\end{lemma}

\begin{proof}
  This is contained in \cite[IV, §2, Proof of 4.4]{DG}.
\end{proof}

\section{Algebraicity of relative self-equivalences}

In this section, we recall results of Espic--Saleh \cite{ES} on the algebraicity of the groups $\Eaut[\rat A](\rat X)$ of homotopy classes of self-equivalences of a rational space $\rat X$ relative to a subspace $\rat A$, while reformulating them using the functor-of-points perspective on algebraic groups.
(In the absolute case, this goes back to Sullivan \cite[Theorem~10.3]{Sul} and Wilkerson \cite[Theorem~B]{Wil}.)
In addition, we prove that various constructions are compatible with this algebraic group structure.

\subsection{Minimal relative dg Lie algebras}

We will need a relative version of the classical notion of a minimal model of a dg Lie algebra.
This has been introduced and studied by Espic--Saleh \cite[§3]{ES} (inspired by work of Cirici--Roig \cite{CR}).
As usual for minimal models, their fundamental property is that a (relative) quasi-isomorphism is already an isomorphism (cf.\ \cite[Theorem~3.8]{ES}).

\begin{definition}
  A quasi-free map $L' \to L$ of non-negatively graded dg Lie algebras is \emph{minimal} if the differential of the relative indecomposables $\indec[L'](L)$ of \cref{def:indec} is trivial.
  A quasi-free dg Lie algebra $L$ is \emph{minimal} when the map $0 \to L$ is.
\end{definition}

\begin{remark}
  Our definition of a minimal quasi-free map differs from that of Espic--Saleh \cite[Definition~3.4]{ES}: they require $L'$ to be quasi-free and $L$ to be a colimit of ``KS-extensions'', but do not require the dg Lie algebras to be non-negatively graded.
  However, in the case that $L'$ is quasi-free and $L'$ and $L$ are non-negatively graded, the two definitions agree: any quasi-free map $L' \to L$ can be written as a colimit
  \[ L'  \longto  L' \cop \freelie (V_0)  \longto  L' \cop \freelie (V_0 \dirsum V_1)  \longto  \cdots \]
  of KS-extensions, where $V$ is as in \cref{def:quasi-free} and $V_k$ denotes the part of homological degree $k$.
  (To see that their condition on the induced differential on $V$ agrees with our condition on the differential of $\indec[L'](L)$, one uses \cref{lemma:quasi-free_indec}.)
\end{remark}

Much of what we do below relies on having a minimal dg Lie algebra model for a map of spaces.
Fortunately such models always exist, as the following lemma attests.

\begin{lemma}[Espic--Saleh]
  Let $f \colon L' \to L$ be a map of positively graded dg Lie algebras.
  Then there exists a commutative diagram of dg Lie algebras
  \[
  \begin{tikzcd}
    \widetilde L' \rar{\eq} \dar[swap]{i} & L' \dar{f} \\
    \widetilde L \rar{\eq} & L
  \end{tikzcd}
  \]
  such that $\widetilde L'$ and $i$ are minimal quasi-free and the horizontal maps are quasi-isomorphisms.
\end{lemma}

\begin{proof}
  By the existence of minimal dg Lie models (see e.g.\ \cite[Theorem~22.13]{FHT}), there exists a quasi-isomorphism $\widetilde L' \to L'$ with $\widetilde L'$ positively graded and minimal quasi-free.
  Then the claim follows from \cite[Theorem~3.7]{ES}.
\end{proof}

\subsection{Automorphisms and relative self-equivalences of dg Lie algebras}

In this subsection, we recall how to equip the group $\Aut[L'](L)$ of automorphisms of a dg Lie algebra $L$ relative to $L'$ with an algebraic group structure, and how to use this to do the same for groups of the form $\Eaut[\rat A](\rat X)$.

\begin{definition}
  Let $R$ be a commutative $\QQ$-algebra and $i \colon L' \to L$ and $\rho \colon L \to \Pi$ maps of dg Lie algebras over $R$.
  We denote by $\Aut[L']^R(L)_\rho$ the group of automorphisms $\phi \colon L \to L$ of dg Lie algebras over $R$ such that $\phi \after i = i$ and $\rho \after \phi = \rho$.
  When $R = \QQ$, $L' = 0$, or $\rho = 0$, we omit the respective term from the notation.
\end{definition}

\begin{definition}
  Let $L' \to L$ and $\rho \colon L \to \Pi$ be maps dg Lie algebras.
  We write $\aAut[L'](L)_\rho$ for the functor $\Algfg{\QQ} \to \Grp$ given on objects by $R \mapsto \Aut[L' \tensor R]^R(L \tensor R)_{\rho \tensor R}$ and on morphisms by extension of scalars.
  When $L' = 0$ or $\rho = 0$, we omit it from the notation.
\end{definition}

The following is a slight generalization of \cite[Proposition~4.10]{ES}.
We provide a proof for completeness; it is essentially the same argument as the one of Espic--Saleh, though formulated slightly differently.

\begin{proposition}[Espic--Saleh] \label{prop:aAut_alg}
  Let $f \colon L' \to L$ and $\rho \colon L \to \Pi$ be maps of dg Lie algebras such that $L'$ is finitely generated, $L$ is positively graded and finitely presented (i.e.\ it is the quotient of a finitely generated dg Lie algebra by a finitely generated dg Lie ideal), and $\Pi$ is of finite type.
  Then $\aAut[L'](L)_\rho$ is an algebraic group.
\end{proposition}

\begin{proof}
  Let $m$ be the maximum of the degrees of the generators of $L$ and $L'$ and the degrees of the relations of $L$.
  Then $\aAut[L'](L)$ embeds into the algebraic group $G \defeq \prod_{k = 0}^m \aGL(L_k)$.
  (Note that $L_k$ is finite dimensional since $L$ is finitely generated and positively graded.)
  Let $S \subseteq G$ be the setwise stabilizer of the subspaces $f(L'_k)$ for $k \le m$; this is an algebraic subgroup (see e.g.\ \cite[Proposition~4.3]{Mil}).
  There is an algebraic representation of $S$ on $f(L'_{\le m})$, and we denote by $T$ its kernel; it is the pointwise stabilizer in $G$ of the subspaces $f(L'_k)$.
  Moreover, note that conjugation and precomposition yield an algebraic representation of $G$ on
  \[ \Hom(\shift[-1] L_{\le m}, L_{\le m}) \dirsum \Hom \bigl( (L \tensor L)_{\le m}, L_{\le m} \bigr) \dirsum \Hom(L_{\le m}, \Pi_{\le m}) \]
  where $\Hom$ denotes the vector space of graded linear maps of degree $0$.
  Then we denote by $K$ the stabilizer of the element $(d_L, \liebr \blank \blank, \rho)$; as above, this is an algebraic subgroup.
  Note that $\aAut[L'](L)_\rho = T \intersect K$; in particular it is an algebraic group.
  Here we use that $L$ is presented in degrees $\le m$ to deduce that an automorphism of $L_{\le m}$ of degree $0$ that preserves the Lie bracket extends to an automorphism of the underlying graded Lie algebra of $L$.
\end{proof}

We now recall how to use this algebraic group structure to obtain an algebraic group structure on $\Eaut[\nerve{L_A}] ( \nerve{L_X} )$, in the case that $L_A \to L_X$ is minimal quasi-free.
More precisely, we recall how to lift the map of the following definition to a morphism of algebraic groups.

\begin{definition} \label{def:real}
  Let $L_B \to L_A \to L_X$ and $\rho \colon L_X \to \Pi$ be maps of dg Lie algebras.
  We write
  \[ \real \colon \Aut[L_A](L_X)_\rho  \longto  \Eaut[\nerve{L_A}] \bigl( \nerve{L_X} \bigr)_{\eqcl{\nerve{\rho}}_{\nerve{L_B}}} \]
  for the group homomorphism induced by the functor $\nerve{\blank}$.
\end{definition}

\begin{definition}
  Let $R$ be a commutative $\QQ$-algebra and $L_A \to L_X$ a quasi-free map of positively graded dg Lie algebras over $R$.
  We denote by $\Null[L_A]^R(L_X) \subseteq \Aut[L_A]^R(L_X)$ the subgroup of those automorphisms $f$ such that $f \eq_{L_A}^R \id[L_X]$.
  
  For $L_A \to L_X$ a quasi-free map of positively graded dg Lie algebras, we write $\aNull[L_A](L_X) \subseteq \aAut[L_A](L_X)$ for the subfunctor given on objects by $R \mapsto \Null[L_A \tensor R]^R(L_X \tensor R)$.
\end{definition}

\begin{proposition}[Espic--Saleh] \label{prop:ES}
  Let $L_A \to L_X$ be a minimal quasi-free map of positively graded finitely generated quasi-free dg Lie algebras.
  Then $\aNull[L_A](L_X) \subseteq \aAut[L_A](L_X)$ is a normal unipotent algebraic subgroup.
\end{proposition}

\begin{proof}
  This is contained in the proof of \cite[Theorem~4.13]{ES}.
\end{proof}

\begin{remark}
  It is essential for $L_A \to L_X$ to be minimal to obtain that $\aNull[L_A](L_X)$ is unipotent.
  As a counterexample, consider the dg Lie algebra $L \defeq \freelie (v, dv)$ with $\deg v = 3$ and differential $d(v) = dv$.
  This is quasi-free and contractible and hence $\aNull(L) = \aAut(L)$.
  This is isomorphic to the multiplicative algebraic group, which is not unipotent.
\end{remark}

Note that, if $\rho \colon L_X \to \Pi$ is a map of dg Lie algebras such that the differential of $\Pi$ is trivial, then $\aNull[L_A](L_X) \subseteq \aAut[L_A](L_X)_\rho$ (by \cref{lemma:dgLie_homotopic}).
Hence we can define the following.

\begin{definition} \label{def:areal}
  Let $L_A \to L_X$ be a minimal quasi-free map of positively graded finitely generated quasi-free dg Lie algebras, and $\rho \colon L_X \to \Pi$ a map of dg Lie algebras such that $\Pi$ is of finite type and has trivial differential.
  We define an algebraic group
  \[ \aEaut[L_A]( L_X )_\rho  \defeq  \quot {\aAut[L_A](L_X)_\rho} {\aNull[L_A](L_X)} \]
  and write $\areal \colon \aAut[L_A](L_X)_\rho \to \aEaut[L_A] ( L_X )_\rho$ for the quotient map.
  When $\rho = 0$, we omit it from the notation.
\end{definition}

\begin{proposition}[Espic--Saleh] \label{prop:aEaut}
  Let $L_B \to L_A \to L_X$ be quasi-free maps of positively graded finitely generated quasi-free dg Lie algebras such that $L_A \to L_X$ is minimal, and let $\rho \colon L_X \to \Pi$ be a map of dg Lie algebras such that $\Pi$ is of finite type, positively graded, and has trivial differential.
  Then the group homomorphism $\real$ of \cref{def:real} is a quotient map and there is a unique isomorphism of groups
  \[ \Phi  \colon  \aEaut[L_A] ( L_X )_\rho (\QQ)  \xlongto{\iso}  \Eaut[\nerve{L_A}] \bigl( \nerve{L_X} \bigr)_{\eqcl{\nerve{\rho}}_{\nerve{L_B}}} \]
  such that $\Phi \after \areal_\QQ = \real$.
  In particular the codomain is independent of $L_B$.
\end{proposition}

\begin{proof}
  In the case that $\Pi = 0$, this follows from \cite[(Proof of) Corollary~4.8]{ES} and the fact that, given a normal unipotent algebraic subgroup $U$ of an algebraic group $G$, the canonical map $\quot{G(\QQ)}{U(\QQ)} \to (\quot G U)(\QQ)$ is a group isomorphism (see e.g.\ \cite[Lemma~2.9]{BZ}).
  
  By \cref{lemma:map_MCb}, the functor $\nerve{\blank}$ induces an isomorphism
  \[ \hmap[L_B]{L_X}{\Pi}  \xlongto{\iso}  \hmap[\nerve{L_B}]{\nerve{L_X}}{\nerve{\Pi}} \]
  and this is equivariant with respect to the isomorphism
  \[ \quot {\Aut[L_A](L_X)} {\Null[L_A](L_X)}  \xlongto{\iso}  \Eaut[\nerve{L_A}] ( \nerve{L_X} ) \]
  by construction.
  Hence the latter restricts to an isomorphism of the stabilizer groups of the elements $\eqcl {\rho}$ and $\eqcl {\nerve{\rho}}$, respectively.
  Since the differential of $\Pi$ is trivial, we have $\Null[L_A](L_X) \subseteq \Aut[L_A](L_X)_\rho$, and the quotient is isomorphic to the stabilizer of $\eqcl {\rho}$ in $\aEaut[L_A](L_X)(\QQ)$.
  This implies the claim.
\end{proof}

Combining the preceding proposition with \cref{lemma:Eaut_map}, there is, for $B \to A \to X$ cofibrations and $\xi \colon X \to K$ a map of Kan complexes respectively modeled by $L_B \to L_A \to L_X$ and $\rho$, a canonical isomorphism
\[ \aEaut[L_A] ( L_X )_\rho (\QQ)  \iso  \Eaut[\rat A] (\rat X)_{\eqcl{\rat \xi}_{\rat B}} \]
of groups.
In particular the right-hand side is independent of $B$.

We conclude this subsection by deducing from work of Kupers \cite{Kup} and Espic--Saleh \cite{ES} that the group $\Eaut[A](X)_{\eqcl{\xi}_B}$ is arithmetic.

\begin{proposition}[Espic--Saleh, Kupers] \label{prop:arithmetic}
  Let $B \to A \to X$ be cofibrations of simply connected Kan complexes of the homotopy types of finite CW-complexes, modeled by the quasi-free maps $L_B \to L_A \to L_X$ of positively graded finitely generated quasi-free dg Lie algebras such that $L_A \to L_X$ is minimal.
  Furthermore let $K$ be a simple Kan complex, and $\xi \colon X \to K$ a map modeled by the map $\rho \colon L_X \to \Pi$ of dg Lie algebras such that $\Pi$ is of finite type, positively graded, and has trivial differential.
  Then the map induced by rationalization
  \[ q  \colon  \Eaut[A](X)_{\eqcl{\xi}_B}  \longto  \Eaut[\rat A](\rat X)_{\eqcl{\rat \xi}}  \iso  \aEaut[L_A](L_X)_\rho(\QQ) \]
  has finite kernel and image an arithmetic subgroup of $\aEaut[L_A](L_X)_\rho$.
\end{proposition}

\begin{proof}
  Consider the following pullback diagram of groups
  \[
  \begin{tikzcd}
  	\Eaut[A](X)_{\eqcl{\rat \xi}} \rar[hook] \dar & \Eaut[A](X) \dar \\
  	\Eaut[\rat A](\rat X)_{\eqcl{\rat \xi}} \rar[hook] & \Eaut[\rat A](\rat X)
  \end{tikzcd}
  \]
  where the right-hand vertical map is induced by rationalization.
  The right-hand vertical map has finite kernel and image an arithmetic subgroup of $\aEaut[L_A](L_X)$ by \cite[Proposition~3.4]{Kup} (that the kernel of $q$ is finite was already contained in \cite[Theorem~2.3]{ES}).
  Since the lower horizontal map is the map of $\QQ$-points of an embedding of algebraic groups by \cref{prop:ES}, and the square is a pullback, the left-hand vertical map also has a finite kernel and image an arithmetic subgroup of $\aEaut[L_A](L_X)_\rho$ (see e.g.\ \cite[§1.1]{Ser}).
  
  Now consider the diagram of pointed sets
  \[
  \begin{tikzcd}
  	\Eaut[A](X)_{\eqcl{\xi}_B} \rar[hook] \dar & \Eaut[A](X) \dar[equal] \rar{\xi_*} & \hmap[B]{X}{K} \dar \\
  	\Eaut[A](X)_{\eqcl{\rat \xi}} \rar[hook] & \Eaut[A](X) \rar{(\rat \xi)_*} & \hmap[\rat B]{\rat X}{\rat K}
  \end{tikzcd}
  \]
  where the rows are exact by definition.
  By the same argument as in \cite[Proof of Theorem~2.3]{ES}, the map $\hmap[B] X K \to \hmap[\rat B] {\rat X} {\rat K}$ has finite fibers whenever $B \to X$ is a cofibration of nilpotent simplicial sets with the homotopy types of finite CW-complexes, and $K$ is a simple Kan complex.
  This implies that the left-hand vertical map is the inclusion of a finite-index subgroup.
  Hence its image in $\Eaut[\rat A](\rat X)_{\eqcl{\rat \xi}}$ is an arithmetic subgroup of $\aEaut[L_A](L_X)_\rho$ as well.
\end{proof}

\subsection{Algebraic representations}

In this subsection, we prove that various representations of the groups $\Aut[L_A](L_X)$ and $\Eaut[\rat A](\rat X)$ lift to algebraic representations of the corresponding algebraic groups.
We begin with recalling a result of Kupers, stating that this is the case for the representation on relative (co)homology.

\begin{lemma}[Kupers] \label{lemma:rel_Ho_algebraic}
  Let $A \to X$ be a cofibration of simply connected Kan complexes of the homotopy types of finite CW-complexes, modeled by the minimal quasi-free map $L_A \to L_X$ of positively graded finitely generated quasi-free dg Lie algebras.
  Then there is, for any $n$, a canonical algebraic representation of $\aEaut[L_A](L_X)$ on $\Ho n (X, A; \QQ)$ such that restricting this action along $\Eaut[A](X)  \to  \Eaut[\rat A](\rat X)  \iso  \aEaut[L_A](L_X)(\QQ)$ yields the usual action of $\Eaut[A](X)$ on $\Ho * (X, A; \QQ)$.
\end{lemma}

\begin{proof}
  For cohomology, this is \cite[Proposition~3.6]{Kup}; the proof for homology is analogous.
\end{proof}

We now prove that the action on the relative indecomposables lifts to an algebraic representation with unipotent kernel (for the absolute case, see \cite[Theorem~2.24, (iv)]{BZ}).
Below we will see that, when the map of dg Lie algebras is minimal, this representation is actually isomorphic to the representation on the relative homology.

\begin{lemma} \label{lemma:indecomposables_representation}
  Let $i \colon L' \to L$ be a map of positively graded finitely generated quasi-free dg Lie algebras.
  Then the action of $\Aut[L'](L)$ on $\indec[L'](L)$ canonically lifts to a map of algebraic groups
  \[ \Koppa \colon \aAut[L'](L)  \longto  \aAut \bigl( \indec[L'](L) \bigr) \]
  and the kernel of $\Koppa$ is unipotent.
  
  Furthermore, assume that $i$ is minimal quasi-free.
  Then there is a unique map $\koppa$ of algebraic groups such that the composite
  \[ \aAut[L'](L)  \xlongto{\areal}  \aEaut[L'](L)  \xlongto{\koppa}  \aAut \bigl( \indec[L'](L) \bigr) \]
  is equal to $\Koppa$, and the kernel of $\koppa$ is unipotent as well.
\end{lemma}

\begin{proof}
  First note that $\aAut ( \indec[L'](L) )$ is an algebraic group by \cref{prop:aAut_alg}.
  Now recall that $\indec[L'](L)$ was defined to be the quotient $\quot {\ol L} {\liebr {\ol L} {\ol L}}$, where $\ol L$ is the quotient of $L$ by the Lie ideal generated by $L'$.
  Let $m$ be the maximum of the degrees of the generators of $L$.
  Then $L_{\le m}$ is an algebraic representation of $\aAut[L'](L)$, and $\indec[L'](L)_{\le m} = \indec[L'](L)$ is obtained from it by taking quotients by algebraic subrepresentations.
  This yields a map of algebraic groups $\aAut[L'](L) \to \aGL(\indec[L'](L))$, which factors through the desired map $\Koppa$.
  
  We now prove that the kernel $K$ of $\Koppa$ is unipotent.
  To this end, we denote by $W_{L'}^{i,j} L \subseteq L$ the graded linear subspace spanned by those iterated Lie brackets of at least $i$ homogeneous elements of $L$ such that at least $j$ of these elements lie in the image of $L'$.
  Then we define
  \begin{equation} \label{eq:sub_word-length_filtration}
    F_{L'}^n L  \defeq  \sum_{\substack{i \ge 1, j \ge 0 \\ i + j = n}}  W_{L'}^{i,j} L
  \end{equation}
  as graded linear subspaces of $L$.
  This yields a descending filtration of $L = F_{L'}^1 L$ by graded linear subspaces.
  The filtration is Hausdorff since $L$ is positively graded, and it is preserved by any dg Lie algebra automorphism of $L$ that preserves $L'$.
  
  Now let $R \in \Algfg{\QQ}$ and $\phi \in K(R)$.
  Then, for any $x \in L \tensor R$, we have $\phi(x) = x + l' + l$ for some $l' \in L' \tensor R$ and $l \in \liebr {L \tensor R} {L \tensor R}$ such that $l = 0 = l'$ if $x \in L' \tensor R$.
  This implies that $\phi$ acts trivially on the associated graded of the filtration $F_{L' \tensor R}^n (L \tensor R)$.
  Also note that $W_{L' \tensor R}^{i,j} (L \tensor R) = (W_{L'}^{i,j} L) \tensor R$ and hence $F_{L' \tensor R}^n (L \tensor R) = (F_{L'}^n L) \tensor R$.
  
  Let $m$ be the maximal degree of generators of $L$.
  Then the vector space $L_{\le m}$ is a faithful finite-dimensional (since $L$ is positively graded) algebraic representation of $\aAut[L'](L)$, and hence of $K$.
  Moreover, the argument above proves that the filtration \eqref{eq:sub_word-length_filtration} exhibits $L_{\le m}$ as a unipotent algebraic representation of $K$, and hence that $K$ is unipotent (see e.g.\ \cite[Corollary~14.6]{Mil}).
  
  We now prove that $\Koppa$ factors through $\aEaut[L'](L)$ when $i$ is minimal.
  To this end, let $R \in \Algfg{\QQ}$ and $\phi \in \aAut[L'](L)(R)$.
  Since the functor $\blank \tensor \sForms[1]$ preserves quotients, by \cref{lemma:dgLie_homotopy} a homotopy $\phi \eq^R_{L' \tensor R} \id$ implies that $\Koppa_R(\phi) \eq^R \id$.
  Since $i$ is minimal, the differential of $\indec[L'](L) \tensor R$ is trivial and hence this implies, by \cref{lemma:dgLie_homotopic}, that $\Koppa_R(\phi) = \id$.
  This shows that $\aNull[L'](L) \subseteq \ker(\Koppa)$, and hence that $\koppa$ exists as claimed.
  
  Since $\areal$ is a quotient map, so is $\ker(\Koppa) \to \ker(\koppa)$.
  As $\ker(\Koppa)$ is unipotent, this implies that $\ker(\koppa)$ is unipotent as well (see e.g.\ \cite[Corollary~14.7]{Mil}).
\end{proof}

\begin{lemma} \label{lemma:indec_is_ho}
  Let $A \to X$ be a cofibration of simply connected Kan complexes of the homotopy types of finite CW-complexes, modeled by the minimal quasi-free map $i \colon L_A \to L_X$ of positively graded finitely generated quasi-free dg Lie algebras.
  Then there is an isomorphism of graded $\aEaut[L_A](L_X)(\QQ)$-representations
  \[ \shift[-1] \Ho * (X, A; \QQ)  \iso  \indec[L_A](L_X) \]
  with the actions provided by \cref{lemma:rel_Ho_algebraic,lemma:indecomposables_representation}.
\end{lemma}

\begin{proof}
  We pick isomorphisms $L_A \iso \freelie V$ and $L_X \iso L_A \cop \freelie W$ of graded Lie algebras and graded Lie algebras under $L_A$, respectively.
  Then there is a commutative diagram
  \[
  \begin{tikzcd}
    \CEchains * (L_A) \dar[hook] \rar[two heads]{\eq} & \QQ \dirsum \shift \indec(L_A) \dar \rar{\iso} & \QQ \dirsum \shift V \dar[hook] \\
    \CEchains * (L_X) \rar[two heads]{\eq} & \QQ \dirsum \shift \indec(L_X) \rar{\iso} & \QQ \dirsum \shift (V \dirsum W)
  \end{tikzcd}
  \]
  where the the right-hand isomorphisms are provided by \cref{lemma:quasi-free_indec} and the left-hand horizontal maps are given by projecting onto the cochains of wedge degree $0$ and $1$, and then projecting $\shift L_A$ and $\shift L_X$ further to $\shift \indec(L_A)$ and $\shift \indec(L_X)$, respectively.
  The composites of the rows are well-known to be quasi-isomorphisms (see e.g.\ \cite[Proposition~22.8]{FHT}), and so they induce a quasi-isomorphism
  \[ f  \colon  \quot {\CEchains * (L_X)} {\CEchains * (L_A)}  \xlongto{\eq}  \shift W  \iso  \shift {\indec[L_A](L_X)} \]
  where the last isomorphism is again provided by \cref{lemma:quasi-free_indec} and the differential of the two right-hand terms is trivial since we assumed $i$ to be minimal.
  The map $f$ is equivariant with respect to the canonical actions of $\Aut[L_A](L_X)$.
  After passing to homology, the action on the domain descends to an action of $\Eaut[L_A](L_X)(\QQ)$; this is completely analogous to the case of Chevalley--Eilenberg cochains in \cite[§3.3.4]{Kup}.
  In particular the resulting action is precisely the one of \cref{lemma:rel_Ho_algebraic}.
  This finishes the proof.
\end{proof}

It turns out that the image of the representation on the relative indecomposables is often the maximal reductive quotient of $\aEaut[L'](L)$ (in the absolute case, this is \cite[Corollary~4.2]{BZ}).

\begin{definition} \label{def:ahoEaut}
  Let $i \colon L' \to L$ be a map of positively graded finitely generated quasi-free dg Lie algebras, and $\rho \colon L \to \Pi$ a map of dg Lie algebras such that $\Pi$ is of finite type and has trivial differential.
  We write
  \[ \ahoEaut[L'](L)_\rho  \defeq  \Koppa \bigl( \aAut[L'](L)_\rho \bigr)  \subseteq  \aAut \bigl( \indec[L'](L) \bigr)  \subseteq  \aGL \bigl( \indec[L'](L) \bigr) \]
  for the image of (the restriction of) $\Koppa$ as an algebraic group under $\aAut[L'](L)_\rho$.
  When $i$ is minimal quasi-free, we will also consider $\ahoEaut[L'](L)_\rho = \koppa ( \aEaut[L'](L)_\rho )$ as an algebraic group under $\aEaut[L'](L)_\rho$.
\end{definition}

\begin{lemma} \label{lemma:ahoEaut_reductive}
  Let $L' \to L$ be a minimal quasi-free map of positively graded finitely generated quasi-free dg Lie algebras, and $\rho \colon L \to \Pi$ a map of dg Lie algebras such that $\Pi$ is of finite type and has trivial differential.
  Then the following conditions are equivalent:
  \begin{enumerate}
    \item
    The algebraic representation of $\aAut[L'](L)_\rho$ on $\indec[L'](L)$ is semi-simple.
    \item
    The algebraic representation of $\aEaut[L'](L)_\rho$ on $\indec[L'](L)$ is semi-simple.
    \item
    The algebraic group $\ahoEaut[L'](L)_\rho$ is the maximal reductive quotient of $\aEaut[L'](L)_\rho$.
  \end{enumerate}
\end{lemma}

\begin{proof}
  First note that $\indec[L'](L)$ is semi-simple as a representation of $\aAut[L'](L)_\rho$ if and only if it is semi-simple as a representation of $\aEaut[L'](L)_\rho$ if and only if it is semi-simple as a representation of $\ahoEaut[L'](L)_\rho$.
  In particular 1 and 2 are equivalent.
  That 2 implies 3 follows from \cite[Lemma~2.7]{BZ} and \cref{lemma:indecomposables_representation}.
  To see that 3 implies 1 and 2, note that any finite-dimensional representation of a reductive algebraic group is semi-simple.
\end{proof}

We conclude this subsection by proving that the action of the automorphisms on the derivation dg Lie algebra lifts to an algebraic representation.
This is integral to producing the equivariant algebraic models this paper is about.

\begin{lemma} \label{lemma:Der_algebraic}
  Let $L' \to L$ be a map of finitely generated positively graded quasi-free dg Lie algebras, and let $n \in \ZZ$.
  Then the conjugation action of $\Aut[L'](L)$ on $\Der(L \rel L')_n$ can be canonically extended to a finite-dimensional algebraic representation of the algebraic group $\aAut[L'](L)$.
\end{lemma}

\begin{proof}
  Write $m$ for the maximum of the degrees of the generators of $L$.
  We first note that the vector space $\Der(L \rel L')_n$ embeds into $\Hom (L_{\le m}, L_{\le m + n})$ since a derivation is determined by its values on the generators.
  Hence $\Der(L \rel L')_n$ is finite dimensional.
  
  Moreover there is, for $R \in \Algfg{\QQ}$, a natural isomorphism of $R$-modules
  \[ \Der(L \rel L')_n \tensor R  \xlongto{\iso}  \Der[R](L \tensor R \rel L' \tensor R)_n \]
  given by sending $\theta \tensor r$ to the $R$-linear derivation $\psi$ such that $\psi(x \tensor r') = \theta(x) \tensor r r'$.
  The target has a natural action of $\Aut[L' \tensor R]^R(L \tensor R)$ by conjugation.
  Via the isomorphism above this assembles into a morphism
  \[ \aAut[L'](L)  \longto  \aGL \bigl( \Der(L \rel L')_n \bigr) \]
  of algebraic groups.
\end{proof}

\subsection{The associated Lie algebras}

In this subsection, we recall identifications of the Lie algebras of the algebraic groups $\aAut[L_A](L_X)$ and $\aEaut[L_A](L_X)$.
This is due to Espic--Saleh, following Block--Lazarev \cite[(Proof of) Theorem~3.4]{BL}.

\begin{proposition}[Espic--Saleh] \label{prop:Lie_aEaut}
  Let $i \colon L_A \to L_X$ be a map of positively graded finitely generated quasi-free dg Lie algebras.
  Then there is an isomorphism of Lie algebras
  \begin{gather*}
    \Lie \bigl( \aAut[L_A](L_X) \bigr)  \xlongto{\iso}  \Cycles 0 \bigl( \Der(L_X \rel L_A) \bigr) \\
    \shortintertext{given by sending}
    \varphi \in \Lie \bigl( \aAut[L_A](L_X) \bigr) \subseteq \Aut[L_A \tensor \Qeps]^\Qeps \bigl( L_X \tensor \Qeps \bigr)
  \end{gather*}
  to the unique derivation $\delta \colon L_X \to L_X$ such that the equation
  \[ \varphi(x \tensor 1) = x \tensor 1 + \delta(x) \tensor \epsilon \]
  holds for all $x \in L_X$.
  
  Moreover, if $i$ is minimal quasi-free, then there exist unique dashed isomorphisms
  \[
  \begin{tikzcd}
    \Lie \bigl( \aNull[L_A](L_X) \bigr) \dar \rar[dashed]{\iso} & \Bound 0 \bigl( \Der(L_X \rel L_A) \bigr) \dar[hook]{\inc} \\
    \Lie \bigl( \aAut[L_A](L_X) \bigr) \dar[swap]{\Lie(\areal)} \rar{\iso} & \Cycles 0 \bigl( \Der(L_X \rel L_A) \bigr) \dar[two heads]{\pr} \\
    \Lie \bigl( \aEaut[L_A](L_X) \bigr) \rar[dashed]{\iso} & \Ho 0 \bigl( \Der(L_X \rel L_A) \bigr) 
  \end{tikzcd}
  \]
  making the diagram commute.
\end{proposition}

\begin{proof}
  The first part can be shown as in \cite[Proof of Lemma~4.12]{ES}.
  The second part follows from \cite[Proof of Theorem~4.13]{ES} since the functor $\Lie$ is right exact (see e.g.\ \cite[Theorem~25.4.11]{TY}).
\end{proof}

In the following lemma, we record what the representations of $\aAut[L_A](L_X)$ on $L_X$ and $\Der(L_X \rel L_A)$ correspond to on the level of Lie algebras.

\begin{lemma} \label{lemma:Der_action_compat}
  Let $L_A \to L_X$ be a quasi-free map of positively graded finitely generated quasi-free dg Lie algebras, and let $n \in \ZZ$.
  Then the following diagram of Lie algebras commutes
  \[
  \begin{tikzcd}
    \Lie \bigl( \aGL \bigl( \Der(L_X \rel L_A)_n \bigr) \bigr) \dar{\iso} & \lar[swap]{\Lie(c)} \Lie \bigl( \aAut[L_A](L_X) \bigr) \rar{\Lie(a)} \dar[swap]{\iso} & \Lie \bigl( \aGL \bigl( (L_X)_n \bigr) \bigr) \dar{\iso} \\
    \gl \bigl( \Der(L_X \rel L_A)_n \bigr) & \lar[swap]{\ad} \Cycles 0 \bigl( \Der(L_X \rel L_A) \bigr) \rar & \gl \bigl( (L_X)_n \bigr)
  \end{tikzcd}
  \]
  where $a$ is the canonical representation of $\aAut[L_A](L_X)$ on $(L_X)_n$, $c$ is the representation of \cref{lemma:Der_algebraic}, the middle vertical map is the isomorphism of \cref{prop:Lie_aEaut}, the outer vertical maps are the isomorphisms of \cref{lemma:Lie_GL}, the bottom-left horizontal map is the adjoint action, and the bottom-right horizontal map is the canonical action of $\Der(L_X \rel L_A)_0$ on $(L_X)_n$.
\end{lemma}

\begin{proof}
  This follows from chasing through the definitions.
\end{proof}

\subsection{Gluing constructions}

In this subsection, we prove that taking pushouts of maps of dg Lie algebras induces maps on the algebraic groups of (relative) automorphisms of these dg Lie algebras.

\begin{notation} \label{not:Lie_gluing}
  We fix quasi-free maps
  \[ L_X \xlongfrom{i} L_B \longfrom L_A \longto L_C \xlongto{j} L_Y \]
  of positively graded finitely generated quasi-free dg Lie algebras.
  Furthermore, we write
  \begin{alignat*}{2}
    L_R  &\defeq  L_B \cop_{L_A} L_C  &\qquad  R & \defeq  \nerve{L_B} \cop_{\nerve{L_A}} \nerve{L_C} \\
    L_P  &\defeq  L_X \cop_{L_A} L_Y  &\qquad  P & \defeq  \nerve{L_X} \cop_{\nerve{L_A}} \nerve{L_Y}
  \end{alignat*}
  for the respective pushout, as well as
  \[ r \colon R \to \nerve{L_R}  \qquad \text{and} \qquad  p \colon P \to \nerve{L_P} \]
  for the induced maps.
\end{notation}

\begin{lemma} \label{lemma:aAut_pushouts}
  In the situation of \cref{not:Lie_gluing}, there is a map of algebraic groups
  \[ \ol \Gamma  \colon  \aAut[L_B](L_X) \times \aAut[L_C](L_Y)  \longto  \aAut[L_R](L_P) \]
  given by taking pushouts of maps and using \cref{lemma:Lie_algebra_extension}.
\end{lemma}

\begin{proof}
  The claim is equivalent to $\ol \Gamma_T(\phi \tensor[S] T, \psi \tensor[S] T) = \ol \Gamma_S(\phi, \psi) \tensor[S] T$ for all morphisms $S \to T$ of $\Algfg{\QQ}$, $\phi \in \aAut[L_B]^R(L_X)(\QQ)(S)$, and $\psi \in \aAut[L_C](L_Y)(\QQ)(S)$.
  This is an elementary verification using the fact that there is a natural isomorphism $(\blank \tensor[S] T) \after (\blank \tensor S)  \iso  \blank \tensor T$ of functors from dg Lie algebras to dg Lie algebras over $T$.
\end{proof}

\begin{lemma} \label{lemma:pushout_indecomposables}
  In the situation of \cref{not:Lie_gluing}, the canonical map $\iota \colon L_R \to L_P$ is quasi-free between quasi-free dg Lie algebras and the canonical map of chain complexes
  \[ \indec[L_B](L_X)  \dirsum  \indec[L_C](L_Y)  \longto  \indec[L_R](L_P) \]
  is an isomorphism.
  In particular, if $i$ and $j$ are minimal, then so is $\iota$.
  
  Moreover, the following diagram of algebraic groups commutes
  \[
  \begin{tikzcd}[column sep = 15]
    \aAut[L_B](L_X) \times \aAut[L_C](L_Y) \ar{rr}{\ol \Gamma} \dar[swap]{\Koppa \times \Koppa} & & \aAut[L_R](L_P) \dar{\Koppa} \\
    \aAut \bigl( \indec[L_B](L_X) \bigr) \times \aAut \bigl( \indec[L_C](L_Y) \bigr) \rar{\dirsum} & \aAut \bigl( \indec[L_B](L_X) \dirsum \indec[L_C](L_Y) \bigr) \rar{\iso} & \aAut \bigl( \indec[L_R](L_P) \bigr)
  \end{tikzcd}
  \]
  where $\Koppa$ is the map of \cref{lemma:indecomposables_representation}.
\end{lemma}

\begin{proof}
  For the first claim, pick isomorphisms $L_X \iso L_B \cop \freelie (V)$ and $L_Y \iso L_C \cop \freelie (W)$ of graded Lie algebras under $L_B$ and $L_C$, respectively.
  Then there is an induced isomorphism $L_P \iso L_R \cop \freelie (V \dirsum W)$ of graded Lie algebras under $L_R$; hence $\iota$ is quasi-free and, by a similar argument, so is $L_R$.
  Moreover, the following diagram of graded vector spaces commutes
  \[
  \begin{tikzcd}
    V \dirsum W \dar[swap]{\iso} \drar[bend left = 17]{\iso} & \\
    \indec[L_B](L_X) \dirsum \indec[L_C](L_Y) \rar & \indec[L_R](L_P)
  \end{tikzcd}
  \]
  where the two isomorphisms are provided by \cref{lemma:quasi-free_indec}.
  Hence the horizontal map is an isomorphism as well.
  The second claim follows from an elementary verification.
\end{proof}

\begin{lemma} \label{lemma:aEaut_pushouts}
  We use \cref{not:Lie_gluing} and assume that $i$ and $j$ are minimal.
  Then there is a unique map $\ol \gamma$ of algebraic groups such that the diagram
  \[
  \begin{tikzcd}
    \aAut[L_B](L_X) \times \aAut[L_C](L_Y) \rar{\ol \Gamma} \dar[swap]{\areal \times \areal} & \aAut[L_R](L_P) \dar{\areal} \\
    \aEaut[L_B](L_X) \times \aEaut[L_C](L_Y) \rar{\ol \gamma} & \aEaut[L_R](L_P)
  \end{tikzcd}
  \]
  commutes.
  Moreover, the following diagram of groups commutes
  \[
  \begin{tikzcd}
    \aEaut[L_B](L_X)(\QQ) \times \aEaut[L_C](L_Y)(\QQ) \rar{\ol \gamma_\QQ} \dar[swap]{\iso} & \aEaut[L_R](L_P)(\QQ) \dar{\iso} \\
    \Eaut[\nerve{L_B}] \bigl( \nerve{L_X} \bigr) \times \Eaut[\nerve{L_C}] \bigl( \nerve{L_Y} \bigr) \rar{\gamma} & \Eaut[\nerve{L_R}] \bigl( \nerve{L_P} \bigr)
  \end{tikzcd}
  \]
  where $\gamma$ is induced by the map of \cref{lemma:aut_pushout_zig-zag} (using \cref{lemma:MC_pushouts,lemma:cofibration_pushouts_special}) and the vertical isomorphisms are those of \cref{prop:aEaut}.
\end{lemma}

\begin{proof}
  For the first statement we need to prove that, when $\phi \eq_{L_B} \phi'$ and $\psi \eq_{L_C} \psi'$, we have that $\phi \cop_{L_A} \psi$ and $\phi' \cop_{L_A} \psi'$ are simplicially homotopic relative to $L_R$ (and similarly over any commutative $\QQ$-algebra).
  This follows from \cref{lemma:dgLie_homotopy}.
  
  The second statement follows from \cref{prop:aEaut} and the fact that the diagram of monoids
  \[
  \begin{tikzcd}
    \Aut[L_B](L_X) \times \Aut[L_C](L_Y) \ar{rr}{\ol \Gamma_\QQ} \dar[swap]{\real \times \real} \ar[bend left, start anchor = -5]{ddr} &[-10] &[-20] \Aut[L_R](L_P) \dar{\real} \\
    \Eaut[\nerve{L_B}] \bigl( \nerve{L_X} \bigr) \times \Eaut[\nerve{L_C}] \bigl( \nerve{L_Y} \bigr) \dar & & \Eaut[\nerve{L_R}] \bigl( \nerve{L_P} \bigr) \dar{\iso} \\
    \Eraut[R] ( P ) & \lar[swap]{\iso} \Eraut[R] ( p ) \rar &  \Eaut[R] \bigl( \nerve{L_P} \bigr)
  \end{tikzcd}
  \]
  commutes.
  The central diagonal map is obtained by sending $(\phi, \psi)$ to the commutative diagram
  \[
  \begin{tikzcd}
    P \rar{\nerve{\phi} \cop \nerve{\psi}} \dar[swap]{p} &[30] P \dar{p} \\
    \nerve{L_P} \rar{\nerve{\phi \cop \psi}} & \nerve{L_P}
  \end{tikzcd}
  \]
  considered as a map $p \to p$.
\end{proof}

\subsection{Forgetful maps} \label{sec:alg_forget}

Consider a sequence $A \to B \to X$ of two cofibrations of spaces.
Then there is a forgetful map $\Eaut[\rat B](\rat X) \to \Eaut[\rat A](\rat X)$ of groups.
In this subsection, we lift this to a map of algebraic groups.
For this, our need to work with minimal models in the subsections above is significant: given two minimal maps of dg Lie algebras $L_A \to L_B \to L_X$, the composite $L_A \to L_X$ does not need to be minimal.
If we choose a minimal model $L_A \to L_X'$ of this map, there is not necessarily a canonical map $\Aut[L_B](L_X) \to \Aut[L_A](L_X')$.
This leads to relevant restrictions down the line.


\begin{definition}
  Let the following be a commutative diagram of positively graded dg Lie algebras
  \[
  \begin{tikzcd}
    L_A \rar \dar[swap]{a} & L_X \dar{f} \\
    L_B \rar & L_Y \rar{\rho} & \Pi
  \end{tikzcd}
  \]
  such that $\Pi$ is of finite type and all other objects are finitely generated and quasi-free.
  Then we denote by
  \[ \aAut[a](f)_\rho  \subseteq  \aAut[L_A](L_X)_{\rho \after f} \times \aAut[L_B](L_Y)_\rho \]
  the algebraic subgroup consisting, at $R \in \Algfg{\QQ}$, of those pairs $(\phi, \psi)$ such that $(f \tensor R) \after \phi = \psi \after (f \tensor R)$.
  We write $\Aut[a](f)_\rho$ for the $\QQ$-points of this algebraic group.
  When $\rho$ is the trivial map we omit it from the notation.
\end{definition}

Note that $\aAut[a](f)_\rho$ is indeed an algebraic group since it is the stabilizer of the element $f$ under the algebraic representation of $\aAut[L_A](L_X)_{\rho \after f} \times \aAut[L_B](L_Y)_\rho$ on $\Hom(L_X, L_Y)$ (cf.\ the proof of \cref{prop:ES}).
The following lemma is a generalization of \cite[Theorem~4.15]{ES}.

\begin{lemma} \label{lemma:aEaut_forget}
  Let the following be a commutative diagram of positively graded finitely generated quasi-free dg Lie algebras
  \[
  \begin{tikzcd}
    L_A \ar[bend right = 15]{drr}[swap]{j} \rar{k} & L_B \rar{i} & L_X \\
    & & L'_X \uar[swap]{m}
  \end{tikzcd}
  \]
  such that $k$, $i$, and $j$ are quasi-free.
  Furthermore assume that $i$ and $j$ are minimal and that $m$ is a quasi-isomorphism.
  Then there is a canonical map $\ol \phi$ of algebraic groups such that the diagram
  \[
  \begin{tikzcd}
    \aAut[L_B](L_X) \dar[swap]{\areal} & \lar \aAut[k](m) \rar & \aAut[L_A](L'_X) \dar{\areal} \\
    \aEaut[L_B](L_X) \ar{rr}{\ol \phi} & & \aEaut[L_A](L'_X)
  \end{tikzcd}
  \]
  of algebraic groups and the diagram
  \[
  \begin{tikzcd}
    \aEaut[L_B](L_X)(\QQ) \ar{rr}{\ol \phi_\QQ} \dar[swap]{\iso} & & \aEaut[L_A](L'_X)(\QQ) \dar{\iso} \\
    \Eaut[\nerve{L_B}] \bigl( \nerve{L_X} \bigr) \rar & \Eaut[\nerve{L_A}] \bigl( \nerve{L_X} \bigr) & \lar[swap]{\iso} \Eaut[\nerve{L_A}] \bigl( \nerve{L'_X} \bigr)
  \end{tikzcd}
  \]
  of groups commute, where in the latter diagram the vertical isomorphisms are those of \cref{prop:aEaut}, and the bottom right-hand horizontal isomorphism is the one of \cref{lemma:Eaut_map} associated to $m$.
\end{lemma}

\begin{proof}
  By \cref{rem:dgLie_homotopy_inverse}, the map $m$ has a simplicial homotopy inverse $n$ relative to $L_A$; it is unique up to simplicial homotopy relative to $L_A$ and again a quasi-isomorphism.
  We first define a natural transformation $\ol \Phi \colon \aAut[L_B](L_X) \to \aEaut[L_A](L'_X)$ of the underlying functors to $\Set$.
  To this end, let $R \in \Algfg{\QQ}$ and $f \in \aAut[L_B](L_X)(R) = \aAut[L_B \tensor R](L_X \tensor R)$.
  Then we set $\ol \Phi_R(f) = (n \tensor \id[R]) \after f \after (m \tensor \id[R])$; it is clear that this is indeed natural in $R$.
  Note that $\ol \Phi_R(f)$ is a quasi-isomorphism and hence an isomorphism by \cite[Theorem~3.8]{ES} since $L_A \to L_X'$ is minimal.
  Moreover, the natural transformation $\areal \after \ol \Phi$ is actually a map of algebraic groups since $m \after n \eq_{L_A} \id[L_X]$ implies that $(m \tensor \id[R]) \after (n \tensor \id[R]) \eq_{L_A \tensor R} \id[L_X \tensor R]$.
  By the same argument the map $\areal \after \ol \Phi$ does not depend on the choice of $n$, the image of $\aNull[L_B](L_X)$ under $\areal \after \ol \Phi$ is trivial, and the following diagram of algebraic groups
  \[
  \begin{tikzcd}
    \aAut[k](m) \rar \dar & \aAut[L_A](L'_X) \dar{\areal} \\
    \aAut[L_B](L_X) \rar{{\areal} \after \ol \Phi} & \aEaut[L_A](L'_X)
  \end{tikzcd}
  \]
  commutes.
  This implies that $\areal \after \ol \Phi$ induces a map $\ol \phi$ of algebraic groups such that the first diagram commutes.
  
  That the second diagram also commutes follows from \cref{prop:aEaut} and the fact that the following diagram of sets commutes
  \[
  \begin{tikzcd}
    \Aut[L_B](L_X) \ar{rr}{\ol \Phi_\QQ} \dar[swap]{\real} & & \Aut[L_A](L'_X) \ar{dd}{\real} \\
    \Eaut[\nerve{L_B}] \bigl( \nerve{L_X} \bigr) \dar & & \\
    \Eaut[\nerve{L_A}] \bigl( \nerve{L_X} \bigr) \rar{\iso} & \hmapeq[\nerve{L_A}] {\nerve{L'_X}} {\nerve{L_X}} & \lar[swap]{\iso} \Eaut[\nerve{L_A}] \bigl( \nerve{L'_X} \bigr)
  \end{tikzcd}
  \]
  since $\nerve{m \after n}$ is simplicially homotopic to the identity relative to $\nerve{L_A}$ (for example by \cref{lemma:dgLie_homotopy} and \cite[Lemma~7.3.8]{Hir}).
\end{proof}

%

\section{Equivariant models}

This section constitutes the main part of the paper.
We provide rational models for classifying spaces $\B \bdlaut{A}{B}{K}(\xi)$ of self-equivalences of bundles and prove that, under certain conditions, these models are compatible with gluing constructions and restricting $A$ and $B$ to smaller spaces.
We furthermore use \cref{sec:block_prelim} to obtain similar results for the classifying spaces $\B \BlDiff[\bdry](M)$ of block diffeomorphisms of high-dimensional manifolds.

\subsection{Nilpotence of relative self-equivalences of bundles}

In this subsection, we prove that the space $\B \bdlaut{A}{B}{K}(\xi)_G$ of \cref{def:bdlaut} is nilpotent (resp.\ virtually nilpotent) when the action of $G$ on the integral (resp.\ rational) homology of the base space $X$ of $\xi$ is nilpotent.
This is the case when $G$ is contained in a unipotent algebraic subgroup of $\Eaut[\rat A](\rat X)$, which is the situation we will apply these results to.
These statements are analogous to, and their proofs use, results of Dror--Zabrodsky \cite{DZ} and Berglund--Zeman \cite[§3.2]{BZ} for $\B \aut(X)$.

\begin{lemma} \label{lemma:Baut_nilpotent}
  Let $i \colon A \to X$ be a cofibration of simply connected Kan complexes of the homotopy types of finite-dimensional CW-complexes, and let $G \subseteq \Eaut[A](X)$ be a subgroup such that its action on $\Ho n (X, A; \ZZ)$ is nilpotent for all $n$.
  Then $\B \aut[A](X)_G$ is nilpotent (recall from \cref{not:components} that $\aut[A](X)_G$ denotes the components of $\aut[A](X)$ belonging to $G$).
\end{lemma}

\begin{proof}
  First, let $\upsilon \colon \Eaut[A](X) \to \Eaut(X)$ be the forgetful map and write $\widetilde G \defeq \inv \upsilon (\upsilon (G))$.
  Since $G \subseteq \widetilde G$, it is enough to prove that $\B \aut[A](X)_{\widetilde G}$ is nilpotent.
  The simplicial monoid $\aut[A](X)_{\widetilde G}$ is the fiber over $i$ of the map
  \[ \aut(X)_{G'}  \xlongto{i^*}  \map(A, X)_i \]
  where $G' \defeq \upsilon(G)$.
  The map $i^*$ is a fibration, since $i$ is a cofibration, and hence the inclusion induces a weak equivalence
  \[ \iota  \colon  \B \aut[A](X)_{\widetilde G}  \xlongto{\eq}  \B \bigl( *, \aut(X)_{G'}, \map(A, X)_i \bigr) \]
  by \cref{lemma:bar_stabilizer}.
  Now, let the following be a commutative diagram of topological spaces as in \cref{lemma:bar_fiber_sequence_action}
  \[
  \begin{tikzcd}
    \gr {\map(A, X)_i}  \rar \dar{\eq}[swap]{\phi} & \gr* {\B \bigl( *, \aut(X)_{G'}, \map(A, X)_i \bigr)} \rar \dar{\eq}[swap]{\psi} & \gr {\B \aut(X)_{G'}} \dar[equal] \\
    F \rar{\iota} & E \rar[two heads]{q} & \gr {\B \aut(X)_{G'}}
  \end{tikzcd}
  \]
  and set $f_0 \defeq \phi(i)$ and $e_0 \defeq \iota(f_0)$.
  It suffices to prove that $E$ is nilpotent.
  
  Note that the action of $G$, and hence $G'$, on $\Ho n (X; \ZZ)$ is nilpotent for all $n$ by the long exact sequence of the pair $(X, A)$ since $G$ acts trivially on $\Ho * (A; \ZZ)$ and nilpotently on $\Ho n (X, A; \ZZ)$ by assumption (here we use that the class of nilpotent $G$-actions is closed under taking subgroups and extensions).
  Hence $G'$ acts nilpotently on $\hg k (X)$ for all $k \ge 1$ (see e.g.~\cite[Proposition~3.2]{BZ}; note that the action of $\Eaut[*](X)$ on $\hg k (X)$ factors through $\Eaut(X)$ since $X$ is simply connected).
  Moreover the base $B \defeq \gr {\B \aut(X)_{G'}}$ is nilpotent by \cite[Theorem~D]{DZ}.
  By \cite[Ch.~II, Proposition~4.4]{BK} (applied to the two fibrations $E \to B \to *$), it is thus enough to prove that the holonomy action of $\hg 1 (E, e_0)$ on $\hg k (F, f_0)$ (which is induced by a group homomorphism $\hg 1 (E, e_0) \to \Eaut[*](F)$; see e.g.\ \cite[p.~20]{MP}) is nilpotent for every $k \ge 1$.
  Now note that the weak equivalence $\psi \after \iota$ induces a group isomorphism $\alpha \colon \widetilde G \to \hg 1 (E, e_0)$.
  By \cref{lemma:bar_fiber_sequence_action}, the canonical pointed action of a self-equivalence $f \in \aut[A](X)_{\widetilde G}$ on $\map(A, X)_i$ corresponds to the pointed action of $\alpha(\eqcl f) \in \hg 1 (E, e_0)$ on $F$, in the sense that $\phi \after f_*$ and $\alpha(\eqcl f)_* \after \phi$ are homotopic relative to $\set i$.
  Hence it is enough to prove that the action of $\widetilde G$ on $\hg k ( \map(A, X), i )$ is nilpotent for every $k \ge 1$.
  
  To this end, consider the Federer spectral sequence of \cref{prop:Federer}
  \[ E_2^{p,q}  =  \begin{lrcases*} \Coho p \bigl( A; \hg q(X) \bigr), & if $p \ge 0$ and $q \ge \max(2, p)$ \\ 0, & otherwise \end{lrcases*}  \longabuts  \hg {q - p} \bigl( \map(A, X), i \bigr) \]
  associated to the map $i \colon A \to X$.
  The group $\widetilde G$ acts on the spectral sequence, and its action on $\hg q(X)$ is nilpotent since it factors through the nilpotent action of $G'$.
  Since the class of nilpotent $\widetilde G$-actions is closed under taking quotients, subgroups, and extensions, and is furthermore closed under taking $\Hom(M, \blank)$ for any abelian group $M$ by \cite[Ch.~I, Proposition~4.15]{HMR}, the action of $\widetilde G$ on $\Coho p ( A, \hg q(X) )$ is nilpotent as well.
  Then, for $q > p$, the group $\widetilde G$ acts nilpotently on $E_\infty^{p,q}$, and thus also on $\hg k ( \map(A, X), i )$ since, for each $k$, only finitely many pairs $(p, q)$ contribute.
\end{proof}

\begin{lemma} \label{lemma:Bautbdl_nilpotent}
  Let $B \to A \to X$ be cofibrations of simply connected Kan complexes of the homotopy types of finite-dimensional CW-complexes, $K$ a simply connected Kan complex, and $\xi \colon X \to K$ a map.
  Furthermore let $G \subseteq \Eaut[A](X)_{\eqcl{\xi}_B}$ be a subgroup such that its action on $\Ho n (X, A; \ZZ)$ is nilpotent for all $n$.
  Then $\B \bdlaut{A}{B}{K}(\xi)_G$ is nilpotent.
\end{lemma}

\begin{proof}
  We first factor $\xi$ as a trivial cofibration $f \colon X \to X'$ followed by a fibration $\xi' \colon X' \to K$.
  Then both of the maps
  \[
  \begin{tikzcd}[row sep = 10, column sep = -40]
    & \dlar[start anchor = south west][swap]{\eq} \B \bigl( \map[B](X', K)_{\xi'}, \aut[A](f)_{\widetilde G}, * \bigr) \drar[start anchor = south east]{\eq} & \\
    \B \bigl( \map[B](X, K)_\xi, \aut[A](X)_G, * \bigr) & & \B \bigl( \map[B](X', K)_{\xi'}, \aut[A](X')_{G'}, * \bigr)
  \end{tikzcd}
  \]
  are weak equivalences by part~(i) of \cref{lemma:rel_aut_span} and \cref{lemma:bar_eq}, where $\widetilde G$ and $G'$ are the subgroups of the respective groups of components corresponding to $G$.
  Thus we can assume without loss of generality that $\xi$ is a fibration.
  
  Consider the following commutative diagram, whose rows are fiber sequences by \cref{lemma:bar_fiber_sequence},
  \[
  \begin{tikzcd}
  	\map[B](X, K)_\xi \rar \dar[equal] & \B \bigl( \map[B](X, K)_\xi, \aut[A](X)_G, * \bigr) \rar \dar & \B \aut[A](X)_G \dar \\
  	\map[B](X, K)_\xi \rar & \B \bigl( \map[B](X, K)_\xi, \aut[B](X)_H, * \bigr) \rar & \B \aut[B](X)_H
  \end{tikzcd}
  \]
  where $H$ is the image of $G$ under the map $\Eaut[A](X) \to \Eaut[B](X)$.
  We want to prove that the total space $E$ of the upper row is nilpotent.
  Since the base space $\B \aut[A](X)_G$ is nilpotent by \cref{lemma:Baut_nilpotent}, it is, as in the proof of that lemma, enough to prove that the holonomy action of $\hg 1 (E)$ on $\hg n ( \map[B](X, K)_\xi )$ is nilpotent for each $n \ge 1$.
  By the diagram above, this action factors through the holonomy action of $\hg 1$ of the total space $E'$ of the lower row, and thus it is enough to prove that this action is nilpotent.
  
  Now consider the map $\xi_* \colon \aut[B](X)_H \to \map[B](X, K)_\xi$ given by postcomposing with $\xi$.
  Since $B \to X$ is a cofibration and $\xi$ is a fibration, the map $\xi_*$ is a fibration as well; we denote its fiber over $\xi$ by $\aut[B//K](X)_H$.
  The map $\iota \colon \B \aut[B//K](X)_H \to E'$ induced by the inclusion is a weak equivalence by \cref{lemma:bar_stabilizer}.
  Moreover, by \cref{lemma:bar_fiber_sequence_action}, for any element $f \in \aut[B//K](X)_H$, the holonomy action of $\iota_*(f)$ on $\hg n ( \map[B](X, K)_\xi )$ equals the action induced by precomposing with $f$.
  Hence it is enough to prove that the canonical action of $H' \defeq \hg 0 ( \aut[B//K](X)_H )$ on $\hg n (\map[B](X, K), \xi)$ is nilpotent for each $n \ge 1$.
  
  To this end, consider the Federer spectral sequence of \cref{prop:Federer}
  \[ E_2^{p,q}  =  \begin{lrcases*} \Coho p \bigl( X, B; \hg q(K) \bigr), & if $p \ge 0$ and $q \ge \max(2, p)$ \\ 0, & otherwise \end{lrcases*}  \longabuts  \hg {q - p} \bigl( \map[B](X, K), \xi \bigr) \]
  associated to the map $\xi \colon X \to K$.
  (Note that $(X, B)$ has the homotopy type of a finite-dimensional relative CW-pair since $B$ and $X$ are simply connected and $\Ho * (X, B)$ is concentrated in finitely many degrees.
  In the absolute case $B = \emptyset$ this follows from \cite[Proposition~4C.1]{Hat}, and the proof of the relative case is similar.)
  The group $H'$ acts on the spectral sequence.
  Since the class of nilpotent $H'$-actions is closed under taking quotients, subgroups, and extensions, and only finitely many pairs $(p, q)$ contribute to each degree of the abutment, it is enough to prove that the $H'$-action on the $E_2$-page is nilpotent.
  This action factors through the action of $H$, which acts nilpotently on $\Ho p (X, A)$ since $G$ does so by assumption.
  Hence $H$ also acts nilpotently on $\Ho p (X, B)$ by the long exact sequence of the triple $(X, A, B)$.
  By \cite[Ch.~I, Proposition~4.15]{HMR} (or, more precisely, the same argument applied to a contravariant functor) the class of nilpotent $H$-actions is closed under taking $\Hom(\blank, M)$ and $\Ext(\blank, M)$ for any abelian group $M$.
  Then the universal coefficient theorem for cohomology completes the proof.
\end{proof}

\begin{lemma} \label{lemma:Baut_virtually_nilpotent}
  Let $B \to A \to X$ be cofibrations of simply connected Kan complexes of the homotopy types of finite CW-complexes, $K$ a simply connected Kan complex, and $\xi \colon X \to K$ a map.
  Furthermore let $G \subseteq \Eaut[A](X)_{\eqcl{\xi}_B}$ be a subgroup such that its representation on $\Ho n (X, A; \QQ)$ is unipotent for all $n$ (i.e.\ it admits a finite filtration by subrepresentations such that the associated graded is a trivial representation).
  Then $\B \bdlaut{A}{B}{K}(\xi)_G$ is virtually nilpotent.
\end{lemma}

\begin{proof}
  By \cite[Lemma~3.4]{BZ}, if $G$ is any group acting on a finitely generated abelian group $H$ such that the induced representation on $H \tensor \QQ$ is unipotent, then $G$ has a finite-index subgroup acting nilpotently on $H$.
  This implies the claim by \cref{lemma:Bautbdl_nilpotent}.
\end{proof}

\subsection{Equivariant models for relative self-equivalences of bundles} \label{sec:model}

In this subsection, we prove the first main result of this paper, stating that $\B \bdlaut A B K (\xi)$ is rationally equivalent to $\hcoinv {\MCb(\lie g)} {\Gamma}$ for a certain dg Lie algebra $\lie g$ equipped with an algebraic action of an arithmetic group $\Gamma$.
From this, we will produce both a cdga model for $\B \bdlaut A B K (\xi)$ as well as draw strong cohomological consequences.
To achieve this, we follow the strategy of Berglund--Zeman \cite{BZ}, who constructed a rational model for $\B \aut(X)$ by modeling a certain nilpotent covering together with the action of its deck transformation group.
We adapt this approach to (a generalization of) a model for $\B \bdlaut{A}{B}{K}(\xi)_{\id}$ due to Berglund \cite{Ber}.
This involves producing a model for the simplicial monoid $\aut[A](X)_G$ for certain subgroups $G \subseteq \Eaut[A](X)_{\eqcl{\xi}_B}$, as well as a model for the relative mapping space $\map[B](X, K)_\xi$ together with its action of $\aut[A](X)_G$.

We begin by recalling a construction of Berglund--Saleh \cite[§3.2]{BS} that explicitly relates realizations of dg Lie subalgebras of $\Der(L_X \rel L_A)$ with the simplicial monoid $\aut[\rat A](\rat X)$.

\newcommand{\expact}{e}
\newcommand{\aexpact}{\bar e}

\begin{lemma} \label{lemma:exp_action}
  Let $L' \to L$ be a map of finitely generated dg Lie algebras, and $\lie g$ a finite-dimensional nilpotent Lie algebra equipped with a nilpotent action $\lie g \to \Cycles 0 {(\Der(L \rel L'))}$.
  Then the map
  \[ \expact  \colon  \exp (\lie g)  \longto  \Aut[L'](L),  \qquad  \theta  \longmapsto  \sum_{n \ge 0} \frac {1} {n!} (\theta \act \blank)^{\after n} \]
  is a group homomorphism.
  When $L$ is quasi-free and positively graded, then $\expact$ canonically extends to a map $\aexpact  \colon  \aexp (\lie g)  \to  \aAut[L'](L)$ of algebraic groups.
\end{lemma}

\begin{proof}
  First note that, for all $\theta \in \lie g$ and homogeneous $x \in L $, the sum $\expact(\theta)(x)$ is finite since $\lie g$ acts nilpotently on $L$.
  Furthermore $\expact(\theta)(x) = x$ for all $x \in L'$.
  That $\expact$ is a well-defined map of groups then follows as in \cite[Proof of Proposition~4.10]{BFMT}.
  For $R \in \Algfg{\QQ}$, we define $\aexpact_R \defeq \expact \colon \exp(\lie g \tensor R) \to \Aut[L' \tensor R](L \tensor R)$.
  It is clear that this yields a map of algebraic groups.
\end{proof}

When $\lie g$ is positively graded, the following lemma is a special case of \cite[Proposition~3.11]{BS} (see also Lindell--Saleh \cite[Definition~3.8]{LS}).
We provide the short proof for completeness.

\begin{lemma}[Berglund--Saleh] \label{lemma:exp_aut_map}
  Let $L' \to L$ be a map of finitely generated nilpotent dg Lie algebras and let $\psi \colon \lie g \to \Der (L \rel L')$ be a map of dg Lie algebras such that both $\lie g$ and the induced action of $\lie g$ on $L$ are nilpotent.
  Then there is a map of simplicial monoids
  \[ \Xi \colon \expb ( \lie g )  \longto  \aut[\MCb(L')] \bigl( \MCb(L) \bigr) \]
  given, in each simplicial degree, by ${\MC} \after \expact$.
  Furthermore assume that $\psi$ is injective and let $G \subseteq \Aut[L'](L)$ be a subgroup such that the conjugation action of $G$ on $\Der(L \rel L')$ preserves $\lie g$ setwise.
  Then $\Xi$ is $G$-equivariant with respect to the conjugation actions on both sides.
\end{lemma}

\begin{proof}
  Note that, for any commutative chain algebra $A$ concentrated in finitely many non-positive degrees, the action of the Lie algebra $\Cycles 0 (\lie g \tensor A)$ on $L \tensor A$ is nilpotent since $\lie g$ acts nilpotently on $L$.
  Then the first claim follows from \cref{lemma:exp_action}.
  The second claim is clear by inspection.
\end{proof}

By work of Berglund--Saleh \cite[Theorem~1.1]{BS}, the identity component of $\aut[\rat A](\rat X)$ is modeled by the positive truncation $\trunc 1 { \Der(L_X \rel L_A) }$.
The following result of Lindell--Saleh \cite{LS} provides us with an explicit description of the map exhibiting this model.

\begin{proposition}[Lindell--Saleh] \label{prop:LS}
  Let $L_A \to L_X$ be a quasi-free map of finitely generated positively graded quasi-free dg Lie algebras.
  Then the map $\Xi$ of \cref{lemma:exp_aut_map} restricts to a weak equivalence
  \[ \expb \bigl( \trunc 1 { \Der(L_X \rel L_A) } \bigr)  \xlongto{\eq}  \aut[\MCb(L_A)] \bigl( \MCb(L_X) \bigr)_{\id} \]
  of simplicial monoids.
\end{proposition}

\begin{proof}
  This is \cite[Proposition~3.12]{LS}.
\end{proof}

We will need a version of this for a larger part of $\aut[\rat A](\rat X)$.
However, since $\expb(\lie g)$ is only defined for nilpotent dg Lie algebras $\lie g$, and so $\B \expb(\lie g) \eq \MCb(\lie g)$ is always a nilpotent space, we can at most hope to model spaces of the form $\B \aut[\rat A](\rat X)_G$ with $G$ nilpotent.
We will now show that a model exists in the case that $G$ arises as the $\QQ$-points of a unipotent algebraic subgroup of $\aEaut[\rat A](\rat X)$.
Using a different approach, such a model was also constructed by Félix--Fuentes--Murillo \cite[Theorem~0.1]{FFM23}.
We begin by defining the corresponding dg Lie subalgebra of $\Der(L_X \rel L_A)$ (see also \cite[Definition~4.3]{FFM23}).

\begin{definition} \label{def:Deru}
  Let $L_A \to L_X$ be a minimal quasi-free map of finitely generated positively graded quasi-free dg Lie algebras, and let $\ol G \subseteq \aEaut[L_A](L_X)$ be an algebraic subgroup.
  We denote by $\Der[\ol G](L_X \rel L_A) \subseteq \Der(L_X \rel L_A)$ the dg Lie subalgebra defined as the pullback
  \[
  \begin{tikzcd}
    \Der[\ol G](L_X \rel L_A) \ar[hook]{rr} \dar[two heads] & & \trunc 0 { \Der(L_X \rel L_A) } \dar[two heads]{\pr} \\
    \Lie(\ol G) \rar[hook] & \Lie \bigl( \aEaut[L_A](L_X) \bigr) \rar{\iso} & \Ho 0 \bigl( \Der(L_X \rel L_A) \bigr)
  \end{tikzcd}
  \]
  with the isomorphism provided by \cref{prop:Lie_aEaut} (note that the functor $\Lie(\blank)$ is left exact, e.g.\ by \cite[10.14]{Mil}).
\end{definition}

The following proposition is a relative version of \cite[Lemma~3.13]{BZ}.
We provide a more detailed proof for the reader's convenience.
See also \cite[Theorem~0.1]{FFM23} for a similar statement.

\begin{proposition} \label{prop:exp_eq}
  Let $L_A \to L_X$ be a minimal quasi-free map of positively graded finitely generated quasi-free dg Lie algebras, and let $\ol U \subseteq \aEaut[L_A](L_X)$ be a unipotent algebraic subgroup.
  Then the dg Lie algebra $\Der[\ol U](L_X \rel L_A)$ is nilpotent and it acts nilpotently on $L_X$.
  Moreover the map $\Xi$ of \cref{lemma:exp_aut_map} restricts to a weak equivalence of simplicial monoids
  \[ \Xi \colon \expb \bigl( \Der[\ol U](L_X \rel L_A) \bigr)  \xlongto{\eq}  \aut[\nerve{L_A}] \bigl( \nerve{L_X} \bigr)_{\ol U(\QQ)} \]
  that is $N$-equivariant with respect to the conjugation actions on both sides, where $N \subseteq \Aut[L_A](L_X)$ is the normalizer of the preimage of $\ol U(\QQ)$.
\end{proposition}

\begin{proof}
  We begin by proving that $\lie g \defeq \Der[\ol U](L_X \rel L_A)$ is nilpotent and acts nilpotently on $L_X$.
  To this end, define $\ol U'$ to be the pullback
  \[
  \begin{tikzcd}
    \ol U' \rar \dar & \aAut[L_A](L_X) \dar[two heads]{\areal} \\
    \ol U \rar{\inc} & \aEaut[L_A](L_X)
  \end{tikzcd}
  \]
  of algebraic groups.
  By \cref{prop:ES} the kernel of $\areal$ is unipotent.
  This implies that $\ol U'$ is an extension of unipotent groups and hence unipotent itself (see e.g.\ \cite[Corollary~14.7]{Mil}).
  Applying the functor $\Lie$ to the diagram above yields the left-hand square in the diagram of Lie algebras
  \[
  \begin{tikzcd}
    \Lie(\ol U') \rar \dar & \Lie \bigl( \aAut[L_A](L_X) \bigr) \dar \rar{\iso} & \Cycles 0 \bigl( \Der(L_X \rel L_A) \bigr) \dar{\pr} \\
    \Lie(\ol U) \rar & \Lie \bigl( \aEaut[L_A](L_X) \bigr) \rar{\iso} & \Ho 0 \bigl( \Der(L_X \rel L_A) \bigr) 
  \end{tikzcd}
  \]
  whose right-hand square is provided by \cref{prop:Lie_aEaut}.
  Since $\Lie$ preserves pullbacks (see e.g.\ \cite[10.14]{Mil}), this induces compatible isomorphisms $\Lie(\ol U') \iso \lie g_0$ and $\Lie(\ol U) \iso \Ho 0 (\lie g)$.
  It follows from \cref{lemma:Der_action_compat} that the representation of $\lie g_0$ on $\lie g_n = \Der(L_X \rel L_A)_n$ is nilpotent for all $n$, since any finite-dimensional algebraic representation of a unipotent algebraic group induces a nilpotent representation of its Lie algebra (see e.g.\ \cite[14.31]{Mil}).
  Hence $\lie g$ is nilpotent.
  Similarly, again using \cref{lemma:Der_action_compat}, we see that the representation of $\lie g_0$ on $(L_X)_n$ is nilpotent for all $n$, and hence $\lie g$ acts nilpotently on $L_X$.
  
  Assuming that $\Xi$ is well defined, i.e.\ that it lands in the correct components, \cref{prop:LS,lemma:htpy_exp} imply that it is an isomorphism on $\hg{k}$ for $k \ge 1$ since the inclusion $\trunc 1 {\Der(L_X \rel L_A)} \to \lie g$ induces an isomorphism on homology in positive degrees.
  We will now prove that it is indeed well defined and in fact induces an isomorphism on $\hg 0$.
  To this end, consider the diagram of groups
  \[
  \begin{tikzcd}
    \exp (\lie g_0) \dar[two heads] & \lar[equal] \exp[0] (\lie g) \rar{\expact} \dar[two heads] & \Aut[L_A](L_X) \dar{\nerve{\blank}} \\
    \exp \bigl( \Ho 0 (\lie g) \bigr) \dar{\iso} & \lar[swap]{\iso} \hg 0 \bigl( \expb (\lie g) \bigr) \rar{\Xi_*} & \Eaut[\nerve{L_A}] \bigl( \nerve{L_X} \bigr) \\
    \exp \bigl( \Lie(\ol U) \bigr) \ar{rr}{\iso} & & \ol U(\QQ) \uar[swap]{\inc}
  \end{tikzcd}
  \]
  where $\expact$ is as in \cref{lemma:exp_action} and the upper left-hand square is provided by \cref{lemma:htpy_exp} and commutes.
  To finish the proof, it suffices to show that the bottom square commutes as well.
  Since the upper right-hand square commutes by definition, this is equivalent to the outer square commuting.
  That square is isomorphic to the $\QQ$-points of a corresponding diagram of algebraic groups.
  Since $\aexp(\lie g_0)$ is unipotent and hence connected (see e.g.\ \cite[Corollary~14.15]{Mil}), it suffices to show that the diagram of algebraic groups commutes after applying the functor $\Lie$ (see e.g.\ \cite[II, §6, 2.1 (b)]{DGeng}).
  This yields the outer pentagon of the diagram of Lie algebras
  \[
  \begin{tikzcd}
    \Lie \bigl( \aexp (\lie g_0) \bigr) \ar{rrr}{\Lie(\aexpact)}[below, name = T]{} \ar{ddd} & & & \dlar[end anchor = 10]{\iso} \Lie \bigl( \aAut[L_A](L_X) \bigr) \ar{dd} \\
     & \lie g_0 \ar[phantom, to = T]{}{(*)} \ular{\iso} \rar \dar & \Cycles 0 \bigl( \Der(L_X \rel L_A) \bigr) \dar & \\
     & \Lie(\ol U) \dlar[end anchor = 15][swap]{\iso} \drar[near end]{\id} \rar & \Ho 0 \bigl( \Der(L_X \rel L_A) \bigr) & \lar[swap]{\iso}  \Lie \bigl( \aEaut[L_A](L_X) \bigr) \\
    \Lie \bigl( \aexp \bigl( \Lie(\ol U) \bigr) \bigr) \ar{rr}{\iso} & & \Lie(\ol U) \ar[bend right = 15]{ur} &
  \end{tikzcd}
  \]
  where everything commutes except possibly the square marked $(*)$.
  In that square the upper composite is given by
  \[ \theta  \longmapsto  \theta \tensor \epsilon  \longmapsto  \sum_{n \ge 0} \frac {\bigl( (\theta \tensor \epsilon) \act \blank \bigr)^n} {n!}  \longmapsto  \theta \]
  by the explicit descriptions of the maps involved given in \cref{lemma:Lie_exp,lemma:exp_action,prop:Lie_aEaut}.
  In particular this upper composite is equal to the lower map, which is just the inclusion.
  This finishes the proof that $\Xi$ is a weak-equivalence.
  
  Lastly, we prove that the map is $N$-equivariant.
  First note that by assumption the conjugation action of $N \subseteq \Aut[L_A](L_X)$ on $\aut[\nerve{L_A}] ( \nerve{L_X} )$ restricts to $\aut[\nerve{L_A}] ( \nerve{L_X} )_{\ol U(\QQ)}$.
  By \cref{lemma:exp_aut_map} it is thus enough to prove that the conjugation action of $N$ on $\Der(L_X \rel L_A)$ preserves $\lie g$ setwise.
  This follows from the isomorphisms of \cref{prop:Lie_aEaut} being compatible with the conjugation actions of $\Aut[L_A](L_X)$, which can be seen by inspection.
\end{proof}

So far we have obtained a model for simplicial monoids of the form $\aut[A](X)_U$ (when $A$ and $X$ are rational spaces).
To be able to model the space $\B \bdlaut{A}{B}{K}(\xi)_U$ of \cref{def:bdlaut} we furthermore need a model for the relative mapping space $\map[B](X, K)_\xi$ together with its action of $\aut[A](X)_U$.
This is what we will provide now, following Berglund \cite[§3.6]{Ber}.
We begin with a helpful lemma (essentially due to Berglund \cite{Ber15}) on modeling specific connected components of a given space.

\begin{definition}
	Let $(L, d)$ be a dg Lie algebra, $L' \subseteq L$ a dg Lie ideal, and $\tau \in \MC(L)$ a Maurer--Cartan element.
	Then we write $(L')^\tau$ for the dg Lie algebra $(L', d_\tau \defeq d + \liebr \tau \blank)$.
\end{definition}

\begin{lemma}[Berglund] \label{lemma:MC_component_iso}
  Let $(L, d)$ be a dg Lie algebra equipped with a nilpotent action $\alpha$ of a nilpotent dg Lie algebra $\lie g$, and let $\tau \in \MC(L)$ such that $\theta \act \tau = 0$ for all $\theta \in \lie g_0$.
  Then the map
  \begin{align*}
    \MCb(\trunc 0 {L^\tau})  &\xlongto{\eq}  \MCb(L)_{\tau} \\
    x  &\longmapsto  x + \tau \tensor 1
  \end{align*}
  is a weak equivalence of simplicial sets.
  Moreover this map is $\expb(\lie g)$-equivariant when the domain is equipped with the action of \cref{lemma:MC_outer} associated to the outer action of $\lie g$ on $\trunc 0 {L^\tau}$ given by $\alpha$ and $\chi_\tau(\theta) \defeq \tau \act \theta$, and the codomain is equipped with the action of \cref{lemma:exp_aut_map}.
\end{lemma}

\begin{proof}
  Consider the two maps
  \[ \MCb(\trunc 0 {L^\tau})  \xlongto{\eq}  \MCb(L^\tau)_0  \xlongto{\iso}  \MCb(L)_{\tau} \]
  where the first map is induced by the inclusion and the second map is given by $x \mapsto x + \tau \tensor 1$ in each simplicial degree.
  The first map is a weak equivalence by work of Berglund \cite[Corollary~4.11]{Ber15} and Getzler \cite[Corollary~5.11]{Get}, and the second map is an isomorphism by an argument analogous to \cite[Proof of Proposition~4.9]{Ber15}.
  The assumption that $\theta \act \tau = 0$ for all $\theta \in \Cycles{0}(\lie g)$ implies that the $\expb(\lie g)$-action on $\MCb(L)$ of \cref{lemma:exp_aut_map} indeed restricts to an action on the connected component containing $\tau$.
  It is immediate from chasing through the definitions that the second map is $\expb(\lie g)$-equivariant when $\MCb(L^\tau)_0$ is equipped with the action of \cref{lemma:MC_outer} corresponding to the outer action $\chi_\tau$ (which clearly fixes the $0$-simplex $0$).
  
  We will now prove that $\chi_\tau$ indeed restricts to an outer action on $\MCb(\trunc 0 {L^\tau})$, i.e.\ that $\theta \act \tau \in \trunc 0 {L^\tau}$ for all $\theta \in \lie g$.
  This is true for $\theta \in \lie g_0$ by assumption and true for $\theta \in \lie g_{\ge 2}$ by degree reasons.
  For $\theta \in \lie g_1$ we have that
  \begin{align*}
    d_\tau(\theta \act \tau) &= d(\theta) \act \tau - \theta \act d(\tau) + \liebr{\tau}{\theta \act \tau} \\
    &= \tfrac{1}{2} \theta \act \liebr \tau \tau + \liebr{\tau}{\theta \act \tau} \\
    &= \tfrac{1}{2} \bigl( \liebr {\theta \act \tau} \tau - \liebr \tau {\theta \act \tau} \bigr) + \liebr{\tau}{\theta \act \tau} \\
    &= 0
  \end{align*}
  which implies the claim and finishes the proof.
\end{proof}

\begin{definition}
	Let $C$ be a cocommutative dg coalgebra and $L$ a dg Lie algebra.
	Then we equip the chain complex $\Hom(C, L)$ with the dg Lie algebra structure with Lie bracket
	\[ \liebr f g  \defeq  \liebr \blank \blank \after (f \tensor g) \after \Delta \]
	where $\Delta$ is the comultiplication of $C$.
	When $C' \subseteq C$ is a dg subcoalgebra, we denote by $\Hom(C, L \rel C') \subseteq \Hom(C, L)$ the dg Lie ideal of those morphisms that vanish on $C'$.
	
	A Maurer--Cartan element of $\Hom(C, L)$ is called a \emph{twisting morphism}.
	We write $\pi \colon \CEchains{*}(L) \to L$ for the map of degree $-1$ that is the identity on $\shift L$ and trivial otherwise, and call it the \emph{universal twisting morphism}.
\end{definition}

\begin{proposition} \label{prop:map_model}
  Let $L_B \to L_A \to L_X$ be quasi-free maps of positively graded finitely generated quasi-free dg Lie algebras such that $L_A \to L_X$ is minimal, and let $\rho \colon L_X \to \Pi$ be a map of dg Lie algebras such that $\Pi$ is of finite type, positively graded, and has trivial differential.
  Then there is a weak equivalence of simplicial sets with a right $\Aut[L_B](L_X)_{\rho}$-action
  \[ \Theta \colon \cMCb \bigl( \trunc 0 { \Hom \bigl( \CEchains * (L_X), \Pi \rel[\big] \CEchains * (L_B) \bigr)^{\rho \after \pi} } \bigr)  \xlongto{\eq}  \map[\nerve{L_B}] \bigl( \nerve{L_X}, \nerve{\Pi} \bigr)_{\nerve{\rho}} \]
  where, on the left-hand side, we use the complete filtration induced by the complete filtration $\Pi = \Pi_{\ge 1} \supseteq \Pi_{\ge 2} \supseteq \dots$ of dg Lie algebras.
  
  Moreover, let $\ol U \subseteq \aEaut[L_A](L_X)_\rho$ be a unipotent algebraic subgroup.
  Then $\Theta$ is equivariant with respect to the map of \cref{prop:exp_eq}
  \[ \Xi \colon \expb \bigl( \Der[\ol U](L_X \rel L_A) \bigr)  \xlongto{\eq}  \aut[\nerve{L_A}] \bigl( \nerve{L_X} \bigr)_{\ol U(\QQ)} \]
  where the action of the domain (on each stage of the limit) is the action of \cref{lemma:MC_outer} associated to the restriction of the outer action $(\rho \after \pi)_*$ given by
  \begin{equation} \label{eq:Der_Hom_action}
  \begin{aligned}
    (f \act \theta)(\shift x_1 \wedge \dots \wedge \shift x_n)  &\defeq  \sum_{k = 1}^n (-1)^\epsilon f \bigl( \shift x_1 \wedge \dots \wedge \shift \theta (x_k) \wedge \dots \wedge \shift x_n \bigr) \\
    (\rho \after \pi)_*(\theta)  &\defeq  (\rho \after \pi) \act \theta
  \end{aligned}
  \end{equation}
  where $\theta \in \Der(L_X)$ and $f \in \Hom ( \CEchains * (L_X), \Pi )$, and we set $\epsilon \defeq \deg \theta (k + \deg {x_1} + \dots + \deg {x_{k-1}})$.
\end{proposition}

\begin{proof}
  By \cite[(Proof of) Theorem~3.16]{Ber} there is a natural weak equivalence of simplicial sets
  \begin{equation} \label{eq:map_eq}
    \MC \bigl( \Hom \bigl( \CEchains * (L_X), \Pi \tensor \sForms \bigr) \bigr)  \xlongto{\eq}  \map \bigl( \MCb(L_X), \MCb(\Pi) \bigr)
  \end{equation}
  given by the adjoint of the restriction of the map given in simplicial degree $n$ by
  \begin{equation} \label{eq:map_eq_formula}
  \begin{aligned}
  	\Hom \bigl( \CEchains * (L_X), \Pi \tensor \sForms[n] \bigr)_{-1} \times (L_X \tensor \sForms[n])_{-1}  &\longto  (\Pi \tensor \sForms[n])_{-1} \\
  	(f, x \tensor b)  &\longmapsto  \sum_{k \ge 0} (-1)^{\binom k 2 \deg b} \frac{1}{k!} f \bigl( (\shift x)^{\wedge k} \bigr) \cdot b^k
  \end{aligned}
  \end{equation}
  where $\cdot$ denotes the multiplication map $\Pi \tensor \sForms[n] \tensor \sForms[n] \to \Pi \tensor \sForms[n]$.
  Note that the sign in the sum is only $-1$ if $\deg b$ is odd and $k > 1$.
  But then $b^k = 0$; hence the sign can also be omitted.
  The map \eqref{eq:map_eq} sends the $0$-simplex $\rho \after \pi$ to $\MCb(\rho)$ by construction.
  Furthermore, by \cite[Remark~3.17]{Ber}, there is an isomorphism
  \begin{equation} \label{eq:cMCb_Hom}
    \cMCb \bigl( \Hom \bigl( \CEchains * (L_X), \Pi \bigr) \bigr)  \xlongto{\iso}  \MC \bigl( \Hom \bigl( \CEchains * (L_X), \Pi \tensor \sForms \bigr) \bigr)
  \end{equation}
  where on the left-hand side we use the complete filtration with $k$-th stage $\Hom ( \CEchains * (L_X), \Pi_{\ge k} )$.
  The isomorphism is in simplicial degree $n$ induced by the map sending a sequence $(f_k)_{k \ge 1}$ with $f_k \in \Hom ( \CEchains * (L_X), \Pi_{\le k - 1} )_p \tensor \sForms[n]$ to the map $\CEchains * (L_X) \to \Pi \tensor \sForms[n]$ sending $x \in \CEchains q (L_X)$ to $(\ev[x] \tensor \id) (f_{q+p+1})$.
  We thus obtain the following commutative diagram of simplicial sets
  \begin{equation} \label{eq:map_eq_diag}
  \begin{tikzcd}
    \cMCb \bigl( \trunc 0 { \Hom \bigl( \CEchains * (L_X), \Pi \bigr)^{\rho \after \pi} } \bigr) \dar[swap]{\eq} \rar & \cMCb \bigl( \trunc 0 { \Hom \bigl( \CEchains * (L_B), \Pi \bigr)^{\rho \after \pi} } \bigr) \dar{\eq} \\
    \cMCb \bigl( \Hom \bigl( \CEchains * (L_X), \Pi \bigr) \bigr)_{\rho \after \pi} \dar[swap]{\eq} \rar & \cMCb \bigl( \Hom \bigl( \CEchains * (L_B), \Pi \bigr) \bigr)_{\rho \after \pi} \dar{\eq} \\
    \map \bigl( \MCb(L_X), \MCb(\Pi) \bigr)_{\MCb(\rho)} \rar & \map \bigl( \MCb(L_B), \MCb(\Pi) \bigr)_{\MCb(\rho)}
  \end{tikzcd}
  \end{equation}
  where the top vertical weak equivalences are (at each stage of the limit) those of \cref{lemma:MC_component_iso} (note that each map in the respective towers is a fibration).
  Taking horizontal fibers yields the desired weak equivalence $\Theta$ (since both the top and the bottom horizontal map are fibrations).
  Using the formula \eqref{eq:map_eq_formula}, one checks it to be $\Aut[L_B](L_X)_{\rho}$-equivariant.

  We claim that, for any nilpotent dg Lie subalgebra $\lie g \subseteq \Der(L_X)$, the action of $\lie g$ on
  \[ \quot { \Hom \bigl( \CEchains * (L_X), \Pi \bigr) } { \Hom \bigl( \CEchains * (L_X), \Pi_{\ge k} \bigr) } \]
  induced by \eqref{eq:Der_Hom_action} is nilpotent for every $k$.
  To see this, note that the representation of $\lie g_0$ on $(\shift L_X)^{\wedge p} \subseteq \CEchains * (L_X)$ is nilpotent for all $p$.
  Since $L_X$ is non-negatively graded, this implies that the representation of $\lie g_0$ on $\CEchains q (L_X)$ is nilpotent for all $q$, which in turn implies the claim.
  Hence there is an action of $\expb(\lie g)$ on $\cMCb ( \Hom ( \CEchains * (L_X), \Pi ) )$ given at each stage of the limit by the action of \cref{lemma:MC_outer}.
  
  By chasing through the definitions one sees that the composite of the maps \eqref{eq:cMCb_Hom} and \eqref{eq:map_eq} is equivariant with respect to the map $\Xi \colon \expb(\lie g) \to \aut[*](\MCb(L_X))$ of \cref{lemma:exp_aut_map} (see also \cite[Proposition~3.18]{Ber}).
  Moreover, if $\lie g \subseteq \Der(L_X \rel L_B)$, then the lower square of the diagram above is $\expb(\lie g)$-equivariant before passing to the indicated components when the right-hand column is equipped with the trivial action.
  
  We will now prove that the action of any $\theta \in \lie g \defeq \Der[\ol U](L_X \rel L_A)$ of degree $0$ annihilates the element $\rho \after \pi \in \Hom ( \CEchains * (L_X), \Pi )$; by \cref{lemma:MC_component_iso}, this will imply that the whole diagram \eqref{eq:map_eq_diag} is $\expb(\lie g)$-equivariant.
  To this end, note that $\theta \act (\rho \after \pi) = \pm \rho \after \theta \after \pi$, so that it is enough to prove that $\rho(\theta(x)) = 0$ for all $x \in L_X$.
  Since $\rho$ is a map of Lie algebras and $\theta(\blank)$ is a derivation, it is enough to prove that $\rho(\theta(x)) = 0$ for all generators $x \in L_X$; let $N \in \NN$ be the maximal degree of a generator of $L_X$.
  Writing $\ol U'$ for the preimage of $\ol U$ in $\aAut[L_A](L_X)$, there is a commutative diagram of Lie algebras
  \[
  \begin{tikzcd}
    \Lie(\ol U') \rar \dar[swap]{\iso} & \Lie \bigl( \aAut[L_A](L_X) \bigr) \rar \dar{\iso} & \Lie \bigl( \aGL \bigl( \Hom \bigl( (L_X)_{\le N}, \Pi_{\le N} \bigr) \bigr) \bigr) \dar{\iso} \\
    \Der[\ol U](L_X \rel L_A)_0 \rar & \Cycles 0 \bigl( \Der(L_X \rel L_A) \bigr) \rar & \gl \bigl( \Hom \bigl( (L_X)_{\le N}, \Pi_{\le N} \bigr) \bigr)
  \end{tikzcd}
  \]
  where the middle and right-hand vertical isomorphisms are those of \cref{prop:Lie_aEaut} and \cref{lemma:Lie_GL}, respectively, and the left-hand square is provided by the proof of \cref{prop:exp_eq}.
  The right-hand horizontal maps are given by acting via precomposition, and the right-hand square is seen to commute by chasing through the definitions.
  Since $\ol U \subseteq \aEaut[L_A](L_X)_\rho$, we have $\ol U' \subseteq \aAut[L_A](L_X)_\rho$, and so the upper composite lands in the Lie algebra of the stabilizer of $\rho$.
  Thus the lower composite lands in the Lie subalgebra of those endomorphisms vanishing on $\rho$; this finishes the proof.
\end{proof}

\begin{lemma} \label{lemma:Hom_eq}
	In the situation of \cref{prop:map_model}, if $\Pi$ is abelian and $\rho$ vanishes on $L_B$, then there is an $\expb(\Der[\ol U](L_X \rel L_A))$-equivariant weak equivalence of simplicial sets
	\[ \MCb \bigl( \trunc 0 { \Hom \bigl( \shift \indec[L_B](L_X), \Pi \bigr) } \bigr)  \xlongto{\eq}  \cMCb \bigl( \trunc 0 { \Hom \bigl( \CEchains * (L_X), \Pi \rel[\big] \CEchains * (L_B) \bigr)^{\rho \after \pi} } \bigr) \]
  (recall that $\indec[L_B](L_X)$ denotes relative indecomposables, see \cref{def:indec}).
  Here the action of $\expb(\Der[\ol U](L_X \rel L_A))$ on the domain is the action of \cref{lemma:MC_outer} associated to the outer action $\tilde \rho_*$ given by $(f \act \theta)(\shift \eqcl x) \defeq (-1)^{\deg \theta} f( \shift \eqcl {\theta(x)} )$ and $\tilde \rho_*(\theta) \defeq \tilde \rho \act \theta$, where $\tilde \rho \colon \shift \indec[L_B](L_X) \to \Pi$ is the map induced by $\rho$.
\end{lemma}

\begin{proof}
	First note that $\Hom ( \CEchains * (L_X), \Pi )$ is abelian since $\Pi$ is (in particular the twist by $\rho \after \pi$ can be omitted), and that the map
	\[ \pr^*  \colon  \Hom \bigl( \quot {\CEchains * (L_X)} {\CEchains * (L_B)}, \Pi \bigr)  \longto  \Hom \bigl( \CEchains * (L_X), \Pi \rel[\big] \CEchains * (L_B) \bigr) \]
	is an isomorphism of chain complexes.
	By the proof of \cref{lemma:indec_is_ho}, the projection
  \[ p \colon \quot {\CEchains * (L_X)} {\CEchains * (L_B)} \longto \shift \indec[L_B](L_X) \]
  is a quasi-isomorphism, and hence the map
	\[ p^* \colon \Hom \bigl( \shift \indec[L_B](L_X), \Pi \bigr)  \longto  \Hom \bigl( \quot {\CEchains * (L_X)} {\CEchains * (L_B)}, \Pi \bigr) \]
	is a quasi-isomorphism as well.
  It is compatible with the respective outer actions of the derivation dg Lie algebra $\Der[\ol U](L_X \rel L_A)$, and filtration preserving when each side is equipped with the respective complete filtration induced by the filtration $\Pi_{\ge k}$ of $\Pi$.
  Hence we obtain a weak equivalence
  \[ \cMCb \bigl( \trunc 0 { \Hom \bigl( \shift \indec[L_B](L_X), \Pi \bigr) } \bigr)  \xlongto{\eq}  \cMCb \bigl( \trunc 0 { \Hom \bigl( \CEchains * (L_X), \Pi \rel[\big] \CEchains * (L_B) \bigr) } \bigr) \]
  that is $\expb(\Der[\ol U](L_X \rel L_A))$-equivariant.
  Since $\shift \indec[L_B](L_X)$ is finite-dimensional, the filtration of $\trunc 0 { \Hom \bigl( \shift \indec[L_B](L_X), \Pi \bigr) }$ stabilizes in each homological degree.
  This implies the claim.
\end{proof}

\begin{remark} \label{rem:more_general_with_CE}
  For the rest of this paper we work in the setting of \cref{lemma:Hom_eq}, i.e.\ we assume that $\Pi$ is abelian and that $\rho$ vanishes on $L_B$.
  This simplifies everything that follows and yields nicer results.
  However, by continuing to use the realization
  \[ \cMCb \bigl( \trunc 0 { \Hom \bigl( \CEchains * (L_X), \Pi \rel[\big] \CEchains * (L_B) \bigr)^{\rho \after \pi} } \bigr) \]
  one could likely carry out all arguments in the more general situation as well.
  This would mainly require versions of \cref{lemma:MC_outer,lemma:MC_forms} for $\cMCb$.
\end{remark}

\begin{proposition} \label{prop:bdlaut_model}
  In the situation of \cref{lemma:Hom_eq}, the simplicial set
  \begin{equation} \label{eq:bdlaut_MC}
  	\B \bdlaut{\nerve{L_A}}{\nerve{L_B}}{\nerve{\Pi}} \bigl( \nerve{\rho} \bigr)_{\ol U(\QQ)}
  \end{equation}
  is weakly equivalent to the Maurer--Cartan space of the dg Lie algebra
  \[ \lie g_{L_A}^{L_B}(\rho)_{\ol U}  \defeq  \trunc 0 { \Hom \bigl( \shift \indec[L_B](L_X), \Pi \bigr) }  \rsemidir[\tilde \rho_*] \Der[\ol U](L_X \rel L_A) \]
  where the outer action is given by $(f \act \theta)(\shift \eqcl x) \defeq (-1)^{\deg \theta} f( \shift \eqcl {\theta \act x} )$ and $\tilde \rho_*(\theta) \defeq \tilde \rho \act \theta$, where $\tilde \rho \colon \shift \indec[L_B](L_X) \to \Pi$ is the map induced by $\rho$.
  This weak equivalence can be lifted to a weak equivalence of simplicial sets with an action of the normalizer $N$ of the preimage of $\ol U(\QQ)$ in $\Aut[L_A](L_X)_{\rho}$.
\end{proposition}

\begin{proof}
	By \cref{prop:exp_eq,prop:map_model,lemma:Hom_eq}, there is a weak equivalence between the simplicial set \eqref{eq:bdlaut_MC} and
	\[ \B \Bigl( \MCb \bigl( \trunc 0 { \Hom \bigl( \shift \indec[L_B](L_X), \Pi \bigr) } \bigr), \expb \bigl( \Der[\ol U](L_X \rel L_A) \bigr), * \Bigr) \]
	where the middle term acts on the left-hand term via the action of \cref{lemma:MC_outer} associated to the outer action $\tilde \rho_*$.
	Then \cref{lemma:MC_outer} implies the first claim.
	For the equivariance, note that the outer action $\tilde \rho_*$ is compatible with the actions of $N$; the result then follows from the equivariance of \cref{prop:exp_eq,prop:map_model} and the naturality of \cref{lemma:Hom_eq}.
\end{proof}

The preceding proposition shows that when $U$ arises as the $\QQ$-points of a normal unipotent subgroup, we can model the space $\B \bdlaut{A}{B}{K}(\xi)_U$ (when $A$, $B$, and $X$ are rational spaces) together with its conjugation action of $\Aut[L_A](L_X)_\rho$.
This space sits in a fiber sequence
\[ \B \bdlaut{A}{B}{K}(\xi)_U  \longto  \B \bdlaut{A}{B}{K}(\xi)  \longto  \B \bigl( \quot {\Eaut[A](X)_{\eqcl {\xi}}} U \bigr) \]
and, following Berglund--Zeman \cite{BZ}, we would like to say that we can recover the total space from the fiber together with its action of $\quot {\Eaut[A](X)_{\eqcl {\xi}}} U$.
However, this action is not strict and can not be directly related to the action of $\Aut[L_A](L_X)_{\rho}$.
To circumvent this problem, below we recall from \cite[§3.5]{BZ} that $\B \bdlaut{A}{B}{K}(\xi)_U$ is, as an $\Aut[L_A](L_X)_{\rho}$-space, weakly equivalent to
\[ \B \bigl( \map[B](X, K)_{\xi}, \aut[A](X), \quot {\Eaut[A](X)_{\eqcl {\xi}}} U \bigr) \]
where the action is from the right via the map $\real$ of \cref{def:real}.
Using \cref{lemma:hquot_hcoinv} this proves that, when the group homomorphism $\Aut[L_A](L_X)_{\rho} \to \quot {\Eaut[A](X)_{\eqcl {\xi}}} U$ has a section, we can recover $\B \bdlaut{A}{B}{K}(\xi)$ from the dg Lie algebra appearing in \cref{prop:bdlaut_model} together with its $\Aut[L_A](L_X)_{\rho}$-action.
We will see below that such a section always exists when $U$ is the unipotent radical of $\Eaut[A](X)_{\eqcl{\xi}}$.

\begin{definition}
  Let $G$ be a group-like simplicial monoid, $H \subseteq \hg{0}(G)$ a subgroup, and $M$ a right $G$-module in simplicial sets.
  We define a simplicial set
  \[ \B[H] (M, G)  \defeq  \B \bigl( M, G, \quot {\hg{0}(G)} H \bigr) \]
  where $G$ acts on the set $\quot {\hg{0}(G)} H$ via the canonical map $G \to \hg{0}(G)$.
  Note that $\B[H] (M, G)$ admits a canonical right action of the group $\quot {\Norm[\hg{0}(G)](H)} H$.
  
  When $B \to A \to X$ are cofibrations and $\xi \colon X \to K$ is a map of Kan complexes, we write
  \[ \B[H] \bdlaut{A}{B}{K}(\xi)_G  \defeq  \B[H] \bigl( \map[B](X, K)_\xi, \aut[A](X)_G \bigr) \]
  where $H \subseteq G \subseteq \Eaut[A](X)_{\eqcl{\xi}_B}$ are subgroups.
\end{definition}

\begin{lemma} \label{lemma:submonoid_bar}
  Let $G$ be a group-like simplicial monoid, $H \subseteq \hg{0}(G)$ a subgroup, and $M$ a right $G$-module in simplicial sets.
  Then the following map, induced by the inclusion $G_H \to G$,
  \[ \B( M, G_H, * )  \longto  \B \bigl( M, G, \quot {\hg 0 (G)} H \bigr) = \B[H] (M, G) \]
  is a weak equivalence (recall that $G_H$ denotes the components of $G$ belonging to $H$).
\end{lemma}

\begin{proof}
  This follows from \cref{lemma:bar_stabilizer} applied to the left $G$-module $M$, the right $G$-module $\quot {\hg 0 (G)} H$, and the class of the identity element $\eqcl 1 \in \quot {\hg 0 (G)} H$.
  Note that each component of $\B \bigl( M, G, \quot {\hg 0 (G)} H \bigr)$ contains an element of the form $(m, \eqcl 1)$ with $m \in M_0$, since $G$ acts transitively on $\quot {\hg 0 (G)} H$.
\end{proof}

\begin{lemma}[Berglund--Zeman] \label{lemma:left_conj_eq}
  Let $f \colon G \to H$ be a map of simplicial monoids, $M$ a right $G$-module in simplicial sets, $D$ a discrete group, and $a \colon D \to G$ a map of simplicial monoids.
  Then there is a weak equivalences in the homotopy category of simplicial sets with a $D$-action
  \[ B(M, G, H)_{\mathrm{conj}}  \xlonghto{\eq}  B(M, G, H)_{\mathrm{right}} \]
  where on left-hand side $D$ acts from the right on $M$ and by conjugation on both $G$ and $H$, and on the right-hand side $D$ only acts on $H$ on the right.
  Moreover, this zig-zag consists of monoidal natural transformations of oplax symmetric monoidal functors from the category of tuples $(D, G, H, M, a, f)$ to the category of pairs $(D, X)$ of a group $D$ and a simplicial set $X$ with a $D$-action (both of which we equip with the cartesian monoidal structure).
\end{lemma}

\begin{proof}
  For every left $G$-module $N$, \cite[Lemma~3.18]{BZ} provides a natural zig-zag of $D$-equivariant homotopy equivalences between $B(M, G, N)$, where $D$ acts on the right on $M$, on the left on $N$, and by conjugation on $G$, and the same space with the trivial action.
  Applying this to $N = H$ yields the zig-zag in the statement of the lemma.
  To see that this zig-zag is equivariant with respect to the stated $D$-actions, one notes that the left $G$-module $H$ has a compatible right $H$-module structure and then one uses that the zig-zag is natural in $N$.
  An elementary verification shows that the construction of the zig-zag provided in \cite[Lemma~3.18]{BZ} is monoidal in the claimed sense.
\end{proof}

So far, we have seen how to model the space $\B \bdlaut{\rat A}{\rat B}{\rat K}(\rat \xi)$ by modeling the space $\B[U] \bdlaut{\rat A}{\rat B}{\rat K}(\rat \xi)$ together with its $\quot {\Eaut[\rat A](\rat X)_{\eqcl {\rat \xi}}} U$-action.
By the same argument, we can recover $\B \bdlaut{A}{B}{K}(\xi)$ from any space $\B[G] \bdlaut{A}{B}{K}(\xi)$ together with its $\quot {\Eaut[A](X)_{\eqcl{\xi}_B}} G$-action.
The following lemma shows how to relate these two constructions.
It is a generalization of \cite[Proposition~3.10]{BZ}.
We first fix the following notation.

\begin{notation} \label{not:group_triple}
  Let $B \to A \to X$ be cofibrations of simply connected Kan complexes of the homotopy types of finite CW-complexes, modeled by the quasi-free maps $L_B \to L_A \to L_X$ of positively graded finitely generated quasi-free dg Lie algebras such that $L_A \to L_X$ is minimal.
  Furthermore let $K$ be a simply connected Kan complex, and $\xi \colon X \to K$ a map modeled by the map $\rho \colon L_X \to \Pi$ of dg Lie algebras such that $\Pi$ is of finite type, positively graded, abelian, and has trivial differential.
  Furthermore, let $G \subseteq \Eaut[A](X)_{\eqcl{\xi}_B}$ be a finite-index subgroup, $\ol U \subseteq \aEaut[L_A](L_X)_\rho$ a unipotent algebraic subgroup, and $N \subseteq \Eaut[L_A](L_X)_\rho$ a subgroup containing $U \defeq \ol U(\QQ)$ as a normal subgroup.
  We write $q_G \colon G \to \Eaut[\rat A](\rat X)_{\eqcl{\rat \xi}} \iso \Eaut[L_A](L_X)_\rho$ for the map induced by rationalization and denote by $\lie g_{L_A}^{L_B}(\rho)_{\ol U}$ the dg Lie algebra of \cref{prop:bdlaut_model} (equipped with its action of the preimage of $N$ in $\Aut[L_A](L_X)_\rho$).
\end{notation}

\begin{remark} \label{rem:Pi}
  The assumptions of \cref{not:group_triple} imply that $\Pi$ is isomorphic to $\hg * (\Loops K) \tensor \QQ$ considered as an abelian graded Lie algebra with trivial differential.
  A simply connected space $K$ admits a model of this form if and only if it is rationally equivalent to a product of Eilenberg--Mac Lane spaces.
  If this is the case, we can model any map of Kan complexes $\xi \colon X \to K$ up to homotopy by a map $\rho \colon L_X \to \hg * (\Loops K) \tensor \QQ$ (assuming that $X$ is simply connected).
  Factoring $\xi$ as a cofibration $\xi' \colon X \to K'$ followed by a trivial fibration $f \colon K' \to K$, we can arrange $\rho$ to model $\xi'$ strictly; moreover $f$ induces a weak equivalence $\B \bdlaut{A}{B}{K}(\xi') \to \B \bdlaut{A}{B}{\B \SO}(\xi)$.
\end{remark}

\begin{lemma} \label{lemma:aut_rat_eq}
  In the situation of \cref{not:group_triple}, rationalization induces a rational equivalence
  \[ \B[\inv{q_G}(U)] \bdlaut{A}{B}{K}(\xi)_G  \xlongto{\req}  \B[U] \bdlaut{\rat A}{\rat B}{\rat K}(\rat \xi) \]
  of simplicial sets with a $\quot {\inv{q_G}(N)} {\inv{q_G}(U)}$-action, where the action on the target is induced by $q$.
\end{lemma}

\begin{proof}
	By \cref{lemma:rel_Ho_algebraic}, the action of $U$ on $\Ho n (X, A; \QQ)$ extends to an algebraic representation of $\ol U$ for any $n$.
	Since $\ol U$ is unipotent, so is this algebraic representation (see e.g.\ \cite[Proposition~14.3]{Mil}).
	Hence the action of $U$, and thus $\inv{q_G}(U)$, on $\Ho n (X, A; \QQ)$ is nilpotent.
	This implies, by \cref{lemma:Baut_virtually_nilpotent}, that $\B \aut[A](X)_{\inv{q_G}(U)}$ and $\B \bdlaut{A}{B}{K}(\xi)_{\inv{q_G}(U)}$ are virtually nilpotent.
	Furthermore note that there is a natural isomorphism $\rHo * (\rat X; \ZZ) \iso \rHo * (\rat X; \QQ)$ (this follows, for example, from \cite[Ch.~V, Proposition~3.3]{BK} combined with \cite[Ch.~X, Corollary~3.3]{GJ}), and thus the same is true for the relative homology of the pair $(\rat X, \rat A)$.
	By \cite[Theorem~E]{Wal} this also implies that $\rat B$, $\rat A$, and $\rat X$ have the homotopy types of finite-dimensional CW-complexes.
	Hence $\B \aut[\rat A](\rat X)_U$ and $\B \bdlaut{\rat A}{\rat B}{\rat K}(\rat \xi)_U$ are nilpotent by \cref{lemma:Bautbdl_nilpotent}.
	
  Now consider the map
  \[ \B \aut[A](X)_{\id}  \longto  \B \aut[\rat A](\rat X)_{\id} \]
  induced by rationalization.
  It is a rational homology equivalence by \cite[Lemma~3.3]{ER}\footnote{They do not use the same model for rationalization as we do, but the same argument applies.} and \cref{lemma:bar_eq}; since both domain and codomain are simply connected, it is also a rational homotopy equivalence.
  Thus the map
  \begin{equation} \label{eq:aut_rat}
  	\B \aut[A](X)_{\inv{q_G}(U)}  \longto  \B \aut[\rat A](\rat X)_U
  \end{equation}
  induces a $\QQ$-isomorphism (see \cref{def:Q-iso}) on $\hg k$ for $k \ge 2$.
  We will now prove that it also induces a $\QQ$-isomorphism on $\hg 1$.
  
  It follows from \cref{prop:arithmetic} that the group $U \intersect \im(q_G)$ is an arithmetic subgroup of $\ol U$ (for example by \cite[§1.1]{Ser}), and hence its inclusion into $U$ is a $\QQ$-isomorphism by \cite[Lemma~2.13]{BZ}.
  Since the kernel of $q_G$ is finite, also by \cref{prop:arithmetic}, this implies that the composite
  \[ \inv{q_G}(U)  \longsurj  U \intersect \im(q_G)  \longincl  U \]
  is a $\QQ$-isomorphism, too.
  Hence the map \eqref{eq:aut_rat} is a rational homotopy equivalence.

	Now consider the commutative diagram
	\begin{equation} \label{eq:bdlaut_fib_map}
	\begin{tikzcd}
		\map[B](X, K)_{\xi} \rar \dar[swap]{\req} & \B \bdlaut{A}{B}{K}(\xi)_{\inv{q_G}(U)} \rar \dar & \B \aut[A](X)_{\inv{q_G}(U)} \dar{\req} \\
		\map[\rat B](\rat X, \rat K)_{\rat \xi} \rar & \B \bdlaut{\rat A}{\rat B}{\rat K}(\rat \xi)_U \rar & \B \aut[\rat A](\rat X)_U
	\end{tikzcd}
	\end{equation}
	whose rows are fiber sequences by \cref{lemma:bar_fiber_sequence}.
  Note that the two upper right-hand terms are virtually nilpotent and all other terms are nilpotent by \cref{lemma:Baut_virtually_nilpotent,lemma:Bautbdl_nilpotent} (for the two mapping spaces, we apply the latter lemma to the cases $A = X$ and $A = \rat X$).
	The left-hand vertical map is a rational homotopy equivalence by the same argument as in \cite[Proof of Lemma~3.3]{ER}.
	It follows from \cite[Theorem~3.1]{HMR} that there is a five lemma for $\QQ$-isomorphisms between nilpotent groups.
	By passing to finite-index subgroups of $\hg 1$ in the upper row, we can arrange for all groups in the map of long exact sequences induced by \eqref{eq:bdlaut_fib_map} to be nilpotent.
	Since a map of groups is a $\QQ$-isomorphism if and only if it is a $\QQ$-isomorphism when restricted to a fixed finite-index subgroup, this implies that the middle vertical map in \eqref{eq:bdlaut_fib_map} is a rational homotopy equivalence.
  Then, by \cref{lemma:req}, it is also a rational homology equivalence.
  This implies the claim by \cref{lemma:submonoid_bar}.
\end{proof}

Putting everything we have done so far together, we obtain the following result.
As explained above, this provides a model for the space $\B \bdlaut{A}{B}{K}(\xi)$ when there is a normal subgroup $U \subseteq \Eaut[L_A](L_X)_{\rho}$ such that the group homomorphism $\Aut[L_A](L_X)_\rho \to \quot {\Eaut[L_A](L_X)_{\rho}} U$ has a section.
We will see below that this can always be arranged by choosing $U$ to be the $\QQ$-points of the unipotent radical of $\aEaut[L_A](L_X)_{\rho}$.

\begin{proposition} \label{thm:model}
  In the situation of \cref{not:group_triple}, there are two morphisms
  \[ \B[\inv{q_G}(U)] \bdlaut{A}{B}{K}(\xi)_G \xlongto{\req} \B[U] \bdlaut{\rat A}{\rat B}{\rat K}(\rat \xi) \xlonghto{\eq} \MCb \bigl( \lie g_{L_A}^{L_B}(\rho)_{\ol U} \bigr) \]
  where the first one is a rational equivalence of simplicial sets with a $\quot {\inv{q_G}(N)} {\inv{q_G}(U)}$-action and the second one is a weak equivalence in the homotopy category of simplicial sets with an action of the preimage $N'$ of $N$ in $\Aut[L_A](L_X)_\rho$.
\end{proposition}

\begin{proof}
  The first map is provided by \Cref{lemma:aut_rat_eq}.
  Furthermore, \cref{prop:bdlaut_model,lemma:submonoid_bar,lemma:left_conj_eq} yield a zig-zag
  \[ \MCb (\lie g_{L_A}^{L_B}(\rho)_{\ol U})  \xlonghto{\eq}  \B[\ol U(\QQ)] \bdlaut{\nerve{L_A}}{\nerve{L_B}}{\nerve{\Pi}} \bigl( \nerve{\rho} \bigr) \]
  of weak equivalences of simplicial sets with an $N'$-action.
  Lastly, \cref{thm:Baut_functor} provides (using \cref{lemma:rat}) weak equivalences
  \begin{align*}
  	\aut[\rat A](\rat X)  &\xlonghto{\eq}  \aut[\nerve{L_A}] \bigl( \nerve{L_X} \bigr) \\
  	\map[\rat B](\rat X, \rat K)  &\xlonghto{\eq}  \map[\nerve{L_B}] \bigl( \nerve{L_X}, \nerve{\Pi} \bigr)
  \end{align*}
  such that the lower map is equivariant with respect to the upper map in the appropriate sense, and $\rat \xi$ corresponds to $\nerve{\rho}$ under the lower weak equivalence.
  Combining everything, we obtain the desired zig-zag of weak equivalences of simplicial sets with an $N'$-action.
\end{proof}

We will now, as promised, prove the existence of a normal unipotent subgroup $U$ such that the group homomorphism $\Aut[L_A](L_X)_\rho \to \quot {\Eaut[\rat A](\rat X)_{\eqcl{\rat \xi}}} U$ has a section.
We begin by introducing some notation for the resulting quotients and the corresponding derivation dg Lie algebras.

\begin{definition} \label{def:redEaut}
  Let $L_A \to L_X$ and $\rho \colon L_X \to \Pi$ be maps of dg Lie algebras as in \cref{not:group_triple}.
  We define $\aredEaut[L_A] (\rho)$ to be the maximal reductive quotient of $\aEaut[L_A] (L_X)_\rho$ and set $\Deru[\rho](L_X \rel L_A)  \defeq  \Der[\ol U](L_X \rel L_A)$ where $\ol U \subseteq \aEaut[L_A] (L_X)_\rho$ is the unipotent radical.
  When $\rho$ is trivial, we simply write $\aredEaut[L_A] (L_X)$ and $\Deru(L_X \rel L_A)$.
\end{definition}

\begin{lemma} \label{lemma:areal_section}
  In the situation of \cref{def:redEaut} the composite map of algebraic groups
  \[ \bar r \colon \aAut[L_A](L_X)_\rho  \xlongto{\areal}  \aEaut[L_A] (L_X)_\rho  \xlongto{\ol \pr}  \aredEaut[L_A] (\rho) \]
  has a section $\bar s$.
  In particular the composite map of groups
  \[ \Aut[L_A](L_X)_\rho  \xlongto{\real}  \Eaut[\nerve{L_A}] \bigl( \nerve{L_X} \bigr)_{\eqcl{\nerve{\rho}}}  \iso  \aEaut[L_A] (L_X)_\rho (\QQ)  \xlongto{\ol \pr_\QQ}  \aredEaut[L_A] (\rho) (\QQ) \]
  has a section $s$.
\end{lemma}

\begin{proof}
  By definition, the map $\ol \pr$ is a quotient map with kernel a unipotent algebraic subgroup.
  The same is true for $\areal$ by \cref{prop:ES}.
  Hence the kernel of $\bar r$ is an extension of unipotent algebraic groups and thus unipotent (see e.g.\ \cite[Corollary~14.7]{Mil}).
  Since $\aredEaut[L_A] (\rho)$ is reductive, this implies by \cref{lemma:unipotent_radical} that the kernel of $\bar r$ is the unipotent radical of $\aAut[L_A](L_X)_\rho$.
  Then it follows from \cref{prop:Mostow} that $\bar r$ has a section.
\end{proof}

The following theorem provides the promised model for $\B \bdlaut{A}{B}{K}(\xi)$ and is the first main result of this paper.
It is a bundle version of \cite[Corollary~3.32]{BZ}.

\begin{theorem} \label{cor:Baut_eq}
  Let $B \to A \to X$ be cofibrations of simply connected Kan complexes of the homotopy types of finite CW-complexes, modeled by the quasi-free maps $L_B \to L_A \to L_X$ of positively graded finitely generated quasi-free dg Lie algebras such that $L_A \to L_X$ is minimal.
  Furthermore let $K$ be a simply connected Kan complex, and $\xi \colon X \to K$ a map modeled by the map $\rho \colon L_X \to \Pi$ of dg Lie algebras such that $\Pi$ is of finite type, positively graded, abelian, and has trivial differential and such that $\rho(L_B) = 0$.
  Moreover let $G \subseteq \Eaut[A](X)_{\eqcl{\xi}_B}$ be a finite-index subgroup and $\redim{G}$ the image of $G$ in $\aredEaut[L_A] (\rho)(\QQ)$.
  Lastly choose a section $s$ as in \cref{lemma:areal_section}.
  Then there is a rational equivalence in the homotopy category of simplicial sets
  \begin{gather*}
  	\Theta_s \colon \B \bdlaut{A}{B}{K}(\xi)_G  \xlonghto{\req}  \hcoinv { \MCb \bigl( \lie g_{L_A}^{L_B}(\rho) \bigr) } {\redim{G}}
  	\shortintertext{where}
  	\lie g_{L_A}^{L_B}(\rho)  \defeq  \trunc 0 { \Hom \bigl( \shift \indec[L_B](L_X), \Pi \bigr) }  \rsemidir[\tilde \rho_*] \Deru[\rho](L_X \rel L_A)
  \end{gather*}
  is the dg Lie algebra of \cref{prop:bdlaut_model}, on which $\redim{G}$ acts by conjugation via the composite map
  \[ \redim{G}  \longincl  \aredEaut[L_A] (\rho)(\QQ)  \xlongto{s}  \Aut[L_A](L_X)_\rho \]
  of groups.
  Moreover $\Gamma_G$ is an arithmetic subgroup of $\aredEaut[L_A](\rho)$ and $\Theta_s$ induces an isomorphism on (co)homology with respect to any local coefficient system pulled back from $\B \redim{G}$.
\end{theorem}

\begin{proof}
  That $\Gamma_G$ is an arithmetic subgroup of $\aredEaut[L_A] (\rho)$ follows from \cref{prop:arithmetic} since the image of an arithmetic subgroup under a quotient map of algebraic groups is again an arithmetic subgroup (see e.g.\ \cite[§1.1]{Ser}).
  By \cref{thm:model}, the section $s$ yields a rational equivalence in the homotopy category of simplicial sets with a $\redim{G}$-action
  \[ Z \defeq \B \bigl( \map[B](X, K)_\xi, \aut[A](X)_G, \redim{G} \bigr)  \longto  \MCb \bigl( \lie g_{L_A}^{L_B}(\rho) \bigr) \]
  where $\redim{G}$ acts on $\lie g_{L_A}^{L_B}(\rho)$ as in the statement.
  Taking homotopy orbits and using \cref{lemma:hquot_hcoinv} yields the zig-zag
  \[ \B \bdlaut{A}{B}{K}(\xi)_G  \xlongfrom{\eq}  \hcoinv Z {\redim{G}}  \xlonghto{\req}  \hcoinv { \MCb \bigl( \lie g_{L_A}^{L_B}(\rho) \bigr) } {\redim{G}} \]
  where the right-hand zig-zag is a rational homotopy equivalence by \cref{lemma:bar_fiber_sequence}, and it induces an isomorphism on (co)homology with respect to any local coefficient system pulled back from $\B \redim{G}$ by \cref{lemma:hcoinv_coho}.
\end{proof}

The following in particular provides a cdga model for the space $\B \bdlaut{A}{B}{K}(\xi)$.
It is a bundle version of \cite[Theorem~3.39]{BZ} that also incorporates non-trivial coefficients.

\begin{corollary} \label{cor:forms_eq}
  In the situation of \cref{cor:Baut_eq}, let $M$ be a $\redim{G}$-representation in rational vector spaces.
  Then there is a quasi-isomorphism
  \[ \Psi_s  \colon  \Forms * \bigl( \B \bdlaut{A}{B}{K}(\xi)_G; M \bigr)  \xlonghto{\eq}  \Forms * \bigl( \B \redim{G}; \CEcochains * \bigl( \lie g_{L_A}^{L_B}(\rho) ; M \bigr) \bigr) \]
  in the homotopy category of cochain complexes.
  Here $\gamma \in \redim{G}$ acts on the cochain complex $\CEcochains * ( \lie g_{L_A}^{L_B}(\rho); M )$ by $(\inv \gamma, \gamma)$, where the action on $\lie g_{L_A}^{L_B}(\rho)$ is by conjugation via $s$.
  Moreover, the map $\Psi_s$ lifts to a zig-zag of monoidal natural quasi-isomorphisms between lax symmetric monoidal functors from $\redim{G}$-representations to cochain complexes (in particular it is a zig-zag of algebras when $M$ is an algebra).
\end{corollary}

\begin{proof}
  This follows from \cref{cor:Baut_eq} by applying $\Forms * (\blank, M)$ and using \cref{lemma:forms_orbits,lemma:MC_forms}.
  By $M$ we mean the respective pullback of the local system $M$ on $\B \redim{G}$ of \cref{def:BG_local_system}; that all the maps in the zig-zag of \cref{cor:Baut_eq} are compatible with this follows from its construction and \cref{lemma:hquot_hcoinv_local_system}.
\end{proof}

We conclude this subsection by inferring a description of the cohomology of $\B \bdlaut{A}{B}{K}(\xi)$; it is a bundle analogue of \cite[Theorem~3.41]{BZ} that also incorporates non-trivial coefficients.
The statement can be viewed as a strong form of collapse of the Serre spectral sequence associated to the fiber sequence
\[ \B \bdlaut{A}{B}{K}(\xi)_U  \longto  \B \bdlaut{A}{B}{K}(\xi)_G  \longto  \B \redim{G} \]
where $U$ denotes the kernel of the quotient map $G \to \redim{G}$.
Note, in particular, that we obtain an isomorphism of algebras without passing to an associated graded.

\begin{corollary} \label{cor:cohomology}
	In the situation of \cref{cor:Baut_eq}, let $M$ be a $\redim{G}$-representation in rational vector spaces and $\bar s$ a section as in \cref{lemma:areal_section}.
  Then there is an isomorphism
  \[ \Phi_{\bar s}  \colon  \Coho * \bigl( \B \bdlaut{A}{B}{K}(\xi)_G; M \bigr)  \xlongto{\iso}  \Coho * \bigl( \B \redim{G}; \CEcoho * \bigl( \B \bdlaut{A}{B}{K}(\xi); M \bigr) \bigr) \]
  of graded vector spaces.
  Here $\gamma \in \redim{G}$ acts on $\CEcoho * ( \lie g_{L_A}^{L_B}(\rho); M )$ by $(\inv \gamma, \gamma)$, where the action on $\lie g_{L_A}^{L_B}(\rho)$ is by conjugation via $\bar s_\QQ$.
  Moreover, the map $\Phi_{\bar s}$ is a monoidal natural isomorphism between lax symmetric monoidal functors from $\redim{G}$-representations to graded vector spaces (in particular it is an isomorphism of algebras when $M$ is an algebra).
\end{corollary}

\begin{proof}
  This follows by combining \cref{cor:forms_eq,lemma:forms_cohomology,lemma:forms_split_coeff}.
  To apply the latter, we need to prove that $\CEcochains * ( \lie g_{L_A}^{L_B}(\rho) M )$ is split as a cochain complex of $\redim{G}$-representations (see also \cref{rem:local_system_groupoid}).
  Note that $\Deru[\rho](L_X \rel L_A)$ and $\Hom ( \shift \indec[L_B](L_X), \Pi )$ are degree-wise finite-dimensional algebraic representation of $\aAut[L_A](L_X)_\rho$ by \cref{lemma:Der_algebraic,lemma:indecomposables_representation}, respectively.
  Hence the same is true for $\lie g_{L_A}^{L_B}(\rho)$ and, since this dg Lie algebra is non-negatively graded, also for its Chevalley--Eilenberg chains.
  Since $\aredEaut[L_A](\rho)$ is reductive, this implies that $\CEchains * ( \lie g_{L_A}^{L_B}(\rho) )$ is a degree-wise semi-simple $\aredEaut[L_A](\rho)$-representation.
  Hence it is split as a cochain complex of $\redim{G}$-representations (for example by \cite[Lemmas~B.1 and B.2]{BM}).
  By the Universal Coefficient Theorem, this remains true upon applying $\Hom(\blank, M)$, so that $\CEcochains * ( \lie g_{L_A}^{L_B}(\rho); M )$ is split as required.
\end{proof}

\subsection{Equivariant models for block diffeomorphisms} \label{sec:block_models}

In this subsection we use \cref{prop:block_diff} to specialize the results of \cref{sec:model} to obtain a rational model for the classifying space $\B \BlDiff[\bdry](M)$ of block diffeomorphisms of a high-dimensional manifold $M$ with simply connected boundary.
This is one of the main results of this paper.

\begin{theorem} \label{thm:block_diff}
  Let $M$ be a simply connected compact smooth manifold of dimension $n \ge 6$ with simply connected boundary.
  Write $\bdry_0 M \defeq \bdry M \setminus \openDisk{n-1}$ and let the inclusions $\bdry_0 M \to \bdry M \to M$ be modeled by the quasi-free maps $L_0 \to L_\bdry \to L_M$ of positively graded finitely generated quasi-free dg Lie algebras such that $L_\bdry \to L_M$ is minimal.
  Furthermore let the classifying map $\sttang M \colon M \to \B \SO$ of the oriented stable tangent bundle of $M$ be modeled by the map $\rho \colon L_M \to \hg * (\SO) \tensor \QQ$ of dg Lie algebras such that $\rho(L_0) = 0$, and write $\BlGamma{\bdry}(M)$ for the image of $\hg 0 (\B \BlDiff[\bdry](M))$ in $\aredEaut[L_\bdry] (\rho)(\QQ)$.
  Lastly choose a section $s$ as in \cref{lemma:areal_section}.
  Then there is a rational equivalence in the homotopy category of topological spaces
  \begin{gather*}
    \Theta_s \colon \B \BlDiff[\bdry](M)  \xlonghto{\req}  \gr[\big] { \hcoinv {\MCb \bigl( \tilde{\lie g}_{L_\bdry}^{L_0}(\rho) \bigr)} {\BlGamma{\bdry}(M)} }
    \shortintertext{where}
    \tilde{\lie g}_{L_\bdry}^{L_0}(\rho)  \defeq  \trunc 0 { \Hom \bigl( \shift \indec[L_0](L_M), \hg * (\SO) \tensor \QQ \bigr) }  \rsemidir[\tilde \rho_*] \Deru[\rho](L_M \rel L_\bdry)
  \end{gather*}
  is the dg Lie algebra of \cref{prop:bdlaut_model}, on which $\BlGamma{\bdry}(M)$ acts by conjugation via the composite map
  \[ \BlGamma{\bdry}(M)  \longincl  \aredEaut[L_\bdry] (\rho)(\QQ)  \xlongto{s}  \Aut[L_\bdry](L_M)_\rho \]
  of groups.
  Moreover $\BlGamma{\bdry}(M)$ is an arithmetic subgroup of $\aredEaut[L_\bdry](\rho)$ and $\Theta_s$ induces an isomorphism on (co)homology with respect to any local coefficient system pulled back from $\B \BlGamma{\bdry}(M)$.
\end{theorem}

\begin{proof}
  Since the rational cohomology of $\B \SO$ is a free graded commutative algebra on the Pontryagin classes $\pont[i]$, the map $\B \SO \to \prod_{i \ge 1} \K(\QQ, 4i)$ given by the Pontryagin classes is a rational equivalence.
  Hence we can proceed as in \cref{rem:Pi} to obtain a map of dg Lie algebras $\rho \colon L_M \to \hg * (\SO) \tensor \QQ$ that models a cofibrant replacement $\xi \colon \Sing(M) \to K$ of $\Sing(\sttang{M})$.
  Then we have weak equivalences
  \[ \B \bdlaut{\Sing(\bdry)}{\Sing(\bdry_0)}{K}(\xi)  \eq  \B \bdlaut{\Sing(\bdry)}{\Sing(\bdry_0)}{\B \SO} \bigl( \Sing(\sttang{M}) \bigr)  \eq  \B \bdlaut{\bdry}{\bdry_0}{\B \SO}(\sttang{M}) \]
  (using \cref{lemma:Sing_map} for the second one), and the result follows by combining \cref{prop:block_diff,cor:Baut_eq}.
\end{proof}

\begin{remark}
  The algebraic group $\aEaut[L_\bdry](L_M)_\rho$ is the algebraic subgroup of $\aEaut[L_\bdry](L_M)$ consisting of those (homotopy classes of) automorphisms whose induced map on $\Ho * (\indec(L_M)) \iso \rHo * (M; \QQ)$ stabilizes every Pontryagin class $\pont[i](M) \colon \rHo {4i} (M; \QQ) \to \QQ$.
  Similarly, the assumption that $\rho$ can be chosen to vanish on $L_0$ is equivalent to requiring the rational Pontryagin classes of $\bdry_0 M$ to vanish.
\end{remark}

\begin{corollary} \label{cor:forms_block_diff}
  In the situation of \cref{thm:block_diff}, let $P$ be a $\BlGamma{\bdry}(M)$-representation in rational vector spaces.
  Then there is a quasi-isomorphism
  \[ \Psi_s  \colon  \Forms * \bigl( \B \BlDiff[\bdry](M); P \bigr)  \xlonghto{\eq}  \Forms * \bigl( \B \BlGamma{\bdry}(M); \CEcochains * \bigl( \tilde{\lie g}_{L_\bdry}^{L_0}(\rho); P \bigr) \bigr) \]
  in the homotopy category of cochain complexes.
  Here $\gamma \in \BlGamma{\bdry}(M)$ acts on the cochain complex $\CEcochains * ( \tilde{\lie g}_{L_\bdry}^{L_0}(\rho); P )$ by $(\inv \gamma, \gamma)$, where the action on $\tilde{\lie g}_{L_\bdry}^{L_0}(\rho)$ is by conjugation via $s$.
  Moreover, the map $\Psi_s$ lifts to a zig-zag of monoidal natural quasi-isomorphisms between lax symmetric monoidal functors from $\BlGamma{\bdry}(M)$-representations to cochain complexes (in particular it is a zig-zag of algebras when $P$ is an algebra).
\end{corollary}

\begin{proof}
  This follows from \cref{thm:block_diff} as in the proof of \cref{cor:forms_eq}.
\end{proof}

\begin{corollary} \label{cor:coho_block_diff}
  In the situation of \cref{thm:block_diff}, let $P$ be a $\BlGamma{\bdry}(M)$-representation in rational vector spaces and $\bar s$ a section as in \cref{lemma:areal_section}.
  Then there is an isomorphism
  \[ \Phi_{\bar s}  \colon  \Coho * \bigl( \B \BlDiff[\bdry](M); P \bigr)  \xlongto{\iso}  \Coho * \bigl( \B \BlGamma{\bdry}(M); \CEcoho * \bigl( \tilde{\lie g}_{L_\bdry}^{L_0}(\rho); P \bigr) \bigr) \]
  of graded vector spaces.
  Here $\gamma \in \BlGamma{\bdry}(M)$ acts on $\CEcoho * ( \tilde{\lie g}_{L_\bdry}^{L_0}(\rho); P )$ by $(\inv \gamma, \gamma)$, where the action on $\tilde{\lie g}_{L_\bdry}^{L_0}(\rho)$ is by conjugation via $\bar s_\QQ$.
  Moreover, the map $\Phi_{\bar s}$ is a monoidal natural isomorphism between lax symmetric monoidal functors from $\BlGamma{\bdry}(M)$-representations to graded vector spaces (in particular it is an isomorphism of algebras when $P$ is an algebra).
\end{corollary}

\begin{proof}
  This follows from \cref{cor:forms_block_diff} as in the proof of \cref{cor:cohomology}.
\end{proof}

\subsection{Compatibility with gluing constructions} \label{sec:gluing}

In this subsection, we prove that the constructions of \cref{sec:model} are compatible with taking pushouts of maps, i.e.\ that the map of spaces
\[ \B \bdlaut{A}{B}{K}(\xi) \times \B \bdlaut{A'}{B'}{K}(\xi')  \longto  \B \bdlaut{A \cop_C A'}{B \cop_C B'}{K}(\xi \cop_C \xi') \]
corresponds to a map of dg Lie algebras under the rational equivalence of \cref{cor:Baut_eq}.
We begin by fixing the extensive notation we will use throughout.

\begin{notation} \label{not:gluing}
  Let the following be cofibrations of simply connected Kan complexes with the homotopy types of finite CW-complexes
  \[ 
  \begin{tikzcd}[column sep = 20, row sep = 10]
  	X & \lar[tail][swap]{\iota} A & \lar[tail][swap]{\kappa} B & \lar[tail] C \rar[tail] & B' \rar[tail]{\kappa'} & A' \rar[tail]{\iota'} & X'
  \end{tikzcd}
  \]
  modeled by the quasi-free maps
  \[
  \begin{tikzcd}[column sep = 20, row sep = 10]
  	L_X & \lar[tail][swap]{i} L_A & \lar[tail][swap]{j} L_B & \lar[tail] L_C \rar[tail] & L_{B'} \rar[tail]{j'} & L_{A'} \rar[tail]{i'} & L_{X'}
  \end{tikzcd}
  \]
  of positively graded finitely generated quasi-free dg Lie algebras.
  We define
  \[ L_Q \defeq L_B \cop_{L_C} L_{B'}  \qquad  L_R \defeq  L_A \cop_{L_C} L_{A'}  \qquad  L_P \defeq  L_X \cop_{L_C} L_{X'} \]
  and let the following vertical maps be a trivial projective cofibration in $\sSet^{\lincat 2}$ 
  \[
  \begin{tikzcd}
  	B \cop_C B' \rar[tail] \dar[tail]{\eq} & A \cop_C A' \rar[tail] \dar[tail]{\eq} & X \cop_C X' \dar[tail]{\eq} \\
  	Q \rar[tail] & R \rar[tail] & P
  \end{tikzcd}
  \]
  such that $Q$, $R$, and $P$ are fibrant.
  Let $K$ be a simply connected Kan complex modeled by an abelian finite-type dg Lie algebra $\Pi$ with trivial differential.
  Furthermore, let $\zeta \colon P \to K$ be a cofibration and $\sigma \colon L_P \to \Pi$ a map such that $\rho \defeq \restrict \sigma {L_X}$ models $\xi \defeq \restrict \zeta X$ and $\rho' \defeq \restrict \sigma {L_{X'}}$ models $\xi' \defeq \restrict \zeta {X'}$.
  We also write
  \[ \mu  \colon  \vertpair {\aut[A](X) \times \aut[A'](X')} {\map[B](X, K) \times \map[B'](X', K)}  \longhto  \vertpair {\aut[R](P)} { \map[Q] \bigl( P, \FR(K) \bigr) } \]
  for the zig-zag of \cref{lemma:aut_pushout_zig-zag} of pairs of a simplicial monoid and a right module over it,
  \[ \gamma  \colon  \Eaut[A](X) \times \Eaut[A'](X')  \longto  \Eaut[R](P) \]
  for the map of groups induced by $\mu$ on $\hg{0}$,
  \begin{gather*}
    \ol \gamma  \colon  \aEaut[L_A](L_X) \times \aEaut[L_{A'}](L_{X'})  \longto  \aEaut[L_R](L_P) \\
    \ol \Gamma  \colon  \aAut[L_A](L_X) \times \aAut[L_{A'}](L_{X'})  \longto  \aAut[L_R](L_P)
  \end{gather*}
  for the maps of algebraic groups of \cref{lemma:aAut_pushouts} and \cref{lemma:aEaut_pushouts}, and we set $\Gamma \defeq \ol \Gamma_\QQ$.
\end{notation}

\begin{remark}
  In the situation of \cref{not:gluing}, by arguments similar to the proof of \cref{lemma:rat} (and using \cref{lemma:projective_cofibration_lincat}), the upper horizontal maps in the following diagram in the category of functors from $\lincat 2$ to simplicial sets
  \newsavebox{\DiagA}
  \begin{lrbox}{\DiagA}
  	\begin{tikzcd}[sep = small]
  		\rat B \cop_{\rat C} \rat B' \dar \\
  		\rat A \cop_{\rat C} \rat A' \dar \\
  		\rat X \cop_{\rat C} \rat X'
  	\end{tikzcd}
  \end{lrbox}
  \newsavebox{\DiagB}
  \begin{lrbox}{\DiagB}
  	\begin{tikzcd}[sep = small]
  		\rat {(B \cop_C B')} \dar \\
  		\rat {(A \cop_C A')} \dar \\
  		\rat {(X \cop_C X')}
  	\end{tikzcd}
  \end{lrbox}
  \newsavebox{\DiagC}
  \begin{lrbox}{\DiagC}
  	\begin{tikzcd}[sep = small]
  		\rat Q \dar \\
  		\rat R \dar \\
  		\rat P
  	\end{tikzcd}
  \end{lrbox}
  \newsavebox{\DiagD}
  \begin{lrbox}{\DiagD}
  	\begin{tikzcd}[sep = small]
  		\nerve{L_B} \cop_{\nerve{L_C}} \nerve{L_{B'}} \dar \\
  		\nerve{L_A} \cop_{\nerve{L_C}} \nerve{L_{A'}} \dar \\
  		\nerve{L_X} \cop_{\nerve{L_C}} \nerve{L_{X'}}
  	\end{tikzcd}
  \end{lrbox}
  \newsavebox{\DiagE}
  \begin{lrbox}{\DiagE}
  	\begin{tikzcd}[sep = small]
  		\nerve{L_Q} \dar \\
  		\nerve{L_R} \dar \\
  		\nerve{L_P}
  	\end{tikzcd}
  \end{lrbox}
  \[
  \begin{tikzcd}
    \left( \usebox{\DiagA} \right) \rar[tail]{\req} \dar[swap]{\req} & \left( \usebox{\DiagB} \right) \rar[tail]{\req} & \left( \usebox{\DiagC} \right) \dar[dashed]{\req} \\
    \left( \usebox{\DiagD} \right) \ar{rr}{\req} & & \left( \usebox{\DiagE} \right)
  \end{tikzcd}
  \]
  are projective cofibrations.
  Moreover all non-dashed maps are pointwise rational homology equivalences (by \cref{lemma:rat,lemma:MC_pushouts}), and hence we can choose a dashed rational homology equivalence as indicated that makes the diagram commute.
  We consider the induced maps $L_Q \to L_R \to L_P$ to model the maps $Q \to R \to P$ in this way.
  Furthermore note that the right-hand square in the diagram
  \[
  \begin{tikzcd}
    \rat X \cop_{\rat C} \rat X' \dar[swap] \rar[tail]{\req} & \rat P \dar \rar[tail]{\rat \zeta} \drar[phantom]{\eq} & \rat K \dar \\
    \nerve {L_X} \cop_{\nerve {L_C}} \nerve {L_{X'}} \rar & \nerve {L_P} \rar{\nerve{\sigma}} & \nerve \Pi
  \end{tikzcd}
  \]
  commutes up to homotopy since the left-hand and outer squares commute.
  Since $\rat \zeta$ is a cofibration, we can change $\rat K \to \nerve \Pi$ to a homotopic map such that the right-hand square commutes strictly; we will consider $K$ to model $\Pi$ in this way.
\end{remark}

We begin by proving that the map $\Xi$ of \cref{lemma:exp_aut_map} is compatible with gluing constructions.

\begin{lemma} \label{prop:exp_eq_natural}
  We use \cref{not:gluing} and let
  \[ \lie g_X \subseteq \Der(L_X \rel L_A)  \qquad \text{and} \qquad  \lie g_{X'} \subseteq \Der(L_{X'} \rel L_{A'}) \]
  be two nilpotent dg Lie subalgebras that act nilpotently on $L_X$ and $L_{X'}$, respectively, and write
  \[ \lie g_P  \defeq  g(\lie g_X \times \lie g_{X'})  \subseteq  \Der(L_P \rel L_R) \]
  (note that $\lie g_P$ is nilpotent and that it acts nilpotently on $L_P$).
  Then the following diagram in the homotopy category of simplicial monoids commutes
  \[
  \begin{tikzcd}
    \expb (\lie g_X \times \lie g_{X'}) \dar[swap]{g_*} \rar{\iso} & \expb (\lie g_X) \times \expb(\lie g_{X'}) \rar{\Xi \times \Xi} & \aut[\nerve{L_A}] \bigl( \nerve{L_X} \bigr) \times \aut[\nerve{L_{A'}}] \bigl( \nerve{L_{X'}} \bigr) \dar[squiggly] \\
    \expb (\lie g_P) \ar{rr}{\Xi} & & \aut[{\nerve{L_R}}] \bigl( \nerve{L_P} \bigr)
  \end{tikzcd}
  \]
  where $\Xi$ is the map of \cref{lemma:exp_aut_map}, the right-hand vertical map is the one of \cref{lemma:aut_pushout_zig-zag}, and $g \colon \lie g_X \times \lie g_{X'} \to \lie g_P$ is the map of dg Lie algebras given by $(\phi, \psi) \mapsto \tilde \phi + \tilde \psi$, where $\tilde \phi$ is the unique derivation of $L_P$ that extends $\phi$ and vanishes on $L_{X'}$, and analogously for $\tilde \psi$.
  
  Furthermore, let $G_X \subseteq \Aut[L_A](L_X)$ be a subgroup such that its conjugation action preserves $\lie g_X$ setwise, and analogously for $G_{X'} \subseteq \Aut[L_{A'}](L_{X'})$.
  Then the diagram lifts to a commutative diagram in the homotopy category of simplicial monoids with an action of $G_X \times G_{X'}$, where the action on each of the terms is by conjugation.
\end{lemma}

\begin{proof}
  First note that \cref{lemma:aut_pushout_zig-zag} is indeed applicable by \cref{lemma:MC_pushouts,lemma:cofibration_pushouts_special}.
  We define
  \[ N  \defeq  \nerve{L_A} \cop_{\nerve{L_C}} \nerve{L_{A'}}  \quad \text{and} \quad  M  \defeq  \nerve{L_X} \cop_{\nerve{L_C}} \nerve{L_{X'}} \]
  and write $m \colon M \to \nerve{L_P}$ for the induced map.
  Consider the diagram of simplicial sets
  \[
  \begin{tikzcd}
    \expb (\lie g_X) \times \expb(\lie g_{X'}) \ar{rr}{\Xi \times \Xi} &[-10] &[-35] \aut[\nerve{L_A}] \bigl( \nerve{L_X} \bigr) \times \aut[\nerve{L_{A'}}] \bigl( \nerve{L_{X'}} \bigr) \ar{dd} \\
    \expb (\lie g_X \times \lie g_{X'}) \dar[swap]{g_*} \uar{\iso} \drar[start anchor = -5, dashed] & & \\
    \expb (\lie g_P) \dar[swap]{\Xi} & \raut[N](m) \dar \rar{\eq} & \raut[N](M) \dar{m_*} \\
    \aut[\nerve{L_R}] \bigl( \nerve{L_P} \bigr) \rar{\eq} & \aut[N] \bigl( \nerve{L_P} \bigr) \rar{m^*} & \mapreq[N] \bigl( M, \nerve{L_P} \bigr)
  \end{tikzcd}
  \]
  where the lower right-hand square is a pullback by definition.
  The outer square is seen to commute by expanding the definitions, so we obtain the dashed map as indicated.
  It is a map of simplicial monoids by \cref{lemma:monoid_pullback}.
  This yields the desired diagram.
  The last claim, about its equivariance, is clear, using that $\Xi$ is equivariant by \cref{lemma:exp_aut_map}.
\end{proof}

\begin{lemma} \label{lemma:Hom_map_gluing}
	In the situation of \cref{not:gluing}, the following diagram commutes in the homotopy category of simplicial sets
	\[
	\begin{tikzcd}
		\vertbin {\MCb \bigl( \trunc 0 {\Hom \bigl( \shift \indec[L_B](L_X), \Pi \bigr)} \bigr)} \times {\MCb \bigl( \trunc 0 {\Hom \bigl( \shift \indec[L_{B'}](L_{X'}), \Pi \bigr)} \bigr)} \rar{\eq} \dar & \vertbin {\map[\nerve{L_B}] \bigl( \nerve{L_X}, \nerve{\Pi} \bigr)_{\nerve{\rho}}} \times {\map[\nerve{L_{B'}}] \bigl( \nerve{L_{X'}}, \nerve{\Pi} \bigr)_{\nerve{\rho'}}} \dar[squiggly] \\
		\MCb \bigl( \trunc 0 {\Hom \bigl( \shift \indec[L_Q](L_P), \Pi \bigr)} \bigr) \rar{\eq} & \map[\nerve{L_Q}] \bigl( \nerve{L_P}, \nerve{\Pi} \bigr)_{\nerve{\sigma}}
	\end{tikzcd}
	\]
	where the horizontal weak equivalences are (products of) composites of the maps of \cref{prop:map_model,lemma:Hom_eq}, the left-hand vertical map comes from \cref{lemma:pushout_indecomposables}, and the right-hand vertical map is the one of \cref{lemma:aut_pushout_zig-zag}.
	Moreover, together with \cref{prop:exp_eq_natural}, this yields a commutative diagram in the homotopy category of pairs of a simplicial monoid and a right module over it.
\end{lemma}

\begin{proof}
	By chasing through their definitions, we see that the composite
	\[ \MCb \bigl( \trunc 0 {\Hom \bigl( \shift \indec[L_B](L_X), \Pi \bigr)} \bigr)  \xlongto{\eq}  \map[\nerve{L_B}] \bigl( \nerve{L_X}, \nerve{\Pi} \bigr)_{\nerve{\rho}} \]
	of the maps of \cref{prop:map_model,lemma:Hom_eq} is given by the adjoint of the restriction of the map given in simplicial degree $n$ by
  \begin{equation} \label{eq:MC_Hom_eq}
	\begin{aligned}
		\bigl( \Hom \bigl( \shift \indec[L_B](L_X), \Pi \bigr) \tensor \sForms[n] \bigr)_{-1} \times (L_X \tensor \sForms[n])_{-1}  &\longto  (\Pi \tensor \sForms[n])_{-1} \\
		(f \tensor a, x \tensor b)  &\longmapsto  (-1)^{\deg a \deg x} f (\shift \eqcl{x}) \tensor a b + \rho(x) \tensor b
	\end{aligned}
  \end{equation}
	and analogously for $L_{X'}$ and $L_P$.
	Using this formula, one obtains the desired commutativity similarly to the proof of \cref{prop:exp_eq_natural}.
\end{proof}

Next, we prove that rationalization is compatible with the map $\mu$.
This implies that the same is true for the rational homology equivalence of \cref{lemma:aut_rat_eq}.

\begin{lemma} \label{lemma:aut_rat_eq_natural}
  In the situation of \cref{not:gluing}, the following diagram commutes in the homotopy category of pairs of a simplicial monoid and a right module over it
  \[
  \begin{tikzcd}
    \vertpair {\aut[A](X) \times \aut[A'](X')} {\map[B](X, K) \times \map[B'](X', K)} \rar \dar[squiggly][swap]{\mu} & \vertpair {\aut[\rat A](\rat X) \times \aut[\rat A'](\rat X')} {\map[\rat B](\rat X, \rat K) \times \map[\rat B'](\rat X', \rat K)} \dar[squiggly] \\
    \vertpair {\aut[R](P)}  {\map[Q] \bigl( P, K \bigr)} \rar & \vertpair {\aut[\rat R](\rat P)} {\map[\rat Q] \bigl( \rat P, \rat K \bigr)}
  \end{tikzcd}
  \]
  where the horizontal maps are induced by rationalization and the vertical maps are those of \cref{lemma:aut_pushout_zig-zag}.
\end{lemma}

\begin{proof}
  First note that \cref{lemma:aut_pushout_zig-zag} is indeed applicable by \cref{lemma:rat,lemma:cofibration_pushouts_special}.
  We will prove the statement for the simplicial monoids; the proof for the pairs is analogous.
  To this end consider the diagram
  \[
  \begin{tikzcd}[column sep = 15]
  	{\aut[A](X)} \times {\aut[A'](X')} \rar \ar{dd} & \aut[A \cop_C A'](X \cop_C X') \rar[squiggly] \dar & \aut[A \cop_C A'](P) \dar & \lar[swap]{\eq} \aut[R](P) \dar \\
  	 & \aut[\rat{(A \cop_C A')}] \bigl( \rat{(X \cop_C X')} \bigr) \rar[squiggly]{\eq} & \aut[\rat{(A \cop_C A')}](\rat P) \dar{\eq} & \aut[\rat R](\rat P) \lar[swap]{\eq} \dlar[end anchor = north east]{\eq} \\
  	{\aut[\rat A](\rat X)} \times {\aut[\rat A'](\rat X')} \rar & \aut[\rat A \cop_{\rat C} \rat A'](\rat X \cop_{\rat C} \rat X') \rar[squiggly] \uar[squiggly] & \aut[\rat A \cop_{\rat C} \rat A'](\rat P)
  \end{tikzcd}
  \]
  where the zig-zags are those of \cref{lemma:rel_aut_span}.
  The left-hand pentagon commutes in the homotopy category of simplicial monoids by an argument similar to the proof of \cref{prop:exp_eq_natural}, and it is easy to see that the same is true for the upper right-hand squares.
  The lower middle square commutes by \cref{lemma:rel_aut_naturality,lemma:rel_aut_span} since the vertical map
  \[
  \begin{tikzcd}
  	\rat A \cop_{\rat C} \rat A' \dar \rar & \rat X \cop_{\rat C} \rat X' \dar \\
    \rat {(A \cop_C A')} \rar & \rat {(X \cop_C X')}
  \end{tikzcd}
  \]
  is a trivial projective cofibration in the arrow category by arguments similar to the proof of \cref{lemma:rat}.
\end{proof}

We can now combine the preceding lemmas to prove that the maps of \cref{thm:model} are compatible with $g$ and $\mu$.

\begin{proposition} \label{thm:model_natural}
  We use \cref{not:gluing} and assume that $i$ and $i'$ are minimal.
  Furthermore, let $(G, \ol U, N)$, $(G', \ol U', N')$, and $(H, \ol V, M)$ be three tuples as in \cref{not:group_triple}, associated to $(B, A, \xi)$, $(B', A', \xi')$, and $(Q, R, \zeta)$, respectively.
  Assume that $\gamma(G \times G') \subseteq H$, that $\ol \gamma(\ol U \times \ol U') \subseteq \ol V$, and that $\ol \gamma_\QQ(N \times N') \subseteq M$.
  Then there is a diagram
  \[
  \begin{tikzcd}
    \B[\inv{q_G}(U)] \bdlaut{A}{B}{K}(\xi)_G \times \B[\inv{q_{G'}}(U')] \bdlaut{A'}{B'}{K}(\xi')_{G'} \rar[squiggly]{\mu_*} \dar[swap]{\req} & \B[\inv{q_H}(V)] \bdlaut{R}{Q}{K}(\zeta)_H \dar{\req} \\
    \B[U] \bdlaut{\rat A}{\rat B}{\rat K}(\rat \xi) \times \B[U'] \bdlaut{\rat A'}{\rat B'}{\rat K}(\rat \xi') \rar[squiggly] \dar[squiggly][swap]{\eq} & \B[V] \bdlaut{\rat R}{\rat Q}{\rat K}(\rat \zeta) \dar[squiggly]{\eq} \\
    \MCb \bigl( \lie g_{L_A}^{L_B}(\rho)_{\ol U} \bigr) \times \MCb \bigl( \lie g_{L_{A'}}^{L_{B'}}(\rho')_{\ol U'} \bigr) \rar{g_*} & \MCb \bigl( \lie g_{L_R}^{L_Q}(\sigma)_{\ol V} \bigr)
  \end{tikzcd}
  \]
  such that the upper square commutes in the homotopy category of simplicial sets with an action of $\quot {\inv {q_G}(N)} {\inv {q_G}(U)} \times \quot {\inv {q_{G'}}(N')} {\inv {q_{G'}}(U')}$, the lower square commutes in the homotopy category of simplicial sets with an action of the preimage of $N \times N'$ in $\Aut[L_A](L_X)_\rho \times \Aut[L_{A'}](L_{X'})_{\rho'}$, and the two lower vertical maps are (the product of) the corresponding maps of \cref{thm:model}.
  Here the horizontal map is some morphism in the homotopy category of simplicial sets with an action of $\quot N U \times \quot {N'} {U'}$, and $g \colon \lie g_{L_A}^{L_B}(\rho)_{\ol U} \times \lie g_{L_{A'}}^{L_{B'}}(\rho')_{\ol U'} \to \lie g_{L_R}^{L_Q}(\sigma)_{\ol V}$ is the map of dg Lie algebras given by \cref{lemma:pushout_indecomposables} on the $\Hom$-factors and as in \cref{prop:exp_eq_natural} on the $\Der$-factors.
\end{proposition}

\begin{proof}
  First note that chasing through the definitions shows that $\Lie(\ol \Gamma)$ corresponds to $\Cycles 0 (g)$ under the isomorphism of \cref{prop:Lie_aEaut}, and hence $\Lie(\ol \gamma)$ corresponds to $\Ho 0 (g)$.
  This implies that there is indeed a map
  \[ \Der[\ol U](L_X \rel L_A) \times \Der[\ol U'](L_{X'} \rel L_{A'})  \longto  \Der[\ol V](L_P \rel L_R) \]
  as claimed.
  
  The upper square of the desired diagram is obtained from \cref{lemma:aut_rat_eq_natural}.
  We now construct the bottom square.
  \cref{prop:exp_eq_natural,lemma:Hom_map_gluing,lemma:MC_outer,lemma:submonoid_bar,lemma:left_conj_eq} yield a commutative diagram in the homotopy category of simplicial sets with an action of the preimage of $N \times N'$ in $\Aut[L_A](L_X)_\rho \times \Aut[L_{A'}](L_{X'})_{\rho'}$
  \[
  \begin{tikzcd}
    \MCb \bigl( \lie g_{L_A}^{L_B}(\rho)_{\ol U} \bigr) \times \MCb \bigl( \lie g_{L_{A'}}^{L_{B'}}(\rho')_{\ol U'} \bigr) \rar{\eq} \dar[swap]{\iso}  & \B[U] \bdlaut{\nerve{L_A}}{\nerve{L_B}}{\nerve{\Pi}} \bigl( \nerve{\rho} \bigr) \times \B[U'] \bdlaut{\nerve{L_{A'}}}{\nerve{L_{B'}}}{\nerve{\Pi}} \bigl( \nerve{\rho'} \bigr) \dar{\iso} \\
    \MCb \Bigl( \lie g_{L_A}^{L_B}(\rho)_{\ol U} \times \lie g_{L_{A'}}^{L_{B'}}(\rho')_{\ol U'} \Bigr) \dar[swap]{g*} & \B[U \times U'] \Bigl( \bdlaut{\nerve{L_A}}{\nerve{L_B}}{\nerve{\Pi}} \bigl( \nerve{\rho} \bigr) \times \bdlaut{\nerve{L_{A'}}}{\nerve{L_{B'}}}{\nerve{\Pi}} \bigl( \nerve{\rho'} \bigr) \Bigr) \dar[squiggly] \\
    \MCb \bigl( \lie g_{L_R}^{L_Q}(\sigma)_{\ol V} \bigr) \rar[squiggly]{\eq} & \B[V] \bdlaut{\nerve{L_R}}{\nerve{L_Q}}{\nerve{\Pi}} \bigl( \nerve{\sigma} \bigr)
  \end{tikzcd}
  \]
  where the bottom right-hand vertical map is induced by the one of \cref{lemma:aut_pushout_zig-zag}.
  Combining this with \cref{lemma:zig-zag_pushouts}, we obtain the desired diagram.
\end{proof}

To be able to apply \cref{thm:model_natural} to obtain a compatibility of the model of \cref{cor:Baut_eq} with gluing constructions, we need the map $\ol \gamma$ of \cref{lemma:aEaut_pushouts} to induce a map on maximal reductive quotients.
The following lemma provides a criterion for when this is the case.
It was one of our motivations for introducing the algebraic groups $\ahoEaut[L_A](L_X)_\rho$.

\begin{lemma} \label{lemma:aredEaut_gluing_map}
  We use \cref{not:gluing} and assume that $i$ and $i'$ are minimal.
  Furthermore assume that the algebraic representations, provided by \cref{lemma:indecomposables_representation}, of $\aAut[L_A](L_X)_\rho$ on $\indec[L_A](L_X)$ and of $\aAut[L_{A'}](L_{X'})_{\rho'}$ on $\indec[L_{A'}](L_{X'})$ are semi-simple.
  Then the (restriction of the) map of \cref{lemma:aEaut_pushouts}
  \begin{gather*}
    \ol \gamma  \colon  \aEaut[L_A](L_X)_\rho \times \aEaut[L_{A'}](L_{X'})_{\rho'}  \longto  \aEaut[L_R](L_P)_\sigma
    \shortintertext{induces a map}
    \ol \nu  \colon  \aredEaut[L_A](\rho) \times \aredEaut[L_{A'}](\rho')  \longto  \aredEaut[L_R](\sigma)
  \end{gather*}
  of algebraic groups.
\end{lemma}

\begin{proof}
  By \cref{lemma:pushout_indecomposables}, the map $\ol \gamma$ induces a map
  \[ \ahoEaut[L_A](L_X)_\rho \times \ahoEaut[L_{A'}](L_{X'})_{\rho'}  \longto  \ahoEaut[L_R](L_P)_\sigma \]
  of the algebraic groups of \cref{def:ahoEaut}.
  Moreover the following maps of algebraic groups, which exist since the kernel of the map $\Koppa$ of \cref{lemma:indecomposables_representation} is unipotent,
  \[ \ahoEaut[L_A](L_X)_\rho  \longto  \aredEaut[L_A](\rho)  \qquad \text{and} \qquad  \ahoEaut[L_{A'}](L_{X'})_{\rho'}  \longto  \aredEaut[L_{A'}](\rho') \]
  are isomorphisms by \cref{lemma:ahoEaut_reductive}.
  Hence there is a map of algebraic groups
  \[ \aredEaut[L_A](\rho) \times \aredEaut[L_{A'}](\rho')  \iso  \ahoEaut[L_A](L_X)_\rho \times \ahoEaut[L_{A'}](L_{X'})_{\rho'}  \longto  \ahoEaut[L_R](L_P)_\sigma  \longto  \aredEaut[L_R](\sigma) \]
  compatible with $\ol \gamma$, which finishes the proof.
\end{proof}

For later reference, we record the outcome of the preceding lemma in the following assumption.

\begin{assumption} \label{ass:gluing}
  In the situation of \cref{not:gluing}, we additionally assume that $i$ and $i'$ are minimal, and that $\ol \gamma$ induces a map
  \[ \ol \nu \colon \aredEaut[L_A](\rho) \times \aredEaut[L_{A'}](\rho')  \longto  \aredEaut[L_R](\sigma) \]
  of algebraic groups.
  By \cref{lemma:aredEaut_gluing_map}, the latter condition is fulfilled if the algebraic representations of $\aAut[L_A](L_X)_\rho$ on $\indec[L_A](L_X)$ and of $\aAut[L_{A'}](L_{X'})_{\rho'}$ on $\indec[L_{A'}](L_{X'})$ are semi-simple (for which \cref{lemma:ahoEaut_reductive} provides an equivalent characterization).
\end{assumption}

\begin{remark} \label{rem:gluing_condition}
  An induced map $\ol \nu$ as in \cref{ass:gluing} does not always exist.\footnote{The authors would like to thank the participants of the WIQI topology seminar at the MPI Bonn for helpful discussions regarding this question.}
  For a counterexample, let $n \ge 2$ and consider $L_A = L_{A'} = L_C \defeq \freelie(a)$, $L_X \defeq (\freelie(a, x, y), d)$, and $L_{X'} \defeq (\freelie(a, x'), d')$ with $\deg a \defeq n$ and $\deg x = \deg y = \deg{x'} \defeq n + 1$ as well as $d(x) = a$ and $d'(x') = a$ (and otherwise trivial differentials).
  The algebraic group $\aAut[L_A](L_X)$ is isomorphic to the algebraic subgroup of $\aGL(\linspan {x, y})$ consisting of the matrices of the form
  \[ \begin{pmatrix} 1 & 0 \\ \alpha & \beta \end{pmatrix} \]
  with respect to the ordered basis $(x, y)$.
  Hence its unipotent radical $\ol U$ consists of those matrices $u$ such that $\beta = 1$.
  On the other hand, consider the pushout $L_P \defeq L_X \cop_{L_A} L_{X'} \iso (\freelie(a, x, x', y), \delta)$ where $\delta(x) = \delta(x') = a$ and $\delta$ is trivial otherwise.
  The algebraic group $\aAut[L_A](L_P)$ is isomorphic to the algebraic subgroup of $\aGL(\linspan {x, x', y})$ consisting of the matrices of the form
  \[ \begin{pmatrix} 1 & 0 & 0 \\ \alpha_1 & \beta_1 & \gamma_1 \\ \alpha_2 & \beta_2 & \gamma_2 \end{pmatrix} \]
  with respect to the ordered basis $(x, x - x', y)$.
  Sending such a matrix to the matrix
  \[ \begin{pmatrix} \beta_1 & \gamma_1 \\ \beta_2 & \gamma_2 \end{pmatrix} \]
  defines a surjective map of algebraic groups $\ol \pi \colon \aAut[L_A](L_P) \to \aGL[2]$.
  An element $u$ of $\ol U(R)$ as above is sent by $(\ol \pi \after \ol \Gamma)_R$ to the element
  \[ \begin{pmatrix} 1 & 0 \\ \alpha & 1 \end{pmatrix} \]
  which is non-trivial for $\alpha \neq 0$.
  This is a contradiction to $\ol \Gamma_R(u)$ being an element of the unipotent radical of $\aAut[L_A](L_P)$.
\end{remark}

Recall that the model of \cref{cor:Baut_eq} fundamentally depends on the choice of a section of the group homomorphism $\Aut[L_A](L_X)_\rho \to \aredEaut[L_A](\rho)(\QQ)$.
The following lemma records the simple, but integral, observation that these sections can be chosen to be compatible with the gluing maps $\gamma$ and $\Gamma$.

\begin{lemma} \label{lemma:choose_compatibly_gluing}
  We use \cref{not:gluing} and assume that \cref{ass:gluing} is fulfilled.
  Moreover, let $\bar s$ and $\bar s'$ be sections of the composite maps of algebraic groups
  \begin{gather*}
    \aAut[L_A](L_X)_\rho  \xlongto{\areal}  \aEaut[L_A](L_X)_\rho  \longto  \aredEaut[L_A] ( \rho ) \\
    \aAut[L_{A'}](L_{X'})_{\rho'}  \xlongto{\areal}  \aEaut[L_{A'}] ( L_{X'} )_{\rho'}  \longto  \aredEaut[L_{A'}] ( \rho' )
  \end{gather*}
  respectively (such sections exist by \cref{lemma:areal_section}).
  Then there exists a section $\bar t$ of the composite map of algebraic groups
  \[ \aAut[L_R](L_P)_\sigma  \xlongto{\areal}  \aEaut[L_R] ( L_P )_\sigma  \longto  \aredEaut[L_R] ( \sigma ) \]
  such that $\ol \Gamma \after (\bar s \times \bar s') = \bar t \after \ol \nu$.
\end{lemma}

\begin{proof}
  This follows from \cref{lemma:compatible_Levi} since products of reductive algebraic groups are again reductive (see e.g.\ \cite[27.2.3]{TY}).
\end{proof}

We are now ready to state the compatibility of the model for $\B \bdlaut{A}{B}{K}(\xi)$ of \cref{cor:Baut_eq} with gluing constructions.
Properly interpreted, this means that the map $g$ models the map $\mu$.
Afterwards, we will deduce that the cdga model of \cref{cor:forms_eq} and the identification of the cohomology of \cref{cor:cohomology} are both compatible with gluing constructions as well.

\begin{theorem} \label{cor:Baut_eq_natural}
  We use \cref{not:gluing} and assume that \cref{ass:gluing} is fulfilled.
  Furthermore, let $G \subseteq \Eaut[A](X)_{\eqcl{\xi}_B}$, $G' \subseteq \Eaut[A'](X')_{\eqcl{\xi'}_B}$, and $H \subseteq \Eaut[R](P)_{\eqcl{\zeta}_Q}$ be finite-index subgroups such that $\gamma(G \times G') \subseteq H$, and denote by $\redim{G}$, $\redim{G'}$, and $\redim{H}$ their images in $\aredEaut[L_A] (\rho)(\QQ)$, $\aredEaut[L_{A'}] (\rho')(\QQ)$, and $\aredEaut[L_P] (\sigma)(\QQ)$, respectively.
  Furthermore, let $\bar s$, $\bar s'$, and $\bar t$ be sections as in \cref{lemma:choose_compatibly_gluing}.
  Then there is a commutative diagram in the homotopy category of simplicial sets
  \[
  \begin{tikzcd}
    \B \bdlaut{A}{B}{K}(\xi)_G \times \B \bdlaut{A'}{B'}{K}(\xi')_{G'} \rar{\mu_*} \dar[squiggly][swap]{\req} & \B \bdlaut{R}{Q}{K}(\zeta)_H \dar[squiggly]{\req} \\
    \hcoinv { \MCb \bigl( \lie g_{L_A}^{L_B}(\rho) \bigr) } { \redim{G} }  \times  \hcoinv { \MCb( \bigl( \lie g_{L_{A'}}^{L_{B'}}(\rho') \bigr) } { \redim{G'} } \rar{g_*} & \hcoinv { \MCb \bigl( \lie g_{L_R}^{L_Q}(\sigma) \bigr) } { \redim{H} }
  \end{tikzcd}
  \]
  such that the two vertical maps are those of \cref{cor:Baut_eq}.
  Here $g \colon \lie g_{L_A}^{L_B}(\rho) \times \lie g_{L_{A'}}^{L_{B'}}(\rho') \to \lie g_{L_R}^{L_Q}(\sigma)$ is the map of dg Lie algebras given by \cref{lemma:pushout_indecomposables} on the $\Hom$-factors and as in \cref{prop:exp_eq_natural} on the $\Der$-factors.
  The group $\redim{G}$ acts on $\lie g_{L_A}^{L_B}(\rho)$ by conjugation via the group homomorphism $\bar s_\QQ$, and analogously $\redim{G'}$ acts via $\bar s'_\QQ$ and $\redim{H}$ via $\bar t_\QQ$.
\end{theorem}

\begin{proof}
  This follows from \cref{thm:model_natural,lemma:hquot_hcoinv}, as in the proof of \cref{cor:Baut_eq}.
\end{proof}

\begin{corollary} \label{cor:forms_eq_natural}
  In the situation of \cref{cor:Baut_eq_natural}, let $M$, $M'$, and $N$ be representations in rational vector spaces of $\redim{G}$, $\redim{G'}$, and $\redim{H}$, respectively, and let $m \colon N \to M \tensor M'$ be a map that is equivariant with respect to $\ol \nu_\QQ \colon \redim{G} \times \redim{G'} \to \redim{H}$.
  Then the following diagram commutes in the homotopy category of cochain complexes
  \[
  \begin{tikzcd}
    &[-180] \Forms * \bigl( \B \bdlaut{A}{B}{K}(\xi)_G \times \B \bdlaut{A'}{B'}{K}(\xi')_{G'}; M \tensor M' \bigr) &[-110] \\
    {\Forms * \bigl( \B \bdlaut{A}{B}{K}(\xi)_G; M \bigr)} \tensor {\Forms * \bigl( \B \bdlaut{A'}{B'}{K}(\xi')_{G'}; M' \bigr)} \urar[bend left, start anchor = north, end anchor = -174][near start]{\eq} \dar[squiggly][swap]{\eq} & & \Forms * \bigl( \B \bdlaut{R}{Q}{K}(\zeta)_H; N \bigr) \dar[squiggly]{\eq} \ular[bend right, start anchor = north, end anchor = -6] \\
    {\Forms * \bigl( \B \redim{G}; \CEcochains * \bigl( \lie g_{L_A}^{L_B}(\rho) ; M \bigr) \bigr)} \tensor {\Forms * \bigl( \B \redim{G'}; \CEcochains * \bigl( \lie g_{L_{A'}}^{L_{B'}}(\rho') ; M' \bigr) \bigr)} \drar[bend right, start anchor = south, end anchor = 173][swap, near start]{\eq} & & \Forms * \bigl( \B \redim{H}; \CEcochains * \bigl( \lie g_{L_R}^{L_Q}(\sigma); N \bigr) \bigr) \dlar[bend left, start anchor = south, end anchor = 8] \\
    & \Forms * \bigl( \B \redim{G} \times \B \redim{G'}; \CEcochains * \bigl( \lie g_{L_A}^{L_B}(\rho) \times \lie g_{L_{A'}}^{L_{B'}}(\rho'); M \tensor M' \bigr) \bigr) &
  \end{tikzcd}
  \]
  where the indicated maps are quasi-isomorphisms and the vertical maps are those of \cref{cor:forms_eq}.
  Moreover, this diagram can be lifted to a commutative diagram in the homotopy category of lax symmetric monoidal functors from the category of tuples $(M, M', N, m)$ to cochain complexes (with monoidal natural transformations as morphisms, and the pointwise weak equivalences).
  Here $\gamma \in \redim{G}$ acts on the cochain complex $\CEcochains * ( \lie g_{L_A}^{L_B}(\rho); M )$ by $(\inv \gamma, \gamma)$, where the action on $\lie g_{L_A}^{L_B}(\rho)$ is by conjugation via the group homomorphism $\bar s_\QQ$, and analogously $\redim{G'}$ acts via $\bar s'_\QQ$ and $\redim{H}$ via $\bar t_\QQ$.
\end{corollary}

\begin{proof}
  This follows from \cref{cor:Baut_eq_natural,lemma:forms_orbits,lemma:MC_forms,lemma:hquot_hcoinv_local_system}, as in the proof of \cref{cor:forms_eq}.
  To prove that the two left-hand diagonal maps are quasi-isomorphisms, note that the groups $\redim{G}$ and $\redim{G'}$ are arithmetic by \cref{cor:Baut_eq} and hence their classifying spaces have the homotopy types of CW-complexes with finitely many cells in each dimension by \cite[Theorem~1.3]{BKK}.
  Then the Künneth Theorem for cohomology with local coefficients (see e.g.\ \cite[Theorem~1.7]{Gre06}\footnote{The result is only stated over $\ZZ$ there, but the same argument works over any principal ideal domain.}), combined with \cref{lemma:forms_split_coeff,lemma:CE_products}, implies that the bottom left-hand map is a quasi-isomorphism.
  The top left-hand map is then also a quasi-isomorphism since the the topmost term is quasi-isomorphic to the bottommost term (compatibly with the left-hand map) by \cref{lemma:hcoinv_coho,lemma:forms_orbits,lemma:MC_forms}.
\end{proof}

\begin{corollary} \label{cor:cohomology_natural}
  In the situation of \cref{cor:Baut_eq_natural}, let $M$, $M'$, and $N$ be representations in rational vector spaces of $\redim{G}$, $\redim{G'}$, and $\redim{H}$, respectively, and let $m \colon N \to M \tensor M'$ be a map that is equivariant with respect to $\ol \nu_\QQ \colon \redim{G} \times \redim{G'} \to \redim{H}$.
  Then the following diagram of graded vector spaces commutes
  \[
  \begin{tikzcd}
    &[-180] \Coho * \bigl( \B \bdlaut{A}{B}{K}(\xi)_G \times \B \bdlaut{A'}{B'}{K}(\xi')_{G'}; M \tensor M' \bigr) &[-110] \\
    {\Coho * \bigl( \B \bdlaut{A}{B}{K}(\xi)_G; M \bigr)} \tensor {\Coho * \bigl( \B \bdlaut{A'}{B'}{K}(\xi')_{G'}; M' \bigr)} \urar[bend left, start anchor = north, end anchor = -174][near start]{\iso} \dar[swap]{\iso} & & \Coho * \bigl( \B \bdlaut{R}{Q}{K}(\zeta)_H; N \bigr) \dar{\iso} \ular[bend right, start anchor = north, end anchor = -6] \\
    {\Coho * \bigl( \B \redim{G}; \CEcoho * \bigl( \lie g_{L_A}^{L_B}(\rho); M \bigr) \bigr)} \tensor {\Coho * \bigl( \B \redim{G'}; \CEcoho * \bigl( \lie g_{L_{A'}}^{L_{B'}}(\rho'); M' \bigr) \bigr)} \drar[bend right, start anchor = south, end anchor = 173][swap, near start]{\iso} & & \Coho * \bigl( \B \redim{H}; \CEcoho * \bigl( \lie g_{L_R}^{L_Q}(\sigma); N \bigr) \bigr) \dlar[bend left, start anchor = south, end anchor = 8] \\
    & \Coho * \bigl( \B \redim{G} \times \B \redim{G'}; \CEcoho * \bigl( \lie g_{L_A}^{L_B}(\rho) \times \lie g_{L_{A'}}^{L_{B'}}(\rho'); M \tensor M' \bigr) \bigr) &
  \end{tikzcd}
  \]
  where the indicated maps are isomorphisms and the vertical maps are those of \cref{cor:cohomology}.
  Moreover, this diagram consists of monoidal natural transformations between lax symmetric monoidal functors from the category of tuples $(M, M', N, m)$ to graded vector spaces.
  Here $\gamma \in \redim{G}$ acts on the graded vector space $\CEcoho * ( \lie g_{L_A}^{L_B}(\rho); M )$ by $(\inv \gamma, \gamma)$, where the action on $\lie g_{L_A}^{L_B}(\rho)$ is by conjugation via the group homomorphism $\bar s_\QQ$, and analogously $\redim{G'}$ acts via $\bar s'_\QQ$ and $\redim{H}$ via $\bar t_\QQ$.
\end{corollary}

\begin{proof}
  This follows from \cref{cor:forms_eq_natural,lemma:forms_cohomology,lemma:forms_split_coeff}, as in the proof of \cref{cor:cohomology}.
\end{proof}

\subsection{Compatibility with forgetful maps} \label{sec:forget}

In this subsection, we study how the constructions of \cref{sec:model} relate to restriction of the fixed subspace, i.e.\ maps of the form
\begin{equation} \label{eq:forget}
  \bdlaut{A}{B}{K}(\xi)  \longto  \bdlaut{A'}{B'}{K}(\xi)
\end{equation}
where $A' \subseteq A$ and $B' \subseteq B$ are subcomplexes such that $B' \subseteq A'$.
Here we run into the issue described at the beginning of \cref{sec:alg_forget} (see there for more explanation): a composite of minimal quasi-free maps of dg Lie algebras is not necessarily minimal.
Due to this, we do not obtain a direct map of dg Lie algebras corresponding to \eqref{eq:forget}, but only a zig-zag.
We begin by fixing the notation we will use throughout.

\begin{notation} \label{not:forget}
  Let the following be a commutative diagram of simply connected Kan complexes with the homotopy types of finite CW-complexes
  \[
  \begin{tikzcd}
  	B \rar[tail] & A \rar[tail]{\iota} & X \rar{\xi} & K \\
  	B' \uar[tail] \rar[tail] & A' \uar[tail][swap]{\alpha} & &
  \end{tikzcd}
  \]
  where the indicated maps are cofibrations.
  Furthermore assume that this diagram is modeled by the left-hand and upper parts of the following commutative diagram of positively graded dg Lie algebras
  \[
  \begin{tikzcd}
  	L_B \rar[tail] & L_A \rar[tail]{i} & L_X \rar{\rho} & \Pi \\
  	L_{B'} \uar[tail] \rar[tail] & L_{A'} \uar[tail][swap]{a} \rar[tail]{i'} & L'_X \uar[tail]{\eq}[swap]{m} \urar[bend right][swap]{\rho'} &
  \end{tikzcd}
  \]
  where $\Pi$ is of finite type, abelian, and has trivial differential, all other objects are finitely generated and quasi-free, the indicated maps are quasi-free, and $m$ is a quasi-isomorphism.
  If $i$ and $i'$ are minimal, we furthermore write
  \[ \phi \colon \Eaut[A](X) \longto \Eaut[A'](X)  \qquad \text{and} \qquad  \ol \phi \colon \aEaut[L_A](L_X) \longto \aEaut[L_{A'}](L'_X) \]
  for the forgetful map and the map of \cref{lemma:aEaut_forget}, respectively.
\end{notation}

\begin{remark}
  Since $i \after a$ and $i'$ are cofibrations, the quasi-isomorphism $m$ has a simplicial homotopy inverse $m' \colon L_X \to L'_X$ relative to $L_{A'}$ (see \cref{rem:dgLie_homotopy_inverse}).
  We consider $L'_X$ to model $X$ via the composite map
  \begin{equation} \label{eq:L'-model}
    \rat X  \xlongto{\eq}  \nerve {L_X}  \xlongto{\nerve {m'}}  \nerve {L'_X}
  \end{equation}
  of Kan complexes
  Furthermore note that $\rho' \after m' \eq \rho$; however the differential of $\Pi$ is trivial, and hence $\rho' \after m' = \rho$ (using \cref{lemma:dgLie_homotopic}).
\end{remark}

The map $\Xi$ of \cref{lemma:exp_aut_map} is compatible with forgetful maps in the following sense.

\begin{lemma} \label{lemma:exp_eq_forget}
  In the situation of \cref{not:forget}, let
  \[ \lie g \subseteq \Der(L_X \rel L_A)  \qquad \text{and} \qquad  \lie g' \subseteq \Der(L'_X \rel L_{A'}) \]
  be two nilpotent dg Lie subalgebras that act nilpotently on $L_X$ and $L'_X$, respectively, and define the dg Lie subalgebra $\lie h \subseteq \lie g \times \lie g'$ to be the pullback of the span
  \[ \lie g  \xlongto{m^*}  \Derrel{m}(L'_X, L_X \rel L_{A'})  \xlongfrom{m_*}  \lie g' \]
  of chain complexes.
  (Note that $\lie h$ is nilpotent since $\lie g \times \lie g'$ is.)
  
  Then there is a unique dashed map such that the diagram of simplicial monoids
  \[
  \begin{tikzcd}
  	\expb (\lie g) \dar[swap]{\Xi} & \lar \expb (\lie h) \dar[dashed] \rar & \expb (\lie g') \dar{\Xi} \\
    \aut[\nerve{L_A}] \bigl( \nerve{L_X} \bigr) & \lar \aut[\nerve{a}] \bigl( \nerve{m} \bigr) \rar & \aut[\nerve{L_{A'}}] \bigl( \nerve{L'_X} \bigr)
  \end{tikzcd}
  \]
  commutes.
  Furthermore, let $G \subseteq \Aut[L_A](L_X)$ be a subgroup such that its conjugation action preserves $\lie g$ setwise, and analogously for $G' \subseteq \Aut[L_{A'}](L'_X)$.
  Then the group $\Aut[a](m) \intersect (G \times G')$ acts on the whole diagram by conjugation.
\end{lemma}

\begin{proof}
  Consider the diagram of simplicial sets
  \[
  \begin{tikzcd}
    \expb(\lie h) \ar{rr} \ar{dd} \drar[dashed, end anchor = north west] & & \expb(\lie g) \dar{\Xi} \\
    & \aut[\nerve{a}](m) \rar \dar & \aut[\nerve{L_A}] \bigl( \nerve{L_X} \bigr) \dar{m_*} \\
    \expb(\lie g') \rar{\Xi} & \aut[\nerve{L_{A'}}] \bigl( \nerve{L'_X} \bigr) \rar{m^*} & \map[\nerve{L_{A'}}] \bigl( \nerve{L'_X}, \nerve{L_X} \bigr)
  \end{tikzcd}
  \]
  where the lower right-hand square is a pullback by definition.
  The outer square is seen to commute by expanding the definitions, so we obtain the desired dashed map as indicated.
  It is a map of simplicial monoids by \cref{lemma:monoid_pullback}.
  The claim about the action follows from \cref{lemma:exp_aut_map}.
\end{proof}

We now establish a notation for the dg Lie algebra $\lie h$ above in the case that $\lie g$ and $\lie g'$ arise from \cref{def:Deru}.

\begin{definition}
  In the situation of \cref{not:forget}, let $\ol U \subseteq \aEaut[L_A](L_X)_\rho$ and $\ol V \subseteq \aEaut[L_{A'}](L'_X)_{\rho'}$ be two algebraic subgroups.
  We define the dg Lie subalgebra
  \[ \Der[\ol U, \ol V](m \rel a) \subseteq \Der(L_X \rel L_A) \times \Der(L'_X \rel L_{A'}) \]
  to be the pullback of the cospan
  \begin{equation} \label{eq:Der_cospan}
    \Der[\ol U](L_X \rel L_A)  \xlongto{m^*}  \Derrel{m}(L'_X, L_X \rel L_{A'})  \xlongfrom{m_*}  \Der[\ol V](L'_X \rel L_{A'})
  \end{equation}
  of chain complexes.
  When $\ol U$ and $\ol V$ are the respective unipotent radical, we write $\Deru[\rho](m \rel a) \defeq \Der[\ol U, \ol V](m \rel a)$.
\end{definition}

Note that, when $\ol U$ and $\ol V$ are unipotent algebraic subgroups, then $\Der[\ol U, \ol V](m \rel a)$ is a nilpotent dg Lie algebra, since it is a dg Lie subalgebra of $\Der[\ol U](L_X \rel L_A) \times \Der[\ol V](L'_X \rel L_{A'})$, which is nilpotent by \cref{prop:exp_eq}.

\begin{lemma} \label{lemma:Hom_map_forget}
	In the situation of \cref{not:forget}, the following diagram of simplicial sets commutes
	\[
	\begin{tikzcd}
		\MCb \bigl( \trunc 0 {\Hom \bigl( \shift \indec[L_B](L_X), \Pi \bigr)} \bigr) \rar{\eq} \dar & \map[\nerve{L_B}] \bigl( \nerve{L_X}, \nerve{\Pi} \bigr)_{\nerve{\rho}} \dar \\
		\MCb \bigl( \trunc 0 {\Hom \bigl( \shift \indec[L_{B'}](L'_X), \Pi \bigr)} \bigr) \rar{\eq} & \map[\nerve{L_{B'}}] \bigl( \nerve{L'_X}, \nerve{\Pi} \bigr)_{\nerve{\rho'}}
	\end{tikzcd}
	\]
	where the horizontal weak equivalences are composites of the maps of \cref{prop:map_model} and \cref{lemma:Hom_eq}.
	Moreover, this diagram is $\expb(\lie h)$-equivariant for any dg Lie algebra $\lie h$ as in \cref{lemma:exp_eq_forget}.
\end{lemma}

\begin{proof}
	The commutativity follows directly from the formula \eqref{eq:MC_Hom_eq} (found in the proof of \cref{lemma:Hom_map_gluing}) for the horizontal maps.
	For the equivariance it is, by \cref{lemma:exp_eq_forget}, enough to show that the left-hand vertical map is equivariant.
	This is clear from the description of the actions of \cref{lemma:Hom_eq}.
\end{proof}

We now prove a compatibility of \cref{thm:model} with restriction maps.
Note that, on the level of dg Lie algebras, we obtain neither a direct map, nor a map in the homotopy category.
See \cref{rem:forget_no_map} below for some comments on the reasons for this.

\begin{proposition} \label{thm:model_forget}
  We use \cref{not:forget} and assume that $i$ and $i'$ are minimal.
  Furthermore, let $(G, \ol U, N)$ and $(H, \ol V, M)$ be two tuples as in \cref{not:group_triple}, associated to $(B, A, \xi)$ and $(B', A', \xi)$, respectively, and assume that $\phi(G) \subseteq H$, that $\ol \phi(\ol U) \subseteq \ol V$, and that $\ol \phi_\QQ(N) \subseteq M$.
  We write
  \[ \lie h_{\ol U, \ol V}  \defeq  \trunc 0 { \Hom \bigl( \shift \indec[L_B](L_X), \Pi \bigr) }  \rsemidir[\tilde \rho_*] \Der[\ol U, \ol V](m \rel a) \]
  where the outer action $\tilde \rho_*$ is as in \cref{prop:bdlaut_model} (via the projection to $\Der[\ol U](L_X \rel L_A)$).
  Then, in the diagram
  \[
  \begin{tikzcd}
    \B[\inv{q_G}(U)] \bdlaut{A}{B}{K}(\xi)_G \rar{\req} \ar{dd} & \B[U] \bdlaut{\rat A}{\rat B}{\rat K}(\rat \xi) \ar{dd} \rar[squiggly]{\eq} & \MCb \bigl( \lie g_{L_A}^{L_B}(\rho)_{\ol U} \bigr) \\
    & & \MCb \bigl( \lie h_{\ol U, \ol V} \bigr) \uar \dar \\
    \B[\inv{q_G}(V)] \bdlaut{A'}{B'}{K}(\xi)_H \rar{\req} & \B[V] \bdlaut{\rat A'}{\rat B'}{\rat K}(\rat \xi) \rar[squiggly]{\eq} & \MCb \bigl( \lie g_{L_{A'}}^{L_{B'}}(\rho')_{\ol V} \bigr)
  \end{tikzcd}
  \]
  the left-hand square commutes in the category of simplicial sets with an action of $\quot {\inv{q_G}(N)} {\inv{q_G}(U)}$, and the right-hand square commutes in the homotopy category of simplicial sets with an action of the preimage of $N \times M$ in $\Aut[a](m)_\rho$.
  Here the two right-hand horizontal maps are those of \cref{thm:model}.
\end{proposition}

\begin{proof}
  Commutativity of the left-hand square is clear, so we focus on the right-hand square.
  Combining \cref{lemma:exp_eq_forget,lemma:Hom_map_forget} with \cref{lemma:MC_outer,prop:exp_eq,lemma:submonoid_bar,lemma:left_conj_eq} yields a commutative diagram
  \[
  \begin{tikzcd}
    \MCb \bigl( \lie g_{L_A}^{L_B}(\rho)_{\ol U} \bigr) \rar[squiggly]{\eq} & \B[\ol U(\QQ)] \bdlaut{\nerve{L_A}}{\nerve{L_B}}{\nerve{\Pi}} \bigl( \nerve{\rho} \bigr) \\
    \MCb \bigl( \lie h_{\ol U, \ol V} \bigr) \rar[squiggly] \dar \uar & \B[W] \Bigl( \map[\nerve{L_B}] \bigl( \nerve{L_X}, \nerve{\Pi} \bigr)_{\nerve{\rho}},  \aut[\nerve{a}] \bigl( \nerve{m} \bigr)_{\eqcl{\nerve{\rho}}} \Bigr) \dar \uar \\
    \MCb \bigl( \lie g_{L_{A'}}^{L_{B'}}(\rho')_{\ol V} \bigr) \rar[squiggly]{\eq} & \B[\ol V(\QQ)] \bdlaut{\nerve{L_{A'}}}{\nerve{L_{B'}}}{\nerve{\Pi}} \bigl( \nerve{\rho'} \bigr)
  \end{tikzcd}
  \]
  in the homotopy category of simplicial sets with an action of the preimage of $N \times M$ in $\Aut[a](m)_\rho$.
  Here $W$ is the image of the map $\hg 0 ( \expb (\lie h_{\ol U, \ol V}) ) \to \hg 0 ( \aut[\nerve{a}](\nerve{m}) )$ induced by the map of \cref{lemma:exp_eq_forget}.
  Furthermore, consider the following diagram in the homotopy category of pairs of a simplicial monoid and a right module over it
  \[
  \begin{tikzcd}
    \vertpair {\aut[\nerve{L_A}] \bigl( \nerve{L_X} \bigr)} {\map[\nerve{L_B}] \bigl( \nerve{L_X}, \nerve{\Pi} \bigr)} \rar[equal] & \vertpair {\aut[\nerve{L_A}] \bigl( \nerve{L_X} \bigr)} {\map[\nerve{L_B}] \bigl( \nerve{L_X}, \nerve{\Pi} \bigr)} \dar & \lar[squiggly][swap]{\eq} \vertpair {\aut[\rat A](\rat X)} {\map[\rat B] ( \rat X, \rat K )} \dar \\
     & \vertpair {\aut[\nerve{L_{A'}}] \bigl( \nerve{L_X} \bigr)} {\map[\nerve{L_{B'}}] \bigl( \nerve{L_X}, \nerve{\Pi} \bigr)} & \lar[squiggly][swap]{\eq} \ar[squiggly, end anchor = north east]{ldd}{\eq} \vertpair {\aut[\rat A'](\rat X)} {\map[\rat B'] ( \rat X, \rat K )} \\
    \vertpair {\aut[\nerve{a}] \bigl( \nerve{m} \bigr)} {\map[\nerve{L_B}] \bigl( \nerve{L_X}, \nerve{\Pi} \bigr)} \ar{uu} \dar \rar & \vertpair {\aut[\nerve{L_{A'}}] \bigl( \nerve{m} \bigr)} {\map[\nerve{L_{B'}}] \bigl( \nerve{L_X}, \nerve{\Pi} \bigr)} \dar{\eq} \uar[swap]{\eq} \\
    \vertpair {\aut[\nerve{L_{A'}}] \bigl( \nerve{L'_X} \bigr)} {\map[\nerve{L_{B'}}] \bigl( \nerve{L'_X}, \nerve{\Pi} \bigr)} \rar[equal] & \vertpair {\aut[\nerve{L_{A'}}] \bigl( \nerve{L'_X} \bigr)} {\map[\nerve{L_{B'}}] \bigl( \nerve{L'_X}, \nerve{\Pi} \bigr)} &
  \end{tikzcd}
  \]
  where the three right-hand zig-zags are those of \cref{thm:Baut_functor}, with the diagonal one being induced by the map \eqref{eq:L'-model}.
  The upper right-hand square commutes in the homotopy category of simplicial sets by \cref{lemma:zig-zag_forget}, and the bottom right-hand triangle commutes by \cref{thm:Baut_functor} since $\nerve{m} \after \nerve{m'} \eq \id$ under $\nerve {L_{A'}}$ and over $\nerve \Pi$ (note that $\nerve \blank$ preserves right homotopies since it preserves products, fibrations, and weak equivalences).
  Combining the two preceding diagrams, we obtain the desired statement.
\end{proof}

\begin{remark} \label{rem:forget_no_map}
  Unfortunately, it seems unlikely to the authors that the map of dg Lie algebras $\Der[\ol U, \ol V](m \rel a) \to \Der[\ol U](L_X \rel L_A)$ is a quasi-isomorphism in general.
  The reason is that the map $m^*$ in \eqref{eq:Der_cospan} is not always surjective in positive degrees.
  This precludes us from concluding that the pullback of the quasi-isomorphism $m_*$ is again a quasi-isomorphism.
  One could try instead to use dg Lie subalgebras of the zig-zag
  \[ \Der(L_X \rel L_A)  \longto  \Der(L_X \rel L_{A'})  \longfrom  \Der(m \rel L_{A'})  \longto  \Der(L'_X \rel L_{A'}) \]
  but the authors did not find a nilpotent dg Lie subalgebra of $\Der(L_X \rel L_{A'})$ that contains the image of $\Der[\ol U](L_X \rel L_A)$ such that the appropriate restriction of $m^*$ remains surjective.
\end{remark}

Similarly to the preceding subsection, we need $\ol \phi$ to induce a map $\aredEaut[L_A](\rho) \to \aredEaut[L_{A'}](\rho')$ of the maximal reductive quotients to be able to obtain a compatibility of the model for $\B \bdlaut{A}{B}{K}(\xi)$ of \cref{cor:Baut_eq} with restriction maps.
We will see in \cref{rem:forget_no_induced} below that the existence of this induced map is not automatic, so we will need to assume it in the following.
Moreover, we need to equip the zig-zag of derivation dg Lie algebras of \cref{thm:model_forget} (applied to the case where $\ol U$ and $\ol V$ are the respective unipotent radical)
\begin{equation} \label{eq:Deru_zig-zag}
  \lie g_{L_A}^{L_B}(\rho)  \longfrom  \lie h  \longto  \lie g_{L_{A'}}^{L_{B'}}(\rho')
\end{equation}
with an $\aAut[L_A](L_X)_\rho$-action, compatible with an action of $\aAut[L_{A'}](L'_X)_{\rho'}$ on the right-hand term.
However there is, in general, no obvious map $\aAut[L_A](L_X)_\rho \to \aAut[L_{A'}](L'_X)_{\rho'}$.
To circumvent this, we will assume that the projection $\ol \pi_A \colon \aAut[a](m)_\rho \to \aAut[L_A](L_X)_\rho$ has a section (since we assumed $m$ to be quasi-free and hence injective, this is equivalent to $\ol \pi_A$ being an isomorphism).
Explicitly this means that, given $R \in \Algfg{\QQ}$, every automorphism of $L_X \tensor R$ under $L_A \tensor R$ and over $\Pi \tensor R$ restricts to an automorphism of $L'_X \tensor R$
This is a fairly strong assumption which results from our need to work with minimal models.

\begin{assumption} \label{ass:forget}
  We use \cref{not:forget} and assume that $i$ and $i'$ are minimal and that $\ol \phi$ induces a map
  \[ \ol \nu \colon \aredEaut[L_A](\rho)  \longto  \aredEaut[L_{A'}](\rho') \]
  of algebraic groups.
  This is the case, for example, when $\aEaut[L_A](L_X)_\rho$ is already reductive.
  
  We furthermore assume that the map
  \[ \ol \pi_A \colon \aAut[a](m)_\rho  \longto  \aAut[L_A](L_X)_\rho \]
  is an isomorphism of algebraic groups.
  This yields a composite map of algebraic groups
  \[ \ol \Phi  \colon  \aAut[L_A](L_X)_\rho  \xlongto{\inv{(\ol \pi_A)}}  \aAut[a](m)_\rho  \longto  \aAut[L_{A'}](L'_X)_{\rho'} \]
  such that $\areal \after \ol \Phi = \ol \phi \after \areal$ (by \cref{lemma:aEaut_forget}).
\end{assumption}

\begin{remark} \label{rem:forget_no_induced}
  A map $\ol \nu$ as above does not always exist.
  For a counterexample, consider $V \defeq \shift[1] \QQ$, $W \defeq V \dirsum V$, and $i \colon V \to W$ the inclusion of the first summand.
  Then we consider the sequence of graded Lie algebras
  \[ 0  \longto  \freelie V  \xlongto{i_*}  \freelie W \]
  equipped with the trivial differentials; in particular each of the two maps, as well as their composite, is minimal.
  
  The algebraic group $\aAut[\freelie V](\freelie W)$ is isomorphic to the algebraic subgroup $\ol G$ of $\aGL(\QQ^2)$ of automorphisms that fix the first summand pointwise.
  This algebraic group is not reductive: the algebraic subgroup $\ol U$ of upper triangular matrices with $1$'s on the diagonal is non-trivial, normal, and unipotent.
  In fact, this is already the unipotent radical and the maximal reductive quotient is isomorphic to the multiplicative group $\aGL(\QQ)$.
  On the other hand, the algebraic group $\aAut(\freelie W)$ is isomorphic to $\aGL(\QQ^2)$, which is already reductive.
  In particular, there can not be an induced map $\quot {\ol G} {\ol U} \to \aGL(\QQ^2)$.
\end{remark}

The following lemma provides the promised action on the zig-zag \eqref{eq:Deru_zig-zag}.

\begin{lemma} \label{lemma:choose_compatibly_forget}
  We use \cref{not:forget} and assume that \cref{ass:forget} is fulfilled.
  Furthermore, let $\bar s$ be a section of the composite map of algebraic groups
  \[ \aAut[L_A](L_X)_\rho  \xlongto{\areal}  \aEaut[L_A] ( L_X )_\rho  \longto  \aredEaut[L_A] ( \rho ) \]
  (such a section exists by \cref{lemma:areal_section}).
  Then there exists a section $\bar t$ of the composite map of algebraic groups
  \[ \aAut[L_{A'}](L'_X)_{\rho'}  \xlongto{\areal}  \aEaut[L_{A'}] ( L'_X )_{\rho'}  \longto  \aredEaut[L_{A'}] ( \rho' ) \]
  such that $\bar t \after \ol \nu = \ol \Phi \after \bar s$.
\end{lemma}

\begin{proof}
  This follows from \cref{lemma:compatible_Levi}.
\end{proof}

We are now ready to state the compatibility of the model for $\B \bdlaut{A}{B}{K}(\xi)$ of \cref{cor:Baut_eq} with restriction maps.
Afterwards, we will also deduce such a compatibility for the cdga model of \cref{cor:forms_eq} and the identification of its cohomology of \cref{cor:cohomology}.

\begin{theorem} \label{cor:Baut_eq_forget}
  We use \cref{not:forget} and assume that \cref{ass:forget} is fulfilled.
  Furthermore, let $G \subseteq \Eaut[A](X)_{\eqcl{\xi}_B}$ and $H \subseteq \Eaut[A'](X)_{\eqcl{\xi}_{B'}}$ be finite-index subgroups such that $\phi(G) \subseteq H$, and denote by $\redim{G}$ and $\redim{H}$ their images in $\aredEaut[L_A] (\rho)(\QQ)$ and $\aredEaut[L_{A'}] (\rho')(\QQ)$, respectively.
  Furthermore let $\bar s$ and $\bar t$ be sections as in \cref{lemma:choose_compatibly_forget}.
  We write
  \[ \lie h  \defeq  \trunc 0 { \Hom \bigl( \shift \indec[L_B](L_X), \Pi \bigr) }  \rsemidir[\tilde \rho_*] \Deru[\rho](m \rel a) \]
  where the outer action $\tilde \rho_*$ is as in \cref{prop:bdlaut_model} (via the projection to $\Der[\ol U](L_X \rel L_A)$).
  Then there is a commutative diagram in the rational homotopy category of simplicial sets
  \[
  \begin{tikzcd}
    \B \bdlaut{A}{B}{K}(\xi)_G \ar{dd} \rar[squiggly]{\req} & \hcoinv { \MCb \bigl( \lie g_{L_A}^{L_B}(\rho) \bigr) } {\redim{G}} \\
    & \hcoinv { \MCb ( \lie h ) } {\redim{G}} \uar \dar \\
    \B \bdlaut{A'}{B'}{K}(\xi)_H \rar[squiggly]{\req} & \hcoinv { \MCb \bigl( \lie g_{L_{A'}}^{L_{B'}}(\rho') \bigr) } {\redim{H}}
  \end{tikzcd}
  \]
  where the two horizontal maps are those of \cref{cor:Baut_eq}.
  Here $\redim{G}$ acts on $\lie g_{L_A}^{L_B}(\rho)$ by conjugation via the group homomorphism $\bar s_\QQ$, on $\lie h$ via $\bar s_\QQ$ and $(\ol \Phi \after \bar s)_\QQ$, and $\redim{H}$ acts on $\lie g_{L_{A'}}^{L_{B'}}(\rho')$ via $\bar t_\QQ$.
\end{theorem}

\begin{proof}
  This follows from \cref{thm:model_forget,lemma:hquot_hcoinv}, as in the proof of \cref{cor:Baut_eq}.
\end{proof}

\begin{corollary} \label{cor:forms_eq_forget}
  In the situation of \cref{cor:Baut_eq_forget}, let $M$ and $N$ be two representations in rational vector spaces of $\redim{G}$ and $\redim{H}$, respectively, and let $f \colon N \to M$ be a map that is equivariant with respect to $\ol \nu_\QQ \colon \redim{G} \to \redim{H}$.
  Then there is a commutative diagram in the homotopy category of cochain complexes
  \[
  \begin{tikzcd}
    \Forms * \bigl( \B \bdlaut{A}{B}{K}(\xi)_G; M \bigr) \rar[squiggly]{\eq} & \Forms * \bigl( \B \redim{G}; \CEcochains * \bigl( \lie g_{L_A}^{L_B}(\rho); M \bigr) \bigr) \dar \\
    & \Forms * \bigl( \B \redim{G}; \CEcochains * ( \lie h; M ) \bigr) \\
    \Forms * \bigl( \B \bdlaut{A'}{B'}{K}(\xi)_H; N \bigr) \ar{uu} \rar[squiggly]{\eq} & \Forms * \bigl( \B \redim{H}; \CEcochains * \bigl( \lie g_{L_{A'}}^{L_{B'}}(\rho'); N \bigr) \bigr) \uar
  \end{tikzcd}
  \]
  where the two horizontal maps are those of \cref{cor:forms_eq}.
  Moreover, this diagram can be lifted to a commutative diagram in the homotopy category of lax symmetric monoidal functors from the category of tuples $(M, N, f)$ to cochain complexes (with monoidal natural transformations as morphisms, and the pointwise weak equivalences).
  Here $\gamma \in \redim{G}$ acts on the cochain complex $\CEcochains * ( \lie g_{L_A}^{L_B}(\rho); M )$ by $(\inv \gamma, \gamma)$, where the action on $\lie g_{L_A}^{L_B}(\rho)$ is by conjugation via the group homomorphism $\bar s_\QQ$, it similarly acts on $\lie h$ via $\bar s_\QQ$ and $(\ol \Phi \after \bar s)_\QQ$, and $\redim{H}$ acts on $\lie g_{L_{A'}}^{L_{B'}}(\rho')$ via $\bar t_\QQ$.
\end{corollary}

\begin{proof}
  This follows from \cref{cor:Baut_eq_forget,lemma:forms_orbits,lemma:MC_forms,lemma:hquot_hcoinv_local_system}, as in the proof of \cref{cor:forms_eq}.
\end{proof}

\begin{corollary} \label{cor:cohomology_forget}
  In the situation of \cref{cor:Baut_eq_forget}, let $M$ and $N$ be two representations in rational vector spaces of $\redim{G}$ and $\redim{H}$, respectively, and let $f \colon N \to M$ be a map that is equivariant with respect to $\ol \nu_\QQ \colon \redim{G} \to \redim{H}$.
  Then there is a commutative diagram in the category of graded vector spaces
  \[
  \begin{tikzcd}
    \Coho * \bigl( \B \bdlaut{A}{B}{K}(\xi)_G; M \bigr) \rar{\iso} & \Coho * \bigl( \B \redim{G}; \CEcoho * \bigl( \lie g_{L_A}^{L_B}(\rho); M \bigr) \bigr) \dar \\
    & \Coho * \bigl( \B \redim{G}; \CEcoho * ( \lie h; M ) \bigr) \\
    \Coho * \bigl( \B \bdlaut{A'}{B'}{K}(\xi)_H; N \bigr) \ar{uu} \rar{\iso} & \Coho * \bigl( \B \redim{H}; \CEcoho * \bigl( \lie g_{L_{A'}}^{L_{B'}}(\rho'); N \bigr) \bigr) \uar
  \end{tikzcd}
  \]
  where the two horizontal maps are those of \cref{cor:cohomology}.
  Moreover, this diagram consists of monoidal natural transformations between lax symmetric monoidal functors from the category of tuples $(M, N, f)$ to graded vector spaces.
  Here $\gamma \in \redim{G}$ acts on $\CEcoho * ( \lie g_{L_A}^{L_B}(\rho); M )$ by $(\inv \gamma, \gamma)$, where the action on $\lie g_{L_A}^{L_B}(\rho)$ is by conjugation via the group homomorphism $\bar s_\QQ$, it similarly acts on $\lie h$ via $\bar s_\QQ$ and $(\ol \Phi \after \bar s)_\QQ$, and $\redim{H}$ acts on $\lie g_{L_{A'}}^{L_{B'}}(\rho')$ via $\bar t_\QQ$.
\end{corollary}

\begin{proof}
  This follows from \cref{cor:forms_eq_forget,lemma:forms_cohomology,lemma:forms_split_coeff}, as in the proof of \cref{cor:cohomology}.
\end{proof}

\subsection{The case of manifolds with boundary a sphere} \label{sec:manifolds}

Given a simply connected (high-dimensional) manifold with simply connected boundary, we obtained rational models for the classifying spaces $\B \aut[\bdry](M)$ and $\B \BlDiff[\bdry](M)$ in \cref{sec:model,sec:block_models}.
However, in the case that the boundary is a sphere, Berglund--Madsen \cite[Theorem~1.2]{BM} obtained related, but simpler, models for the spaces $\B \aut[\bdry](M)_{\id}$ and $\B \BlDiff[\bdry](M)_{\id}$ in terms of a certain derivation dg Lie algebra $\Der(\freelie V \rel \omega_V)$.
In this subsection, we prove that under certain assumptions there are models for the whole spaces $\B \aut[\bdry](M)$ and $\B \BlDiff[\bdry](M)$ in terms of $\Der(\freelie V \rel \omega_V)$.

\begin{remark}
  Suitably formulated, the results of this subsection for $\B \aut[\bdry](M)$ would also apply more generally to Poincaré duality complexes, but we will restrict ourselves to manifolds since the formulation of the results is more convenient.
\end{remark}

We begin by recalling the definitions necessary for introducing the dg Lie algebra $\Der(\freelie V \rel \omega_V)$.

\begin{definition}
  Let $m \in \ZZ$.
  A \emph{graded symplectic form} of degree $-m$ on a graded vector space $V$ is a unimodular bilinear form of degree $-m$
  \[ \iprod \blank \blank  \colon  V \tensor V  \longto  \QQ \]
  such that $\iprod v w = (-1)^{\deg v \deg w + 1} \iprod w v$ (i.e.\ it is graded anti-symmetric).
  We denote by $\Sp m$ the category of finite-dimensional graded vector spaces equipped with a graded symplectic form of degree $-m$, with morphisms those linear maps of degree $0$ that preserve the bilinear form.
\end{definition}

\begin{definition}
  Let $m \in \ZZ$ and $V \in \Sp m$.
  Furthermore choose a homogeneous basis $\alpha_1, \dots, \alpha_n$ of $V$.
  Then we set
  \[ \omega_V  \defeq  \frac 1 2 \sum_{i = 0}^n \liebr {\dualbasis{\alpha_i}} {\alpha_i}  \in  \freelie V \]
  where $\dualbasis{\alpha_i}$ denotes the dual basis with respect to the inner product on $V$.
  This element $\omega_V$ is independent of the choice of basis (see e.g.\ \cite[p.~91]{BM}).
\end{definition}

The following result of Berglund--Madsen (building on work of Stasheff \cite[Theorem~2]{Sta}) is the starting point of this section.
It provides a particularly nice dg Lie algebra model for manifolds with boundary a sphere.

\begin{proposition}[Stasheff, Berglund--Madsen] \label{prop:Stasheff}
  Let $M$ be an oriented simply connected compact topological manifold of dimension $n \ge 2$ such that $\bdry M \iso \Sphere {n-1}$, and set $V_M \defeq \shift[-1] \rHo * (M; \QQ)$.
  Then there exists a differential $\delta_M$ on $\freelie V_M$ such that $L_M \defeq ( \freelie V_M, \delta_M )$ is a minimal quasi-free dg Lie algebra and the element $\omega_M \defeq \omega_{V_M} \in \freelie V_M$ is a cycle such that the map $\freelie(\omega_M) \to L_M$ models the inclusion $\bdry M \to M$.
  Here $V_M$ is equipped with the graded symplectic form of degree $-(n - 2)$ given by $\iprod {\shift[-1] a} {\shift[-1] b} \defeq (-1)^{\deg a} \iprod[\intersect] a b$, where $\iprod[\intersect] \blank \blank$ denotes the intersection pairing.
\end{proposition}

\begin{proof}
  This is \cite[Theorem~3.11]{BM}.
\end{proof}

The inclusion $\freelie(\omega_V)  \to  \freelie V$ is not quasi-free, and hence the results of the preceding subsections do not apply to it.
To amend this, we will (as in \cite[Proof of Theorem~3.12]{BM}) replace it by a quasi-isomorphic map that is minimal quasi-free.

\begin{definition} \label{def:tilde_L}
  Let $m \in \ZZ$, $V \in \Sp m$, and $\delta$ a differential on $\freelie V$ such that $L \defeq (\freelie V, \delta)$ is a dg Lie algebra with $\delta(\omega_V) = 0$.
  Then we write
  \[ \widetilde L  \defeq  \bigl( \freelie (V \dirsum \linspan{\beta_L, \gamma_L}), \tilde \delta \bigr) \]
  for the dg Lie algebra where $\deg {\beta_L} = m$, $\deg {\gamma_L} = m + 1$, and $\tilde \delta$ is uniquely determined by specifying that $\tilde \delta(\beta_L) = 0$, $\tilde \delta(\gamma_L) = \omega_V - \beta_L$, and $\tilde \delta \restrict {} V = \restrict \delta V$.
  It comes equipped with two maps
  \[ \freelie(\beta_L)  \xlongto{i_L}  \widetilde L  \xlongto{p_L}  L \]
  where $\freelie(\beta_L)$ is equipped with the trivial differential, and $p_L$ is uniquely determined by $p_L(\beta_L) = \omega_V$, $p_L(\gamma_L) = 0$, and $p_L(v) = v$ for all $v \in V$.
  Note that $i_L$ is a quasi-free map and that $p_L$ is a quasi-isomorphism (e.g.\ by \cite[Proposition~22.12]{FHT}).
  If $L$ is a minimal, then so is $i_L$.
\end{definition}

We will now relate the algebraic groups $\aAut[\omega_V](L)$ and $\aAut[\beta_L](\widetilde L)$.
In particular we will prove that their maximal reductive quotients agree.

\newcommand{\ext}[1][]{{\xi_{#1}}}
\newcommand{\aext}[1][]{{\ol \xi_{#1}}}

\begin{definition} \label{def:Aut_tilde_L}
  In the situation of \cref{def:tilde_L}, we denote by
  \[ \ext[L] \colon \Aut[\omega_V](L)  \longto  \Aut[\beta_L](\widetilde L) \]
  the group homomorphism given by extending by the identity on $\beta_L$ and $\gamma_L$.
  Moreover, when $V$ is positively graded, we denote by
  \[ \aext[L] \colon \aAut[\omega_V](L)  \longto  \aAut[\beta_L](\widetilde L) \]
  the canonical extension of $\ext[L]$ to a map of algebraic groups (using \cref{lemma:Lie_algebra_extension}).
\end{definition}

\begin{lemma} \label{lemma:aAut_omega}
  Let $m \ge 2$, $V \in \Sp{m}$ concentrated in positive degrees, and $\delta$ a differential on $\freelie V$ such that $L \defeq (\freelie V, \delta)$ is a minimal quasi-free dg Lie algebra with $\delta(\omega_V) = 0$.
  Then the composite
  \[ \aAut[\omega_V](L)  \xlongto{\aext[L]}  \aAut[\beta_L](\widetilde L)  \xlongto{\areal}  \aEaut[\beta_L](\widetilde L) \]
  is a quotient map of algebraic groups.
\end{lemma}

\begin{proof}
  First note that the image of $\aext[L]$ contains $\aAut[\beta_L,\gamma_L](\widetilde L)$.
  Hence it is enough to prove that any element $\phi \in \aAut[\beta_L](\widetilde L)(R)$ is homotopic relative to $\beta_L$ to an automorphism $\psi$ such that $\psi(\gamma_L \tensor 1) = \gamma_L \tensor 1$.
  This is equivalent to finding an element $\nu \in \aNull[\beta_L](\widetilde L)(R)$ such that $(\nu \after \phi)(\gamma_L \tensor 1) = \gamma_L \tensor 1$.
  For ease of notation we will do so for $R = \QQ$; the general case is analogous.
  It follows from the proof of \cref{prop:exp_eq} (applied to the trivial algebraic subgroup of $\aEaut[\beta_L](\widetilde L)$) that the map $\aexpact$ of \cref{lemma:exp_action} restricts to a map
  \[ \aexp \bigl( \Bound 0 \bigl( \Der(\widetilde L \rel \beta_L) \bigr) \bigr)  \longto  \aNull[\beta_L](\widetilde L) \]
  of algebraic groups; we will use this to construct a map $\nu$ as desired.
  
  For degree reasons, we have that $\phi(\gamma_L) = r \gamma_L + \alpha$ for some $r \in \QQ$, and that $\alpha$ is, for some $k \ge 2$, contained in the part $\freelie[\ge k](V) \subseteq L$ of word-length $\ge k$.
  We obtain
  \[ \phi(\omega_V) - \beta_L  =  \phi(\omega_V - \beta_L)  =  \phi \tilde \delta(\gamma_L)  =  \tilde \delta \phi(\gamma_L)  =  \tilde \delta(r \gamma_L + \alpha)  =  r \omega_V - r \beta_L + \tilde \delta (\alpha) \]
  and since $\phi(\omega_V)$, $\omega_V$, and $\tilde \delta (\alpha)$ are all contained in $L$, this implies that $r = 1$.
  By \cite[Corollary~3.10]{BM}, the evaluation map $\ev[\omega_V] \colon \Der(L) \to L$ is surjective onto the decomposables $\liebr{L}{L} \subseteq L$.
  Hence we can pick some $\theta \in \Der(L)_1$ such that $\theta(\omega_V) = - \alpha$.
  Let $\tilde \theta \in \Der(\widetilde L \rel \beta_L)_1$ be the derivation obtained by extending $\theta$ with $\tilde \theta(\beta_L) = 0$ and $\tilde \theta(\gamma_L) = 0$.
  We claim that $(\aexpact_\QQ(\liebr {\tilde \delta} {\tilde \theta}) \after \phi)(\gamma_L) - \gamma_L \in \freelie[\ge k + 1](V)$.
  Assuming this, we can iterate this procedure to find a $\nu \in \Null[\beta_L](\widetilde L)$ such that $(\nu \after \phi)(\gamma_L) - \gamma_L \in \freelie[\ge m + 2](V)_{m + 1} = 0$, which is what we wanted to show.
  To prove the claim, we compute
  \begin{align*}
    \bigl( \aexpact_\QQ ( \liebr {\tilde \delta} {\tilde \theta} ) \after \phi \bigr) (\gamma_L) &= \sum_{n \ge 0} \frac {1} {n!} \liebr {\tilde \delta} {\tilde \theta}^{\after n}(\gamma_L + \alpha) \\
    &=  \gamma_L + \alpha + \sum_{n \ge 1} \frac {1} {n!} \Bigl( \liebr {\tilde \delta} {\tilde \theta}^{\after n - 1} \bigl( \tilde \delta \tilde \theta(\gamma_L) + \tilde \theta \tilde \delta(\gamma_L) \bigr) + \liebr {\tilde \delta} {\tilde \theta}^{\after n}(\alpha) \Bigr) \\
    &=  \gamma_L + \alpha + \sum_{n \ge 1} \frac {1} {n!} \Bigl( \liebr {\tilde \delta} {\tilde \theta}^{\after n - 1} ( - \alpha ) + \liebr {\tilde \delta} {\tilde \theta}^{\after n}(\alpha) \Bigr) \\
    &=  \gamma_L + \sum_{n \ge 1} \Bigl( \frac {1} {n!} - \frac {1} {(n + 1)!} \Bigr) \liebr {\tilde \delta} {\tilde \theta}^{\after n}(\alpha)
  \end{align*}
  and note that, since $L$ is minimal by assumption, the differential $\tilde \delta$ strictly increases bracket length; hence the same is true for $\liebr {\tilde \delta} {\tilde \theta}$.
\end{proof}

The following lemma, essentially due to Berglund--Madsen, relates the derivation dg Lie algebras of the two different models appearing above.

\begin{lemma}[Berglund--Madsen] \label{lemma:Der_omega}
  Let $m \ge 2$ be an integer, $V \in \Sp m$ concentrated in positive degrees, and $\delta$ a differential on $\freelie V$ such that $L \defeq (\freelie V, \delta)$ is a dg Lie algebra with $\delta(\omega_V) = 0$.
  Then the following map, given by extending a derivation by $0$ on $\beta_L$ and $\gamma_L$,
  \[ \Xi  \colon  \Der(L \rel \omega_V)  \longto  \Der(\widetilde L \rel \beta_L) \]
  is a morphism of dg Lie algebras and induces an isomorphism on homology in non-negative degrees.
\end{lemma}

\begin{proof}
  That the map is a map of dg Lie algebras is an elementary verification.
  Now consider the diagram of chain complexes
  \[
  \begin{tikzcd}
    \Der(L \rel \omega_V) \drar[dashed, start anchor = south east, end anchor = north west] \ar[bend left = 13]{drr}{\Xi} \ar[bend right]{ddr}[swap]{\id} & & \\[-5]
     & \Der \bigl( p_L \rel \beta_L \bigr) \dar[swap]{\pr[2]} \rar{\pr[1]} & \Der ( \widetilde L \rel \beta_L ) \dar{(p_L)_*} \\
     & \Der(L \rel \omega_V) \rar{p_L^*} & \Derrel{p_L}(\widetilde L, L \rel \beta_L)
  \end{tikzcd}
  \]
  where the lower right-hand square is a pullback by definition.
  The outer square commutes as well, and hence we obtain a dashed map of chain complexes as indicated that makes both triangles commute.
  Berglund--Madsen \cite[Proof of Theorem~3.12]{BM} showed that $\pr[2]$ is a quasi-isomorphism and that $\pr[1]$ induces an isomorphism on homology in positive degrees.
  Carefully going through their argument, one sees that it in fact shows that $\pr[1]$ induces an isomorphism on homology in non-negative degrees, which finishes the proof.
\end{proof}

We are now ready to state the first main result of this subsection.
Combined with \cref{cor:Baut_eq} it yields a model for $\B \aut[\bdry](M)$ in terms of $\Der(\freelie V \rel \omega_V)$ when the boundary of $M$ is a sphere.

\begin{definition}
  Let $m \ge 2$, $V \in \Sp{m}$ concentrated in positive degrees, and $\delta$ a differential on $\freelie V$ such that $L \defeq (\freelie V, \delta)$ is a minimal quasi-free dg Lie algebra with $\delta(\omega_V) = 0$.
  Then we write $\Deru(L \rel \omega_V) \subseteq \Der(L \rel \omega_V)$ for the preimage of $\Deru(\widetilde L \rel \beta_L)$ under $\Xi$.
\end{definition}

\begin{theorem} \label{thm:manifolds}
  Let $M$ be an oriented simply connected compact topological manifold of dimension $n \ge 3$ such that $\bdry M \iso \Sphere {n-1}$, and let $L_M \defeq ( \freelie V_M, \delta_M )$ be a model for $M$ as in \cref{prop:Stasheff}.
  Write $\arithEaut[\bdry](M)$ for the image of $\Eaut[\bdry](M)$ in $\aredEaut[\beta_M](\widetilde L_M)(\QQ)$ and choose a section $\bar s$ of the composite map of algebraic groups
  \[ \aAut[\omega_M](L_M)  \xlongto{\aext[M]}  \aAut[\beta_M](\widetilde L_M)  \longto  \aredEaut[\beta_M](\widetilde L_M) \]
  (this exists by \cref{lemma:aAut_omega}).
  Then there is a $\arithEaut[\bdry](M)$-equivariant quasi-isomorphism of nilpotent dg Lie algebras
  \[ \Deru(\widetilde L_M \rel \beta_M)  \eq  \Deru(L_M \rel \omega_M) \]
  where $\arithEaut[\bdry](M)$ acts on both sides by conjugation: on the left via $\aext[M] \after \bar s$ and on the right via $\bar s$.
\end{theorem}

\begin{proof}
  First note that by \cref{prop:Stasheff} the map $\freelie(\beta_M) \to L_M$ that sends $\beta_M$ to $\omega_M$ indeed models the inclusion $\inc[\bdry M] \colon \bdry M \to M$.
  Using the construction of \cref{def:tilde_L}, we factor it as a quasi-free minimal map $i_M \colon \freelie(\beta_M) \to \widetilde L_M$ followed by a surjective quasi-isomorphism $p_M \colon \widetilde L_M \to L_M$.
  Lifting the weak equivalence $\rat{\Sing(M)} \to \nerve{L_M}$ along $\nerve{p_M}$ relative to $\rat{\Sing(\bdry M)}$, we see that $i_M$ models $\inc[\bdry M]$ as well.
  
  By construction of $\Deru(L_M \rel \omega_M)$, the map of \cref{lemma:Der_omega} restricts to a quasi-isomorphism
  \[ \Xi  \colon  \Deru(L_M \rel \omega_M)  \longto  \Deru(\widetilde L_M \rel \beta_M) \]
  of dg Lie algebras.
  It is equivariant with respect to the map $\ext[M] \colon \Aut[\omega_M](L_M)  \to  \Aut[\beta_M](\widetilde L_M)$ by inspection.
\end{proof}

\begin{remark} \label{rem:trivial_diff}
  When $M$ is formal and the multiplication of $\rCoho * (M; \QQ)$ is trivial, then any differential $\delta_M$ as in \cref{prop:Stasheff} is trivial (see \cite[Proof of Corollary~3.13]{BM}).
  In this case the statement of \cref{thm:manifolds} can be simplified further.
  First note that the following composite is an isomorphism of algebraic groups
  \[ \lambda \colon \aSP[m] (V)  \xlongto{\freelie}  \aAut[\omega_V](L)  \xlongto{\aext[L]}  \aAut[\beta_L](\widetilde L)  \xlongto{\Koppa}  \ahoEaut[\beta_L](\widetilde L) \]
  where $\Koppa$ and $\ahoEaut$ are as in \cref{lemma:indecomposables_representation} and \cref{def:ahoEaut}, respectively, and $\aSP[m](V)$ denotes the algebraic group of automorphisms of $V$ that preserve its symplectic form.
  (To see this, note that $\lambda$ is an embedding, that by \cref{lemma:aAut_omega} every automorphism in $\aAut[\beta_L](\widetilde L)$ is homotopic to an automorphism $\psi$ that also fixes $\omega_V$ and $\gamma_L$, that $\omega_V \in V \tensor V$ is dual to the symplectic form of $V$ (cf.\ \cite[p.~91]{BM}), and that hence the automorphism of $V$ induced by $\psi$ preserves the symplectic form.)
  Since $\aSP[m](V)$ is reductive, this implies (using \cref{lemma:indecomposables_representation,lemma:unipotent_radical}) that $\ahoEaut[\beta_L](\widetilde L)$ is the maximal reductive quotient of $\aAut[\beta_L](\widetilde L)$.
  Hence (using \cref{lemma:indec_is_ho}) the group $\arithEaut[\bdry](M)$ is the image of $\Eaut[\bdry](M)$ in the automorphisms of $\rHo * (M; \QQ)$, and one can let it act on $L_M = \freelie V_M$ through its canonical action on $V_M = \shift[-1] \rHo * (M; \QQ)$.
  Furthermore, one can identify $\Deru(L_M \rel \omega_M)$ as the dg Lie subalgebra of $\Der(L_M \rel \omega_M)$ that in degree $0$ consists of those derivations $\theta$ such that ${\pr[V_M]} \after \theta \after \inc[V_M] = 0$, where $\pr[V_M] \colon \freelie V_M \to \freelie[1] V_M = V_M$ is the projection.
\end{remark}

We now state the second main result of this subsection.
Combined with \cref{thm:block_diff} it yields a rational model for $\B \BlDiff[\bdry](M)$ in terms of $\Der(\freelie V \rel \omega_V)$.
If $M$ is formal and the multiplication of $\rCoho * (M; \QQ)$ is trivial, then there is again a further simplification analogous to \cref{rem:trivial_diff}.

\begin{definition}
  Let $M$ be an oriented simply connected compact smooth manifold of dimension $n \ge 6$ such that $\bdry M$ is homeomorphic to $\Sphere {n-1}$, and let $L_M$ be a model for $M$ as in \cref{prop:Stasheff}.
  We write $\pont \colon \widetilde L_M \to \hg * (\SO) \tensor \QQ$ for the map of dg Lie algebras that is given by the Pontryagin class
  \[ \pont[i] \colon (V_M)_{4i - 1} = \Ho {4i} (M; \QQ) \longto \QQ \iso \hg {4i - 1}(\SO) \tensor \QQ \]
  on $(V_M)_{4i - 1}$ and that vanishes otherwise (in particular on $\beta_M$ and $\gamma_M$).
  We also write $\Deru[\pont](L_M \rel \omega_M) \subseteq \Der(L_M \rel \omega_M)$ for the preimage of $\Deru[\pont](\widetilde L_M \rel \beta_M)$ under $\Xi$.
\end{definition}

\begin{theorem} \label{thm:manifolds_block}
  Let $M$ be an oriented simply connected compact smooth manifold of dimension $n \ge 6$ such that $\bdry M$ is homeomorphic to $\Sphere {n-1}$, and let $L_M \defeq ( \freelie V_M, \delta_M )$ be a model for $M$ as in \cref{prop:Stasheff}.
  Write $\redEBlaut[\bdry](M)$ for the image of $\hg{0} (\BlDiff[\bdry](M))$ in $\aredEaut[\beta_M](\pont)(\QQ)$, and choose a section $\bar s$ of the composite map of algebraic groups
  \[ \aAut[\omega_M](L_M)_{\pont}  \xlongto{\aext[M]}  \aAut[\beta_M](\widetilde L_M)_{\pont}  \longto  \aredEaut[\beta_M](\pont) \]
  (this exists by \cref{lemma:aAut_omega}).
  Then the dg Lie algebra $\tilde{g}_{\freelie(\beta_M)}^0(\rho)$ of \cref{thm:block_diff} is $\redEBlaut[\bdry](M)$-equivariantly quasi-isomorphism to the nilpotent dg Lie algebra
  \[ \tilde{\lie g}_{\bdry}(M)  \defeq  \trunc 0 { \Hom \bigl( \shift V_M, \hg * (\SO) \tensor \QQ \bigr) } \rsemidir[{\pont}_*] \Deru[\pont](L_M \rel \omega_M) \]
  where the outer action is given by $(f \act \theta)(\shift v) \defeq (-1)^{\deg \theta} f( \shift ({\pr} \after \theta)(v) )$ and ${\pont}_*(\theta) \defeq \pont \act \theta$, where $\pr \colon L_M \to \freelie[1] V_M = V_M$ is the projection.
  Here $\redEBlaut[\bdry](M)$ acts on $\tilde{g}_{\freelie (\beta_M)}^0(\rho)$ via $\aext[M] \after \bar s$ and on $\tilde{\lie g}_{\bdry}(M)$ via $\bar s$.
\end{theorem}

\begin{proof}
  First note that, proceeding as in the proof of \cref{thm:manifolds}, the two maps of spaces $\Disk{n-1} \to \bdry M \to M$ are indeed modeled by the quasi-free maps of dg Lie algebras $0 \to \freelie(\beta_M) \to \widetilde L_M$.
  Furthermore, as in the proof of \cref{thm:block_diff}, the classifying map $\sttang{M} \colon M \to \B \SO$ of the oriented stable tangent bundle of $M$ is modeled up to homotopy by some map $\widetilde L_M \to \hg * (\SO) \tensor \QQ$, which we can choose to be $\pont$ by its construction.
  
  By construction of $\Deru[\pont](L_M \rel \omega_M)$, the map of \cref{lemma:Der_omega} restricts to a quasi-isomorphism
  \[ \Xi  \colon  \Deru[\pont](L_M \rel \omega_M)  \longto  \Deru[\pont](\widetilde L_M \rel \beta_M) \]
  of dg Lie algebras.
  Furthermore, by \cref{lemma:quasi-free_indec}, the inclusion $V_M \dirsum \linspan{\beta_L, \gamma_L} \to \indec(\widetilde L_M)$ is an isomorphism of chain complexes when the source is equipped with the differential such that $d(\gamma_L) = \beta_L$ and that vanishes otherwise; hence the projection defines a quasi-isomorphism $q \colon \indec(\widetilde L_M) \to V_M$.
  This assembles into a commutative diagram of dg Lie algebras with exact rows
  \[
  \begin{tikzcd}
    0 \rar & \trunc 0 { \Hom \bigl( \shift V_M, \hg * (\SO) \tensor \QQ \bigr) } \rar \dar[swap]{q^*} & \lie g \rar \dar & \Deru[\pont](L_M \rel \omega_M) \rar \dar{\Xi} & 0 \\
    0 \rar & \trunc 0 { \Hom \bigl( \shift \indec(\widetilde L_M), \hg * (\SO) \tensor \QQ \bigr) } \rar & \tilde {\lie g} \rar & \Deru[\pont](\widetilde L_M \rel \beta_M) \rar & 0
  \end{tikzcd}
  \]
  such that both of the outer vertical maps are quasi-isomorphisms.
  Hence the map $\lie g \to \tilde {\lie g}$ is a quasi-isomorphism as well.
  It is $\Aut[\omega_M](L_M)_\rho$-equivariant by inspection.
\end{proof}

\subsection{Compatibility with boundary connected sums}

In this subsection, we use the results of \cref{sec:gluing,sec:forget} to prove that the models of \cref{sec:manifolds} are compatible with taking boundary connected sums (note that the boundary of a boundary connected sum of two manifolds with boundary a sphere is again a sphere).

\begin{theorem} \label{thm:manifolds_gluing}
  Let $M$ and $N$ be two $n$-dimensional manifolds as in \cref{thm:manifolds} and let $L_M \defeq ( \freelie V_M, \delta_M )$ and $L_N \defeq ( \freelie V_N, \delta_N )$ be models as in \cref{prop:Stasheff}.
  Denote by $M \bdryconnsum N$ the boundary connected sum of $M$ and $N$, and set $L_{M \bdryconnsum N} \defeq L_M \cop L_N$.
  Assume that the canonical map of algebraic groups
  \[ \ol \gamma \colon \aAut[\omega_M](L_M) \times \aAut[\omega_N](L_N)  \longto  \aAut[\omega_{M \bdryconnsum N}](L_{M \bdryconnsum N}) \]
  induces a map $\ol \nu$ on maximal reductive quotients (this is the case when the algebraic representations of $\aAut[\omega_M](L_M)$ and $\aAut[\omega_N](L_N)$ on $V_M$ and $V_N$ are semi-simple; for example in the situation of \cref{rem:trivial_diff}).
  Lastly choose sections $\bar s_M$, $\bar s_N$, and $\bar s_{M \bdryconnsum N}$ as in \cref{thm:manifolds} such that $\bar s_{M \bdryconnsum N} \after \ol \nu = \ol \gamma \after (\bar s_M \times \bar s_N)$ (this exists by \cref{lemma:compatible_Levi}).
  Then there is a commutative diagram in the homotopy category of simplicial sets
  \[
  \begin{tikzcd}
    \B \aut[\bdry](M) \times \B \aut[\bdry](N) \rar \dar[squiggly][swap]{\req} & \B \aut[\bdry](M \bdryconnsum N) \dar[squiggly]{\req} \\
    \vertbin { \hcoinv { \MCb \bigl( \Deru(L_M \rel \omega_M) \bigr) } { \arithEaut[\bdry](M) } } \times { \hcoinv { \MCb \bigl( \Deru(L_N \rel \omega_N) \bigr) } { \arithEaut[\bdry](N) } } \rar{g_*} & \hcoinv { \MCb \bigl( \Deru(L_{M \bdryconnsum N} \rel \omega_{M \bdryconnsum N}) \bigr) } { \arithEaut[\bdry](M \bdryconnsum N) }
  \end{tikzcd}
  \]
  where the vertical maps are obtained by combining \cref{cor:Baut_eq,thm:manifolds}, and
  \[ g \colon \Deru(L_M \rel \omega_M) \times \Deru(L_N \rel \omega_N)  \longto  \Deru(L_{M \bdryconnsum N} \rel \omega_{M \bdryconnsum N}) \]
  is the map of dg Lie algebras given by $(\phi, \psi) \mapsto \tilde \phi + \tilde \psi$, where $\tilde \phi$ is the unique derivation of $L_{M \bdryconnsum N}$ that extends $\phi$ and vanishes on $L_N$, and analogously for $\tilde \psi$.
  The group $\arithEaut[\bdry](M)$ acts on $\Deru(L_M \rel \omega_M)$ via $\bar s_M$, and analogously for $\arithEaut[\bdry](N)$ and $\arithEaut[\bdry](M \bdryconnsum N)$.
\end{theorem}

\begin{proof}
  We write $\theta \defeq \bdry M \union \bdry N \subset M \bdryconnsum N$ and begin with some technical considerations.
  There is a commutative diagram of simplicial sets
  \begin{equation} \label{eq:Sing_bryconnsum}
    \begin{tikzcd}
      \Sing(\bdry M) \cop_{\Sing(\Disk {n-1})} \Sing(\bdry N) \rar{\eq}[swap]{a} \dar & \Sing(\theta) \dar \\
      \Sing(M) \cop_{\Sing(\Disk {n-1})} \Sing(N) \rar{\eq}[swap]{f} & \Sing(M \bdryconnsum N)
    \end{tikzcd}
  \end{equation}
  where the two horizontal maps are weak equivalences (see e.g.\ \cite[Proposition~13.5.5]{Hir}).
  Using \cref{lemma:projective_cofibration_lincat}, an elementary verification shows that the horizontal maps are a projective cofibration in $\sSet^{\lincat 1}$.
  Writing $A \to X$ for the left-hand column above, we obtain the following commutative diagram in the homotopy category of simplicial monoids
  \[
  \begin{tikzcd}
    \aut[\bdry](M) \times \aut[\bdry](N) \ar{rr}{\eq}[swap]{\sigma \times \sigma} \ar{ddd} \drar &[-20] &[-135] \aut[\Sing(\bdry M)] \bigl( \Sing(M) \bigr) \times \aut[\Sing(\bdry N)] \bigl( \Sing(N) \bigr) \dar \ar[squiggly, bend left, end anchor = north east, start anchor = -10]{dddl} \\
    & \aut[A](f) \rar{\eq} \dar & \aut[A](X) \\
    & \aut[A] \bigl( \Sing(M \bdryconnsum N) \bigr) & \\
    \aut[\theta](M \bdryconnsum N) \rar{\eq}[swap]{\sigma} \dar & \aut[\Sing(\theta)] \bigl( \Sing(M \bdryconnsum N) \bigr) \dar \uar[swap]{\eq} & \\
    \aut[\bdry](M \bdryconnsum N) \rar{\eq}[swap]{\sigma} & \aut[\Sing ( \bdry(M \bdryconnsum N) )] \bigl( \Sing(M \bdryconnsum N) \bigr) &
  \end{tikzcd}
  \]
  where $\sigma$ is the weak equivalence of \cref{lemma:Sing_map} and the right-hand zig-zag is the one of \cref{lemma:aut_pushout_zig-zag}.
  Using this diagram, we will identify its left-hand and right-hand sides for the rest of this proof.
  
  Since $M \bdryconnsum N \eq M \vee N$, \cref{thm:model_natural} provides a diagram relating the two maps
  \begin{gather*}
    \B \aut[\bdry](M) \times \B \aut[\bdry](N)  \longto  \B \aut[\theta](M \bdryconnsum N) \\
    \Deru(\widetilde L_M \rel \beta_M) \times \Deru(\widetilde L_N \rel \beta_N)  \xlongto{g}  \Der[\ol U] \bigl( \widetilde L_M \cop \widetilde L_N \rel \freelie(\beta_M, \beta_N) \bigr)
  \end{gather*}
  where $\ol U \subseteq \aAut[\freelie(\beta_M, \beta_N)](\widetilde L_M \cop \widetilde L_N)$ is the image of the product of the unipotent radicals of $\aAut[\beta_M](\widetilde L_M)$ and $\aAut[\beta_N](\widetilde L_N)$ (note that the image of $\aAut[\beta_M](\widetilde L_M) \times \aAut[\beta_N](\widetilde L_N)$ is contained in the normalizer of $\ol U$).
  Now we consider the diagram of dg Lie algebras
  \[
  \begin{tikzcd}
    \freelie(\beta_M, \beta_N) \rar & \widetilde L_M \cop \widetilde L_N \\
    \freelie(\beta_{M \bdryconnsum N}) \uar{a} \rar & \widetilde L_{M \bdryconnsum N} \uar[swap]{m}
  \end{tikzcd}
  \]
  where $m$ restricts to the identity on $L_M$ and $L_N$, it sends $\gamma_{M \bdryconnsum N}$ to $\gamma_M + \gamma_N$, and both $a$ and $m$ send $\beta_{M \bdryconnsum N}$ to $\beta_M + \beta_N$.
  Note that $a$ models the map $\bdry(M \bdryconnsum N) \to \theta$ (i.e.\ the pinch map $\Sphere{n-1} \to \Sphere{n-1} \vee \Sphere{n-1}$).
  Applying \cref{thm:model_forget} to this situation, we obtain a diagram that relates the forgetful map
  \[ \B \aut[\theta](M \bdryconnsum N)  \longto  \B \aut[\bdry](M \bdryconnsum N) \]
  to the right-hand zig-zag in the following diagram of dg Lie algebras
  \begin{equation} \label{eq:Deru_comparison}
    \begin{tikzcd}
      \Deru(\widetilde L_M \rel \beta_M) \times \Deru(\widetilde L_N \rel \beta_N) \rar{g} & \Der[\ol U] \bigl( \widetilde L_M \cop \widetilde L_N \rel \freelie(\beta_M, \beta_N) \bigr) \\
      \Deru(L_M \rel \omega_M) \times \Deru(L_N \rel \omega_N) \uar{\Xi \times \Xi}[swap]{\eq} \dar[swap]{g} \rar[dashed] & \Der[\ol U,\mathrm u](m \rel k) \uar \dar \\
      \Deru(L_{M \bdryconnsum N} \rel \omega_{M \bdryconnsum N}) \rar{\eq}[swap]{\Xi} & \Deru(\widetilde L_{M \bdryconnsum N} \rel \beta_{M \bdryconnsum N})
    \end{tikzcd}
  \end{equation}
  where a unique dashed map making both squares commute exists as indicated since $m \after \Xi( g(\phi, \psi) ) = g(\Xi(\phi), \Xi(\psi)) \after m$ for all $\phi$ and $\psi$.
  Assembling everything and proceeding as in the proof of \cref{cor:Baut_eq}, we obtain the desired diagram.
\end{proof}

\begin{corollary} \label{cor:manifolds_gluing_forms}
  In the situation of \cref{thm:manifolds_gluing}, let $P$, $Q$, and $K$ be representations in rational vector spaces of $\arithEaut[\bdry](M)$, $\arithEaut[\bdry](N)$, and $\arithEaut[\bdry](M \bdryconnsum N)$, respectively, and let $\mu \colon K \to P \tensor Q$ be a map that is equivariant with respect to the map $\ol \nu_\QQ \colon \arithEaut[\bdry](M) \times \arithEaut[\bdry](N) \to \arithEaut[\bdry](M \bdryconnsum N)$.
  Then the following diagram commutes in the homotopy category of cochain complexes
  \[
  \begin{tikzcd}
    &[-200] \Forms * \bigl( \B \aut[\bdry](M) \times \B \aut[\bdry](N); P \tensor Q \bigr) &[-200] \\
    {\Forms * \bigl( \B \aut[\bdry](M); P \bigr)} \tensor {\Forms * \bigl( \B \aut[\bdry](N); Q \bigr)} \urar[bend left, start anchor = north, end anchor = -178]{\eq} \dar[squiggly][swap]{\eq} & & \Forms * \bigl( \B \aut[\bdry](M \bdryconnsum N); K \bigr) \dar[squiggly]{\eq} \ular[bend right, start anchor = north, end anchor = -2] \\
    \vertbin {\Forms * \Bigl( \B \arithEaut[\bdry](M); \CEcochains * \bigl( \Deru(L_M \rel \omega_M); P \bigr) \Bigr)} \tensor {\Forms * \Bigl( \B \arithEaut[\bdry](N); \CEcochains * \bigl( \Deru(L_N \rel \omega_N); Q \bigr) \Bigr)} \drar[bend right, start anchor = south, end anchor = 171][swap, near start]{\eq} & & \Forms * \Bigl( \B \arithEaut[\bdry](M \bdryconnsum N); \CEcochains * \bigl( \Deru(L_{M \bdryconnsum N} \rel \omega_{M \bdryconnsum N}); K \bigr) \Bigr) \dlar[bend left, start anchor = south, end anchor = 8] \\
    & \Forms * \Bigl( \B \arithEaut[\bdry](M) \times \B \arithEaut[\bdry](N); \CEcochains * \bigl( \Deru(L_M \rel \omega_M) \times \Deru(L_N \rel \omega_N); P \tensor Q \bigr) \Bigr) &
  \end{tikzcd}
  \]
  where the indicated maps are quasi-isomorphisms and the vertical maps are obtained by combining \cref{cor:forms_eq,thm:manifolds}.
  Moreover, this diagram can be lifted to a commutative diagram in the homotopy category of lax symmetric monoidal functors from the category of tuples $(P, Q, K, \mu)$ to cochain complexes (with monoidal natural transformations as morphisms, and the pointwise weak equivalences).
\end{corollary}

\begin{proof}
  This follows from \cref{thm:manifolds_gluing} as in the proof of \cref{cor:forms_eq_natural}.
\end{proof}

\begin{corollary} \label{cor:manifolds_gluing_coho}
  In the situation of \cref{thm:manifolds_gluing}, let $P$, $Q$, and $K$ be representations in rational vector spaces of $\arithEaut[\bdry](M)$, $\arithEaut[\bdry](N)$, and $\arithEaut[\bdry](M \bdryconnsum N)$, respectively, and let $\mu \colon K \to P \tensor Q$ be a map that is equivariant with respect to the map $\ol \nu_\QQ \colon \arithEaut[\bdry](M) \times \arithEaut[\bdry](N) \to \arithEaut[\bdry](M \bdryconnsum N)$.
  Then the following diagram of graded vector spaces commutes
  \[
  \begin{tikzcd}
    &[-200] \Coho * \bigl( \B \aut[\bdry](M) \times \B \aut[\bdry](N) ; P \tensor Q \bigr) &[-201] \\
    \Coho * \bigl( \B \aut[\bdry](M); P \bigr) \tensor \Coho * \bigl( \B \aut[\bdry](N); Q \bigr) \urar[bend left, end anchor = 182]{\iso} \dar[swap]{\iso} & & \Coho * \bigl( \B \aut[\bdry](M \bdryconnsum N); K \bigr) \dar{\iso} \ular[bend right, end anchor = -2] \\
    \vertbin {\Coho * \Bigl( \B \arithEaut[\bdry](M); \CEcoho * \bigl( \Deru(L_M \rel \omega_M); P \bigr) \Bigr)} \tensor {\Coho * \Bigl( \B \arithEaut[\bdry](N); \CEcoho * \bigl( \Deru(L_N \rel \omega_N); Q \bigr) \Bigr)} \drar[bend right, start anchor = -90, end anchor = 171][swap, near start]{\iso} & & \Coho * \Bigl( \B \arithEaut[\bdry](M \bdryconnsum N); \CEcoho * \bigl( \Deru(L_{M \bdryconnsum N} \rel \omega_{M \bdryconnsum N}); K \bigr) \Bigr) \dlar[bend left, start anchor = -90, end anchor = 8]  \\
    & \Coho * \Bigl( \B \arithEaut[\bdry](M) \times \B \arithEaut[\bdry](N); \CEcoho * \bigl( \Deru(L_M \rel \omega_M) \times \Deru(L_N \rel \omega_N); P \tensor Q \bigr) \Bigr) &
  \end{tikzcd}
  \]
  where the indicated maps are isomorphisms and the vertical maps are obtained by combining \cref{cor:cohomology,thm:manifolds}.
  Moreover, this diagram consists of monoidal natural transformations between lax symmetric monoidal functors from the category of tuples $(P, Q, K, \mu)$ to graded vector spaces.
\end{corollary}

\begin{proof}
  This follows from \cref{cor:manifolds_gluing_forms} as in the proof of \cref{cor:cohomology}.
\end{proof}

We now proceed to prove that the model for $\B \BlDiff[\bdry](M)$ of \cref{thm:manifolds_block} is compatible with boundary connected sums as well.

\begin{theorem} \label{thm:manifolds_block_gluing}
  Let $M$ and $N$ be two $n$-dimensional manifolds as in \cref{thm:manifolds_block} and let $L_M \defeq ( \freelie V_M, \delta_M )$ and $L_N \defeq ( \freelie V_N, \delta_N )$ be models as in \cref{prop:Stasheff}.
  Denote by $M \bdryconnsum N$ the boundary connected sum of $M$ and $N$, and set $L_{M \bdryconnsum N} \defeq L_M \cop L_N$.
  Assume that the canonical map of algebraic groups
  \[ \ol \gamma \colon \aAut[\omega_M](L_M)_{\pont} \times \aAut[\omega_N](L_N)_{\pont}  \longto  \aAut[\omega_{M \bdryconnsum N}](L_{M \bdryconnsum N})_{\pont} \]
  induces a map $\ol \nu$ on maximal reductive quotients.
  Lastly choose sections $\bar s_M$, $\bar s_N$, and $\bar s_{M \bdryconnsum N}$ as in \cref{thm:manifolds_block} such that $\bar s_{M \bdryconnsum N} \after \ol \nu = \ol \gamma \after (\bar s_M \times \bar s_N)$ (this exists by \cref{lemma:compatible_Levi}).
  Then there is a commutative diagram in the homotopy category of topological spaces
  \[
  \begin{tikzcd}
    \B \BlDiff[\bdry](M) \times \B \BlDiff[\bdry](N) \rar \dar[squiggly][swap]{\req} & \B \BlDiff[\bdry](M \bdryconnsum N) \dar[squiggly]{\req} \\
    \gr[\big] { \hcoinv { \MCb \bigl( \tilde{\lie g}_{\bdry}(M) \bigr) } { \BlGamma{\bdry}(M) } } \times \gr[\big] { \hcoinv { \MCb \bigl( \tilde{\lie g}_{\bdry}(N) \bigr) } { \BlGamma{\bdry}(N) } } \rar{\tilde g_*} & \gr[\big] { \hcoinv { \MCb \bigl( \tilde{\lie g}_{\bdry}(M \bdryconnsum N) \bigr) } { \BlGamma{\bdry}(M \bdryconnsum N) } }
  \end{tikzcd}
  \]
  where the vertical maps are obtained by combining \cref{thm:block_diff,thm:manifolds_block}, and
  \[ \tilde g \colon \tilde{\lie g}_{\bdry}(M) \times \tilde{\lie g}_{\bdry}(N)  \longto  \tilde{\lie g}_{\bdry}(M \bdryconnsum N) \]
  is the map given by \cref{lemma:pushout_indecomposables} on the $\Hom$-factors and by the map $g$ of \cref{thm:manifolds_gluing} on the $\Der$-factors.
  The group $\BlGamma{\bdry}(M)$ acts on $\tilde{\lie g}_{\bdry}(M)$ via $\bar s_M$, and analogously for $\BlGamma{\bdry}(M)$ and $\BlGamma{\bdry}(M \bdryconnsum N)$.
\end{theorem}

\begin{proof}
  Choosing appropriate disks and a classifying map $\sttang{M \bdryconnsum N} \colon M \bdryconnsum N \to \B \SO$ of the oriented stable tangent bundle of $M \bdryconnsum N$, \cref{prop:block_diff_gluing} yields a commutative diagram
  \[
  \begin{tikzcd}
    \B \BlDiff[\bdry](M) \times \B \BlDiff[\bdry](N) \rar[squiggly]{\req} \dar & \gr[\big] { \B \bdlaut{\bdry}{\bdry_0}{\B \O}(\sttang{M})_{G_M} } \times \gr[\big] { \B \bdlaut{\bdry}{\bdry_0}{\B \O}(\sttang{N})_{G_N} } \dar \\
    \B \BlDiff[\bdry](M \bdryconnsum N) \rar[squiggly]{\req} & \gr[\big] { \B \bdlaut{\bdry}{\bdry_0}{\B \O}(\sttang{M \bdryconnsum N})_{G_{M \bdryconnsum N}} }
  \end{tikzcd}
  \]
  in the homotopy category of topological spaces.
  Let $H \defeq \bdry_0 M \union \bdry_0 N \subset M \bdryconnsum N$ and note that as in \eqref{eq:Sing_bryconnsum} there is a weak equivalence
  \[ \Sing(\bdry_0 M) \cop_{\Sing(\Disk {n-1})} \Sing(\bdry_0 N)  \xlongto{\eq}  \Sing(H) \]
  which we can use to replace the composite
  \[ \B \bdlaut{\bdry}{\bdry_0}{\B\SO}(\sttang M) \times \B \bdlaut{\bdry}{\bdry_0}{\B\SO}(\sttang M)  \longto  \B \bdlaut{\theta}{H}{\B\SO}(\sttang {M \bdryconnsum N})  \longto  \B \bdlaut{\bdry}{\bdry_0}{\B\SO}(\sttang {M \bdryconnsum N}) \]
  by the corresponding maps with $\Sing(\blank)$ applied everywhere, as in the proof of \cref{thm:manifolds_gluing}.
  Factoring $\Sing(\sttang {M \bdryconnsum N})$ as a cofibration followed by a trivial fibration, as in \cref{rem:Pi}, we can then apply \cref{thm:model_natural,thm:model_forget} as in the proof of \cref{thm:manifolds_gluing}.
  Lastly note that the following diagram commutes and is compatible with the corresponding diagram \eqref{eq:Deru_comparison}
  \[
  \begin{tikzcd}
    \vertbin {\trunc 0 { \Hom \bigl( \shift \indec(\widetilde L_M), \hg * (\SO) \tensor \QQ \bigr) }} \times {\trunc 0 { \Hom \bigl( \shift \indec(\widetilde L_N), \hg * (\SO) \tensor \QQ \bigr) }} \rar & \trunc 0 { \Hom \bigl( \shift \indec(\widetilde L_M \cop \widetilde L_N), \hg * (\SO) \tensor \QQ \bigr) } \\
    \vertbin {\trunc 0 { \Hom \bigl( \shift V_M, \hg * (\SO) \tensor \QQ \bigr) }} \times {\trunc 0 { \Hom \bigl( \shift V_N, \hg * (\SO) \tensor \QQ \bigr) }} \uar{q^* \times q^*}[swap]{\eq} \dar \rar & \trunc 0 { \Hom \bigl( \shift \indec(\widetilde L_M \cop \widetilde L_N), \hg * (\SO) \tensor \QQ \bigr) } \uar[equal] \dar \\
    \trunc 0 { \Hom \bigl( \shift (V_M \dirsum V_N), \hg * (\SO) \tensor \QQ \bigr) } \rar{\eq}[swap]{q^*} & \trunc 0 { \Hom \bigl( \shift \indec(\widetilde L_{M \bdryconnsum N}), \hg * (\SO) \tensor \QQ \bigr) }
  \end{tikzcd}
  \]
  where $q$ is the map from the proof of \cref{thm:manifolds_block}.
  This completes the proof.
\end{proof}

\begin{corollary} \label{cor:manifolds_block_gluing_forms}
  In the situation of \cref{thm:manifolds_block_gluing}, let $P$, $Q$, and $K$ be representations in rational vector spaces of $\redEBlaut[\bdry](M)$, $\redEBlaut[\bdry](N)$, and $\redEBlaut[\bdry](M \bdryconnsum N)$, respectively, and let $\mu \colon K \to P \tensor Q$ be a map that is equivariant with respect to the map $\ol \nu_\QQ \colon \redEBlaut[\bdry](M) \times \redEBlaut[\bdry](N) \to \redEBlaut[\bdry](M \bdryconnsum N)$.
  Then the following diagram commutes in the homotopy category of cochain complexes
  \[
  \begin{tikzcd}
    &[-140] \Forms * \bigl( \B \BlDiff[\bdry](M) \times \B \BlDiff[\bdry](N) ; P \tensor Q \bigr) &[-130] \\
    \Forms * \bigl( \B \BlDiff[\bdry](M); P \bigr) \tensor \Forms * \bigl( \B \BlDiff[\bdry](N); Q \bigr) \urar[bend left, end anchor = 182]{\eq} \dar[swap]{\eq} & & \Forms * \bigl( \B \BlDiff[\bdry](M \bdryconnsum N); K \bigr) \dar{\eq} \ular[bend right, end anchor = -2] \\
    \vertbin {\Forms * \Bigl( \B \redEBlaut[\bdry](M); \CEcoho * \bigl( \tilde{\lie g}_{\bdry}(M); P \bigr) \Bigr)} \tensor {\Forms * \Bigl( \B \redEBlaut[\bdry](N); \CEcoho * \bigl( \tilde{\lie g}_{\bdry}(N); Q \bigr) \Bigr)} \drar[bend right, start anchor = -90, end anchor = 172][swap, near start]{\eq} & & \Forms * \Bigl( \B \redEBlaut[\bdry](M \bdryconnsum N); \CEcoho * \bigl( \tilde{\lie g}_{\bdry}(M \bdryconnsum N); K \bigr) \Bigr) \dlar[bend left, start anchor = -90, end anchor = 8] \\
    & \Forms * \Bigl( \B \redEBlaut[\bdry](M) \times \B \redEBlaut[\bdry](N); \CEcoho * \bigl( \tilde{\lie g}_{\bdry}(M) \times \tilde{\lie g}_{\bdry}(N); P \tensor Q \bigr) \Bigr) &
  \end{tikzcd}
  \]
  where the indicated maps are quasi-isomorphisms and the vertical maps are obtained by combining \cref{cor:forms_block_diff,thm:manifolds_block}.
  Moreover, this diagram can be lifted to a commutative diagram in the homotopy category of lax symmetric monoidal functors from the category of tuples $(P, Q, K, \mu)$ to cochain complexes (with monoidal natural transformations as morphisms, and the pointwise weak equivalences).
\end{corollary}

\begin{proof}
  This follows from \cref{thm:manifolds_block_gluing} as in the proof of \cref{cor:forms_eq_natural}.
\end{proof}

\begin{corollary} \label{cor:manifolds_block_gluing_coho}
  In the situation of \cref{thm:manifolds_block_gluing}, let $P$, $Q$, and $K$ be representations in rational vector spaces of $\redEBlaut[\bdry](M)$, $\redEBlaut[\bdry](N)$, and $\redEBlaut[\bdry](M \bdryconnsum N)$, respectively, and let $\mu \colon K \to P \tensor Q$ be a map that is equivariant with respect to the map $\ol \nu_\QQ \colon \redEBlaut[\bdry](M) \times \redEBlaut[\bdry](N) \to \redEBlaut[\bdry](M \bdryconnsum N)$.
  Then the following diagram of graded vector spaces commutes
  \[
  \begin{tikzcd}
    &[-140] \Coho * \bigl( \B \BlDiff[\bdry](M) \times \B \BlDiff[\bdry](N) ; P \tensor Q \bigr) &[-130] \\
    \Coho * \bigl( \B \BlDiff[\bdry](M); P \bigr) \tensor \Coho * \bigl( \B \BlDiff[\bdry](N); Q \bigr) \urar[bend left, end anchor = 182]{\iso} \dar[swap]{\iso} & & \Coho * \bigl( \B \BlDiff[\bdry](M \bdryconnsum N); K \bigr) \dar{\iso} \ular[bend right, end anchor = -2] \\
    \vertbin {\Coho * \Bigl( \B \redEBlaut[\bdry](M); \CEcoho * \bigl( \tilde{\lie g}_{\bdry}(M); P \bigr) \Bigr)} \tensor {\Coho * \Bigl( \B \redEBlaut[\bdry](N); \CEcoho * \bigl( \tilde{\lie g}_{\bdry}(N); Q \bigr) \Bigr)} \drar[bend right, start anchor = -90, end anchor = 172][swap, near start]{\iso} & & \Coho * \Bigl( \B \redEBlaut[\bdry](M \bdryconnsum N); \CEcoho * \bigl( \tilde{\lie g}_{\bdry}(M \bdryconnsum N); K \bigr) \Bigr) \dlar[bend left, start anchor = -90, end anchor = 8] \\
    & \Coho * \Bigl( \B \redEBlaut[\bdry](M) \times \B \redEBlaut[\bdry](N); \CEcoho * \bigl( \tilde{\lie g}_{\bdry}(M) \times \tilde{\lie g}_{\bdry}(N); P \tensor Q \bigr) \Bigr) &
  \end{tikzcd}
  \]
  where the indicated maps are isomorphisms and the vertical maps are obtained by combining \cref{cor:coho_block_diff,thm:manifolds_block}.
  Moreover, this diagram consists of monoidal natural transformations between lax symmetric monoidal functors from the category of tuples $(P, Q, K, \mu)$ to graded vector spaces.
\end{corollary}

\begin{proof}
  This follows from \cref{cor:manifolds_block_gluing_forms} as in the proof of \cref{cor:cohomology}.
\end{proof}

\begin{appendices}
  \section{Model categorical technicalities}

In this appendix, we collect various model categorical results we need throughout the article proper.
The main part is \cref{sec:zig-zag}, where we prove that the assignment $(A \to X) \mapsto \aut[A](X)$ extends to a functor from (a certain subcategory of) the arrow category to the homotopy category of simplicial monoids.

We begin by recalling the following basic lemma, which we will use frequently, as well as a direct consequence.

\begin{lemma} \label{lemma:cofibration_pushouts}
  In a model category, let
  \[
  \begin{tikzcd}
    X \dar[swap]{f} & A \lar \rar \dar{a} & Y \dar{g} \\
    X' & A' \lar \rar & Y'
  \end{tikzcd}
  \]
  be a commutative diagram.
  If $g$ and the induced map $X \cop_{A} A' \to X'$ are both cofibrations, then the induced map $X \cop_{A} Y \to X' \cop_{A'} Y'$ is a cofibration as well.
  This statement also holds when replacing ``cofibration'' with ``trivial cofibration''.
\end{lemma}

\begin{proof}
  The case of cofibrations is \cite[Lemma~7.2.15]{Hir}; the proof for trivial cofibrations is analogous.
\end{proof}

\begin{lemma} \label{lemma:cofibration_pushouts_special}
  In a model category, let
  \[ X \xlongfrom{i} B \longfrom A \longto C \xlongto{j} Y \]
  be maps such that $i$ and $j$ are cofibrations.
  Then the induced map $B \cop_{A} C \to X \cop_{A} Y$ is a cofibration.
\end{lemma}

\begin{proof}
  This follows from \cref{lemma:cofibration_pushouts} since the induced map $B \cop_{A} A \to X$ is isomorphic to $i$.
\end{proof}

\subsection{The projective model structure}

In the following subsections, we will need model structures on certain simple functor categories.
In this subsection, we recall the projective model structure and some of its properties.

\begin{definition}
  Let $\cat M$ be a model category and $\cat I$ a small category.
  The \emph{projective model structure} on the functor category $\cat M^{\cat I}$ is, if it exists, the unique model structure with weak equivalences the pointwise weak equivalences and fibrations the pointwise fibrations.
  A \emph{projective (co)fibration} is a (co)fibration in the projective model structure.
\end{definition}

\begin{remark}
  Let $\cat M$ be a simplicial model category and $\cat I$ a small category.
  If the projective model structure on the functor category $\cat M^{\cat I}$ exists, then it is again a simplicial model category when equipped with its usual simplicial structure (see e.g.\ \cite[Remark~A.3.3.4]{LurHTT}).
\end{remark}

\begin{lemma} \label{lemma:projective_restriction}
  Let $\cat M$ be a model category, and let $r \colon \cat I \to \cat J$ be a right adjoint functor.
  Then the restriction functor $r^* \colon \cat M^{\cat J} \to \cat M^{\cat I}$ preserves cofibrations and trivial cofibrations when both functor categories are equipped with the projective model structures (in particular we assume that they exist).
\end{lemma}

\begin{proof}
  Let $l \colon \cat J \to \cat I$ be the left adjoint of $r$.
  Then the restriction functor $r^*$ is left adjoint to the restriction functor $l^*$.
  Since $l^*$ preserves pointwise fibrations and pointwise weak equivalences, the adjunction $r^* \dashv l^*$ is a Quillen adjunction (see e.g.\ \cite[Proposition~8.5.3]{Hir}).
  Hence the left adjoint $r^*$ preserves cofibrations and trivial cofibrations.
\end{proof}

The following lemma provides a criterion for the existence of the projective model structure for certain indexing categories $\cat I$.
In particular it applies when $\cat I$ is a finite product of finite linearly ordered sets.
To state it, we need the following definition.

\newcommand{\Latch}[1]{{\mathrm L_{#1}}}

\begin{definition}
  Let $\cat C$ be a cocomplete category, $\cat I$ a small category, and $i \in \cat I$ an object.
  We write $\cat I_i \defeq (\cat I \comma i) \setminus \set {\id[i]}$, i.e.\ the category of objects over $i$ without its terminal object; it comes equipped with a forgetful functor $\pr \colon \cat I_i \to \cat I$ and a natural transformation $\pr \to \const[i]$.
  Then the $i$-th \emph{latching space} functor is
  \[ \Latch{i}  \colon  \cat C^{\cat I}  \xlongto{\pr^*}  \cat C^{\cat I_i}  \xlongto{\colim{}}  \cat C \]
  which comes equipped with a natural transformation $\Latch{i} \to \ev[i]$ to the evaluation at $i$.
\end{definition}

\begin{lemma} \label{lemma:projective_cofibration}
  Let $\cat M$ be a model category, and let $\cat I$ be a small category.
  Assume that $\cat I$ admits a functor to an ordinal, considered as a linearly ordered set, such that every non-identity morphism is mapped to a non-identity morphism.
  Then the projective model structure on $\cat M^{\cat I}$ exists.
  Moreover, a natural transformation $F \to G$ of functors $\cat I \to \cat M$ is a projective cofibration if and only if the induced map $F(i) \cop_{\Latch i F} \Latch i G \to G(i)$ is a cofibration for all $i \in \cat I$.
  This statement also holds when replacing ``cofibration'' with ``trivial cofibration''.
  Lastly, every projective cofibration is also a pointwise cofibration.
\end{lemma}

\begin{proof}
  This is \cite[Theorem~5.1.3 and Remark~5.1.7]{Hov}.
\end{proof}

We now spell out some consequences of the preceding lemma for particularly simple indexing categories.

\begin{lemma} \label{lemma:projective_cofibration_lincat}
  Let $\cat M$ be a model category, and let $n \in \NN$.
  A natural transformation $F \to G$ of functors $\lincat n \to \cat M$ is a cofibration with respect to the projective model structure on $\cat M^{\lincat n}$ if and only if $F(0) \to G(0)$ is a cofibration and, for all $0 \le i < n$, the induced map $G(i) \cop_{F(i)} F(i+1) \to G(i + 1)$ is a cofibration.
  This statement also holds when replacing ``cofibration'' with ``trivial cofibration''.
\end{lemma}

\begin{proof}
  This follows from \cref{lemma:projective_cofibration} since the latching space $L_i F$ of a functor $F \colon \lincat n \to \cat M$ at an object $i \in \lincat n$ is isomorphic to the initial object of $\cat M$ when $i = 0$ and to $F(i - 1)$ otherwise.
\end{proof}

\begin{remark} \label{rem:projective_cofibration}
  In a model category $\cat M$, let
  \[
  \begin{tikzcd}
  A \rar{f} \dar[tail][swap]{a} & B \dar{b} \\
  C \rar{g} & D
  \end{tikzcd}
  \]
  be a commutative diagram such that $a$ is a cofibration.
  Then it follows from \cref{lemma:projective_cofibration_lincat} that, in the arrow category $\arcat {\cat M}$ equipped with the projective model structure, if the morphism $(f, g) \colon a \to b$ is a cofibration, then the morphism $(a, b) \colon f \to g$ is one as well.
  In that case both $g$ and $b$ are cofibrations, too.
  This also holds when replacing ``cofibration'' with ``trivial cofibration'' everywhere.
\end{remark}

\begin{lemma} \label{lemma:projective_restriction_pushout}
  Let $\cat M$ be a model category.
  Consider the functor $\cat M^{\lincat 1 \times \lincat 1} \to \arcat {\cat M}$ defined on objects by
  \[ F  \longmapsto  \bigl( F(1, 0) \cop_{F(0,0)} F(0, 1)  \to  F(1, 1) \bigr) \]
  and in the obvious way on morphism.
  This functor preserves both cofibrations and trivial cofibrations with respect to the projective model structures.
\end{lemma}

\begin{proof}
  First note that $\Latch{(1,1)} F \iso F(1, 0) \cop_{F(0,0)} F(0, 1)$.
  Then it is, by \cref{lemma:projective_cofibration_lincat}, enough to show that, given a (trivial) cofibration $F \to G$ in the projective model structure of $\cat M^{\lincat 1 \times \lincat 1}$, both of the induced maps
  \begin{gather*}
    \Latch{(1, 1)} F  \longto  \Latch{(1, 1)} G \\
    F(1, 1) \cop_{\Latch{(1, 1)} F} \Latch{(1, 1)} G  \longto  G(1, 1)
  \end{gather*}
  are (trivial) cofibrations.
  For the second one, this follows from \cref{lemma:projective_cofibration}.
  For the first one, note that $\Latch{(1, 0)} F \iso F(0, 0)$, so that \cref{lemma:projective_cofibration} implies that $F(1, 0) \cop_{F(0, 0)} G(0, 0) \to G(1, 0)$ is a (trivial) cofibration.
  Then \cref{lemma:cofibration_pushouts} implies the desired statement.
\end{proof}

\subsection{Relative mapping complexes}

In this subsection, we prove various basic properties of the relative mapping complexes $\map[A](X, Y)$.
Most of them should be considered to be analogues of the properties of the simplicial model structure on an undercategory $A \comma \cat M$ when we also allow the object $A$ to vary.
For ease of reference, we begin by recalling some of the basic properties of the simplicial model category $A \comma \cat M$.

\begin{lemma} \label{lemma:rel_map}
  Let $\cat M$ be a simplicial model category, $A \in \cat M$ an object, and $x \colon A \to X$, $y \colon A \to Y$, and $z \colon A \to Z$ maps in $\cat M$.
  \begin{itemize} 
    \item
    Let $f \colon X \to Y$ be a map under $A$.
    If $Z$ is fibrant and $f$ is a cofibration, then the map
    \[ \map[A](Y, Z)  \xlongto{f^*}  \map[A](X, Z) \]
    is a fibration of simplicial sets; if $f$ is additionally a weak equivalence, then so is $f^*$.
    If $Z$ is fibrant, $x$ and $y$ are cofibrations, and $f$ is a weak equivalence, then $f^*$ is a weak equivalence.
    
    \item
    Let $g \colon Y \to Z$ be a map under $A$.
    If $x$ is a cofibration and $g$ is a fibration, then the map
    \[ \map[A](X, Y)  \xlongto{g_*}  \map[A](X, Z) \]
    is a fibration of simplicial sets; if $g$ is additionally a weak equivalence, then so is $g_*$.
    If $x$ is a cofibration, $Y$ and $Z$ are fibrant, and $g$ is a weak equivalence, then $g_*$ is a weak equivalence.
  \end{itemize}
\end{lemma}

\begin{proof}
  This follows from \cref{lemma:slice_sm} and \cite[§9.3]{Hir}.
\end{proof}

\begin{lemma} \label{lemma:mapeq_we}
  In a simplicial model category, let
  \[
  \begin{tikzcd}
  A \rar{i} \dar[swap]{a} & Y \dar{f} \\
  A' \rar{i'} & Y'
  \end{tikzcd}
  \]
  be a commutative square, and $A \to X$ and $A' \to Z$ two maps.
  Assume that $f$ is a weak equivalence.
  \begin{itemize}
    \item
    If the map
    \[ f_* \colon \map[A](X, Y) \to \map[A](X, Y') \]
    is a weak equivalence, then so is its restriction $f_*^\eq \colon \mapeq[A](X, Y) \to \mapeq[A](X, Y')$.
    \item
    If the map
    \[ f^* \colon \map[A'](Y', Z) \to \map[A](Y, Z) \]
    is a weak equivalence, then so is its restriction $f^*_\eq \colon \mapeq[A'](Y', Z) \to \mapeq[A](Y, Z)$.
  \end{itemize}
\end{lemma}

\begin{proof}
  We prove the second statement, the proof of the first is similar.
  First note that $f^*$ indeed maps $\mapeq[A'](Y', Z)$ to $\mapeq[A](Y, Z)$, and that $f^*_\eq$ is a weak equivalence onto the connected components it hits.
  So it is enough to show that $f^*_\eq$ is surjective on $\hg 0$.
  Let $\phi \in \mapeq[A](Y, Z)$.
  By assumption, there exists a map $\psi \in \map[A'](Y', Z)$ such that $\psi \after f \eq_A \phi$.
  By \cref{lemma:sh_we}, since $\phi$ is a weak equivalence, so is $\psi \after f$.
  Then, since $f$ is a weak equivalence, so is $\psi$.
  Thus $\psi \in \mapeq[A'](Y', Z)$ and we are done.
\end{proof}

\begin{lemma} \label{lemma:rel_forget}
  In a simplicial model category, let
  \[
  \begin{tikzcd}
  A' \rar{i'} \dar[swap]{a} & X' \dar{f} \\
  A \rar{i} & X
  \end{tikzcd}
  \]
  be a commutative square and $j \colon A \to Y$ a map.
  Assume that $a$ and $f$ are weak equivalences, that $i$ and $i'$ are cofibrations, and that $Y$ is fibrant.
  If $A$ and $A'$ are cofibrant, or if $a$ and $f$ are cofibrations, then the two maps
  \[ f^* \colon \map[A](X, Y)  \longto  \map[A'](X', Y)  \qquad \text{and} \qquad  f^* \colon \mapeq[A](X, Y)  \longto  \mapeq[A'](X', Y) \]
  are weak equivalences of simplicial sets.
\end{lemma}

\begin{proof}
  There is a commutative diagram
  \[
  \begin{tikzcd}
  \map[A](X, Y) \rar \dar & \map(X, Y) \rar{i^*} \dar[swap]{f^*} & \map(A, Y) \dar{a^*} \\
  \map[A'](X', Y) \rar & \map(X', Y) \rar{(i')^*} & \map(A', Y)
  \end{tikzcd}
  \]
  in which the rows are homotopy fiber sequences (based at $j$ and $ja$, respectively) since $i$ and $i'$ are cofibrations and $Y$ is fibrant.
  The conditions guarantee that $f^*$ and $a^*$ are weak equivalences.
  This implies the claim for the first map, and for the second it follows from \cref{lemma:mapeq_we}.
\end{proof}

\begin{lemma} \label{lemma:SM7}
  In a simplicial model category, let $i \colon A \to B$ be a cofibration and $p \colon X \to Y$ a fibration.
  Then the map
  \[ (i^*, p_*) \colon \map(B, X)  \longto  \map(A, X) \times_{\map(A, Y)} \map(B, Y) \]
  is a fibration of simplicial sets.
  If additionally $i$ or $p$ is a weak equivalence, then so is $(i^*, p_*)$.
  If both $i$ and $p$ are weak equivalences, then it restricts to a trivial fibration
  \[ (i^*, p_*) \colon \mapeq(B, X)  \longto  \mapeq(A, X) \times_{\mapeq(A, Y)} \mapeq(B, Y) \]
  of simplicial sets.
\end{lemma}

\begin{proof}
  The first two statements are part of the definition of a simplicial model category.
  The third statement follows as in the proof of \cref{lemma:mapeq_we}.
\end{proof}

\begin{lemma} \label{lemma:rel_SM7}
  In a simplicial model category $\cat M$, let
  \[
  \begin{tikzcd}
    A \rar{i} \dar[swap]{a} & X \dar{f} \\
    A' \rar{i'} & X'
  \end{tikzcd}
  \]
  be a commutative diagram and $j' \colon A' \to Y$ and $g \colon Y \to Z$ two maps.
  Assume that $g$ is a fibration and that $(a, f) \colon i \to i'$ is a cofibration in the arrow category of $\cat M$ equipped with the projective model structure.
  Then the map
  \[ (f^*, g_*) \colon  \map[A'](X', Y)  \longto  \map[A](X, Y) \times_{\map[A](X, Z)} \map[A'](X', Z) \]
  is a fibration of simplicial sets.
  If additionally $g$ or both $a$ and $f$ are weak equivalences, then so is $(f^*, g_*)$.
  If all three of $g$, $a$, and $f$ are weak equivalences, then it restricts to a trivial fibration
  \[ (f^*, g_*) \colon  \mapeq[A'](X', Y)  \longto  \mapeq[A](X, Y) \times_{\mapeq[A](X, Z)} \mapeq[A'](X', Z) \]
  of simplicial sets.
\end{lemma}

\begin{proof}
  We prove the first two statements simultaneously.
  The desired map is obtained by taking fibers of the vertical maps of the following commutative diagram of simplicial sets
  \[
  \begin{tikzcd}
     \map(X', Y) \rar{(f^*, g_*)} \dar[swap]{(i')^*} & \map(X, Y) \times_{\map(X, Z)} \map(X', Z) \dar{i^* \times_{i^*} (i')^* } \\
     \map(A', Y) \rar{(a^*, g_*)} & \map(A, Y) \times_{\map(A, Z)} \map(A', Z)
  \end{tikzcd}
  \]
  over $j'$ and $(j' \after a, g \after j')$, respectively.
  Hence, by the dual of \cref{lemma:cofibration_pushouts}, it is enough to prove that the induced map from $\map(X', Y)$ to the pullback $P$ of the lower right-hand corner is a fibration (resp.\ trivial fibration).
  Noting that the map
  \[ (\id, g_*) \colon \map(A', Y)  \longto  \map(A', Y) \times_{\map(A', Z)} \map(A', Z) \]
  is an isomorphism, we obtain that $P$ is canonically isomorphic to
  \[ \map(X \cop_{A} A', Y) \times_{\map(X \cop_{A} A', Z)} \map(X', Z) \]
  since pullbacks commute with pullbacks and $\map(\blank, T)$ sends colimits to limits for all objects $T$ of $\cat M$ (see e.g.\ \cite[Proposition~9.2.2]{Hir}).
  By \cref{lemma:projective_cofibration_lincat} and assumption, the induced map $q \colon X \cop_{A} A' \to X'$ is a cofibration.
  Since $g \colon Y \to Z$ is a fibration, this implies by \cref{lemma:SM7} that the induced map $p \colon \map(X', Y) \to P$ is a fibration.
  If $g$ is additionally a weak equivalence, then so is $p$; similarly, when both $a$ and $f$ are weak equivalences, then so is $q$ and hence $p$.
  This concludes the proof of the first two statements; the third follows as in the proof of \cref{lemma:mapeq_we}.
\end{proof}

\begin{lemma} \label{lemma:rel_precompose}
  In a simplicial model category $\cat M$, let
  \[
  \begin{tikzcd}
    A \rar{i} \dar[swap]{a} & X \dar{f} \\
    A' \rar{i'} & X'
  \end{tikzcd}
  \]
  be a commutative diagram and $j' \colon A' \to Y$ a map.
  Assume that $Y$ is fibrant and that $(a, f) \colon i \to i'$ is a cofibration in the arrow category of $\cat M$ equipped with the projective model structure.
  Then the map
  \[ f^* \colon \map[A'](X', Y)  \longto  \map[A](X, Y) \]
  is a fibration of simplicial sets.
  If additionally $a$ and $f$ are weak equivalences, then so is $f^*$.
\end{lemma}

\begin{proof}
  This follows from \cref{lemma:rel_SM7} by taking $g \colon Y \to *$ to be the unique map to the terminal object.
\end{proof}

\subsection{The functor \texorpdfstring{$\aut[A](X)$}{aut\_A(X)}} \label{sec:zig-zag}

The following theorem is the main result of this appendix.

\begin{theorem} \label{thm:Baut_functor}
  Let $\cat M$ be a simplicial model category.
  Then the assignment
  \[ (B \xto{i} A \xto{\iota} X \xto{\xi} K)  \longmapsto  \bigl( \aut[A](X), \map[B](X, K) \bigr) \]
  can be canonically extended to a functor from the subcategory $\cat R$ of $\cat M^{\lincat 3}$
  \begin{itemize}
  	\item
  	with objects those functors $\lincat 3 \to \cat M$ as above such that $i$ and $\iota$ are cofibrations, $B$ and $A$ are cofibrant, and $X$ and $K$ are fibrant,
  	\item
  	and morphisms given by pointwise weak equivalences,
  \end{itemize}
  to the core of the homotopy category of pairs $(G, M)$ of a simplicial monoid $G$ and a right $G$-module $M$ in simplicial sets.
  Moreover, given two morphisms of $\cat R$
  \[ f, g \colon (B \to A \to X \to K) \to (B' \to A' \to X' \to K') \]
  such that $B'$ and $A'$ are fibrant and $f$ and $g$ are simplicially homotopic in $\cat M^{\lincat 3}$, the images of $f$ and $g$ under the functor above agree.
\end{theorem}

We begin by proving various technical lemmas we will need for the proof of the theorem above.

\begin{lemma} \label{lemma:rel_aut_span}
  In a simplicial model category $\cat M$, let
  \[
  \begin{tikzcd}
  A \rar{i} \dar[swap]{a} & X \dar{f} \\
  A' \rar{i'} & X'
  \end{tikzcd}
  \]
  be a commutative diagram.
  We consider the induced zig-zag
  \[ \aut[A](X)  \xlongfrom{p}  \aut[a](f)  \xlongto{p'}  \aut[A'](X') \]
  of simplicial monoids.
  \begin{enumerate}
    \item
    Assume that $A' = A$ and $a = \id[A]$, that $f$ is a trivial cofibration, and that $X'$ is fibrant.
    \begin{enumerate}
      \item
      Then $p$ is a trivial fibration.
      \item
      Additionally assume that $i$ is a cofibration and that $X$ is fibrant.
      Then $p'$ is a weak equivalence.
    \end{enumerate}
    \item
    Assume that $f$ is a trivial fibration and that $i$ is a cofibration.
    \begin{enumerate}
      \item
      Then $p'$ is a trivial fibration.
      \item
      Additionally assume that $a$ is a weak equivalence, that $i'$ is a cofibration, that $X'$ is fibrant, and that $A$ and $A'$ are cofibrant.
      Then $p$ is a weak equivalence.
    \end{enumerate}
    \item
    Assume that $f$ is a weak equivalence, that $X'$ is fibrant, and that $(a, f) \colon i \to i'$ is a cofibration in the arrow category of $\cat M$ equipped with the projective model structure.
    \begin{enumerate}
      \item
      Additionally assume that $a$ is a weak-equivalence.
      Then $p$ is a trivial fibration.
      \item
      Additionally assume that $i$ is a cofibration and that $X$ is fibrant.
      Then $p'$ is a weak equivalence.
    \end{enumerate}
  \end{enumerate}
\end{lemma}

\begin{proof}
  We consider the pullback square
  \[
  \begin{tikzcd}
  \aut[a](f) \dar[swap]{p'} \rar{p} & \aut[A](X) \dar{f_*} \\
  \aut[A'](X') \rar{f^*} & \mapeq[A](X, X')
  \end{tikzcd}
  \]
  of simplicial sets.
  Parts~1 and 2.a) follow from \cref{lemma:rel_map,lemma:mapeq_we}.
  Part~2.b) follows by additionally using \cref{lemma:rel_forget}.
  Part~3.a) follows from \cref{lemma:rel_precompose,lemma:mapeq_we}; for part~3.b) we additionally use \cref{lemma:rel_map}.
\end{proof}

\begin{lemma} \label{lemma:rel_aut_naturality}
  In a simplicial model category, let
  \[
  \begin{tikzcd}
  A \rar{i} \dar[swap]{a} & X \rar{f} \dar{g} & Y \dar{h} \\
  A' \rar{i'} & X' \rar{f'} & Y'
  \end{tikzcd}
  \]
  be a commutative diagram.
  Then there is a commutative diagram of simplicial monoids
  \[
  \begin{tikzcd}
  \aut[A](X) & \aut[a](g) \lar \rar & \aut[A'](X') \\
  \aut[A](f) \uar \dar & \aut[a](f, f') \lar \rar \uar \dar & \aut[A'](f') \uar \dar \\ 
  \aut[A](Y) & \aut[a](h) \lar \rar & \aut[A'](Y')
  \end{tikzcd}
  \]
  where, for the middle column, we consider $(f, f') \colon g \to h$ as a map in $\arcat {\cat M}$ under $a$.
  
  Additionally assume that $f$ and $f'$ are trivial cofibrations, that $h$ is a trivial fibration, and that $Y$ and $Y'$ are fibrant.
  Then the three upper vertical maps are trivial fibrations.
\end{lemma}

\begin{proof}
  We prove the second part of the statement; the first part can be obtained in the same way by replacing $\mapeq$ with $\map$ everywhere and omitting any reference to fibrations and weak equivalences.
  
  We have the following commutative diagram
  \[
  \begin{tikzcd}
  \aut[A](X) \rar \dar & \mapeq[A](X, X') \dar & \aut[A'](X') \lar \dar \\
  \mapeq[A](X, Y) \rar & \mapeq[A](X, Y') & \mapeq[A'](X', Y') \lar \\
  \aut[A](Y) \rar \uar[two heads][swap]{\eq} & \mapeq[A](Y, Y') \uar[two heads][swap]{\eq} & \aut[A'](Y') \lar \uar[two heads][swap]{\eq}
  \end{tikzcd}
  \]
  where the indicated maps are trivial fibrations by \cref{lemma:rel_map}.
  Taking pullbacks of the horizontal rows, we thus obtain a commutative diagram
  \[
  \begin{tikzcd}
  \aut[A](X) \dar & \lar \aut[a](g) \dar \rar & \aut[A'](X') \dar \\
  \mapeq[A](X, Y) & \lar \mapeq[a](g, h) \rar & \mapeq[A'](X', Y') \\
  \aut[A](Y) \uar[two heads][swap]{\eq} & \lar \aut[a](h) \uar[two heads][swap]{\eq} \rar & \aut[A'](Y') \uar[two heads][swap]{\eq}
  \end{tikzcd}
  \]
  where the bottom middle vertical map is a trivial fibration by the dual of \cref{lemma:cofibration_pushouts}, using \cref{lemma:SM7} applied to the lower left-hand square of the previous diagram.
  The horizontal maps on the top and bottom are homomorphisms of simplicial monoids by \cref{lemma:monoid_pullback}.
  Now taking pullbacks of the vertical columns, we obtain the desired commutative diagram
  \[
  \begin{tikzcd}
  \aut[A](X) & \aut[a](g) \lar \rar & \aut[A'](X') \\
  \aut[A](f) \uar[two heads]{\eq} \dar & \aut[a](f, f') \lar \rar \uar[two heads]{\eq} \dar & \aut[A'](f') \uar[two heads]{\eq} \dar  \\
  \aut[A](Y) & \aut[a](h) \lar \rar & \aut[A'](Y') 
  \end{tikzcd}
  \]
  of simplicial monoids (again by \cref{lemma:monoid_pullback}).
\end{proof}

\begin{lemma} \label{lemma:zig-zag_composition}
  In a simplicial model category $\cat M$, let
  \[ \iota \colon A \to X  \qquad  \tilde \iota \colon \widetilde A \to \widetilde X  \qquad  \hat \iota \colon \widehat A \to \widehat X  \qquad  \iota' \colon A' \to X' \]
  be four cofibrations such that $A$, $\widetilde A$, $\widehat A$, and $A'$ are cofibrant and $\widetilde X$, $\widehat X$, and $X'$ are fibrant, and let
  \[
  \begin{tikzcd}[column sep = 40]
    \iota \rar{(\tilde a, \tilde f)} \dar[swap]{(\hat a, \hat f)} & \tilde \iota \dar{(\tilde a', \tilde f')} \\
    \hat \iota \rar{(\hat a', \hat f')} & \iota'
  \end{tikzcd}
  \]
  be a commutative diagram in the arrow category of $\cat M$ such that all four maps are pointwise weak equivalences.
  Assume that $(\tilde a, \tilde f)$ is a projective cofibration and that $(\tilde a', \tilde f')$ is a projective fibration.
  Furthermore assume that $(\hat a, \hat f)$ and $(\hat a', \hat f')$ are either both projective fibrations or both projective cofibrations, or that $(\hat a, \hat f)$ is a projective fibration and that $(\hat a', \hat f')$ is a projective cofibration.
  Then the two zig-zags
  \begin{gather*}
    \aut[A](X) \xlongfrom{\eq} \aut[\tilde a](\tilde f) \longto \aut[\widetilde A](\widetilde X) \xlongfrom{\eq} \aut[\tilde a'](\tilde f') \xlongto{\eq} \aut[A'](X') \\
    \aut[A](X) \xlongfrom{\eq} \aut[\hat a](\hat f) \longto \aut[\widehat A](\widehat X) \xlongfrom{\eq} \aut[\hat a'](\hat f') \xlongto{\eq} \aut[A'](X')
  \end{gather*}
  represent the same map in the homotopy category of simplicial monoids.
  (Note that all indicated maps are indeed weak equivalences by \cref{lemma:rel_aut_span}.)
\end{lemma}

\begin{proof}
  Consider the following two squares as objects of the simplicial category $\cat M^{\lincat 1 \times \lincat 1}$
  \[
  \begin{tikzcd}[column sep = 40]
    A \rar{\tilde a} \dar[swap]{\hat a} & \widetilde A \dar{\tilde a'} \\
    \widehat A \rar{\hat a'} & A'
  \end{tikzcd}
  \qquad \text{and} \qquad
  \begin{tikzcd}[column sep = 40]
    X \rar{\tilde f} \dar[swap]{\hat f} & \widetilde X \dar{\tilde f'} \\
    \widehat X \rar{\hat f'} & X'
  \end{tikzcd}
  \]
  and write $P$ for the simplicial monoid of pointwise self-equivalences of the second relative to the first.
  Then there is a commutative diagram of simplicial monoids
  \[
  \begin{tikzcd}
    \aut[A](X) & \lar[swap]{\eq} \aut[\tilde a](\tilde f) \rar & \aut[\widetilde A](\widetilde X) \\
    \aut[\hat a](\hat f) \uar{\eq} \dar[swap]{\hat q} & \lar[swap]{\hat p} P \rar \uar \dar{\hat p'} & \aut[\tilde a'](\tilde f') \dar{\eq} \uar[swap]{\eq} \\
    \aut[\widehat A](\widehat X) & \lar[swap]{\eq} \aut[\hat a'](\hat f') \rar{\eq} & \aut[A'](X')
  \end{tikzcd}
  \]
  and it is enough to prove that $\hat p$ is a weak equivalence.
  
  To this end, we note that there is a pullback square
  \[
  \begin{tikzcd}
    P \dar \rar & \aut[\widetilde A](\widetilde X) \dar{(\tilde f^*, \tilde f'_*)} \\
    \aut[(\hat a, \hat a')](\hat f, \hat f') \rar & \mapeq[A](X, \widetilde X) \times_{\mapeq[A](X, X')} \mapeq[\widetilde A](\widetilde X, X')
  \end{tikzcd}
  \]
  of simplicial sets, where $(\hat a, \hat a')$ and $(\hat f, \hat f')$ denote the respective two maps considered as an object of $\cat M^{\lincat 2}$.
  Our assumptions and \cref{lemma:rel_SM7} imply that the right-hand vertical map is a trivial fibration, and hence so is the left-hand vertical map.
  Now consider the pullback diagram
  \[
  \begin{tikzcd}
    \aut[(\hat a, \hat a')](\hat f, \hat f') \rar \dar & \aut[\hat a](\hat f) \dar{\hat q} \\
    \aut[\hat a'](\hat f') \rar{\hat q'} & \aut[\widehat A](\widehat X)
  \end{tikzcd}
  \]
  of simplicial monoids.
  By \cref{lemma:rel_aut_span} and our assumptions, at least one of $\hat q'$ and $\hat q$ is a trivial fibration.
  In the first case $\hat p$ is a weak equivalence; in the second case both $\hat q$ and $\hat p'$ are weak equivalences.
  In either case we are done.
\end{proof}

We are now ready to prove the main result of this appendix.

\begin{proof}[Proof of \cref{thm:Baut_functor}]
  Let the following be a commutative diagram in $\cat M$
  \[
  \begin{tikzcd}
  	B \rar{i} \dar[swap]{b} & A \rar{\iota} \dar[swap]{a} & X \dar{f} \rar & K \dar{k} \\
  	B' \rar{i'} & A' \rar{\iota'} & X' \rar & K'
  \end{tikzcd}
  \]
  such that $i$, $i'$, $\iota$, and $\iota'$ are cofibrations, $b$, $a$, $f$, and $k$ are weak equivalences, $B$, $B'$, $A$, and $A'$ are cofibrant, and $X$, $X'$, $K$, and $K'$ are fibrant.
  Furthermore, choose a factorization
  \begin{equation} \label{eq:Baut_functor_factorization}
  \begin{tikzcd}
    A \dar[swap]{\iota} \rar{\tilde a} & \widetilde A \dar{\tilde \iota} \rar{\tilde a'} & A' \dar{\iota'} \\
    X \rar{\tilde f} & \widetilde X \rar{\tilde f'} & X'
  \end{tikzcd}
  \end{equation}
  of $(a, f) \colon \iota \to \iota'$ into a trivial cofibration followed by a trivial fibration in the projective model structure on the arrow category of $\cat M$.
  Then we define the induced map
  \[ (b, a, f, k)_*  \colon  \bigl( \aut[A](X), \map[B](X, K) \bigr)  \xlonghto{\eq}  \bigl( \aut[A'](X'), \map[B'](X', K') \bigr) \]
  to be the zig-zag
  \[
  \begin{tikzcd}
    \vertpair {\aut[A](X)} {\map[B](X, K)} & & \vertpair {\aut[A'](X')} {\map[B'](X', K')} \\
    \vertpair {\aut[\tilde a](\tilde f)} {\map[B](\widetilde X, K)} \uar{\eq} \rar{\eq} & \vertpair {\aut[\widetilde A](\widetilde X)} {\map[B](\widetilde X, K')} & \lar{\eq} \vertpair {\aut[\tilde a'](\tilde f')} {\map[B'](X', K')} \uar{\eq}
  \end{tikzcd}
  \]
  considered as an isomorphism in the homotopy category of pairs.
  Note that $\tilde \iota$ is a cofibration by \cref{rem:projective_cofibration}, so that all of the maps above are indeed weak equivalences by \cref{lemma:rel_aut_span} for the spaces of self-equivalences and by \cref{lemma:rel_map,lemma:rel_forget} for the mapping spaces.
  
  We now begin by proving that this construction is independent of the choice of factorization \eqref{eq:Baut_functor_factorization}.
  To this end, let $\iota \to \tilde \iota \to \iota'$ and $\iota \to \hat \iota \to \iota'$ be two factorizations as above.
  They assemble into the outer commutative square of the following diagram
  \[
  \begin{tikzcd}
    \iota \rar[tail] \dar[tail] & \tilde \iota \dar[two heads] \\
    \hat \iota \rar[two heads] \urar[dashed] & \iota'
  \end{tikzcd}
  \]
  in the arrow category of $\cat M$.
  Since the left-hand and the right-hand vertical map are a trivial cofibration and a trivial fibration, respectively, in the projective model structure, we obtain a dashed map as indicated that makes both triangles commute.
  Since all four maps of the outer square are weak equivalences, so is the dashed map.
  Factoring it as a trivial cofibration followed by a trivial fibration, we obtain the following commutative diagram
  \[
  \begin{tikzcd}[row sep = 10]
    \iota \ar[tail]{rr} \ar[tail]{dd} \drar[tail, dashed] & & \tilde \iota \ar[two heads]{dd} \\
     & \bar \iota \urar[two heads] \drar[two heads, dashed] & \\
    \hat \iota \ar[two heads]{rr} \urar[tail] & & \iota'
  \end{tikzcd}
  \]
  where the two dashed maps are defined to be the composites of the appropriate two maps.
  All maps in the diagram are weak equivalences and the indicated maps are cofibrations and fibrations, respectively.
  Also note that $\tilde \iota$, $\hat \iota$, and $\bar \iota$ are all cofibrations by \cref{rem:projective_cofibration}.
  For $\aut[A](X) \hto \aut[A'](X')$ the claim then follows from a repeated application of \cref{lemma:zig-zag_composition}.
  For $\map[B](X, K) \hto \map[B'](X', K')$ we have the commutative diagram
  \[
  \begin{tikzcd}
  	 & \dlar[bend right, start anchor = west, end anchor = north][swap]{\eq} \map[B](\widetilde X, K) \dar[swap]{\eq} \rar{\eq} & \map[B](\widetilde X, K') \dar{\eq} & \\
  	\map[B](X, K) & \lar[swap]{\eq} \map[B](\ol X, K) \dar[swap]{\eq} \rar{\eq} & \map[B](\ol X, K') \dar{\eq} & \lar[swap]{\eq} \ular[bend right, start anchor = north, end anchor = east][swap]{\eq} \dlar[bend left, start anchor = south, end anchor = east]{\eq} \map[B'](X', K') \\
  	 & \ular[bend left, start anchor = west, end anchor = south]{\eq} \map[B](\widehat X, K) \rar{\eq} & \map[B](\widehat X, K') &
  \end{tikzcd}
  \]
  which is compatible with the actions of the monoids appearing above.
  
  Note that it is clear that the construction preserves identities.
  To see that it is compatible with composition, let $(i, \iota, \xi) \to (i', \iota', \xi') \to (i'', \iota'', \xi'')$ be two maps and consider the diagram
  \[
  \begin{tikzcd}
    \iota \ar{rr} \drar[tail] & & \iota' \ar{rr} \drar[tail][near end]{g} & & \iota'' \\
    & \tilde \iota \rar[tail] \urar[two heads][near start]{f} & \bar \iota \rar[two heads] & \hat \iota \urar[two heads] &
  \end{tikzcd}
  \]
  where we factored both maps in the top row as a trivial cofibration followed by a trivial fibration; in addition we factor $g \after f$ in the same way to obtain $\bar \iota$.
  The claim then follows from a repeated application of \cref{lemma:zig-zag_composition} (noting that $\tilde \iota$, $\hat \iota$, and $\bar \iota$ are all cofibrations by \cref{rem:projective_cofibration}) and the commutative diagram
  \[
  \begin{tikzcd}
  	\map[B](X, K) & \lar[swap]{\eq} \map[B](\widetilde X, K) \rar{\eq} & \map[B](\widetilde X, K') \dar[swap]{\eq} & \lar[swap]{\eq} \map[B'](X', K') \\
  	 & \ular{\eq} \map[B](\ol X, K) \drar[swap]{\eq} \uar[swap]{\eq} & \map[B](\widetilde X, K'') & \map[B'](\widehat X, K') \uar[swap]{\eq} \dar{\eq} \\
  	 & & \map[B](\ol X, K'') \uar{\eq} & \lar[swap]{\eq} \map[B'](\widehat X, K'') \\
  	 & & & \ular{\eq} \map[B''](X'', K'') \uar[swap]{\eq}
  \end{tikzcd}
  \]
  which is compatible with the actions of the monoids.
  
  Now let $(i, \iota, \xi)$ and $(i', \iota', \xi')$ be as above such that $B'$ and $A'$ are fibrant, and let $f$ and $g$ be two pointwise weak equivalences $(i, \iota, \xi) \to (i', \iota', \xi')$ such that they are simplicially homotopic in $\cat M^{\lincat 3}$.
  Then, by \cite[Proposition~9.5.23]{Hir}, the maps $f$ and $g$ are left homotopic in the category $\cat M^{\lincat 3}$ equipped with the projective model structure.
  Hence there are, by \cite[Proposition~7.3.4 and Lemma~7.3.6]{Hir} and the assumption that $(i', \iota', \xi')$ is fibrant, an object $\operatorname{Cyl}(i, \iota, \xi)$ and a commutative diagram in $\cat M^{\lincat 3}$
  \[
  \begin{tikzcd}[column sep = 40]
  & \dlar[bend right][swap]{\id} (i, \iota, \xi) \drar[bend left]{f} \dar[tail]{\inc[0]} & \\
  (i, \iota, \xi) & \lar[two heads][swap]{\pr} \operatorname{Cyl}(i, \iota, \xi) \rar & (i', \iota', \xi') \\
  & \ular[bend left]{\id} (i, \iota, \xi) \urar[bend right][swap]{g} \uar[tail][swap]{\inc[1]} &
  \end{tikzcd}
  \]
  such that all maps are weak equivalences and the indicated ones are projective (co)fi\-bra\-tions.
  Note that the $i$- and $\iota$-components of $\operatorname{Cyl}(i, \iota, \xi)$ are cofibrations by \cref{rem:projective_cofibration}; then the claim follows from the functoriality proven above.
\end{proof}

In the rest of this subsection, we record a number of useful properties that the functor of \cref{thm:Baut_functor} enjoys.

\begin{lemma} \label{lemma:Eaut_map}
  Let the following be a commutative diagram in a simplicial model category $\cat M$
  \[
  \begin{tikzcd}
    A \rar{a} \dar[swap]{\iota} & A' \dar{\iota'} \\
    X \rar{f} & X'
  \end{tikzcd}
  \]
  such that $\iota$ and $\iota'$ are cofibrations, $a$ and $f$ are weak equivalences, $A$ and $A'$ are cofibrant, and $X$ and $X'$ are fibrant.
  Then the two maps of sets
  \[ \Eaut[A](X)  \xlongto{f_*}  \hmapeq[A] X {X'}  \xlongfrom{f^*}  \Eaut[A'](X') \]
  are bijections, and the composite $\Phi_f \defeq \inv{(f^*)} \after f_*$ agrees with the map induced by the functor of \cref{thm:Baut_functor} on $\hg 0$.
  In particular $\Phi_f$ is a group isomorphism.
\end{lemma}

\begin{proof}
  That $f_*$ and $f^*$ are bijections follows from \cref{lemma:rel_map,lemma:mapeq_we,lemma:rel_forget}.
  The second claim follows from the fact that, given maps $\iota \to \tilde \iota \to \iota'$ as in the proof of \cref{thm:Baut_functor}, the following diagram of sets commutes
  \[
  \begin{tikzcd}
  \Eaut[\tilde a](\tilde f) \dar \rar & \Eaut[\widetilde A](\widetilde X) \dar & \lar \Eaut[\tilde a'](\tilde f') \dar \\
  \Eaut[A](X) \rar{f_*} & \hmapeq[A] X {X'} & \lar[swap]{f^*} \Eaut[A'](X')
  \end{tikzcd}
  \]
  where the middle vertical map is given by $\phi \mapsto \tilde f' \after \phi \after \tilde f$.
\end{proof}

\begin{lemma} \label{lemma:zig-zag_forget}
  Let the following be a natural weak equivalence between commutative diagrams in a simplicial model category $\cat M$
  \[
  \begin{tikzcd}
  	B \rar[tail] & A \rar[tail] & X \rar & K \\
  	\widehat B \rar[tail] \uar[tail] & \widehat A \uar[tail] &
  \end{tikzcd}
  \qquad \xLongto[\phi]{\eq} \qquad
  \begin{tikzcd}
  	B' \rar[tail] & A' \rar[tail] & X' \rar & K' \\
  	\widehat B' \rar[tail] \uar[tail] & \widehat A' \uar[tail] &
  \end{tikzcd}
  \]
  such that the indicated maps are cofibrations, $X$, $K$, $X'$, and $K'$ are fibrant, and all other objects are cofibrant.
  Then the following diagram commutes in the homotopy category of pairs of a simplicial monoid and a right module over it
  \[
  \begin{tikzcd}
    \vertpair {\aut[A](X)} {\map[B](X, K)} \rar[squiggly]{\eq} \dar & \vertpair {\aut[A'](X')} {\map[B'](X', K')} \dar \\
    \vertpair {\aut[\widehat A](X)} {\map[\widehat B](X, K)} \rar[squiggly]{\eq} & \vertpair {\aut[\widehat A'](X')} {\map[\widehat B'](X', K')}
  \end{tikzcd}
  \]
  where the horizontal maps are those of \cref{thm:Baut_functor}.
\end{lemma}

\begin{proof}
  Consider $\phi$ as a map in the appropriate functor category equipped with the projective model structure, and factor it as a trivial cofibration followed by a trivial fibration.
  By \cref{lemma:projective_restriction} this factorization still consists of a projective cofibration followed by a pointwise fibration when restricted to either $B \to A \to X \to K$ or $\widehat B \to \widehat A \to X \to K$.
  Using these factorizations, we obtain a map between the associated zig-zags, implying the claim.
\end{proof}

Below, we will prove that the functor of \cref{thm:Baut_functor} is, in a certain sense, compatible with taking pushouts.
To be able to make this precise, we need the following lemma.

%

\begin{lemma} \label{lemma:aut_pushout_zig-zag}
  In a simplicial model category $\cat M$, let the following be two commutative diagrams
  \[
  \begin{tikzcd}[sep = small]
    C \rar \dar & B \rar[tail] \dar & A \rar[tail] \ar{dd} & X \ar{ddd} \\
    B' \rar \dar[tail] & Q \drar[tail] & & \\
    A' \ar{rr} \dar[tail] & & R \drar[tail] & \\
    X' \ar{rrr} & & & P
  \end{tikzcd}
  \qquad
  \begin{tikzcd}
  	X \cop_C X' \rar \dar[swap]{p} & K \dar{l} \\
  	P \rar & L
  \end{tikzcd}
  \]
  such that the indicated maps are cofibrations and $P$, $K$, and $L$ are fibrant.
  Furthermore assume that the induced maps $q \colon B \cop_C B' \to Q$, $r \colon A \cop_C A' \to R$, and $p \colon X \cop_C X' \to P$ are trivial cofibrations.
  Then there is a zig-zag of pairs of a simplicial monoid and a right module over it
  \[
  \begin{tikzcd}[row sep = 15, column sep = 15]
    \vertpair {\aut[A](X) \times \aut[A'](X')} {\map[B](X, K) \times \map[B'](X', K)} \dar & & \vertpair {\aut[R](P)} {\map[Q](P, L)} \dar{\eq} \\
    \vertpair {\aut[A \cop_C A'](X \cop_C X')} {\map[B \cop_C B'](X \cop_C X', K)} & \lar[swap]{\eq} \vertpair {\aut[A \cop_C A'](p)} {\map[B \cop_C B'](p, l)} \rar & \vertpair {\aut[A \cop_C A'](P)} {\map[B \cop_C B'](P, L)}
  \end{tikzcd}
  \]
  such that the indicated maps are weak equivalences.
  In particular, this defines a map in the homotopy category of such pairs.
\end{lemma}

\begin{proof}
  That the left-hand horizontal map is a weak equivalence follows from part~1 of \cref{lemma:rel_aut_span} and the same argument applied to mapping spaces.
  That the right-hand vertical map is a weak equivalence follows from \cref{lemma:rel_forget}.
\end{proof}

\begin{lemma} \label{lemma:zig-zag_pushouts}
  In a simplicial model category $\cat M$, let the following be natural pointwise weak equivalences
  \[
  \begin{tikzcd}[sep = small]
  	C \rar[tail] \dar[tail] & B_1 \rar[tail] \dar & A_1 \rar[tail] \ar{dd} & X_1 \ar{ddd} \\
  	B_2 \rar \dar[tail] & Q \drar[tail] & & \\
  	A_2 \ar{rr} \dar[tail] & & R \drar[tail] & \\
  	X_2 \ar{rrr} & & & P
  \end{tikzcd}
  \quad \xLongrightarrow{\eq} \quad
  \begin{tikzcd}[sep = small]
  	C' \rar[tail] \dar[tail] & B_1' \rar[tail] \dar & A_1' \rar[tail] \ar{dd} & X_1' \ar{ddd} \\
  	B_2' \rar \dar[tail] & Q' \drar[tail] & & \\
  	A_2' \ar{rr} \dar[tail] & & R' \drar[tail] & \\
  	X_2' \ar{rrr} & & & P'
  \end{tikzcd}
  \]
  \[
  \begin{tikzcd}
  	X_1 \cop_C X_2 \rar \dar & K \dar \\
    P \rar & L
  \end{tikzcd}
  \quad \xLongrightarrow{\eq} \quad
  \begin{tikzcd}
  	X_1' \cop_C X_2' \rar \dar & K' \dar \\
  	P' \rar & L'
  \end{tikzcd}
  \]
  of diagrams as in \cref{lemma:aut_pushout_zig-zag} (such that the lower natural transformation restricted to the left-hand side of the square is induced by the upper map).
  Additionally assume that $C$ and $C'$ are cofibrant, that $X_1$, $X_2$, $X_1'$, and $X_2'$ are fibrant, and that the indicated maps from $C$ and $C'$ are cofibrations.
  Then the following diagram commutes in the homotopy category of pairs of a simplicial monoid and a right module over it
  \[
  \begin{tikzcd}
    \vertpair {\aut[A_1](X_1) \times \aut[A_2](X_2)} {\map[B_1](X_1, K) \times \map[B_2](X_2, K)} \rar[squiggly] \dar[squiggly][swap]{\eq} & \vertpair {\aut[R](P)} {\map[Q](P, L)} \dar[squiggly]{\eq} \\
    \vertpair {\aut[A_1'](X_1') \times \aut[A_2'](X_2')} {\map[B_1'](X_1', K') \times \map[B_2'](X_2', K')} \rar[squiggly] & \vertpair {\aut[R'](P')} {\map[Q'](P', L')}
  \end{tikzcd}
  \]
  where the horizontal maps are those of \cref{lemma:aut_pushout_zig-zag} and the vertical maps are (products of) those of \cref{thm:Baut_functor}.
\end{lemma}

\begin{proof}
	We prove the statement for the simplicial monoids; incorporating the right modules proceeds by a similar argument.
  We begin by factoring the natural transformation as a trivial cofibration $\kappa$ followed by a trivial fibration $\phi$ in the category of functors from the appropriate diagram category $\cat I$ to $\cat M$, equipped with the projective model structure.
  We denote the objects of the middle diagram by $\widehat C$ etc., and write $M \defeq X_1 \cop_C X_2$, $N \defeq A_1 \cop_C A_2$, and analogously for $M'$, $N'$, $\widehat M$, and $\widehat N$.
  This yields the following commutative diagram
  \[
  \begin{tikzcd}[row sep = 10]
    N \ar{rrrrr} \ar{ddddd}[swap]{g} \drar[swap]{n} & & & & & M \ar{ddddd}{f} \dlar{m} \\
    & \widehat N \ar{rrr} \ar{ddd}[swap]{\hat g} \drar[swap]{n'} & & & \widehat M \ar{ddd}{\hat f} \dlar{m'} & \\
    & & N' \rar \dar[swap]{g'} & M' \dar{f'} & & \\[10]
    & & R' \rar{\iota'} & P' & & \\
    & \widehat R \ar{rrr}{\hat \iota} \urar{r'} & & & \widehat P \ular[swap]{p'} & \\
    R \ar{rrrrr}{\iota} \urar{r} & & & & & P \ular[swap]{p}
  \end{tikzcd}
  \]
  in $\cat M$.
  We note the following:
  \begin{enumerate}
    \item
    The maps $p'$ and $r'$ are trivial fibrations by construction.
    \item
    The object $\widehat P$ is fibrant by the preceding item.
    \item
    The maps $p$ and $r$ are trivial cofibrations by \cref{lemma:projective_cofibration}.
    \item
    The maps $f$, $g$, $f'$, and $g'$ are trivial cofibrations by assumption.
    \item
    The maps $(m, p) \colon f \to \hat f$ and $(n, r) \colon g \to \hat g$ are trivial cofibrations in the projective model structure on the arrow category.
    This is by \cref{lemma:projective_restriction_pushout}, combined with \cref{lemma:projective_cofibration} for the former map and \cref{lemma:projective_restriction} for the latter (note that the appropriate inclusion $\lincat 1 \times \lincat 1 \to \cat I$ is right adjoint).
    \item
    The maps $(f, \hat f) \colon m \to p$ and $(g, \hat g) \colon n \to r$ are trivial cofibrations in the projective model structure on the arrow category by the preceding two items and \cref{rem:projective_cofibration}.
    \item
    The maps $\hat f$ and $\hat g$ are trivial cofibrations by the preceding item and \cref{lemma:projective_cofibration}.
    \item
    The map $n'$ is a weak equivalence since $\hat g$, $r'$, and $g'$ are.
    \item
    The maps $\iota$ and $\iota'$ are cofibrations by assumption.
    \item
    The map $(r, p) \colon \iota \to \hat \iota$ is a cofibration in the projective model structure on the arrow category by \cref{lemma:projective_restriction} (since the corresponding inclusion $\lincat 1 \to \cat I$ is right adjoint).
    \item
    The map $(\iota, \hat \iota) \colon r \to p$ is a cofibration in the projective model structure on the arrow category by the preceding two items and \cref{rem:projective_cofibration}.
    \item
    The map $\hat \iota$ is a cofibration by the preceding item and \cref{lemma:projective_cofibration}.
    \item
    The objects $N$ and $N'$ are cofibrant since $A$ and $A'$ are, and hence $\widehat N$ is cofibrant by item 5.
    \item
    The objects $R$ and $R'$ are cofibrant by items~4 and 13.
  \end{enumerate}
  Now \cref{lemma:rel_aut_naturality} yields the upper three rows of the following commutative diagram
  \[
  \begin{tikzcd}[sep = 15]
    \aut[N](M) & \aut[n](m) \lar \rar & \aut[\widehat N](\widehat M) & \aut[n'](m') \lar \rar & \aut[N'](M') \\
    \aut[N](f) \uar{\eq} \dar & \aut[n](f, \hat f) \lar \rar \uar{\eq} \dar & \aut[\widehat N](\hat f) \uar{\eq} \dar & \aut[n'](\hat f, f') \lar \rar \uar{\eq} \dar & \aut[N'](f') \uar{\eq} \dar \\ 
    \aut[N](P) & \aut[n](p) \lar \rar & \aut[\widehat N](\widehat P) & \aut[n'](p') \lar[swap]{\eq} \rar{\eq} & \aut[N'] \bigl( P' \bigr) \\
    \aut[R](P) \uar{\eq} & \aut[r](p) \lar[swap]{\eq} \rar{\eq} \uar{\eq} & \aut[\widehat R](\widehat P) \uar{\eq} & \aut[r'](p') \lar[swap]{\eq} \rar{\eq} \uar & \aut[R'] \bigl( P' \bigr) \uar
  \end{tikzcd}
  \]
  of simplicial monoids.
  By items~1, 9, 10, and 14 above, the bottom row is a zig-zag as in the proof of \cref{thm:Baut_functor}; in particular these four maps are weak equivalences.
  The other marked maps are weak equivalences by the following arguments:
  \begin{itemize}
    \item
    The three top right-hand vertical maps are weak equivalences by \cref{lemma:rel_aut_naturality} and items~1, 2, 4, and 7 above.
    \item
    The two top left-hand vertical maps are weak equivalences by part~1 of \cref{lemma:rel_aut_span} and items~2, 4, and 6 above.
    \item
    The three marked bottom vertical maps are weak equivalences by \cref{lemma:rel_forget} and items~2, 4, 6, 7, 9, 11, and 12 above.
    \item
    The two marked horizontal maps in the third row are weak equivalences by part~2 of \cref{lemma:rel_aut_span} and items~1, 4, 7, 8, 9, 12, and 13 above.
  \end{itemize}
  
  Lastly, we note that there is a commutative diagram of simplicial monoids
  \[
  \begin{tikzcd}[row sep = 15]
    \aut[B](X) \times \aut[C](Y) \rar & \aut[N](M) \\
    \aut[\kappa_B](\kappa_X) \times \aut[\kappa_C](\kappa_Y) \uar{\eq} \dar[swap]{\eq} \rar & \aut[n](m) \uar \dar \\
    \aut[\widehat B](\widehat X) \times \aut[\widehat C](\widehat Y) \rar & \aut[\widehat N](\widehat M) \\
    \aut[\phi_B](\phi_X) \times \aut[\phi_C](\phi_Y) \uar{\eq} \dar[swap]{\eq} \rar & \aut[n'](m') \uar \dar \\
    \aut[B'] (X') \times \aut[C'](Y') \rar & \aut[N'](M')
  \end{tikzcd}
  \]
  where the horizontal maps are given by taking pushouts of maps, and $\kappa$ and $\phi$ are the natural transformations from above.
  We note that $(\kappa_X, \kappa_B)$ and $(\kappa_Y, \kappa_C)$ are cofibrations in the projective model structure on the arrow category of $\cat M$ by \cref{lemma:projective_cofibration}.
  In particular the two factors of the left-hand column are zig-zags as in the proof of \cref{thm:Baut_functor}.
  Pasting the preceding two diagrams completes the proof.
\end{proof}

\end{appendices}

\printbibliography

\end{document}